\documentclass{amsart}
\usepackage{amsfonts, amsmath, amssymb, graphicx}
\usepackage{enumerate, float}
\usepackage{color}

\newtheorem{theorem}{Theorem}
\theoremstyle{definition}

\newtheorem{definition}{Definition}
\newtheorem{example}{Example}
\newtheorem{lemma}{Lemma}

\newtheorem{remark}{Remark}

\numberwithin{equation}{section}

\begin{document}
\allowdisplaybreaks

\title{Alexander- and Markov-type theorems for virtual trivalent braids}

\author{Carmen Caprau}
\address{Department of Mathematics, California State University, Fresno, CA 93740, USA}
\email{ccaprau@mail.fresnostate.edu}

\author{Abigayle Dirdak}
\address{Department of Mathematics, California State University, Fresno, CA 93740, USA}
\email{abigayledirdak@mail.fresnostate.edu}

\author{Rita Post}
\address{Department of Mathematics, University of Wisconsin-Eau Claire, WI 54768, USA}
\email{postrl@uwec.edu}

\author{Erica Sawyer}
\address{Department of Mathematics, The Evergreen State College, Olympia, WA 98516, USA}
\email{ericas1026@gmail.com}

\date{}
\thanks{This work was supported by NSF Grant DMS-1460151, NSF Grant HRD--1302873, and Simons Foundation collaboration grant $\#$ 355640}

\begin{abstract}
We prove Alexander- and Markov-type theorems for virtual spatial trivalent graphs and virtual trivalent braids. We provide two versions for the Markov-type theorem: one uses an algebraic approach similar to the case of classical braids and the other one is based on $L$-moves.
\end{abstract}
\maketitle

\section{Introduction}\label{sec:intro}

The Alexander theorem~\cite{A} and the Markov theorem~\cite{M} are essential in classical braid theory for understanding the relationship between classical braids and knots or links. The Alexander theorem states that every oriented knot or link can be represented as the closure of a braid. The Markov theorem states that two braids yield equivalent knots or links upon the closure operation if and only if the braids are related by braid isotopy and/or a finite sequence of two moves, namely stabilization and conjugation. These two moves for classical braids are now called the Markov moves. In~\cite{L}, S. Lambropoulou provided a one-move Markov-type theorem for classical braids and links by introducing the (classical) L-moves. We also refer the reader to~\cite{LambRour}, where the $L$-move equivalence for classical braids was established. 

Analogous theorems for virtual braids and virtual knots or links were established by S. Kamada~\cite{K} via Gauss data, and by L. Kauffman and S. Lambropoulou in~\cite{KauLamb}, where they introduced the virtual $L$-moves (or $L_v$-moves). Moreover, the $L_v$-moves were extended to virtual singular braids in~\cite{CPM} and were used to provide Markov-type theorems for this class of braids. To our knowledge, the Alexander theorem for oriented spatial graphs was first proved by K. Kanno and K. Taniyama~\cite{KaTa} (also see~\cite{Is}). The $L$-moves for classical braids extend naturally to trivalent braids, as shown in~\cite{CCD}, where the authors prove Alexander and Markov theorems for trivalent braids and oriented spatial trivalent graphs whose vertices are neither sources nor sinks.

In this paper we consider oriented virtual spatial trivalent graphs and virtual trivalent braids and prove Alexander- and Markov-type theorems for this setting. Our approach makes use of the $L_v$-equivalence defined in~\cite{KauLamb}. In proving analogous theorems for virtual spatial trivalent graphs and braids by means of $L_v$-moves, we must consider the topological properties of virtual spatial trivalent graphs. The set of Reidemeister-type moves for virtual spatial trivalent graph diagrams is richer than the corresponding set of moves from classical or virtual knot theory, as well as that for spatial graph theory (see~\cite{Ka1, Ka2, TM}).

We borrow the braiding algorithm described in~\cite{KauLamb} and extend it to account for trivalent vertices. We use this braiding algorithm to prove the Alexander-type theorem for oriented virtual spatial trivalent knots and links. We then show that in extending the set of $L_v$-moves for virtual braids to the class of virtual trivalent braids, we need to introduce new moves involving trivalent vertices; we refer to these new moves as  trivalent $L_v$-moves (or $TL_v$-moves). 

We define the $TL_v$-equivalence and use it to prove our first Markov-type theorem for virtual trivalent braids. Then through proving a one-to-one correspondence between the $TL_v$-equivalence and a certain algebraic equivalence among virtual trivalent braids, we are able to provide an algebraic Markov-type theorem for this type of braids. The latter theorem is similar in spirit to the original Markov theorem~\cite{M}.

The paper is organized as follows: in Section~\ref{sec:VSTG} we provide a brief review about virtual spatial trivalent graphs. We start Section~\ref{sec:VTB} by introducing virtual trivalent braids, and then we proceed to explain the preparation for braiding and the braiding process for oriented diagrams of virtual spatial trivalent graphs. This braiding algorithm braids any such diagram, and therefore, it provides a proof for the Alexander-type theorem for oriented virtual spatial trivalent graphs. In Section~\ref{ssec:TL_v} we establish the $TL_v$-equivalence for virtual trivalent braids and use it in Section~\ref{ssec:MarkovThm1} to prove our first Markov-type theorem. Finally, in Section~\ref{ssec:MarkovThm2} we state and prove an algebraic Markov-type theorem for virtual trivalent braids.


\section{Virtual spatial trivalent graphs}\label{sec:VSTG}

A \textit{virtual spatial trivalent graph diagram} is a trivalent graph immersed into $\mathbb{R}^2$ with finitely many transverse double points, each of which has information of over/under or virtual crossings as indicated in Fig.~\ref{fig:crossings}. The over/under crossings are also known as classical crossings. Virtual crossings are represented by placing a small circle around the point where the two arcs meet transversely.

\begin{figure}[ht]
\[ \raisebox{0cm}{\includegraphics[height=0.5in]{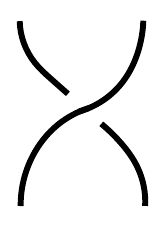}}  \hspace{.5in} \raisebox{0cm}{\includegraphics[height=0.5in]{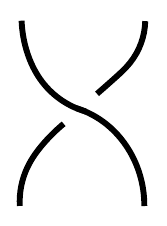}} \hspace{.5in} \raisebox{-.1cm}{\includegraphics[height=0.5in]{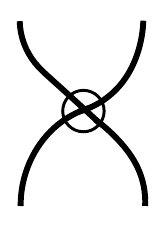}} \]
\caption{Types of Crossings} \label{fig:crossings}
\end{figure}

For brevity, we will refer to a virtual spatial trivalent graph diagram as a virtual STG diagram. Two virtual STG diagrams are called \textit{equivalent} (or \textit{ambient isotopic}) if they are related via a finite sequence of local diagrammatic moves shown in Fig.~\ref{fig:R-moves} (where all possible types of crossings need to be considered) along with planar isotopy. We will collectively refer to these moves as the \textit{extended Reidemeister moves}. For more details on spatial graphs and their topological properties, we refer the reader to~\cite{TM}.

It is easy to see that the extended Reidemeister moves (together with planar isotopy) introduce an equivalence relation on the set of virtual STG diagrams. A \textit{virtual spatial trivalent graph} (virtual STG) is then the \textit{equivalence class} of a virtual STG diagram.

\begin{figure}[ht] 
\begin{center}
  \begin{tabular}{ |l | l | l| }
    \hline
&\\[-5pt]
${\hspace{0.45in}\raisebox{-14pt}{\includegraphics[height=0.4in]{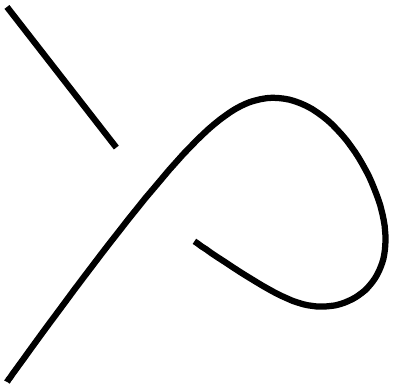}}}\,\, \stackrel{R1}\longleftrightarrow\,\, \raisebox{-13pt}{\includegraphics[height=0.4in]{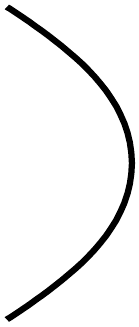}}\,\, \stackrel{R1}\longleftrightarrow\,\,\reflectbox{\raisebox{14pt}{{{\includegraphics[height=0.4in, angle = 180]{poskink}}}}}$

&
${\hspace{0.3in}\raisebox{-13pt}{\includegraphics[height=0.4in]{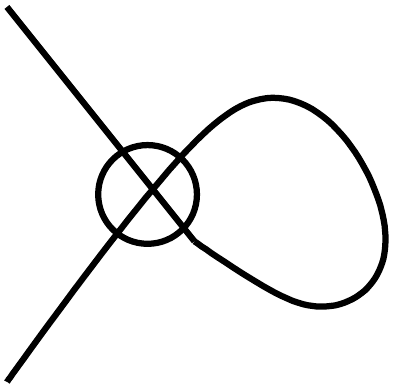}}\,\,  \stackrel{V1}\longleftrightarrow \,\, \raisebox{-13pt}{\includegraphics[height=0.4in]{arc}}}$
\\[20pt]
${\hspace{0.825in}\raisebox{-12pt}{\includegraphics[height=0.4in]{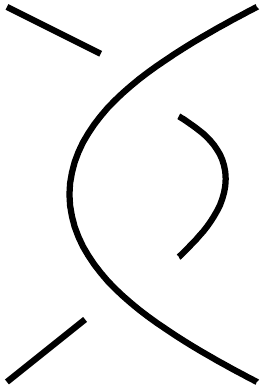}}}\, \stackrel{R2}{\longleftrightarrow}\, \raisebox{-12pt}{\includegraphics[height=0.4in]{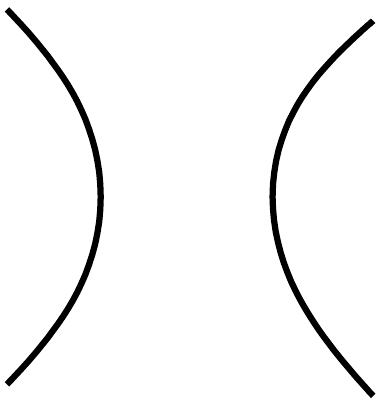}}$
&
${\hspace{0.35in}\raisebox{-13pt}{\includegraphics[height=0.4in]{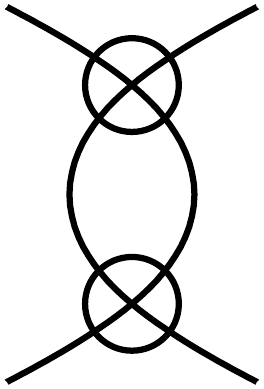}} \,\, \stackrel{V2}{\longleftrightarrow} \,\, \raisebox{-13pt}{\includegraphics[height=0.4in]{A-smoothing}}}$
\\[20pt]
${\hspace{0.55in}\raisebox{-12pt}{\includegraphics[height=0.4in]{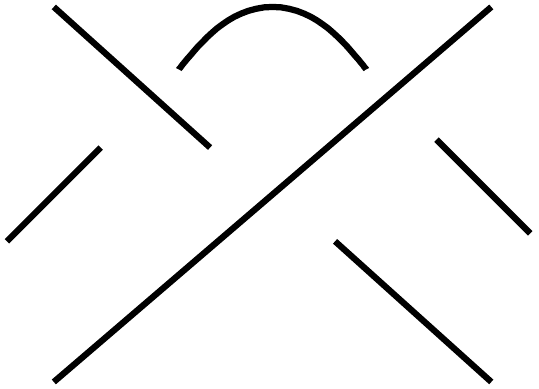}}} \,\,\stackrel{R3}\longleftrightarrow \,\,\raisebox{-12pt}{\includegraphics[height=0.4in]{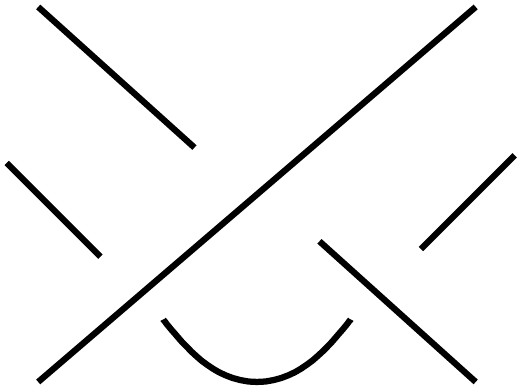}}$
&
${\hspace{0.2cm}\raisebox{-13pt}{\includegraphics[height=0.4in]{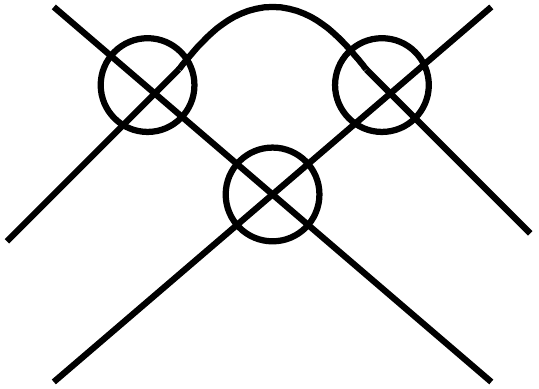}} \,\, \stackrel{V3}{\longleftrightarrow} \,\,   \raisebox{15pt}{\includegraphics[height=0.4in, angle=180]{reid3-virt}}}$
\\[20pt]
${\hspace{0.5in}\raisebox{-13pt}{\includegraphics[height=0.4in]{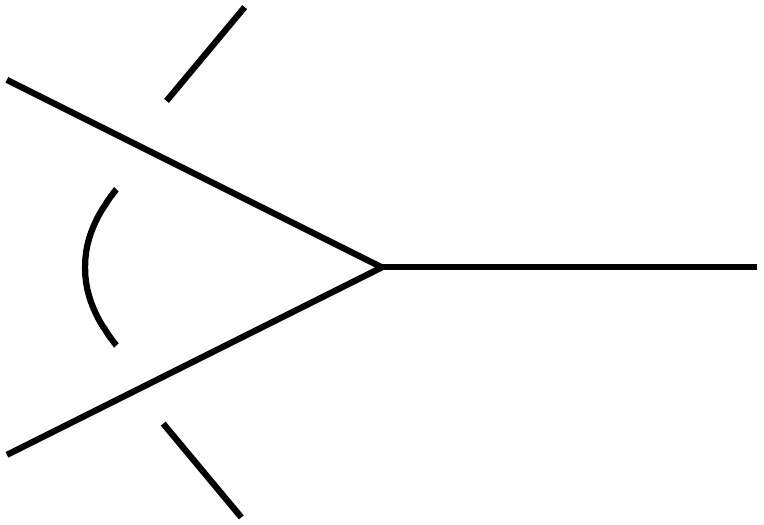}}} \,\, \, \stackrel{R4}\longleftrightarrow \,\,\raisebox{-12pt}{\includegraphics[height=0.4in]{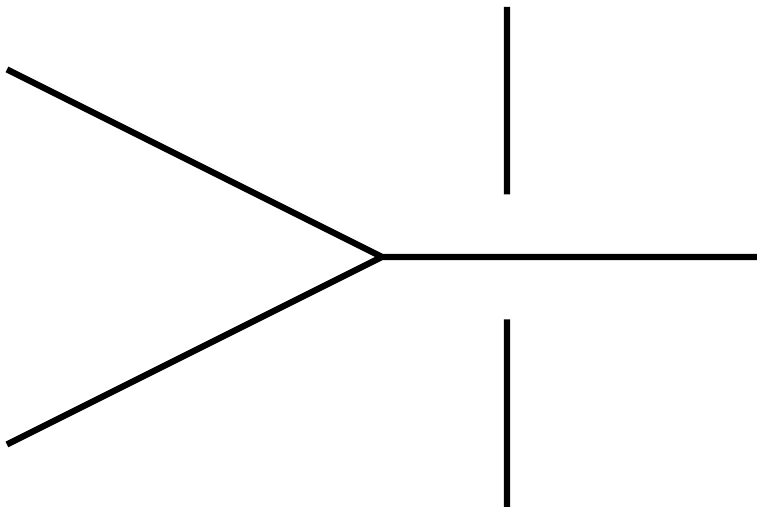}}$
&
${\hspace{0.2cm}\raisebox{-15pt}{\includegraphics[height=0.4in]{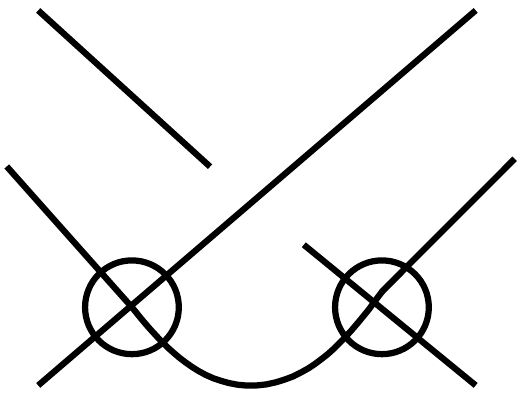}} \hspace{0.2cm}\stackrel{VR3}{\longleftrightarrow} \hspace{0.2cm} \raisebox{15pt}{\includegraphics[height=0.4in, angle=180]{VR3}}}$
\\[20pt]
${\raisebox{-10pt}{\includegraphics[height=0.35in]{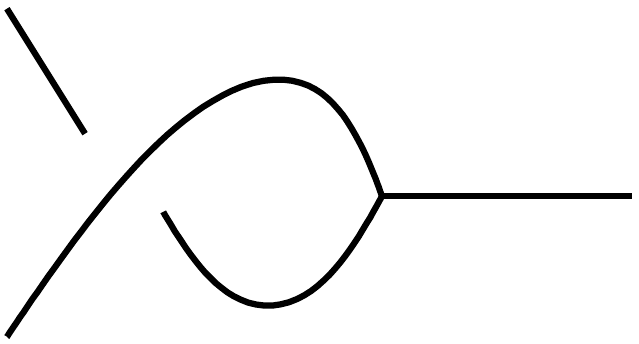}}} \,\,\stackrel{R5}\longleftrightarrow \,\, \raisebox{-10pt}{\includegraphics[height=0.35in]{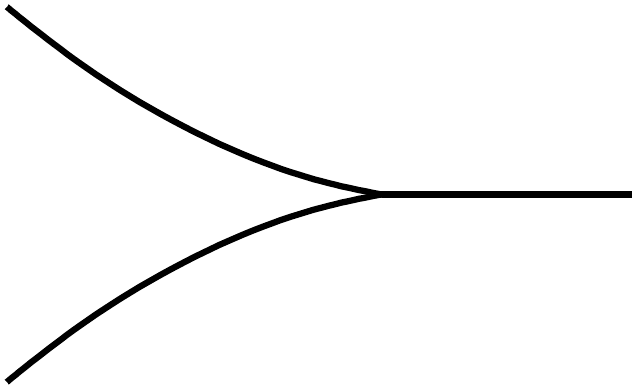}} \,\,\stackrel{R5}\longleftrightarrow \,\,{\raisebox{-12pt}{{{\includegraphics[height=0.35in]{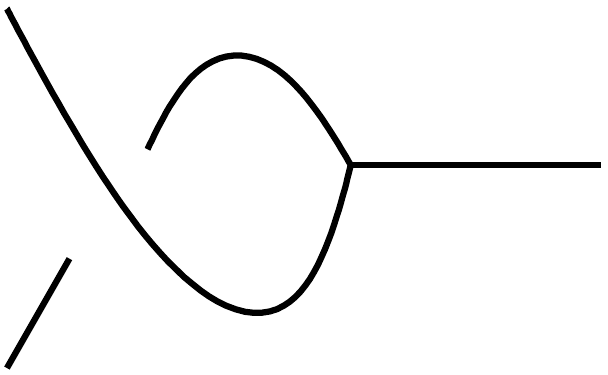}}}}}$
&
$\hspace{0.2cm}{{\raisebox{-15pt}{\includegraphics[height=0.38in]{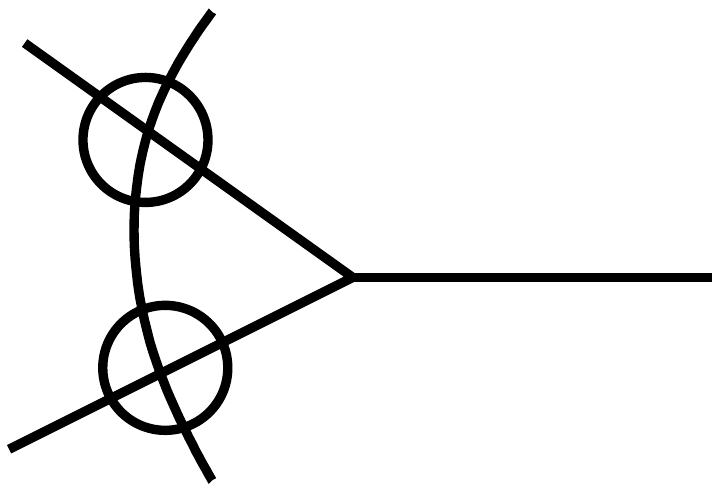}}} \,\,\stackrel{V4}\longleftrightarrow \,\, \raisebox{-15pt}{\includegraphics[height=0.38in]{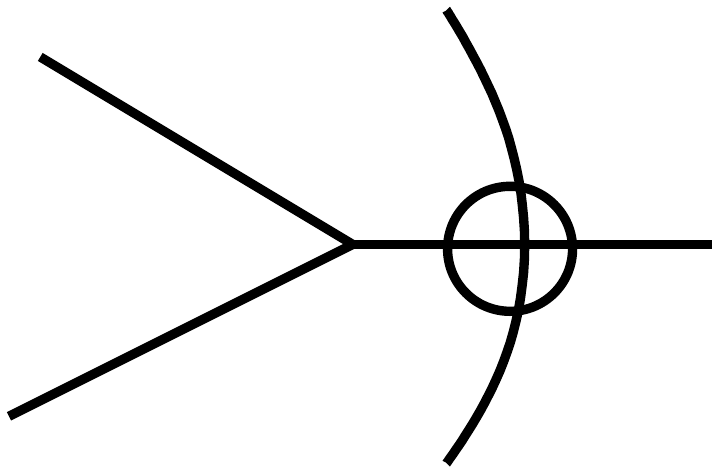}}}$
\\[-10pt] &
          \\ \hline
\end{tabular}
\caption{Extended Reidemeister moves for virtual STG diagrams} \label{fig:R-moves}
\end{center}
\end{figure}

The extended Reidemeister moves involving virtual crossings can be regarded as particular cases of the \textit{detour move} depicted in Fig.~\ref{fig:detour-move}. This move can be applied to any virtual STG diagram without changing its equivalence class, and it consists of taking an arc in a diagram which intersects other edges of the graph only virtually and redrawing that arc to any new location (keeping the endpoints fixed) such that there are virtual crossings everywhere the arc now transversally crosses the diagram (see Fig.~\ref{fig:detour-move}).  Note that the gray box represents any portion of a virtual STG diagram.

\begin{figure}[ht]
\[ \raisebox{-25pt}{\includegraphics[height=0.8in]{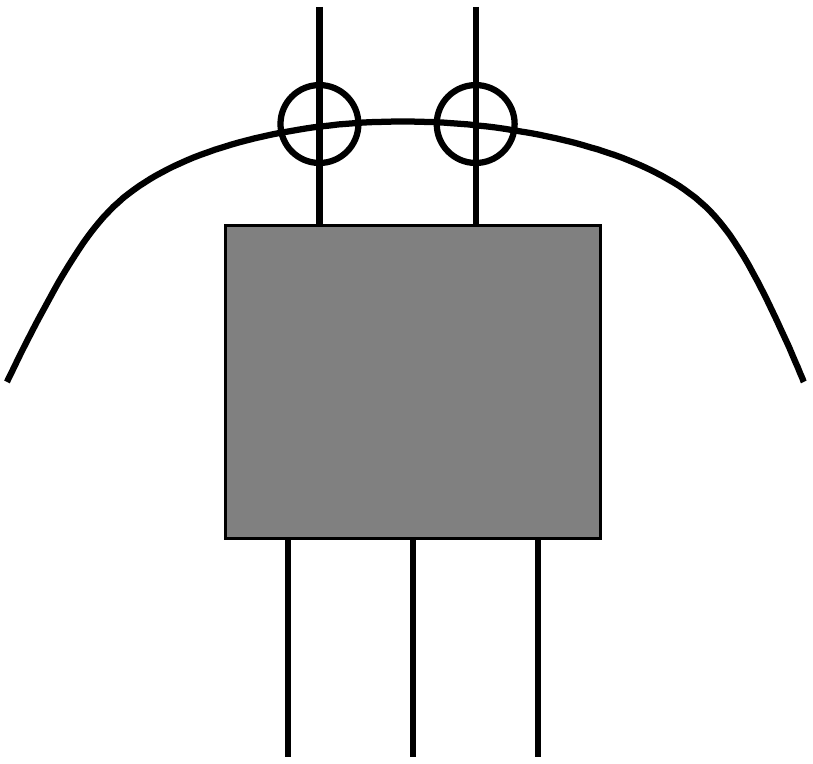}} \hspace{0.3cm}\longleftrightarrow \hspace{0.3cm} \raisebox{-25pt}{\includegraphics[height=0.8in]{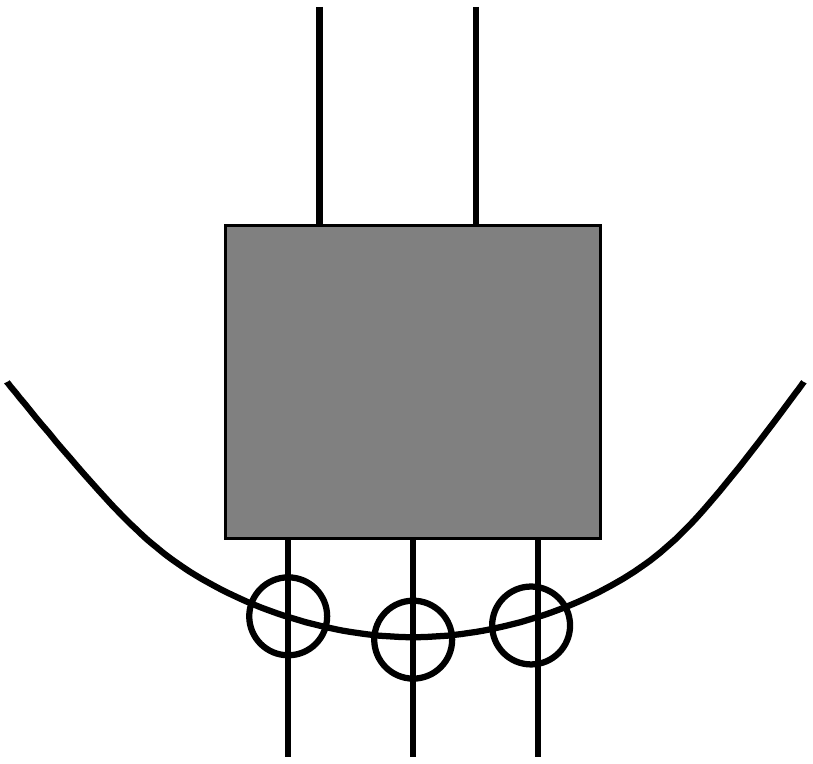}}
\]
\caption{The detour move}\label{fig:detour-move}
\end{figure}

Conversely, the detour move can be derived by applying a finite sequence of the extended Reidemeister moves involving virtual crossings. For more information on the detour move as it applies to virtual knot theory, see~\cite{Ka2}.

We remark that there is a collection of moves resembling the previous moves, which in fact do not represent equivalent virtual STG diagrams;  see Fig.~\ref{fig:bad-moves}. For this reason, these move are called the \textit{forbidden moves} for virtual STG diagrams.   

\begin{figure}[ht]
\[{\raisebox{-17pt}{\includegraphics[height=0.4in, angle=90]{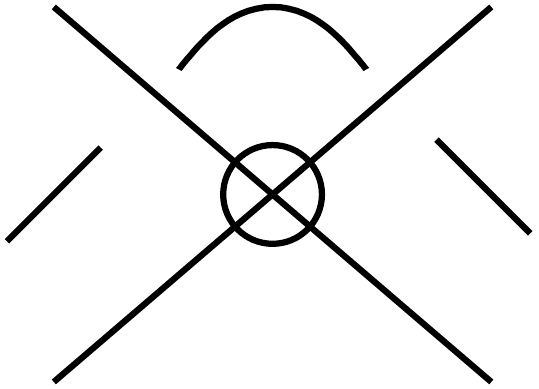}}} \hspace{0.2in} \stackrel{F1}\nleftrightarrow \hspace{0.2in} \raisebox{23pt}{\includegraphics[height=0.4in, angle=270]{fmove1}} \hspace{1.5cm}{\raisebox{-17pt}{\includegraphics[height=0.4in, angle=90]{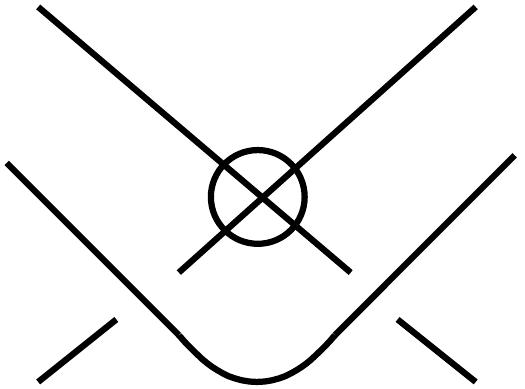}}} \hspace{0.2in} \stackrel{F2}\nleftrightarrow \hspace{0.2in} \raisebox{22pt}{\includegraphics[height=0.4in, angle=270]{fmove2}}
\]
\\[4pt]
\[{\raisebox{-12pt}{\includegraphics[height=0.4in]{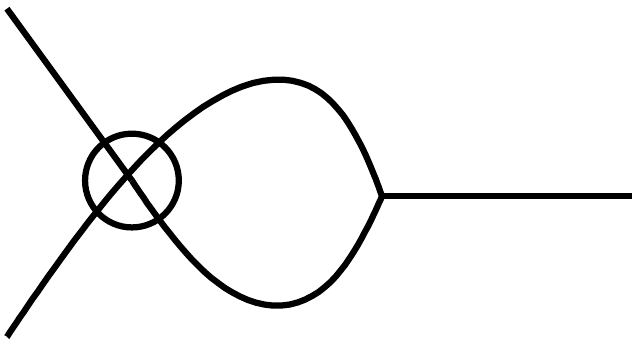}}} \hspace{0.2in} \stackrel{F3}\nleftrightarrow \hspace{0.2in} \raisebox{-12pt}{\includegraphics[height=0.4in]{R5-right}}
\]
\caption{Forbidden moves for virtual STG diagrams} \label{fig:bad-moves}
\end{figure}

We allow a virtual trivalent graph to have no vertices or virtual crossings. That is, we regard the set of classical knots and the set of virtual knots as subsets of the set of virtual spatial trivalent graphs.

In this paper, we work with \textit{well-oriented} virtual spatial trivalent graphs, meaning that all vertices must be either zip or unzip vertices (see Fig.~\ref{fig:orientations}). A \textit{zip vertex} is a trivalent vertex with two edges oriented toward the vertex and the other edge oriented away from the vertex. Similarly, an \textit{unzip vertex} is a trivalent vertex with two edges oriented away from the vertex and one edge oriented toward the vertex.

\begin{figure}[ht]
\[ {\includegraphics[height=0.7in]{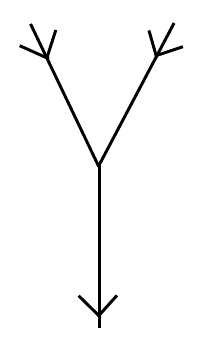}} \hspace{2cm} {\includegraphics[height=0.7in]{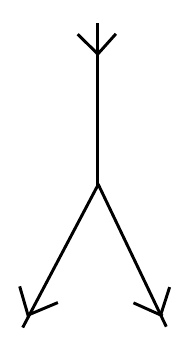}}
\put(-105, -10){\fontsize{7}{7}\text{zip}}
\put(-25, -10){\fontsize{7}{7}\text{unzip}}
\]
\caption{Allowed orientations near a vertex} \label{fig:orientations}
\end{figure}

We remark that any (virtual) spatial trivalent graph has an even number of vertices. Moreover, such a graph can be assigned an orientation so that the graph is well-oriented. Note that, by definition, a well-oriented virtual STG does not contain `sink' or `source' vertices (as depicted in Fig.~\ref{fig:bad-orientations}). 
\begin{figure}[ht]
\[ {\includegraphics[height=0.7in]{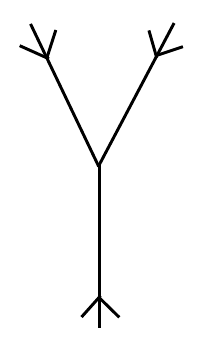}} \hspace{2cm} {\includegraphics[height=0.7in]{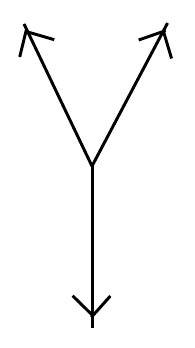}}
\]
\caption{Forbidden orientations near a vertex} \label{fig:bad-orientations}
\end{figure}

Since our graphs are oriented, we work with oriented versions of the extended Reidemeister moves, where the two diagrams in a certain move must have compatible orientations. From here on, all of the virtual spatial graphs graphs considered will be well-oriented.


\section{Virtual trivalent braids} \label{sec:VTB}

Similar to the case of classical knot theory (or virtual knot theory) where one can study classical knots (or virtual knots) by studying classical braids (or virtual braids), we can study virtual spatial trivalent graphs by studying virtual trivalent braids.

 A \textit{virtual trivalent braid} is a braid similar in notion to a classical braid, but may also contain trivalent vertices and virtual crossings. If such a braid has $m$ endpoints on top and $n$ endpoints in the bottom, we refer to it as a virtual trivalent $(m,n)$-braid. We denote by $VTB^{m}_{n}$ the set of all virtual trivalent $(m,n)$-braids. The \textit{identity braid on $n$ strands}, denoted by $1_n$, is the $(n, n)$-braid with $n$ parallel strands (free of vertices and crossings).

The \textit{closure} of a virtual trivalent $(n, n)$-braid is obtained by joining the opposite endpoints of the braid on its plane using non-intersection parallel arcs, as shown in Fig.~\ref{closure}. Note that the closure of a virtual trivalent $(n,n)$-braid is a virtual STG diagram.

\begin{figure}[ht]
\[
\raisebox{-0.8cm}{\includegraphics[height=1.9cm]{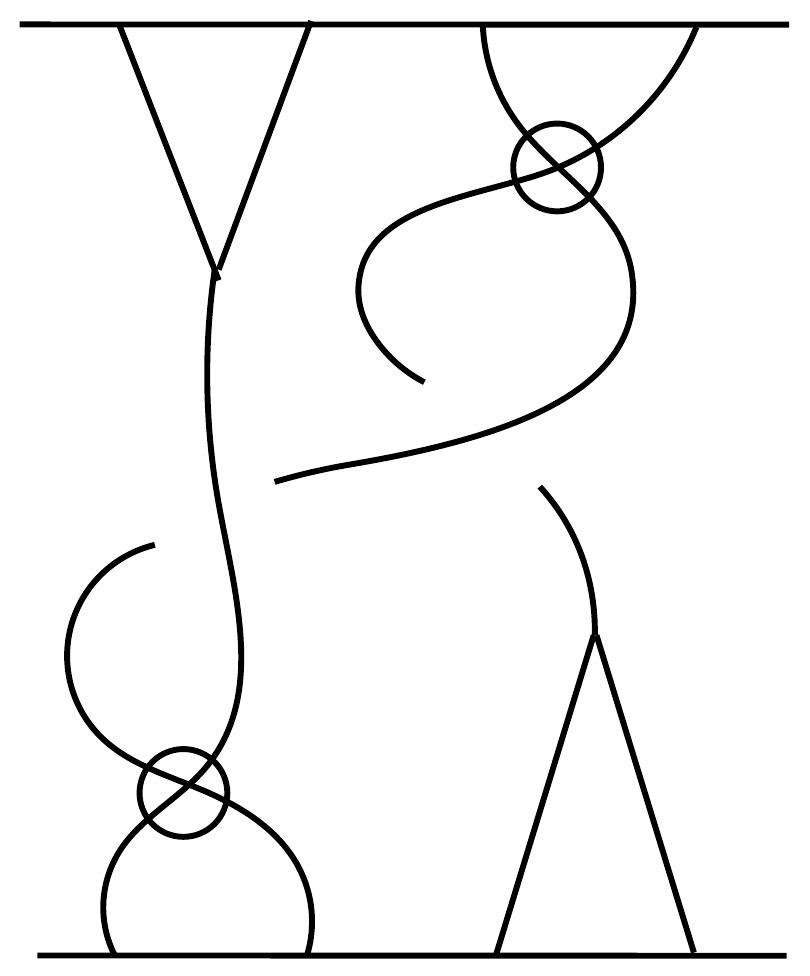}} \hspace{.2cm} \stackrel{\text{closure}}{\longrightarrow}\hspace{.2cm}
\raisebox{-1.75cm}{\includegraphics[height=3.75cm]{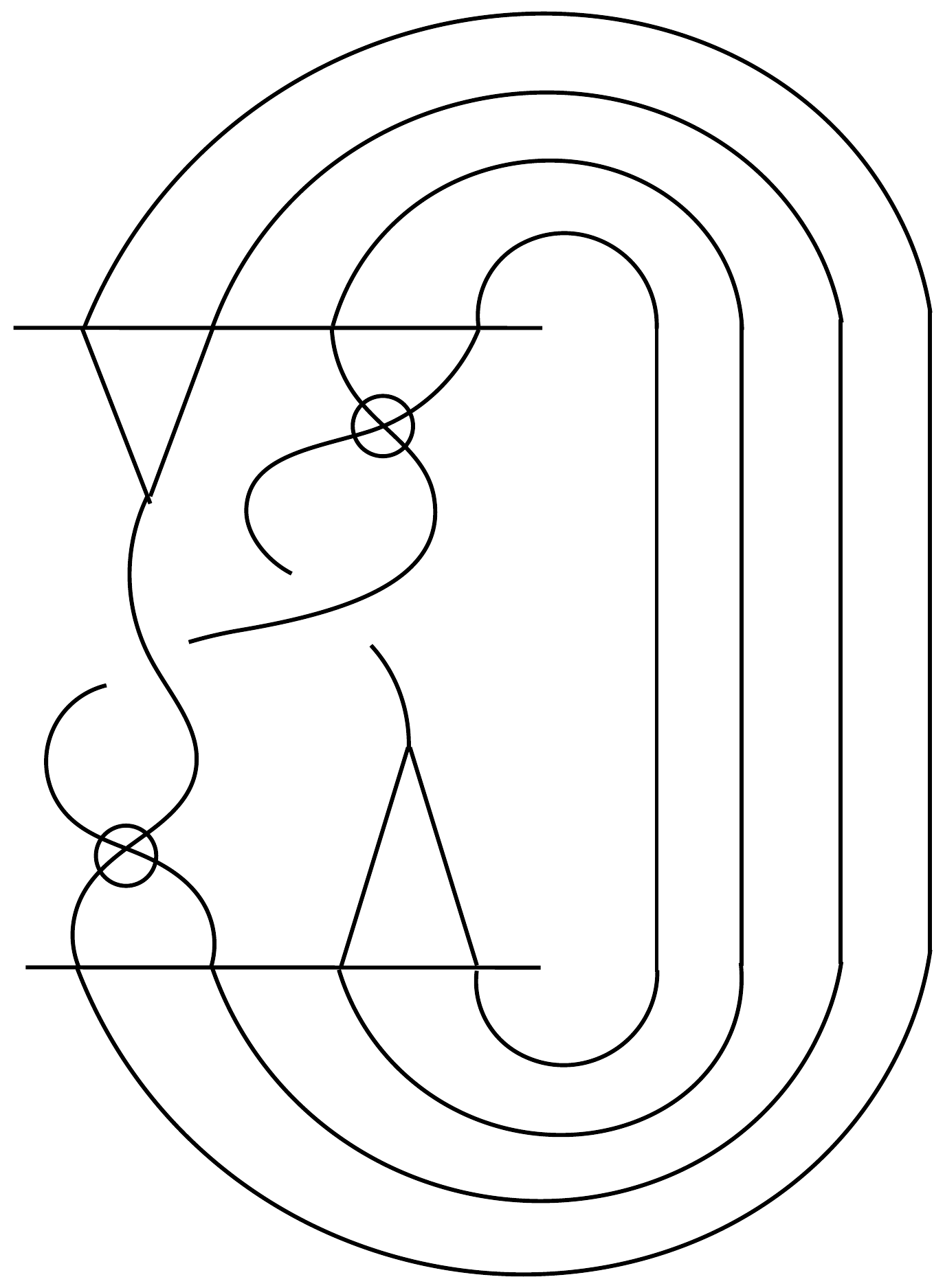}}
\]
\caption{A virtual trivalent braid and its closure}\label{closure}
\end{figure} 

Given two virtual trivalent braids $b_1 \in VTB^m_n$ and $b_2 \in VTB^n_s$, the \textit{composition} $b_1b_2$ is the virtual trivalent braid obtained by placing $b_1$ on top of $b_2$ and connecting the bottom endpoints of $b_1$ with the top endpoints of $b_2$. The resulting braid, $b_1b_2$, is an element of $VTB^m_s$.

The braids $\sigma_i$ , $\sigma^{-1}_i$ and $v_i \in VTB^n_n$ together with $y_i \in VTB^{n}_{n-1}$ and $\lambda_i \in VTB^{n-1}_{n}$ (where $1\leq i \leq n-1$) shown in Fig.~\ref{fig:braidletters} are called \textit{elementary virtual trivalent braids}.

\begin{figure}[ht]
\[  \sigma_i \,\,\, =\,\,\,  \raisebox{-17pt}{\includegraphics[height=.5in]{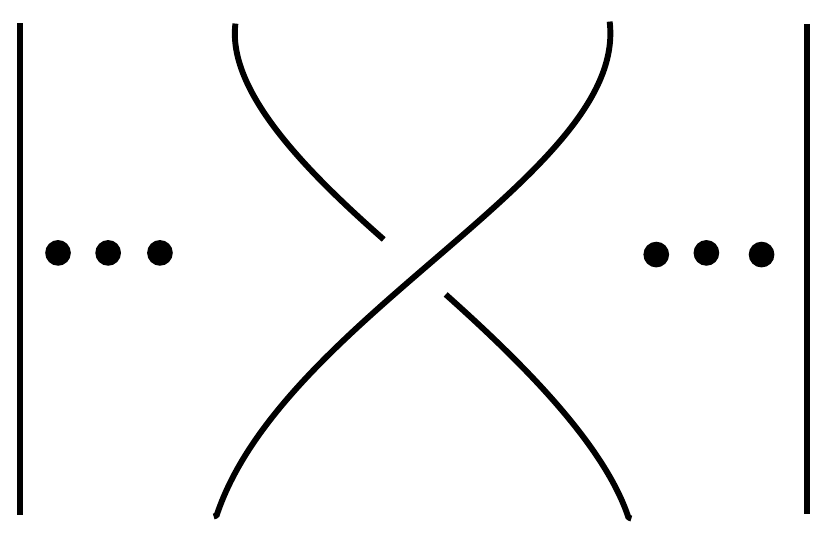}} \hspace{1cm} \sigma_i^{-1} \,\,\,=\,\,\, \raisebox{-17pt}{\includegraphics[height=.5in]{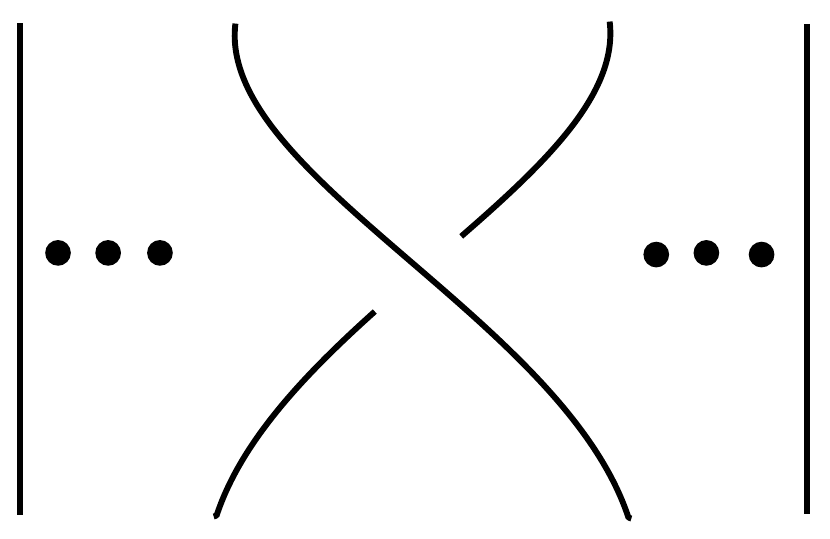}}
\put(-180, 21){\fontsize{7}{7}$1$}
\put(-165, 21){\fontsize{7}{7}$i$}
\put(-150, 21){\fontsize{7}{7}$i+1$}
\put(-127,21){\fontsize{7}{7}$n$}
\put(-58, 21){\fontsize{7}{7}$1$}
\put(-40, 21){\fontsize{7}{7}$i$}
\put(-25, 21){\fontsize{7}{7}$i+1$}
\put(-3,21){\fontsize{7}{7}$n$} \hspace{1cm}
v_i \,\, = \,\,\, \raisebox{-17pt}{\includegraphics[height=.5in]{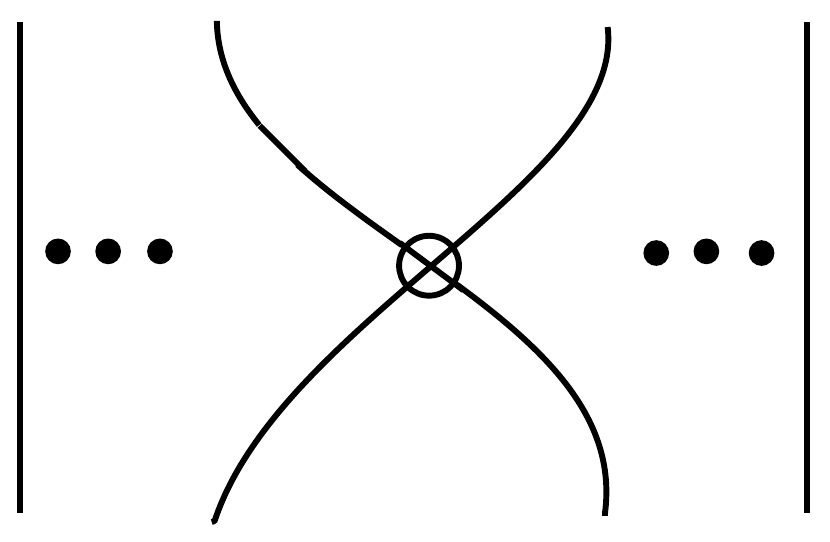}}
\put(-56, 21){\fontsize{7}{7}$1$}
\put(-42, 21){\fontsize{7}{7}$i$}
\put(-25, 21){\fontsize{7}{7}$i+1$}
\put(-3,21){\fontsize{7}{7}$n$}
\]
\[y_i \,\, =\,\,  \raisebox{-17pt}{\includegraphics[height=.5in, width=0.8in]{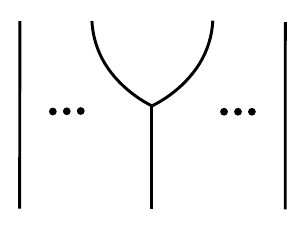}} \hspace{1cm} \lambda_i \,\,\,=\,\,\, \raisebox{-17pt}{\includegraphics[height=.5in, width=0.83in]{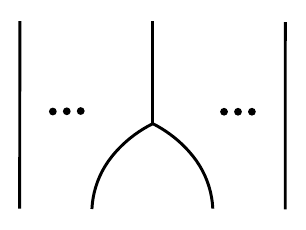}}
\put(-177, 21){\fontsize{7}{7}$1$}
\put(-162, 21){\fontsize{7}{7}$i$}
\put(-146, 21){\fontsize{7}{7}$i+1$}
\put(-125,21){\fontsize{7}{7}$n$}
\put(-58, 21){\fontsize{7}{7}$1$}
\put(-31, 21){\fontsize{7}{7}$i$}
\put(-10,21){\fontsize{7}{7}$n-1$}
\put(-177, -21){\fontsize{7}{7}$1$}
\put(-152, -21){\fontsize{7}{7}$i$}
\put(-130,-21){\fontsize{7}{7}$n-1$}
\put(-58, -21){\fontsize{7}{7}$1$}
\put(-45, -21){\fontsize{7}{7}$i$}
\put(-27, -21){\fontsize{7}{7}$i+1$}
\put(-5,-21){\fontsize{7}{7}$n$}
\]
\caption{Elementary virtual trivalent braids} \label{fig:braidletters}
\end{figure}
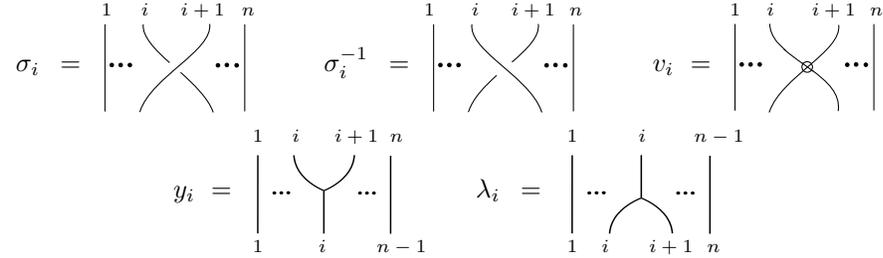

Any virtual trivalent braid can be written as a word in terms of elementary virtual trivalent braids using the composition of braids. For example, the braid depicted in Fig.~\ref{closure} can be represented using the following word: $v_3 y_1 \sigma_2 \sigma_1^{-1} \lambda_3 v_1$. (We use the convention that elementary virtual trivalent braids are `read' from top to bottom with respect to a height function on the plane of the braid).

We consider virtual trivalent braids up to isotopy, a notion that we now define. Braid isotopy is reminiscent of the extended Reidemeister moves for virtual STG diagrams, with the exception that the moves $R1$ and $V1$ are not possible in braid form.

\begin{definition}
Two virtual trivalent $(m, n)$-braids are called \textit{braid equivalent} (or \textit{isotopic}) if they are related by a finite sequence of the braid relations (subword replacements) shown in Fig.~\ref{fig:tbrelns}, together with the following \textit{commuting relations}:

\begin{itemize}
\item $\sigma_i b_j = b_j \sigma_i$, \,\, $v_i b_j = b_j v_i$
\item $y_i b_{j-1} = b_j y_i$,\,\, $\lambda_i b_j = b_{j-1} \lambda_i$
\end{itemize}
where $i+1 <j$ and $b_j \in \{\sigma_j, v_j, y_j, \lambda_j \}$.
\end{definition}

We collectively refer to the commuting relations and the relations in Fig.~\ref{fig:tbrelns} as the \textit{braid isotopy relations}.

\begin{figure}[ht] 
\begin{center}
  \begin{tabular}{ |l | l | l| }
    \hline
$\sigma_i\sigma_i^{-1}=\sigma_i^{-1}\sigma_i=1_n$
&
$v_i^2=1_n$
\\
$\hspace{0.3in}\raisebox{-15pt}{\includegraphics[height=0.5in]{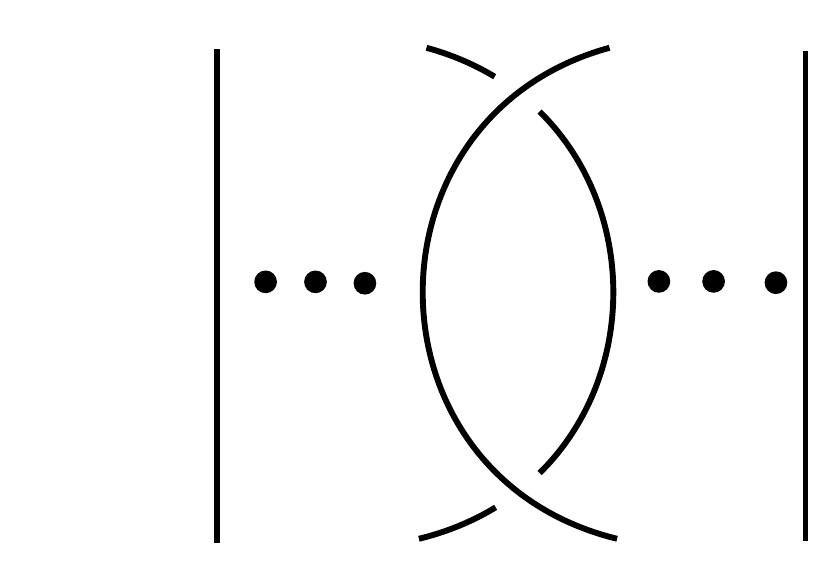}}\,\,  \stackrel{R2}{\sim} \,\, \raisebox{-15pt}{\includegraphics[height=0.5in]{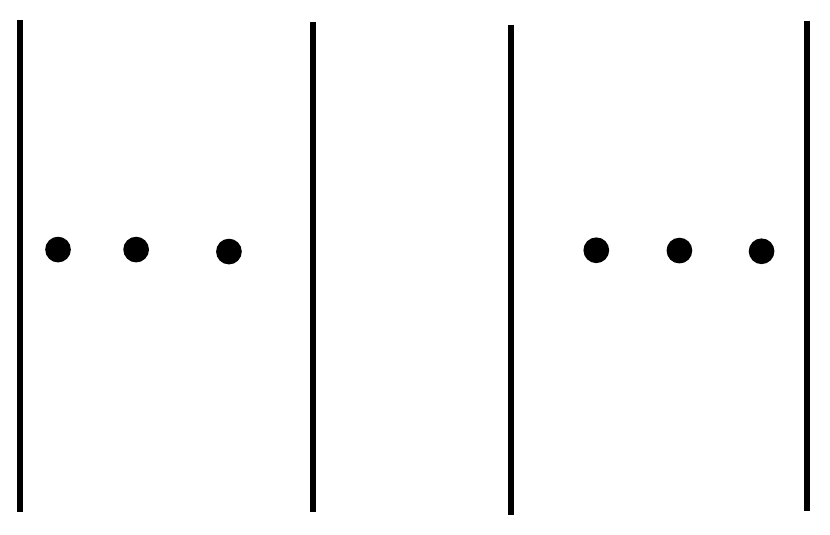}}\hspace{0.2in}$
&
$\hspace{0.2in}\raisebox{-15pt}{\includegraphics[height=0.5in]{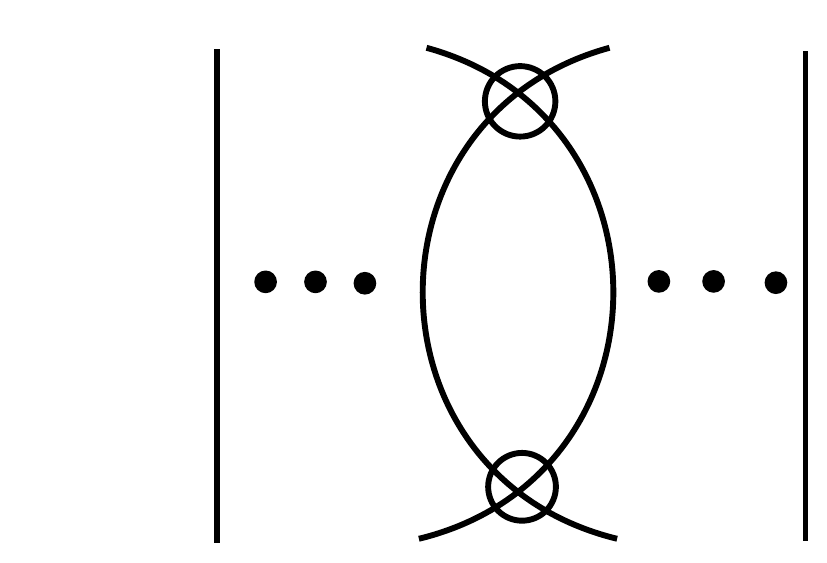}}\,\,  \stackrel{V2}{\sim} \,\, \raisebox{-15pt}{\includegraphics[height=0.5in]{idb}}\hspace{0.2in}$
\\
&
\\
$\sigma_i\sigma_{i+1}\sigma_i = \sigma_{i+1}\sigma_i\sigma_{i+1}$
&
$v_i v_{i+1} v_i = v_{i+1} v_i v_{i+1}$
\\
$\hspace{0.35in}\raisebox{-15pt}{\includegraphics[height=0.5in]{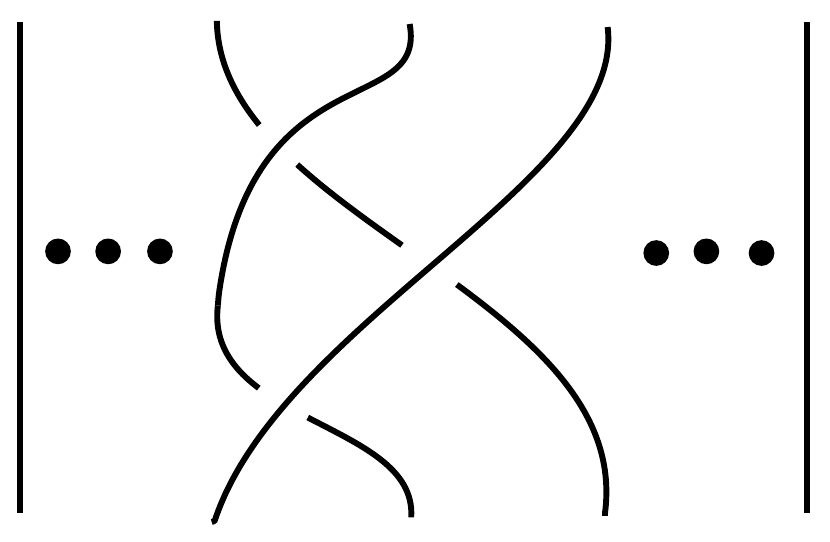}}\,\,  \stackrel{R3}{\sim} \,\, \raisebox{-15pt}{\includegraphics[height=0.5in]{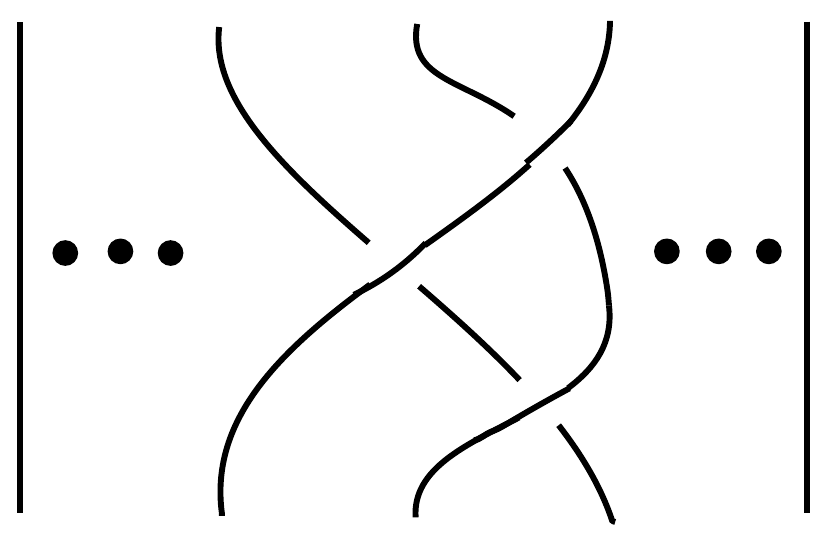}}$
&
$\hspace{0.2in}\raisebox{-15pt}{\includegraphics[height=0.5in]{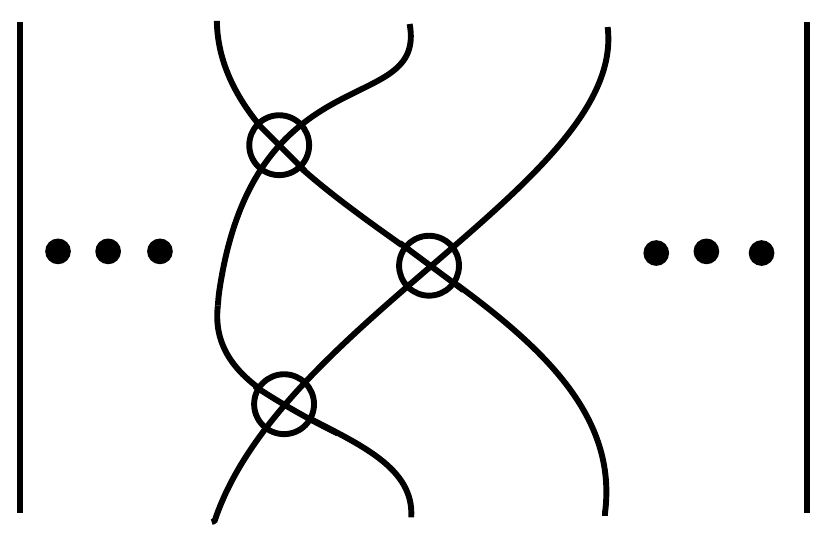}}\,\,  \stackrel{V3}{\sim} \,\, \raisebox{-15pt}{\includegraphics[height=0.5in]{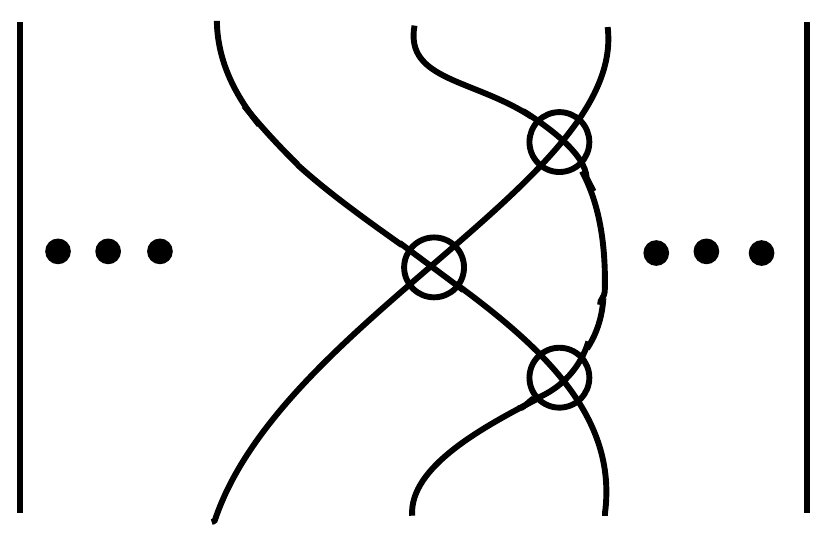}}\hspace{0.2in}$
\\
&
\\
$\sigma_{i+1}\sigma_iy_{i+1}$=$y_i\sigma_i$
&
$v_i\sigma_{i+1}v_i$=$v_{i+1}\sigma_iv_{i+1}$
\\
$\hspace{0.45in}\raisebox{-15pt}{\includegraphics[height=0.5in]{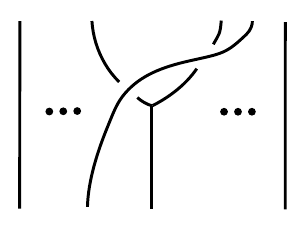}}\,\,  \stackrel{R4}{\sim} \,\, \raisebox{-15pt}{\includegraphics[height=0.5in]{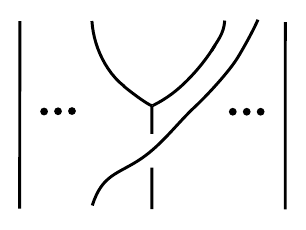}}\hspace{0.2in}$
&
$\hspace{0.2in}\raisebox{-15pt}{\includegraphics[height=0.5in]{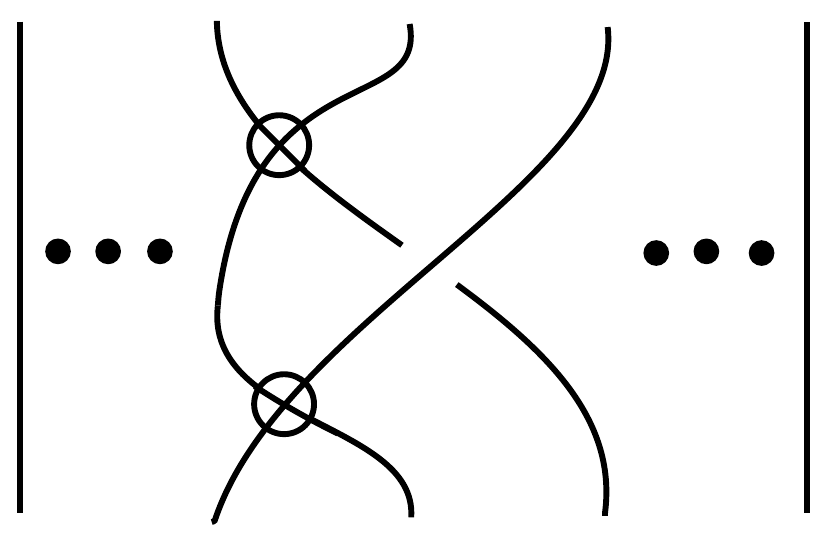}}\,\,  \stackrel{VR3}{\sim} \,\, \raisebox{-15pt}{\includegraphics[height=0.5in]{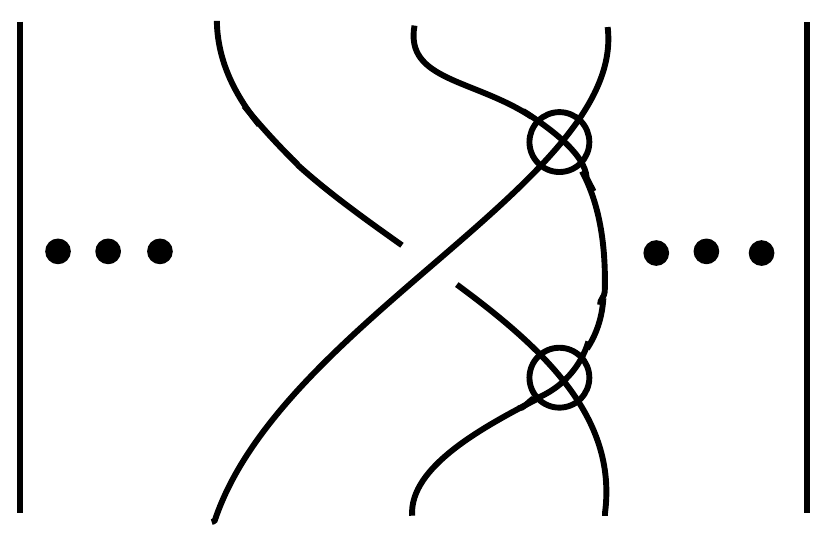}}\hspace{0.2in}$
\\
&
\\
$\sigma_i\lambda_{i+1}=\lambda_i\sigma_{i+1}\sigma_i$
&
$v_{i+1} v_i y_{i+1}=y_i v_i$
\\
$\hspace{0.45in}\raisebox{-15pt}{\includegraphics[height=0.5in]{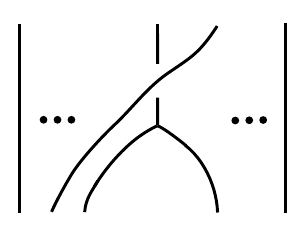}}\,\,  \stackrel{R4}{\sim} \,\, \raisebox{-15pt}{\includegraphics[height=0.5in]{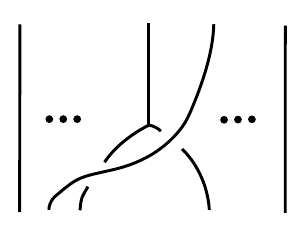}}\hspace{0.2in}$
&
$\hspace{0.3in}\raisebox{-15pt}{\includegraphics[height=0.5in]{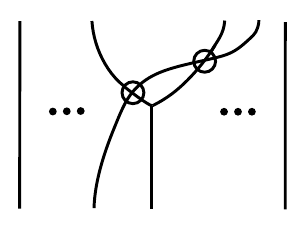}}\,\,  \stackrel{V4}{\sim} \,\, \raisebox{-15pt}{\includegraphics[height=0.5in]{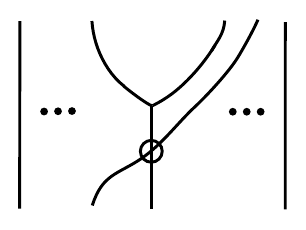}}\hspace{0.2in}$
\\
&
\\
$\sigma_{i}\sigma_{i+1} y_{i}=y_{i+1} \sigma_i$
&
$\lambda_i v_{i+1} v_i$=$v_i \lambda_{i+1}$
\\
$\hspace{0.45in}\raisebox{-15pt}{\includegraphics[height=0.5in]{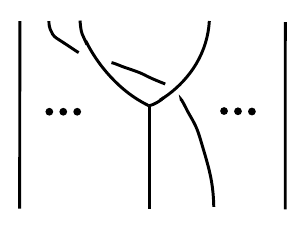}}\,\,  \stackrel{R4}{\sim} \,\, \raisebox{-15pt}{\includegraphics[height=0.5in]{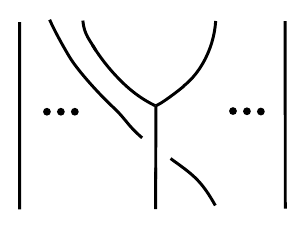}}\hspace{0.2in}$
&
$\hspace{0.3in}\raisebox{-15pt}{\includegraphics[height=0.5in]{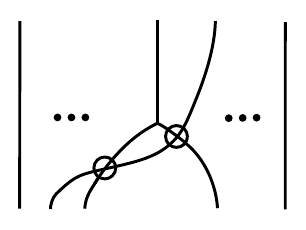}}\,\,  \stackrel{V4}{\sim} \,\, \raisebox{-15pt}{\includegraphics[height=0.5in]{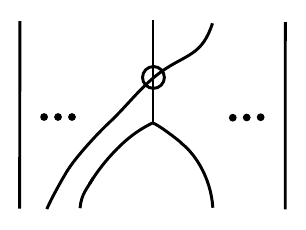}}\hspace{0.2in}$
\\
&
\\
$\sigma_i \lambda_{i}$=$\lambda_{i+1}\sigma_i\sigma_{i+1}$
&
$y_i=\sigma_i y_i$
\\
$\hspace{0.45in}\raisebox{-15pt}{\includegraphics[height=0.5in]{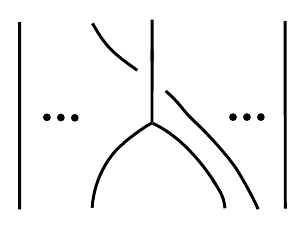}}\,\,  \stackrel{R4}{\sim} \,\, \raisebox{-15pt}{\includegraphics[height=0.5in]{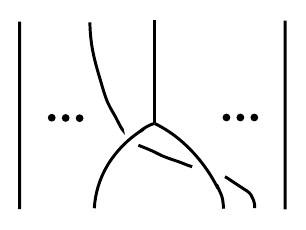}}\hspace{0.2in}$
&
$\hspace{0.45in}\raisebox{-15pt}{\includegraphics[height=0.5in]{y1}}\,\,  \stackrel{R5}{\sim} \,\, \raisebox{-15pt}{\includegraphics[height=0.5in]{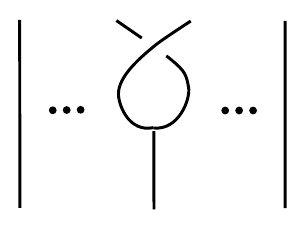}}\hspace{0.2in}$
\\
&
\\
$\lambda_i=\lambda_i \sigma_i$
&
\text{}
\\
$\hspace{0.45in}\raisebox{-15pt}{\includegraphics[height=0.5in]{lambda1}}\,\,  \stackrel{R5}{\sim} \,\, \raisebox{-15pt}{\includegraphics[height=0.5in, width=0.7in]{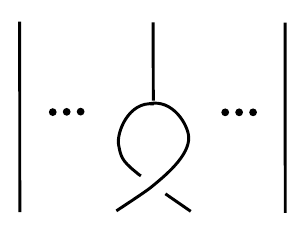}}\hspace{0.2in}$
&
\\
&

         \\ \hline
\end{tabular}
\caption{Braid relations} \label{fig:tbrelns}
\end{center}
\end{figure}


\subsection{Preparation for braiding}

We want to show that every well-oriented virtual spatial trivalent graph can be represented as the closure of some virtual trivalent braid. For this, we follow a braiding process similar to that introduced by Kauffman and Lambropoulou in~\cite{KauLamb}, where they work with virtual braids and oriented virtual knots and links. In order to use this braiding process, several steps must be executed in preparation. First, the vertices in a virtual STG diagram must be in $Y$ or $\lambda$ position. Figure \ref{Vertices} shows both a $Y$-vertex and a $\lambda$-vertex.

\begin{figure}[ht]
\[
\raisebox{-.5cm}{\includegraphics[height=1.5cm]{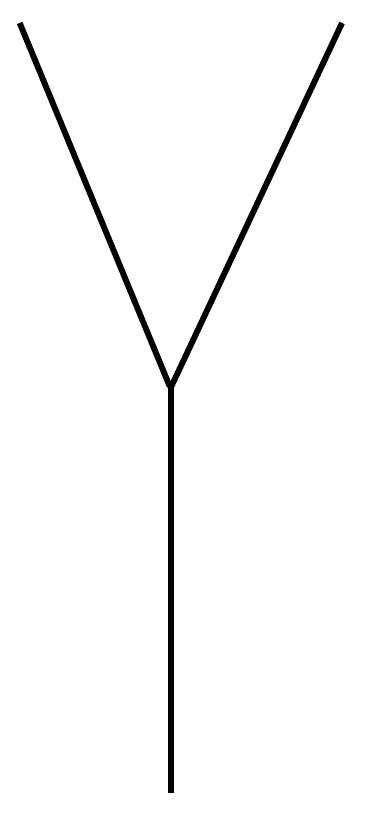}} \hspace{0.5in}
\raisebox{-.5cm}{\includegraphics[height=1.5cm]{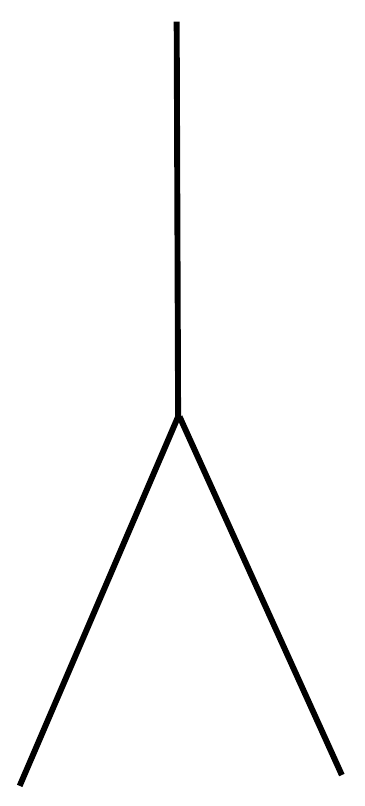}} 
\] 
\caption{$Y$- and $\lambda$-vertices}\label{Vertices}
\end{figure}

If a vertex in a diagram is either a $W$- or $M$-vertex (see Fig.~\ref{wm}), it must be rearranged to be in either the form of a $Y$- or $\lambda$-vertex. Figure~\ref{RearrM} shows how to rearrange a $W$- and $M$-vertex into the form of a $Y$- and $\lambda$-vertex, respectively, with the circled area representing the $Y$-vertex, and respectively, the $\lambda$-vertex. This is done by rotating the vertex counterclockwise while keeping the endpoints fixed.  

\begin{figure}[ht]
\begin{center}
{\includegraphics[height=0.4in, width=.8in]{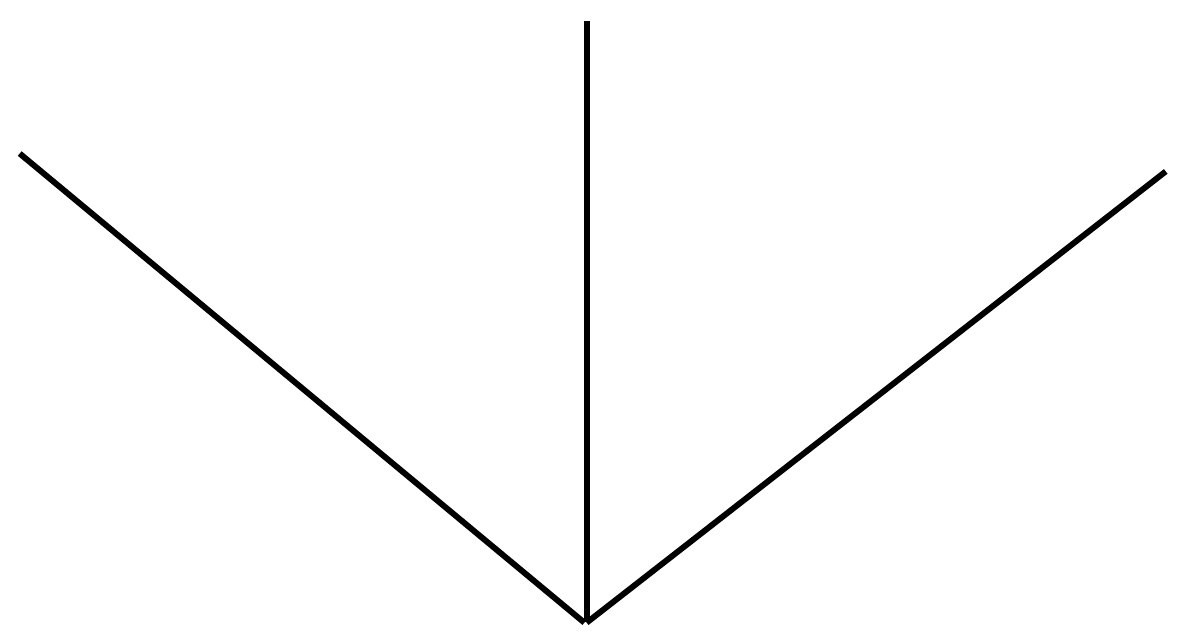}}\,\, \hspace{0.75in} \,\, \, 
\raisebox{1cm}{ \includegraphics[height=.4in, width=.8in, angle=180]{Unoriented-W}}

\end{center}
\caption{$W$- and $M$-vertices}\label{wm}
\end{figure}

\begin{figure}[ht]
\begin{center}

\[ \raisebox{-.2in}{\includegraphics[height=0.5in, width= .8in]{Unoriented-W}}\,\, \rightarrow \,\,  \raisebox{-.2in}{\includegraphics[height=.5in, width=.8in]{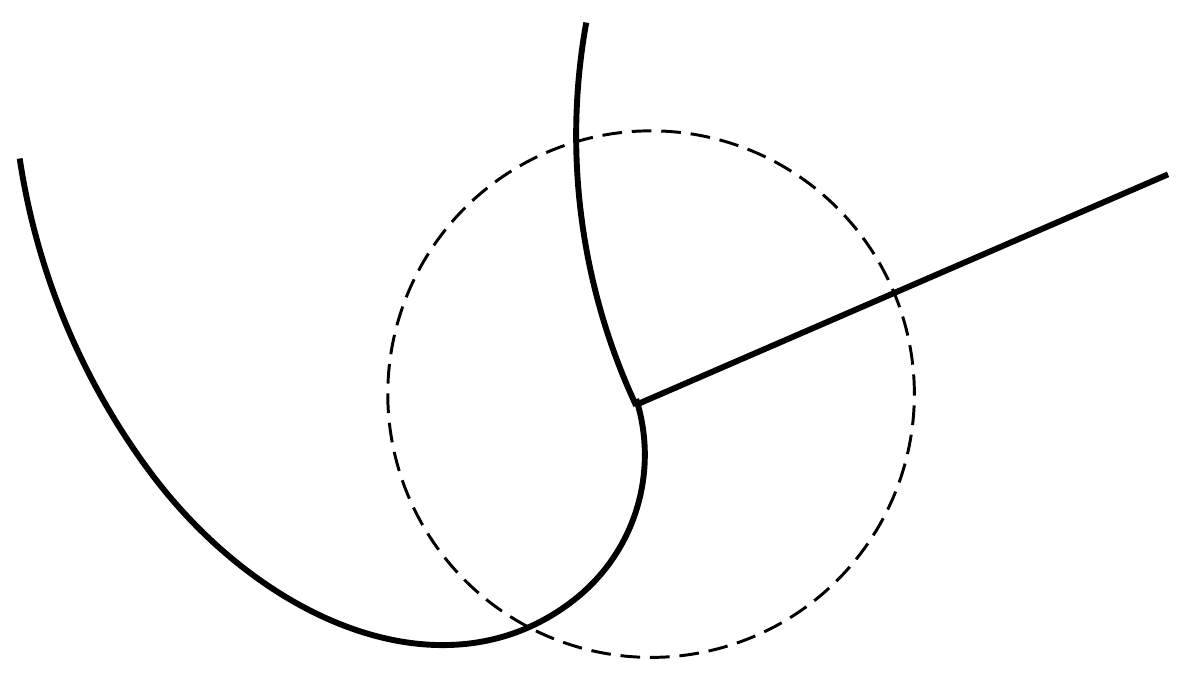}} \hspace{0.75in}
 \raisebox{.2in}{\includegraphics[height=0.5in,width=.8in,  angle=180]{Unoriented-W}}\,\, \rightarrow \,\, 
\raisebox{.2in}{\includegraphics[height=.5in,width=.8in, angle=180]{simplification3}} 
\]

\caption{Rearranging $W$- and $M$-vertices} \label{RearrM}
\end{center}
\end{figure}

Once all vertices are locally arranged into $Y$- or $\lambda$-vertices, the diagram must be put into regular position.

\begin{definition}
A vertex is said to be in \textit {regular position} if it is either a $Y$- or $\lambda$-vertex, and if in a small neighborhood of the vertex, the edges incident with it have downward orientation. A virtual STG diagram is called \textit{regular} if all of its vertices are in regular position. 
\end{definition}

Using swing moves and $R5$ moves, we establish conventions for putting a $Y$- or $\lambda$-vertex into regular position, according to the chart in Fig.~\ref{regpos}. It is important to note that this approach for bringing vertices into regular position does not change the isotopy-type of the original diagram.

\begin{figure}[ht]
\begin{center}
 \begin{tabular}{ |l | l | l| }
    \hline
& &  \\[-.2cm]
$\hspace{.1cm}    \reflectbox{\raisebox{-.2in}{\includegraphics[height=.5in]{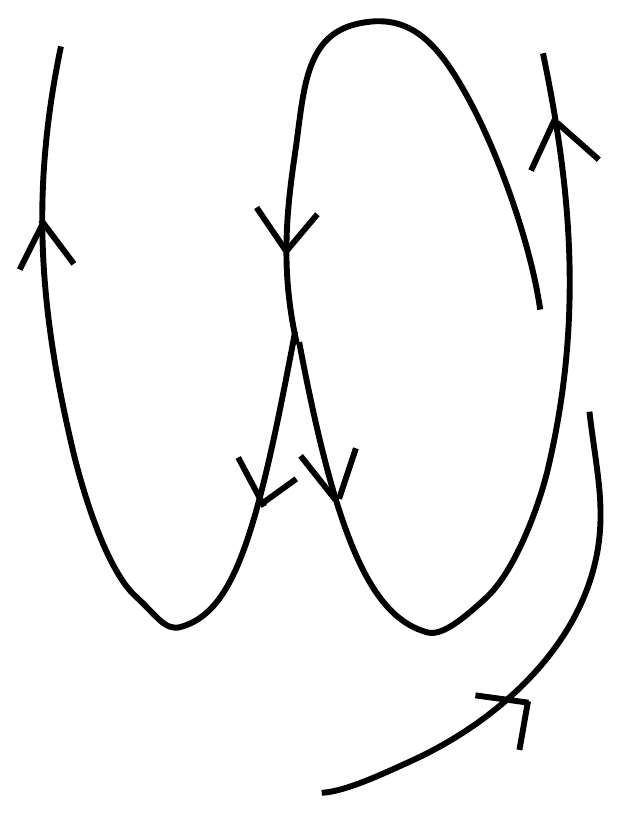}}}  \hspace{.1cm} \,\,\leftarrow \,\, \raisebox{-.2in}{\includegraphics[height=0.5in]{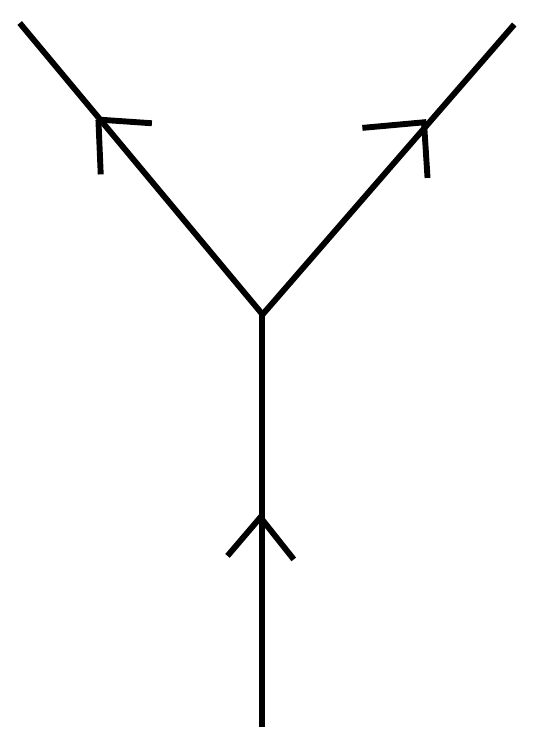}}\,\,  \rightarrow \,\, \raisebox{-.2in}{\includegraphics[height=.5in]{LUUU2b}}$
&
$\raisebox{-.2in}{\includegraphics[height=0.5in]{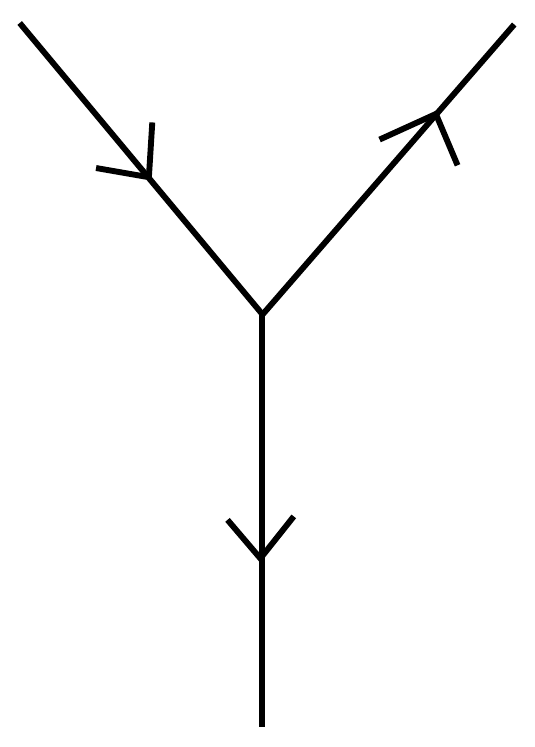}}\,\,  \rightarrow \,\, \raisebox{-.2in}{\includegraphics[height=.5in]{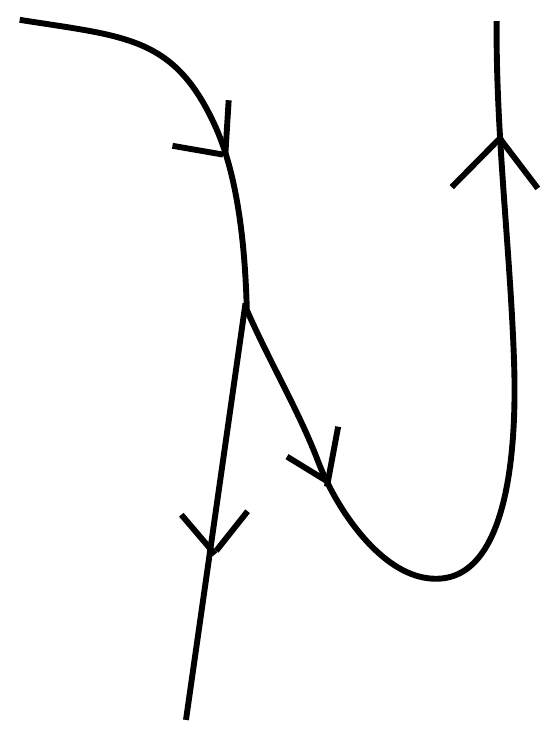}} $
&
$ \raisebox{-.2in}{\includegraphics[height=0.5in]{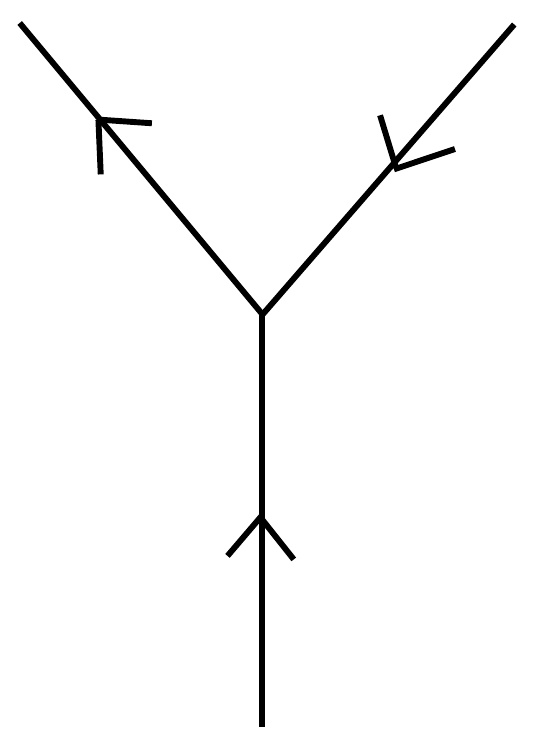}}\,\,  \rightarrow \,\, \raisebox{-.2in}{\includegraphics[height=.5in]{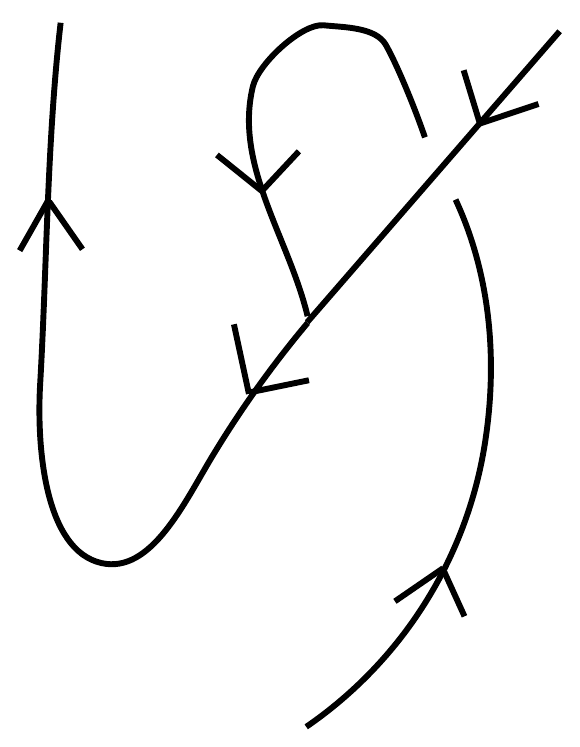}} $
 \\[-.2cm] & & 
 \\ \hline
&  & \\[-.2cm]

$ \hspace{1.2cm}\raisebox{-.2in}{\includegraphics[height=0.5in]{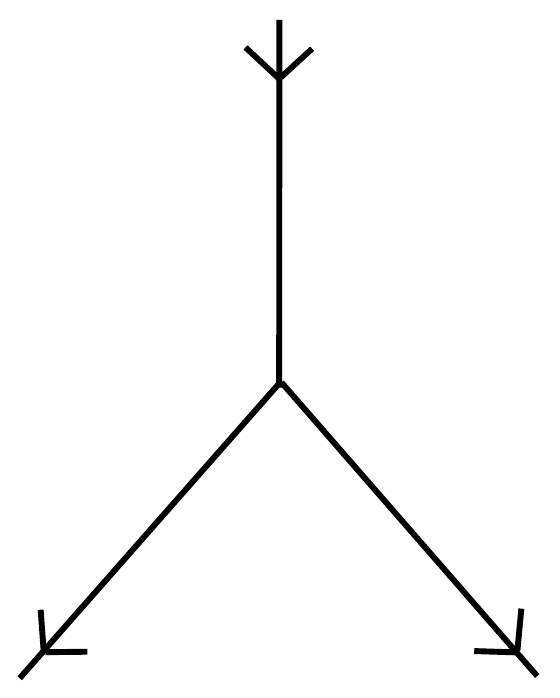}}\,\, \rightarrow \,\, \raisebox{-.2in}{\includegraphics[height=.5in]{LDDD1}} $
&
$\raisebox{-.2in}{\includegraphics[height=0.5in]{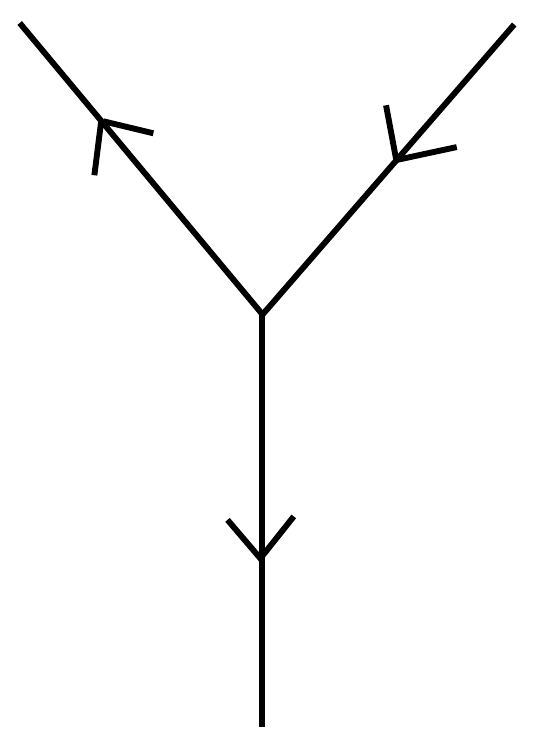}}\,\,  \rightarrow \,\, 
\reflectbox{\raisebox{-.2in}{\includegraphics[height=.5in]{YDUD2}}} $
&
$ \raisebox{-.2in}{\includegraphics[height=0.5in]{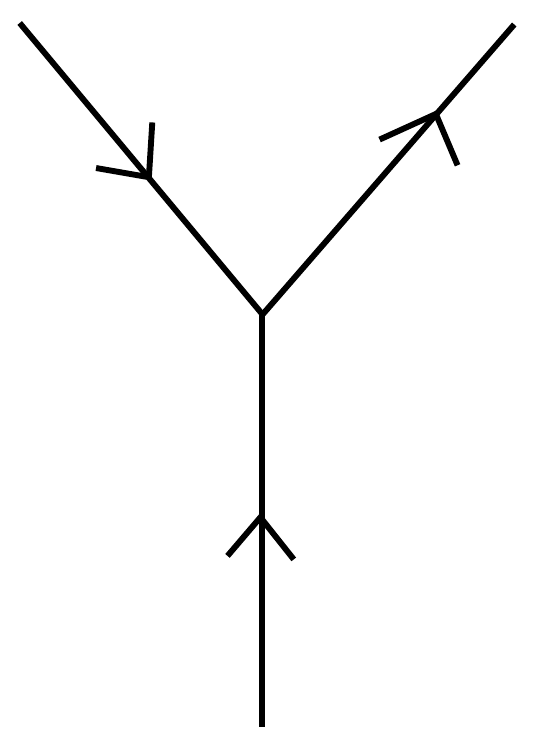}}\,\,  \rightarrow \,\, \raisebox{-.2in}{\includegraphics[height=.5in]{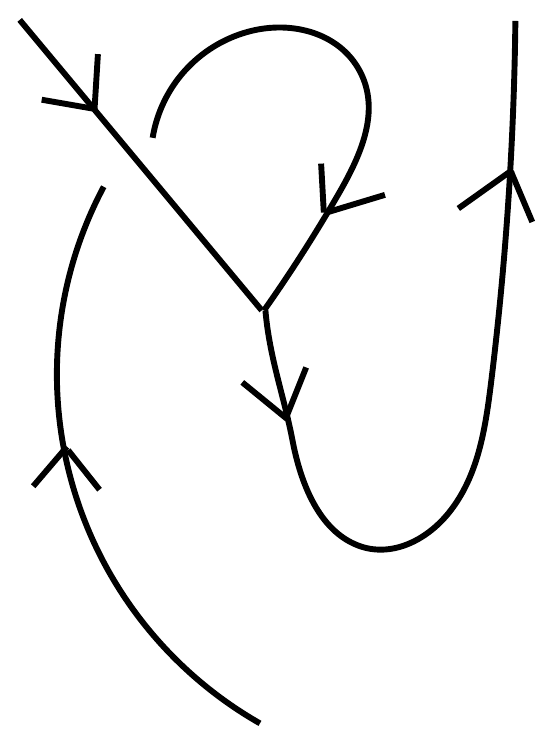}} $

 \\[-.2cm] & & 
 \\ \hline
&  & \\[-.2cm]

$ \hspace{0.1cm}   \reflectbox{\raisebox{-.2in}{\includegraphics[height=.5in]{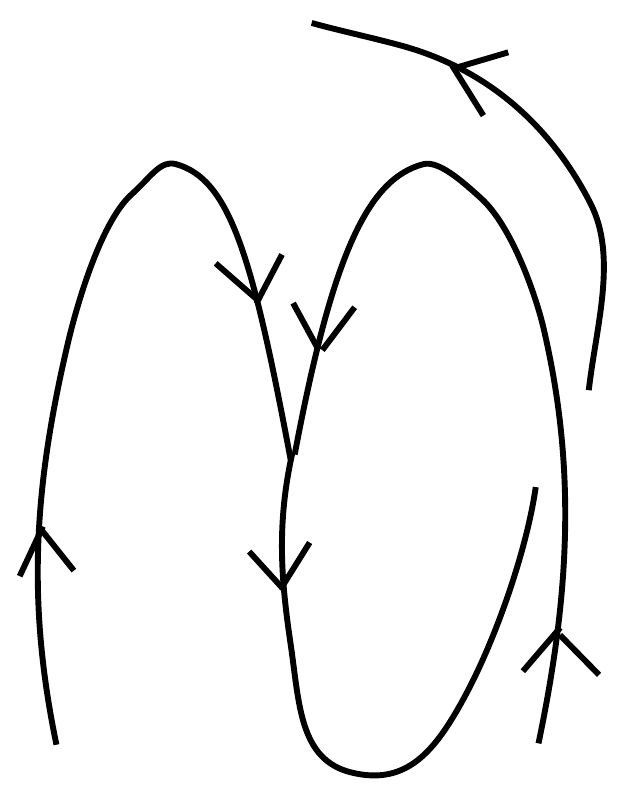}}} \,\, \leftarrow \,\, \raisebox{-.2in}{\includegraphics[height=0.5in]{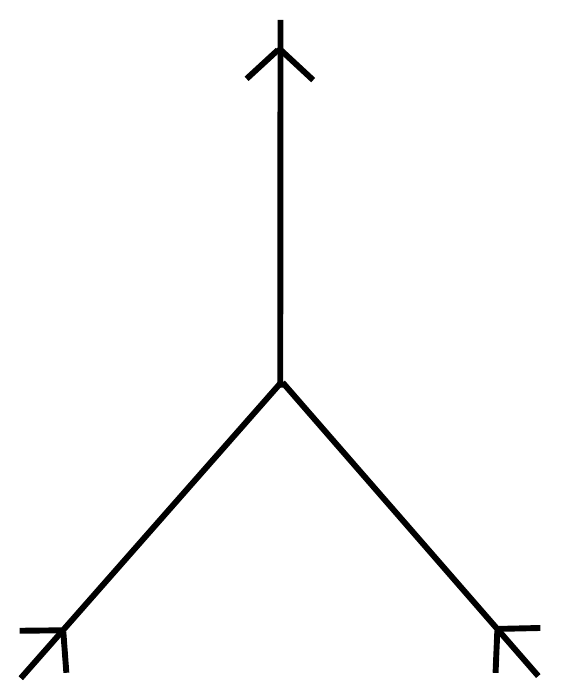}}\,\, \rightarrow \,\, \raisebox{-.2in}{\includegraphics[height=.5in]{YUUU2b}}\hspace{-0.7cm}$
&
$ \raisebox{-.2in}{\includegraphics[height=0.5in]{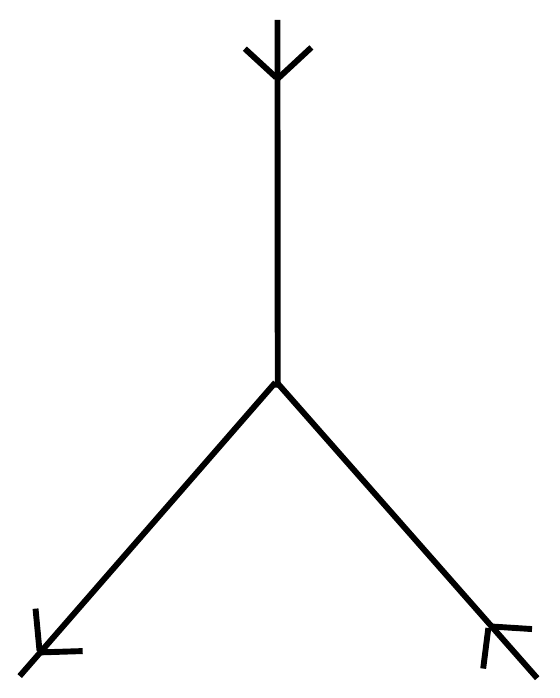}}\,\, \rightarrow \,\, \raisebox{-.2in}{\includegraphics[height=.5in]{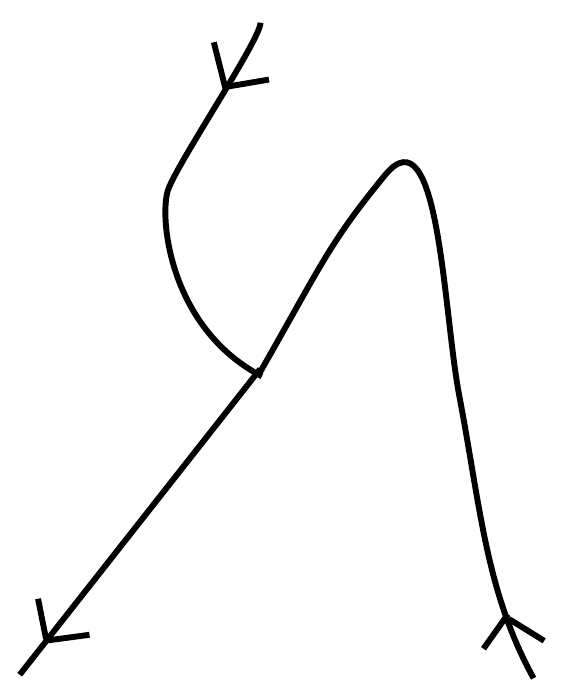}} $
&
$ \raisebox{-.2in}{\includegraphics[height=0.5in]{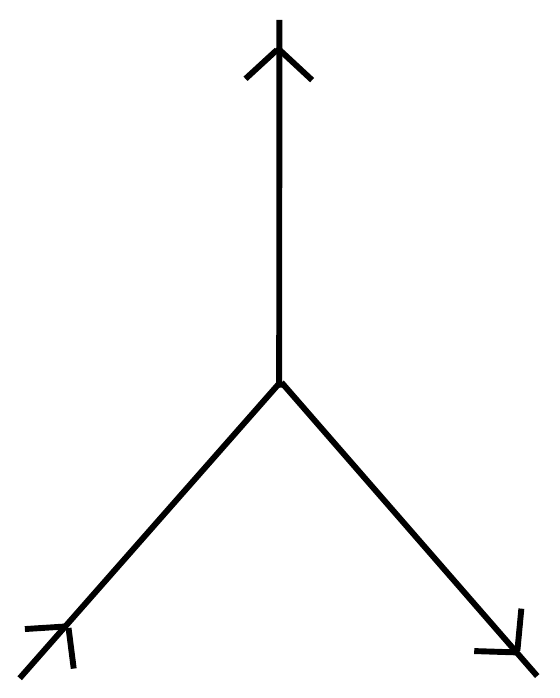}}\,\, \rightarrow \,\, \raisebox{-.2in}{\includegraphics[height=.5in]{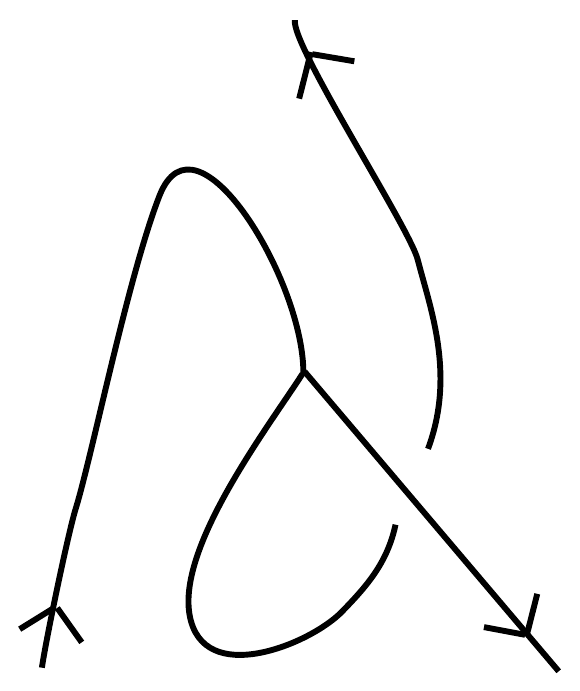}} $
 \\[-.2cm] & & 
 \\ \hline
&  & \\[-.2cm]

$\hspace{1.2cm} \raisebox{-.2in}{\includegraphics[height=0.5in]{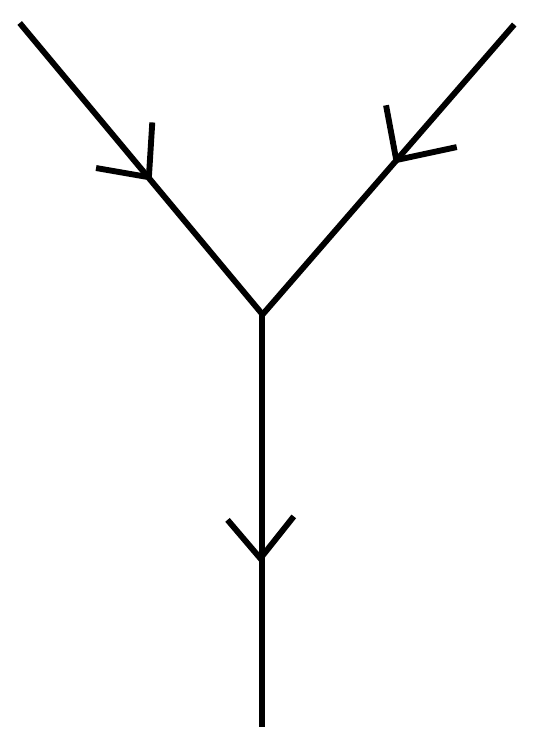}}\,\rightarrow \,\, \raisebox{-.2in}{\includegraphics[height=.5in]{YDDD1}} $

&
$ \raisebox{-.2in}{\includegraphics[height=0.5in]{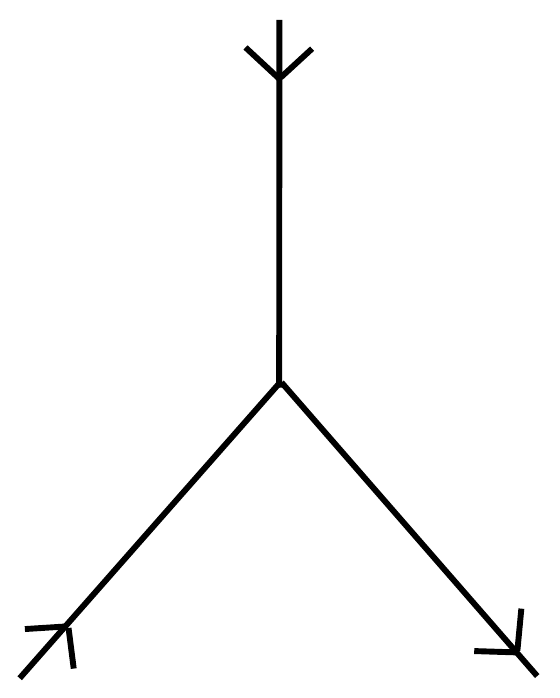}}\,\, \rightarrow \,\, \raisebox{-.2in}{\includegraphics[height=.5in]{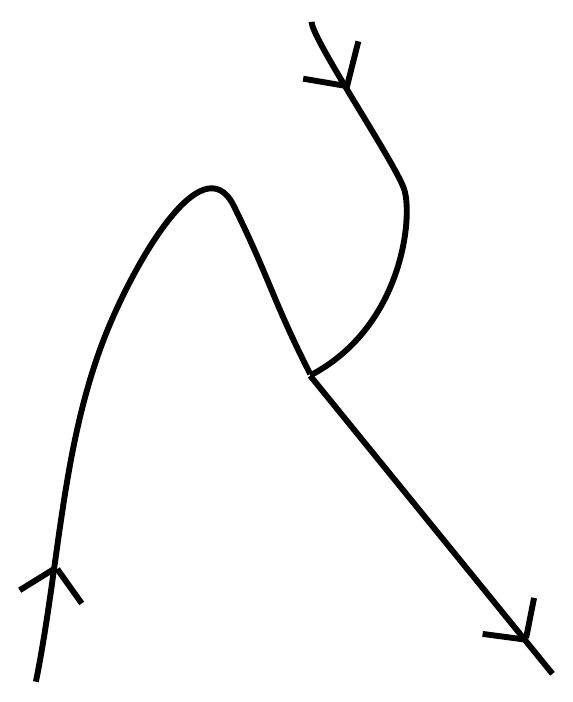}} $
&
$\raisebox{-.2in}{\includegraphics[height=0.5in]{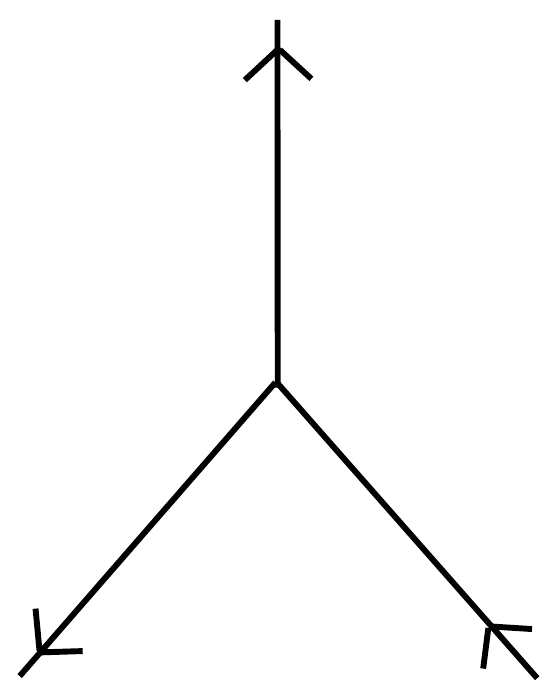}}\,\, \rightarrow \,\, \raisebox{-.2in}{\includegraphics[height=.5in]{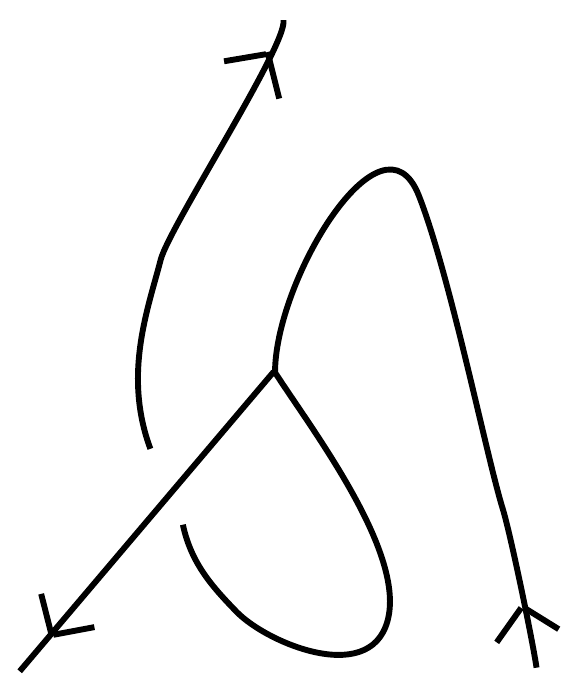}} $
 \\[-.2cm] & & 
 \\ \hline

\end{tabular}
\end{center}
\caption{Arranging $Y$- and $\lambda$-vertices in regular position }\label{regpos}
\end{figure}

We now explain the chart in Fig.~\ref{regpos}.
Note first that if all three edges adjacent to a vertex are oriented downwards, then the vertex is already in regular position and it is therefore left unchanged. If there is exactly one edge with upwards orientation attached to a $Y$- or $\lambda$-vertex, a swing move is applied to that edge to obtain downwards orientation near the vertex. A specific case is shown in Fig.~\ref{excase2}.

\begin{figure}[ht]
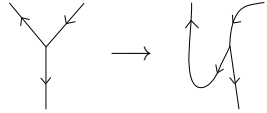

\[
\raisebox{-.7cm}{\includegraphics[height=1.5cm]{YUDD1}} \hspace{.2cm} \longrightarrow \hspace{.3cm} \reflectbox{\raisebox{-.7cm}{\includegraphics[height=1.5cm]{YDUD2}}}
\]
\caption{A swing move near a vertex}\label{excase2}
\end{figure}

If there are two edges with upwards orientation attached to a $Y$- or $\lambda$-vertex, a swing move is applied followed by an $R5$ move (see the third column in Fig.~\ref{regpos}). For a  $Y$-vertex ($\lambda$-vertex), the swing move is applied to the upper edge (lower edge) with upwards orientation; an $R5$ move is then applied between the other edge with upwards orientation and the third edge with downwards orientation. Note that this yields a vertex in regular position. Since there are two possible crossing-types that can be used in an $R5$ move, there are two possible ways to bring this type of vertex to regular position; only one type of possible crossings is shown in each case in Fig.~\ref{regpos}.

If all edges attached to a $Y$- or $\lambda$-vertex have upwards orientation, two swing moves are applied followed by an $R5$ move (see the first column, rows one and three, in Fig.~\ref{regpos}). For a $Y$-vertex ($\lambda$-vertex), the swing moves are applied to the upper (lower) edges attached to the vertex; an $R5$ move is then applied to the bottom (upper) edge with upwards orientation and any of the two upper (bottom) edges attached to the vertex. Since there are two possible sets of edges where the $R5$ move can be applied to and two possible crossing types, there are four possible ways to bring this type of oriented vertex into regular position; only two ways are shown in the chart of Fig.~\ref{regpos} for this type of $Y$- or $\lambda$-vertex.

Once a virtual STG diagram is in regular position, we continue with the preparation for braiding and the braiding algorithm described in~\cite{KauLamb}.\\

A \textit{down-arc} (respectively \textit{up-arc}) in a regular virtual STG diagram is an arc with downward (respectively upward) orientation. A \textit{free up-arc} is an upward-oriented arc that is not part of a crossing and is away from a vertex. 

In order to work with these up-arcs and down-arcs, we subdivide the diagram by marking its arcs with points. The \textit{subdivision} of an arc consists of cutting an arc into smaller arcs, and marking the places where the arc was cut with points, called \textit{subdivision points}. Note that there will be a subdivision point at each local minimum and local maximum in a diagram. Moreover, each crossing with at least one up-arc will be flanked with subdivision points in preparation for braiding. Note that vertices or crossings cannot coincide with subdivision points.

The goal of the braiding algorithm---applied to a regular virtual STG diagram---is to eliminate all up-arcs in the diagram without changing the isotopy-type of the diagram, and such that the resulting diagram is the closure of a virtual trivalent braid (with all braid strands oriented downward). The braiding algorithm requires the diagram to satisfy certain conditions, which are explained in the following definition.\\

\begin{definition}~\label{regular position}
A regular virtual STG diagram is said to be in \textit{general position} if it does not contain any horizontal arcs and no two subdivision points, crossings or vertices are vertically or horizontally aligned. Additionally, no crossing is coincident with a local maximum or minimum.
\end{definition}

Crossings cannot be coincident with a local maximum or minimum since there is a subdivision point at each local maximum and minimum but crossings cannot be subdivision points. In addition,  we require in the definition of regular position that there are no horizontally aligned crossings or vertices so that crossings and vertices lie on different horizontal levels in the corresponding braid obtained at the end of the braiding algorithm.  The reason for requiring that there are no vertically aligned subdivision points, crossings or vertices will become clear once we explain our braiding algorithm.

By applying small planar shifts, if necessary, any regular virtual STG diagram can be transformed into a diagram in general position. These local shifts were called \textit{direction sensitive moves} in~\cite{KauLamb}. Some of these moves allow us to shift a crossing, vertex or subdivision point with respect to the horizontal or vertical direction. Other moves, like the swing moves near crossings and depicted in Fig.~\ref{fig:swing moves}, ensure that no local maximum/minimum point and crossing (classical and virtual) are vertically aligned.

\begin{figure}[ht]
\[ \raisebox{-10pt}{\includegraphics[height=0.4in]{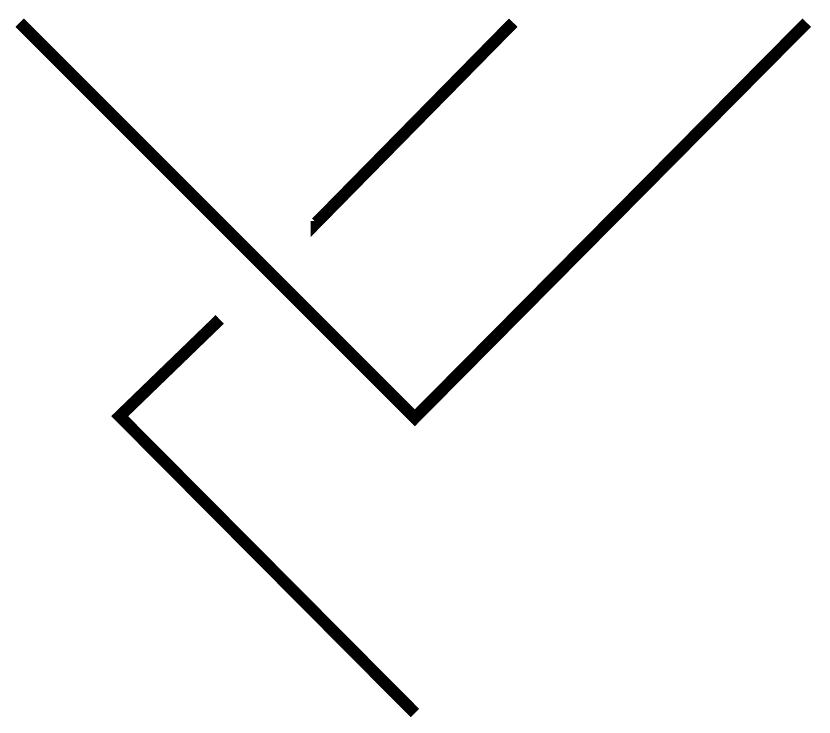}}\,\, \longleftrightarrow \,\, \raisebox{-10pt}{\includegraphics[height=0.4in]{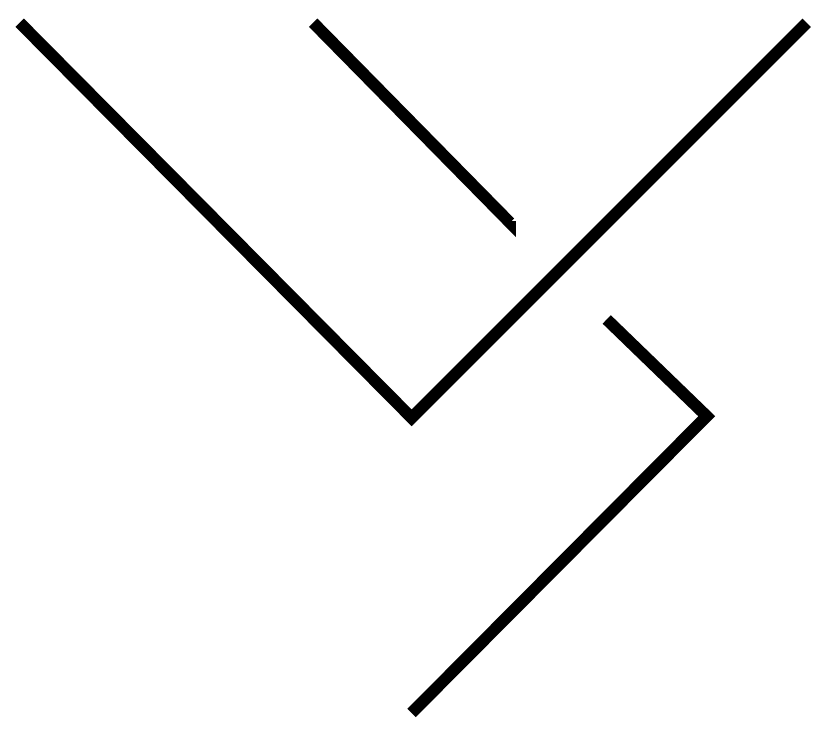}}
 \hspace{0.5cm} 
 \raisebox{-10pt}{\includegraphics[height=0.42in]{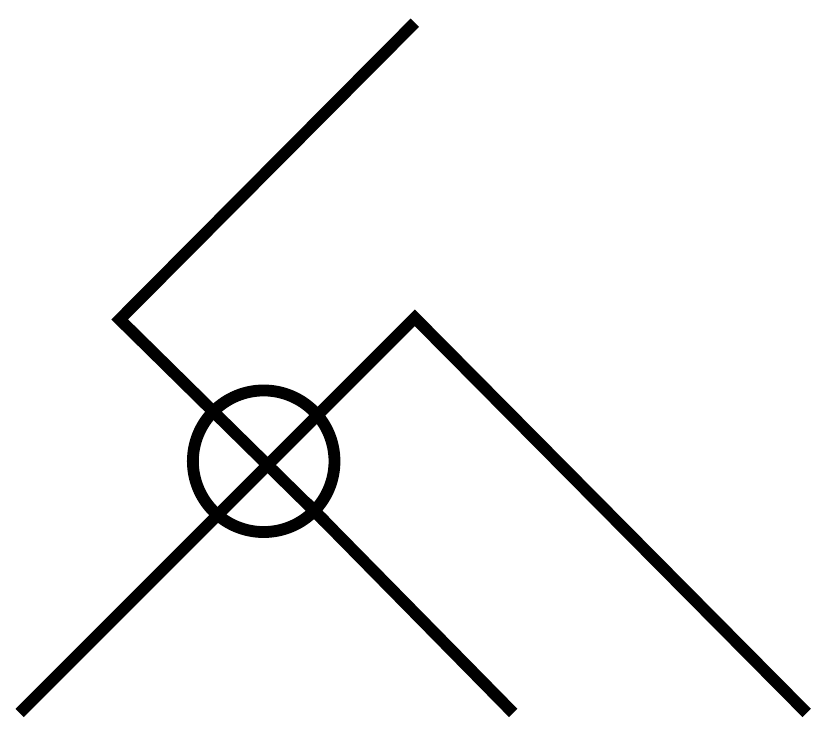}}\,\,\longleftrightarrow \,\, \raisebox{-10pt}{\includegraphics[height=0.42in]{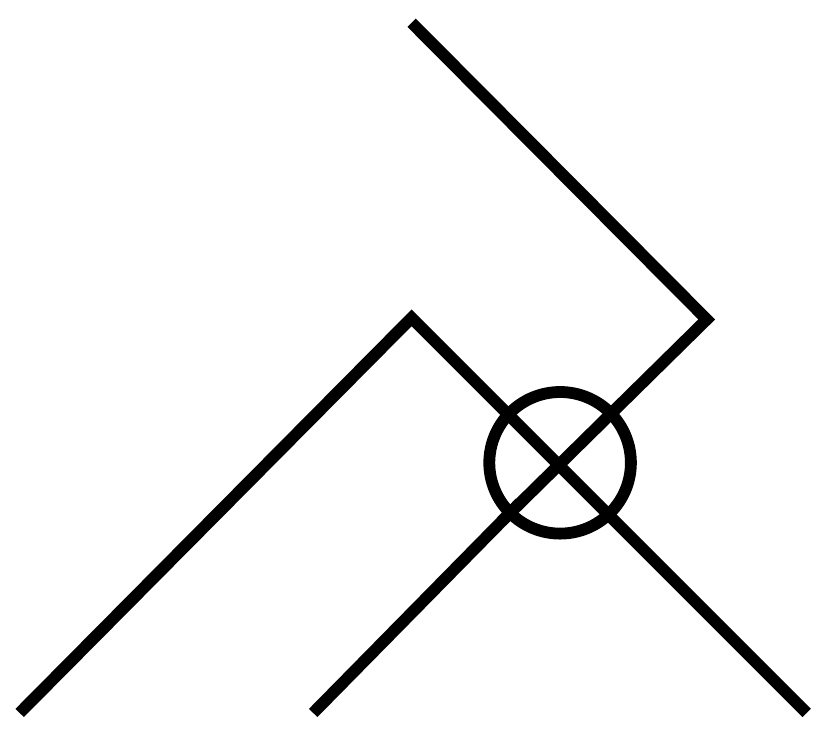}} 
\]
\caption{The swing moves near a crossing}\label{fig:swing moves}
\end{figure}

From here on, all diagrams will be assumed to be in general position. We are now ready to describe our approach to braiding a regular virtual STG diagram in general position.

\subsection{A braiding algorithm and an Alexander-type theorem}

We start with a diagram in regular position and we eliminate the up-arcs while keeping the down-arcs. If a crossing contains at least one up-arc, we braid the crossing as a whole, using
a \textit{crossing box}---a small rectangular box whose diagonals are the arcs of the crossing. This box must be narrow enough so that the vertical regions above and below the box do not intersect with the region of another crossing in the diagram. The free up-arcs are arcs that connect crossing boxes. We braid each crossing containing an up-arc according to the chart in Fig.~\ref{BraidingChart} (see~\cite[Fig. 7]{KauLamb}). When braiding a crossing, a new pair of braid strands are created; these new strands are drawn so that outside of the crossing box they are vertically aligned and so that they intersect virtually any other part of the diagram; this is depicted by placing abstract virtual crossings at the ends of the braid strands.

\begin{figure}[ht] 
\begin{center}
 \begin{tabular}{ |l | l | l| }
    \hline
    $\reflectbox{ \raisebox{5pt}{\includegraphics[height=0.4in]{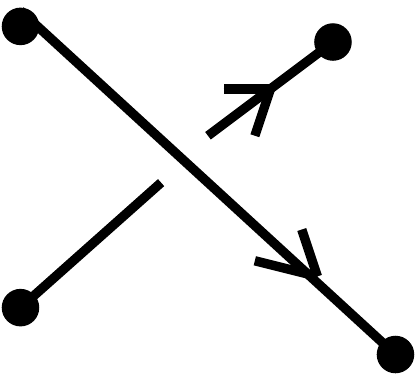}} }\,\, \raisebox{20pt}{$\longrightarrow$} \,\, \reflectbox{\raisebox{-30pt}{\includegraphics[height=1.2in]{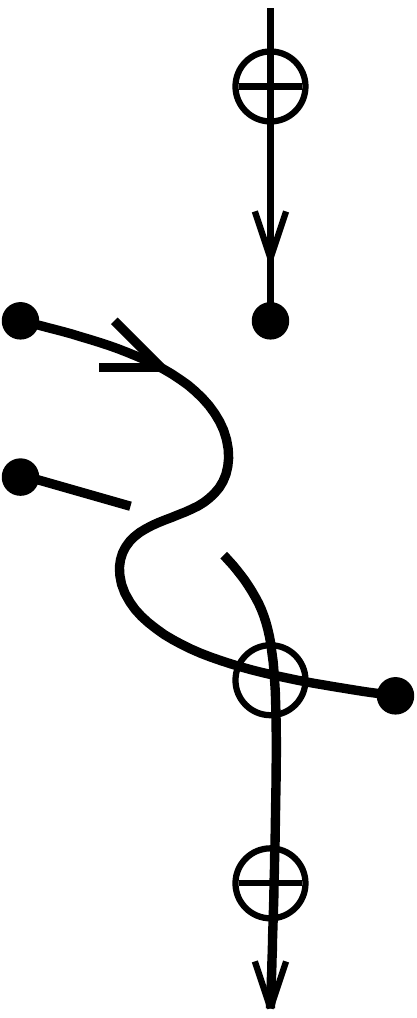}}} $
 &
 $ \reflectbox{\raisebox{5pt}{\includegraphics[height=0.4in]{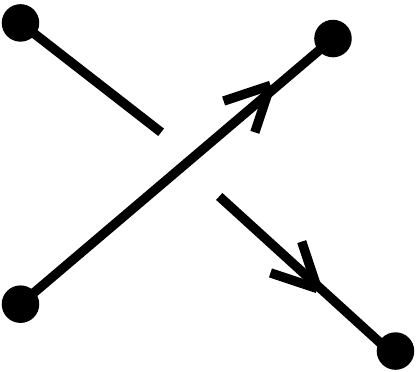}}}  \,\, \raisebox{20pt}{$\longrightarrow$} \,\,  \reflectbox{\raisebox{-30pt}{\includegraphics[height=1.2in]{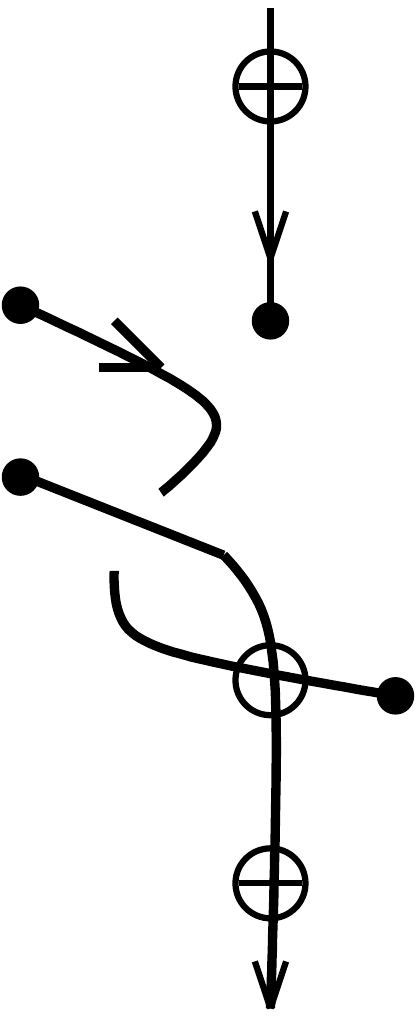}}} $ 
&
 $ \reflectbox{\raisebox{5pt}{\includegraphics[height=0.4in]{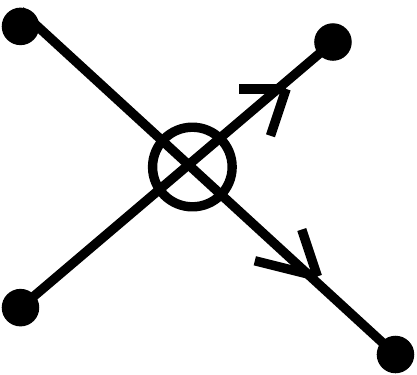}}}  \,\, \raisebox{20pt}{$\longrightarrow$} \,\,  \reflectbox{\raisebox{-25pt}{\includegraphics[height=1.1in]{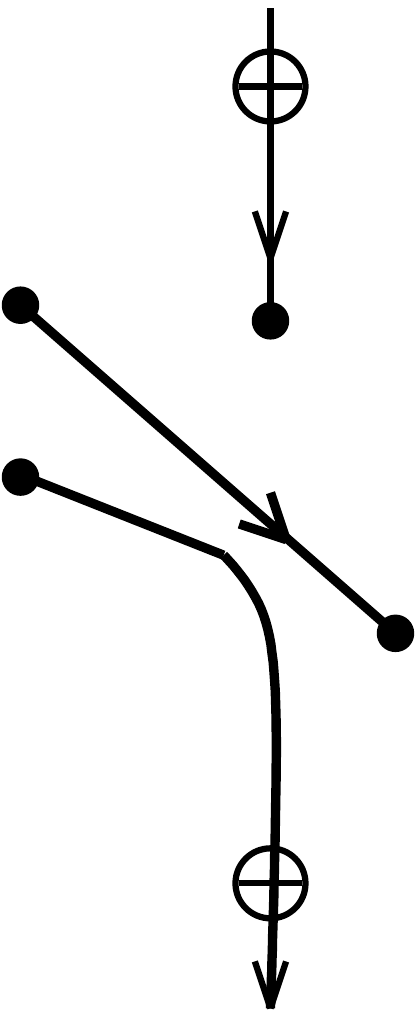}} }$ \\ \hline 

$ \raisebox{5pt}{\includegraphics[height=0.4in]{Virtual2A-new}} \,\, \raisebox{20pt}{$\longrightarrow$} \,\,\raisebox{-25pt}{\includegraphics[height=1.1in]{Virtual2B-new}} $ 
&
$ \raisebox{5pt}{\includegraphics[height=0.4in]{Classical3A-new}}  \,\, \raisebox{20pt}{$\longrightarrow$} \,\, \raisebox{-30pt}{\includegraphics[height=1.2in]{Classical3B-new}} $ 
&
$ \raisebox{5pt}{\includegraphics[height=0.4in]{Classical4A-new}} \,\, \raisebox{20pt}{$\longrightarrow$} \,\, \raisebox{-30pt}{\includegraphics[height=1.2in]{Classical4B-new}} $ \\ 
 \hline
 $ \raisebox{10pt}{\includegraphics[height=0.4in]{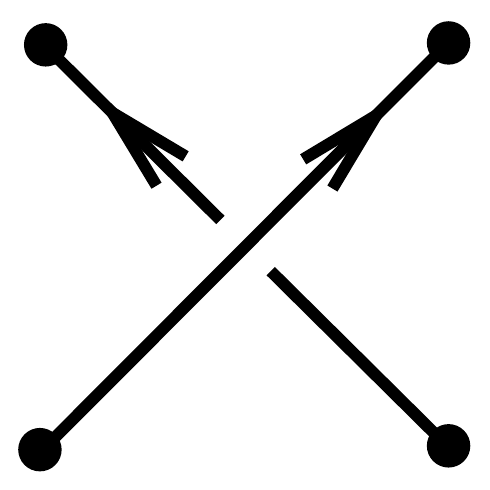}} \,\, \raisebox{25pt}{$\longrightarrow$} \,\, \raisebox{-40pt}{\includegraphics[height=1.4in]{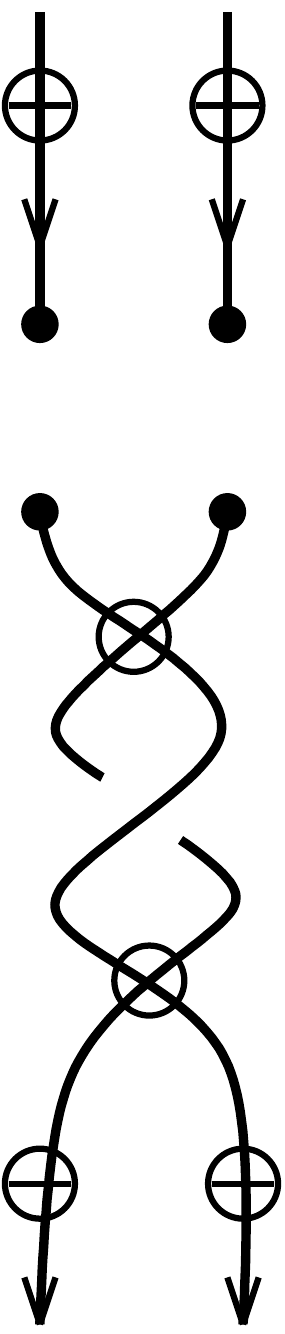}}  $ 
&
  $ \reflectbox{\raisebox{10pt}{\includegraphics[height=0.4in]{Classical5A}}}  \,\, \raisebox{25pt}{$\longrightarrow$} \,\,  \reflectbox{\raisebox{-40pt}{\includegraphics[height=1.4in]{Classical5B-new}} }$  
&
 $ \raisebox{10pt}{\includegraphics[height=0.37in, angle=90]{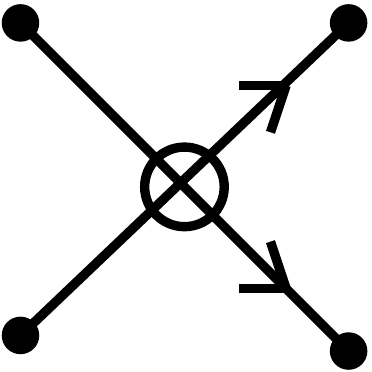}} \,\, \raisebox{20pt}{$\longrightarrow$} \,\, \raisebox{-30pt}{\includegraphics[height=1.2in]{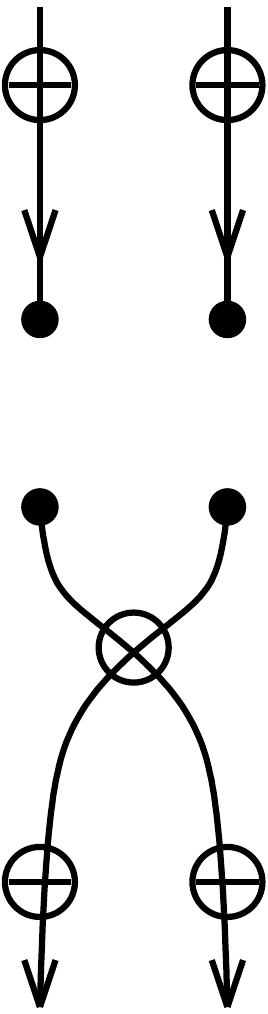}}$ 
          \\ \hline
\end{tabular}
\end{center}
\caption{The braiding chart for crossings}\label{BraidingChart} 
\end{figure} 

It is important to note that connecting the corresponding braid strands (outside of the diagram and around the braid axis) results in a tangle diagram that is isotopic to the original tangle diagram in the crossing box. The narrow condition for crossing boxes ensures that new pairs of braid strands obtained as a result of braiding different crossings do not intersect with each other, and that the order in which we braid the crossings does not matter. 

Once all crossings with at least one up-arc have been braided, we braid the free up-arcs using the \textit{basic braiding move} shown in Fig.~\ref{fig:basic-braiding} (see~\cite[Fig. 9]{KauLamb}). The basic braiding move involves cutting the up-arc at a point and pulling the top endpoint of the cut upward and the bottom endpoint downward so that the new pair of strands are vertically aligned with the upper subdivision point of the original up-arc. Moreover, the new pair of braid strands are extended in such a manner that they cross only virtually with any other arcs in the original diagram. As before, this is abstractly illustrated by marking  the ends of the new braid strands with virtual crossings.

 \begin{figure}[ht]
\[ \raisebox{-15pt}{\includegraphics[height=.5in]{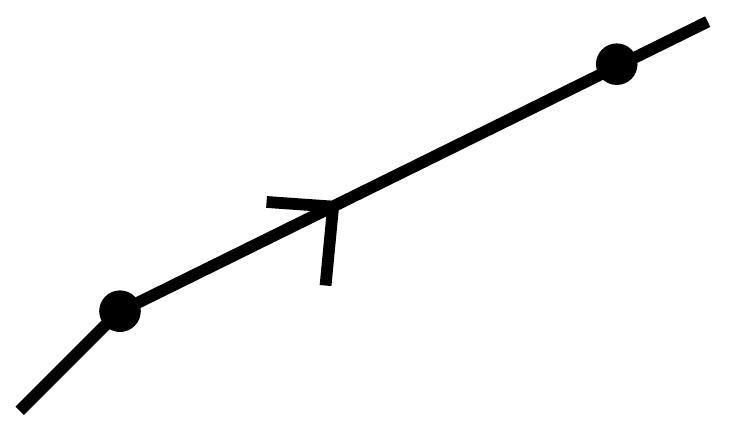}}\,\,{\longrightarrow} \,\,\raisebox{-40pt}{\includegraphics[height=1.2in]{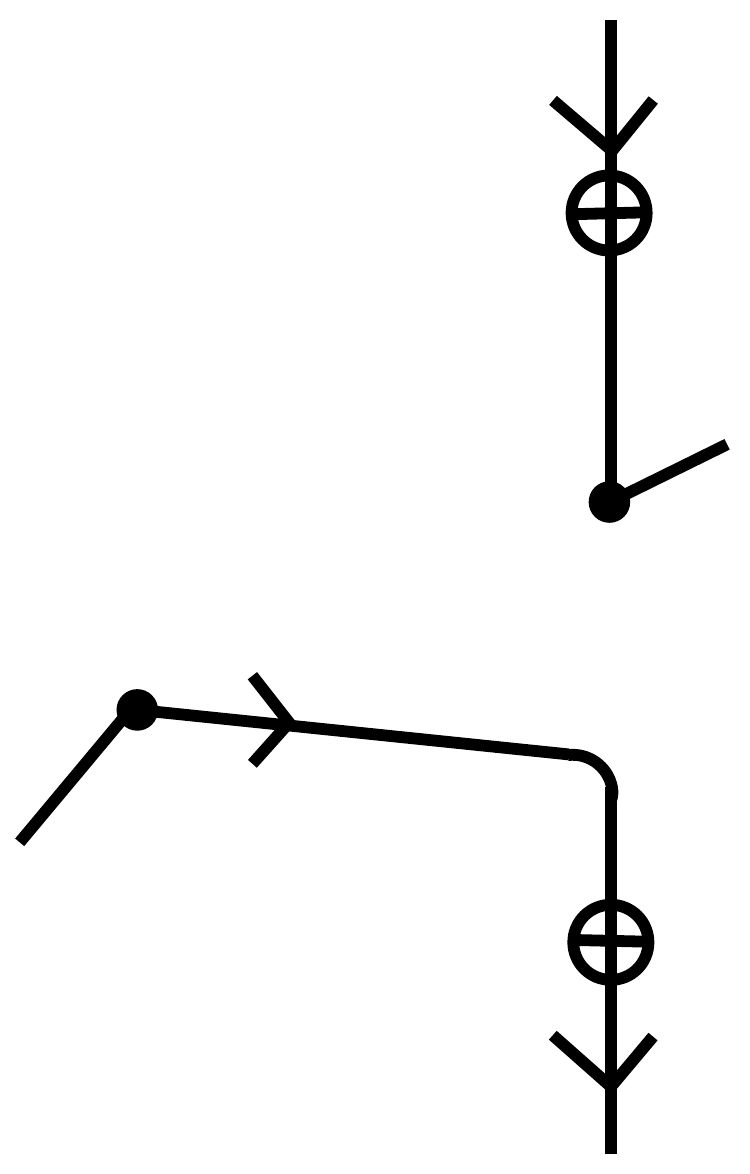}}\,\, \]
\caption{A basic braiding move}\label{fig:basic-braiding}
\end{figure}

It should be clear that connecting the new braid strands yields a virtual trivalent tangle diagram isotopic with the original diagram, via the detour move. Specifically, the basic braiding move creates a loop that is stretched around the axis of the virtual trivalent braid.

We remark that since the original diagram is assumed to be in general position, a subdivision point cannot coincide with a vertex. If a subdivision point on an up-arc was to coincide with a vertex, a basic braiding move (as defined here) would not be allowed. This is because in that case, connecting the corresponding pair of braid strands would result in a diagram that differs from the original diagram by a forbidden move.

We can now explain the reason behind the requirement in the definition of regular position that no subdivision points, crossings or vertices are vertically aligned. If two subdivision points, crossings or vertices are vertically aligned, then the braiding algorithm may result in triple points (a point where three arcs intersect) or braid strands going through a vertex. 

\begin{example} \label{Example1}
We now illustrate the braiding algorithm by applying it to the virtual STG diagram below (see Fig.~\ref{fig:braiding_example}). The notation RP stands for transforming the diagram into regular position and the circled areas represent the crossing or up-arc to be braided in the following step. We use the braiding convention for crossings established in the braiding chart given in Fig.~\ref{BraidingChart}.
\begin{figure}[ht]
\[
\raisebox{-.35in}{\includegraphics[height=.9in]{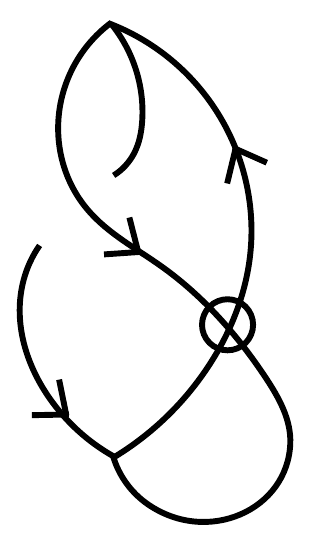}} 
\hspace{0.1in} \stackrel{\text{RP}} {\longrightarrow} \hspace{0.1in} 
\raisebox{-.35in}{\includegraphics[height=.9in]{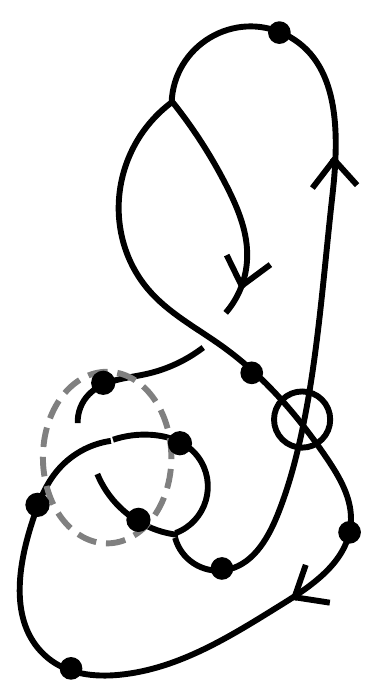}} 
\hspace{0.1in}
\hspace{0.1in}  \stackrel{\text{braiding}}{\longrightarrow}  \hspace{0.1in} 
\raisebox{-.6in}{\includegraphics[height=1.65in]{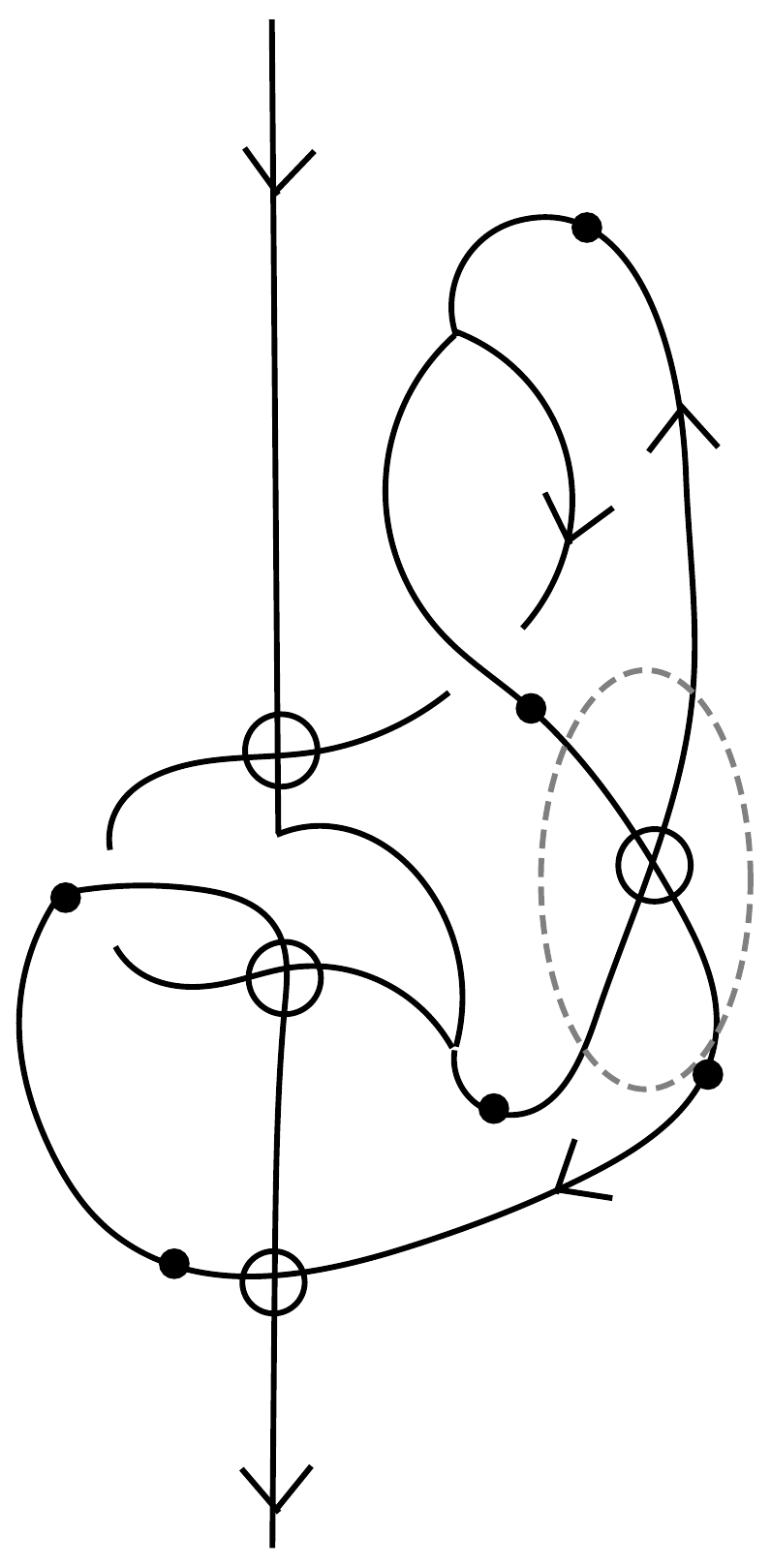}} \hspace{0.1in} 
\hspace{0.1in}  \stackrel{\text{braiding}}{\longrightarrow}  \hspace{0.1in} 
\raisebox{-.6in}{\includegraphics[height=1.7in]{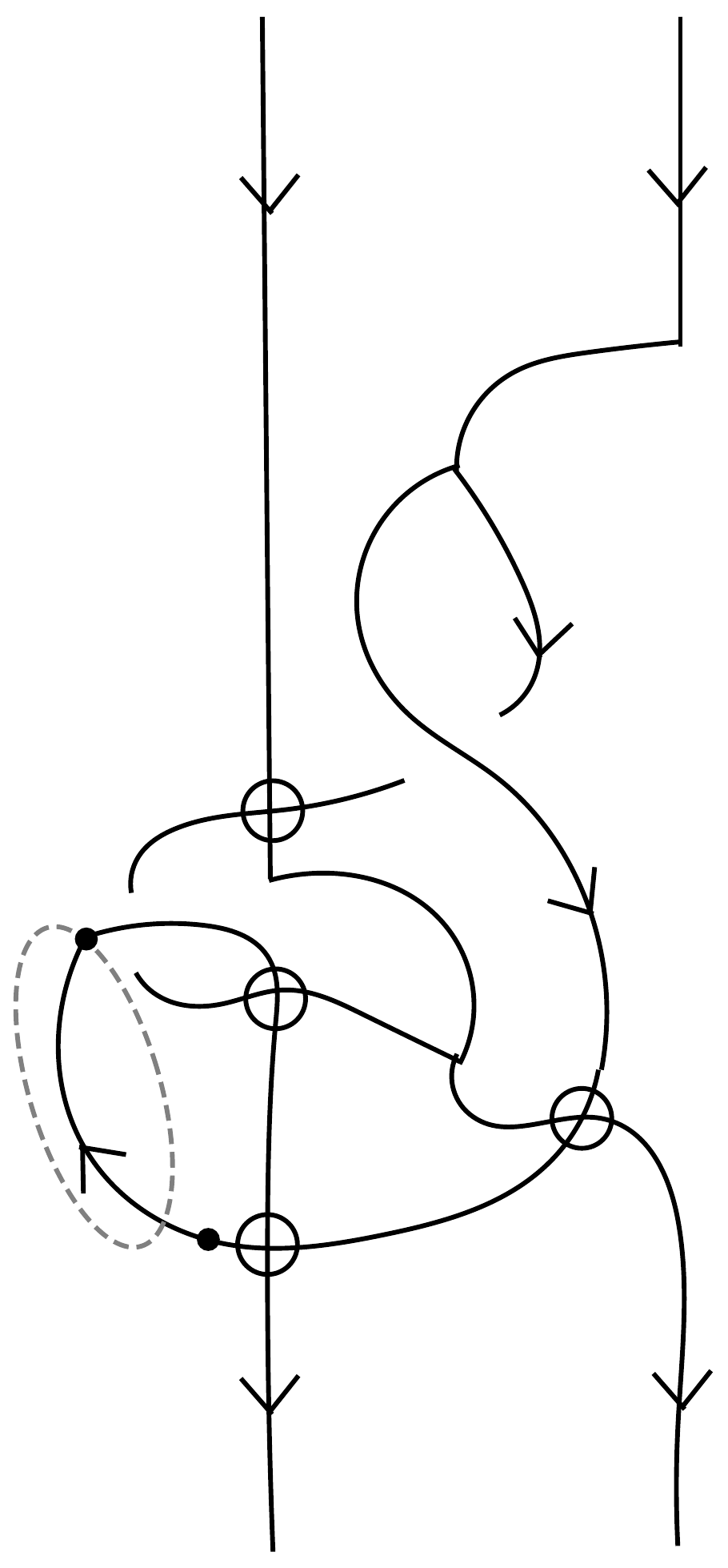}} \hspace{0.1in} 
\]
\[
\hspace{0.1in}  \stackrel{\text{braiding}}{\longrightarrow}  \hspace{0.1in} 
\raisebox{-.6in}{\includegraphics[height=1.7in]{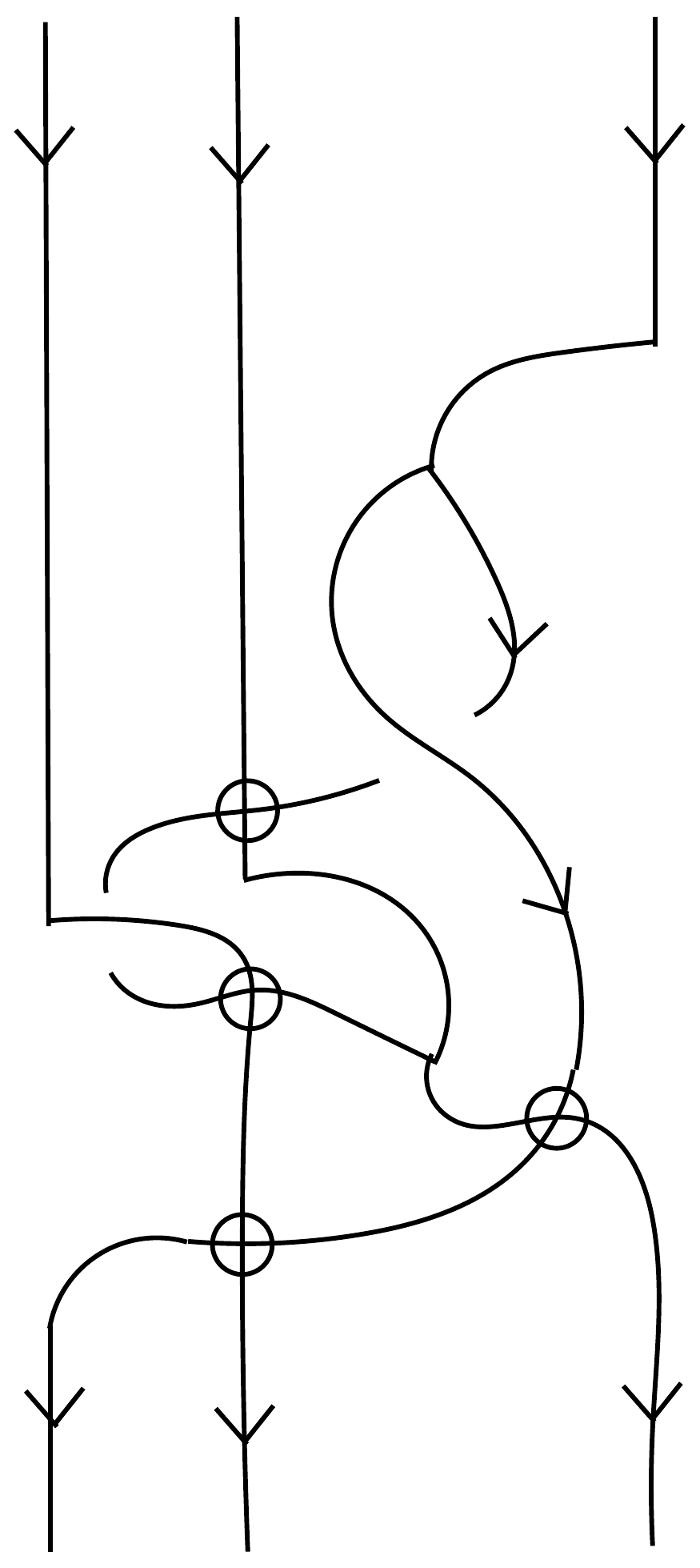}} \hspace{0.1in} 
\hspace{0.1in}  \underset{\text{isotopy}} {\overset{\text{planar}}{\longrightarrow}}
\hspace{0.1in}
\raisebox{-.45in}{\includegraphics[height=1.2in]{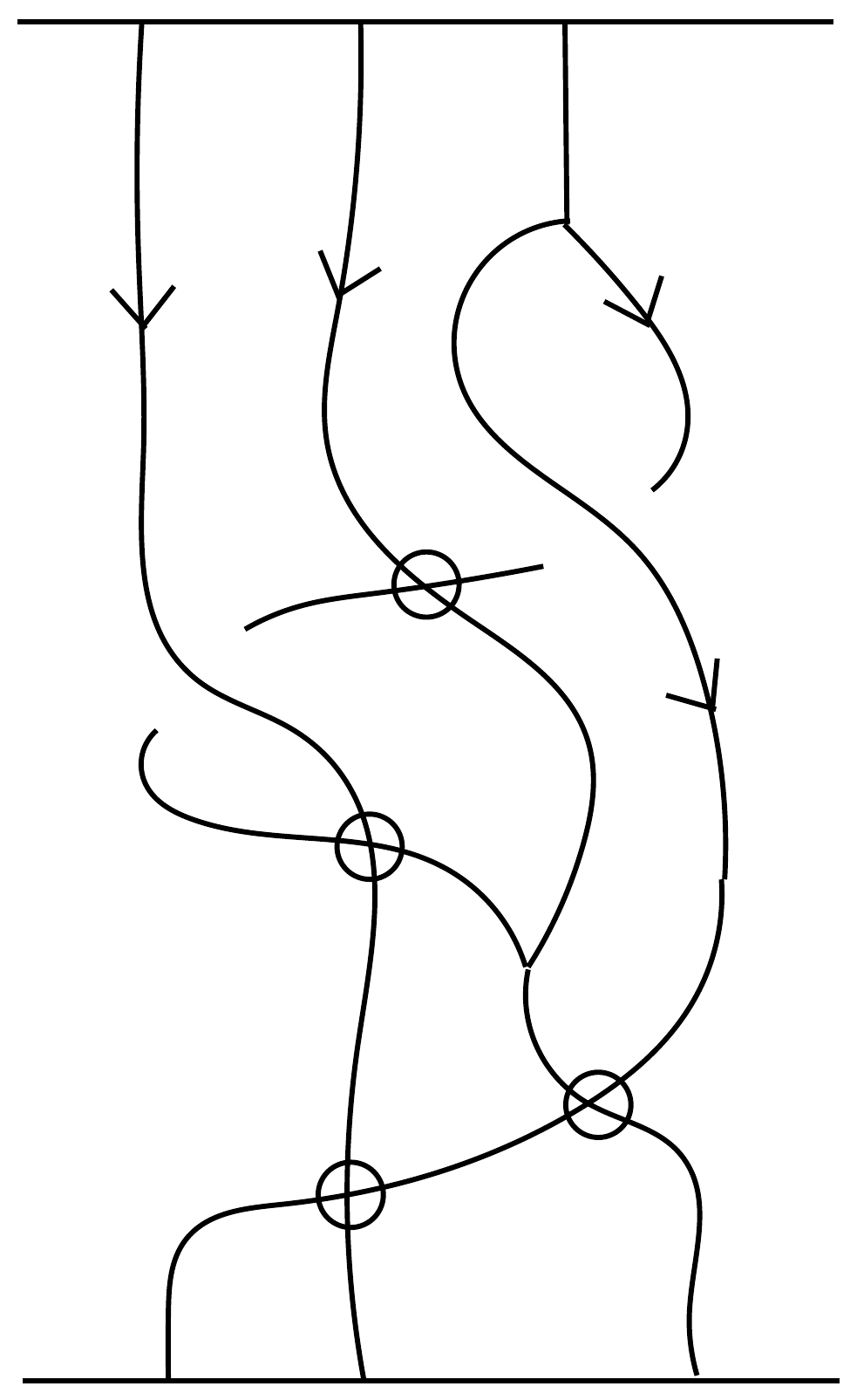}} \hspace{0.1in}
= \lambda_3 \sigma_3^{-1} v_2 \sigma_1^{-1} v_1 y_2 v_2 v_1
\]
\caption{An example of the braiding algorithm}\label{fig:braiding_example}
\end{figure}

The closure of the resulting braid above is a virtual STG diagram that is isotopy equivalent to the original diagram. 
\end{example}

The braiding algorithm described here braids any well-oriented virtual STG diagram in general position, producing a virtual trivalent braid whose closure is isotopic to the original diagram. Therefore, the following theorem holds.\\

\begin{theorem}[Alexander-type Theorem for virtual STGs]~\label{Alexander-type Theorem} 
Every well-oriented virtual spatial trivalent graph can be represented as the closure of a virtual trivalent braid.
\end{theorem}

\section{$TL_v$-equivalence for virtual trivalent braids}

Two non-equivalent virtual trivalent braids may yield isotopic  virtual STG diagrams via the closure operation. Therefore, the next natural consideration is to characterize virtual trivalent braids that have isotopic closures. For this purpose, we introduce an equivalence relation on virtual trivalent braids, $TL_v$-equivalence, which is an extension of the $L_v$-equivalence for virtual braids introduced in~\cite{KauLamb}.

\subsection{Virtual and trivalent $L_v$-moves} \label{ssec:TL_v}

In this section, all of the drawn diagrams are virtual trivalent braids. We start by reviewing a few definitions from~\cite{KauLamb}, (these are the next four definitions), and apply them in the context of virtual trivalent braids.

\begin{definition} \label{basic Lv-move}
A \textit{basic $L_v$-move} on a virtual trivalent braid involves cutting an arc and extending the cut-points such that the upper cut-point is extended downward and the lower cut-point is extended upward. The new pair of braid strands are vertically aligned and intersect virtually with the rest of the braid. As before, we represent this by marking with virtual crossings the locations where the new braid strands intersect the border of the braid box (see Fig.~\ref{basic-Lv}). We regard this move as being able to be performed going the opposite direction, as well.
\end{definition}

\begin{figure}[ht]
\[\raisebox{-.2in}{\includegraphics[height=.7in]{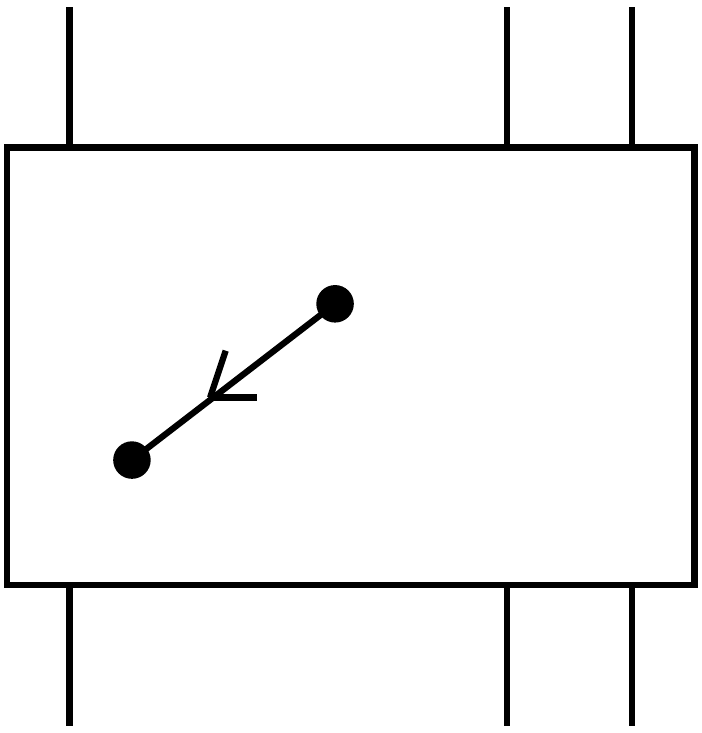}}\,\, 
\stackrel{\text{basic} \, L_v-\text{move}}{\longleftrightarrow} \,\, \raisebox{-.4in}{\includegraphics[height=1in]{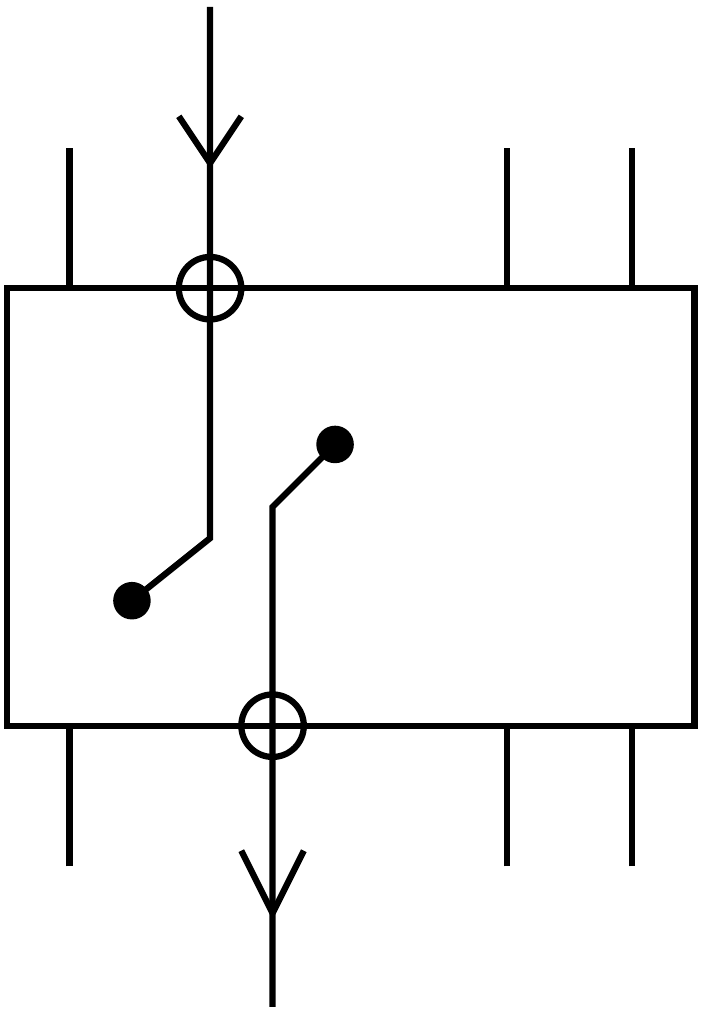}}
\vspace{-.5cm}\]
\caption{Basic $L_v$-move}~\label{basic-Lv}
\end{figure}

Note that the closures of the two braids involved in a basic $L_v$-move differ by a detour move and thus they are isotopic virtual STG diagrams.

Beyond the basic $L_v$-move, we can make other choices when extending the new braid strands obtained by cutting an arc of the braid; these choices allow for the creation of a virtual or real crossing and give rise to the following two types of $L_v$-moves.

\begin{definition}
A \textit{real $L_v$-move} on a virtual trivalent braid is similar to a basic $L_v$-move, except that the new move introduces a classical crossing (positive or negative). Since the crossing can be placed on either right or left side of the original strand, we  refer to these as \textit{right} or \textit{left real $L_v$-moves}. As before, the new braid strands intersect virtually the rest of the braid diagram. Likewise, this move can also be performed going in the opposite direction.  Figure~\ref{real Lv image} shows both left and right real $L_v$-moves.
\end{definition}

\begin{definition}
A \textit{virtual $L_v$-move} (or $vL_v$-move) is similar to the real $L_v$-move only that in this case, the new crossing is virtual. There are two versions of the $vL_v$-move, namely \textit{left} and \textit{right $vL_v$-moves} (see Fig.~\ref{virt Lv}). In addition, the move can also be performed going in the opposite direction.
\end{definition}

\begin{figure}[ht]
\[
\raisebox{-.9cm}{\includegraphics[height=.9in]{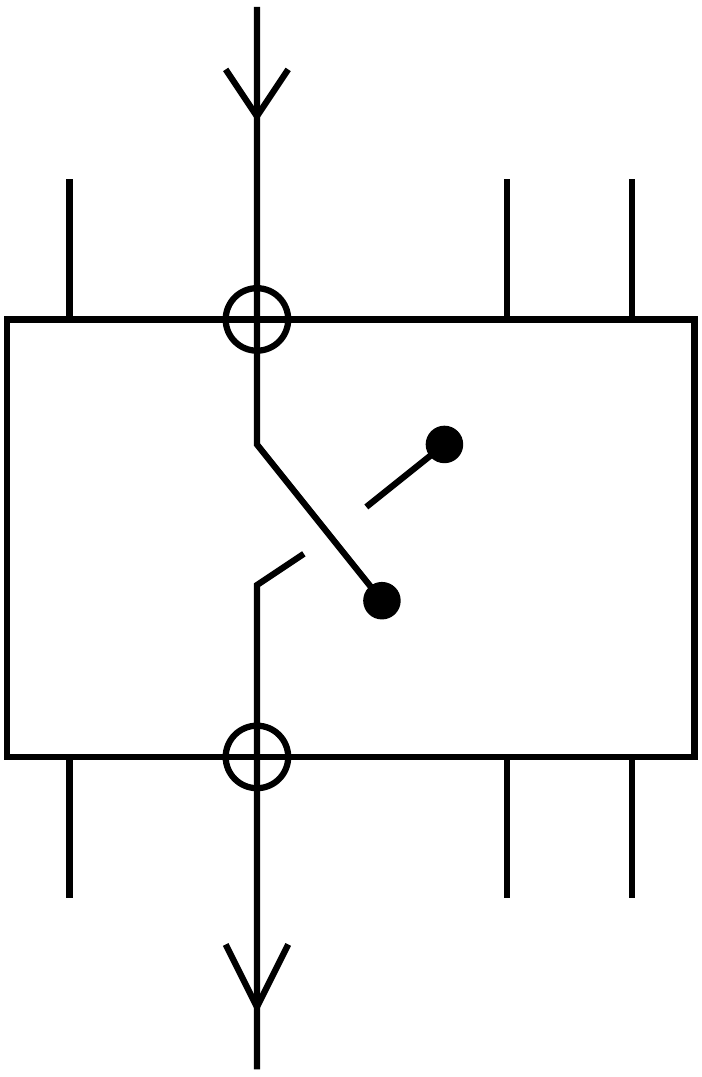}}\,\,
\underset{\text{$L_v$-move}}{\overset{\text{left real}}{\longleftrightarrow}}  \,\,
\raisebox{-.75cm}{\includegraphics[height=.7in]{basic-L1}}\,\,
 \underset{\text{$L_v$-move}}{\overset{\text{right real}}{\longleftrightarrow}} \,\, \raisebox{-.9cm}{\includegraphics[height=0.9in]{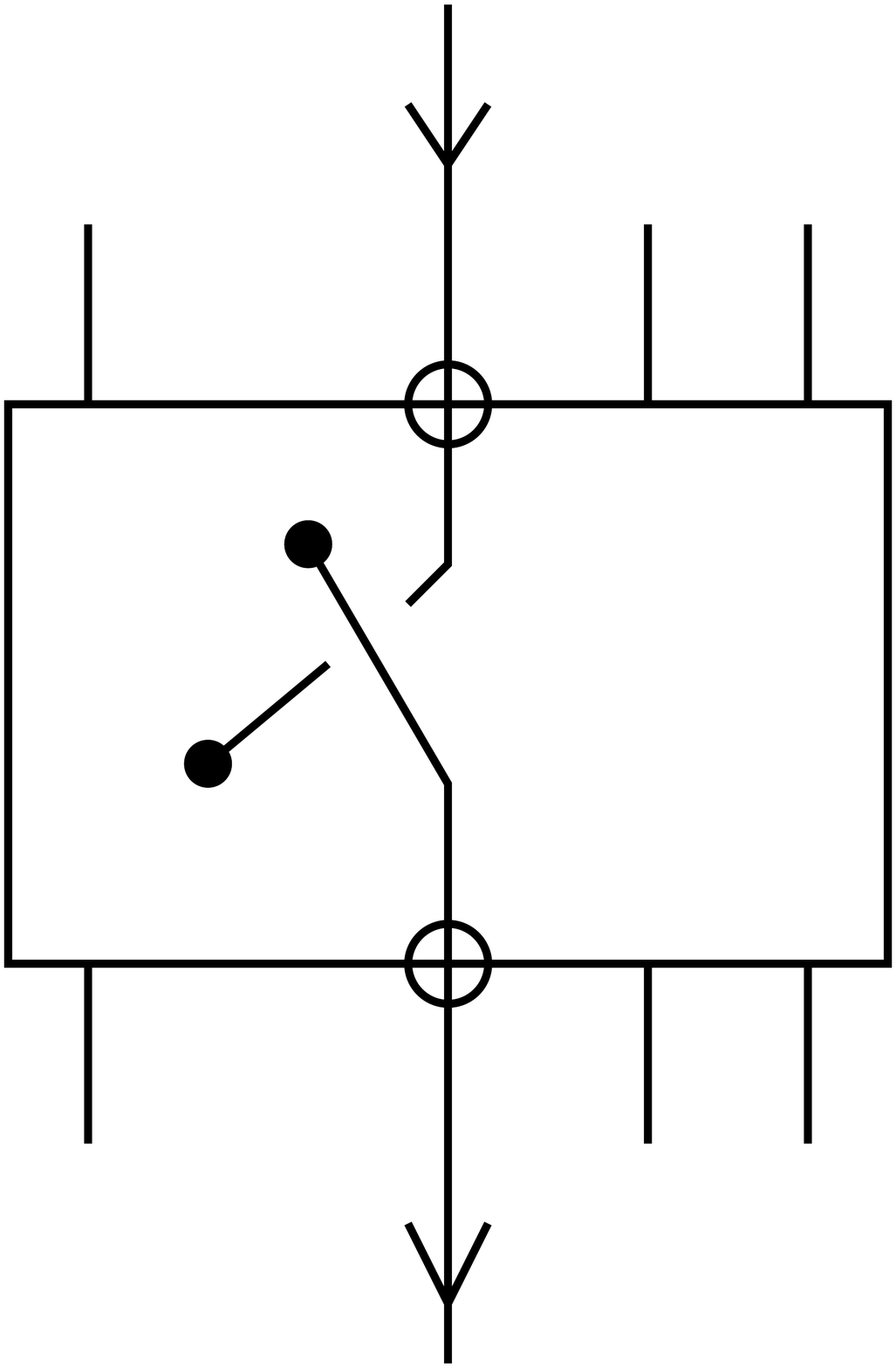}}
\]
\caption{Left and right real $L_v$-moves}\label{real Lv image}
\end{figure}

\begin{figure}[ht]
\[ \raisebox{-32pt}{\includegraphics[height=1in]{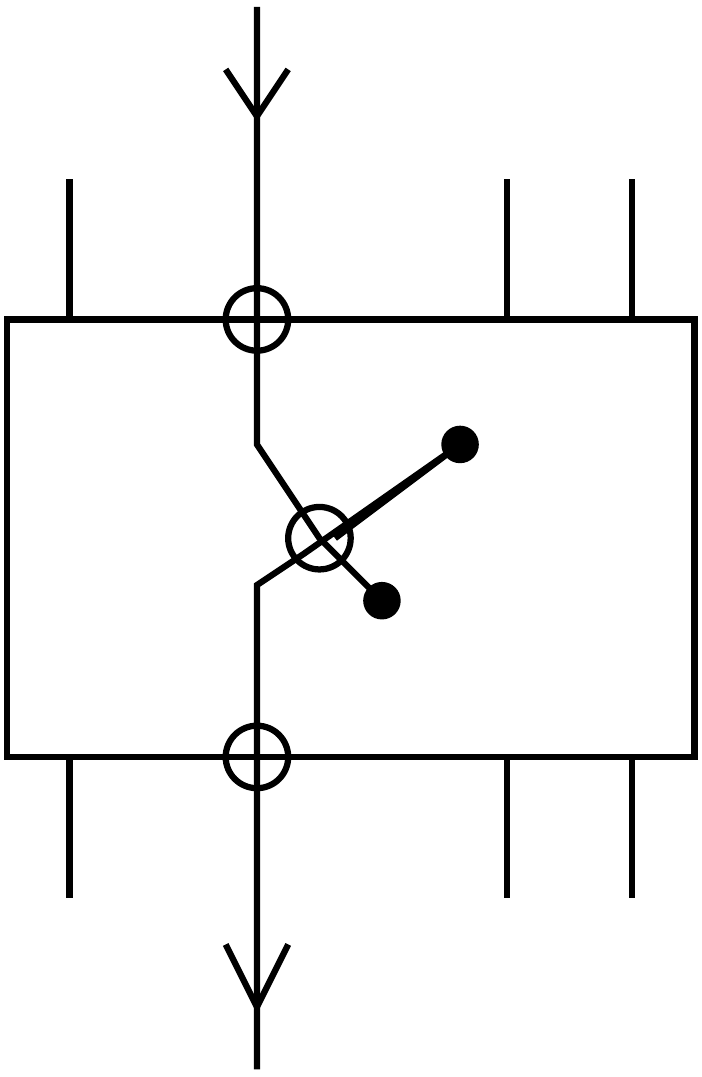}} 
\underset{\text{$L_v$-move}}{\overset{\text{left virtual}}{\longleftrightarrow}}  
 \,\,\raisebox{-22pt}{\includegraphics[height=.75in]{basic-L1}}
\,\,
\underset{\text{$L_v$-move}}{\overset{\text{right virtual}}{\longleftrightarrow}}
\raisebox{-32pt}{\includegraphics[height=1in]{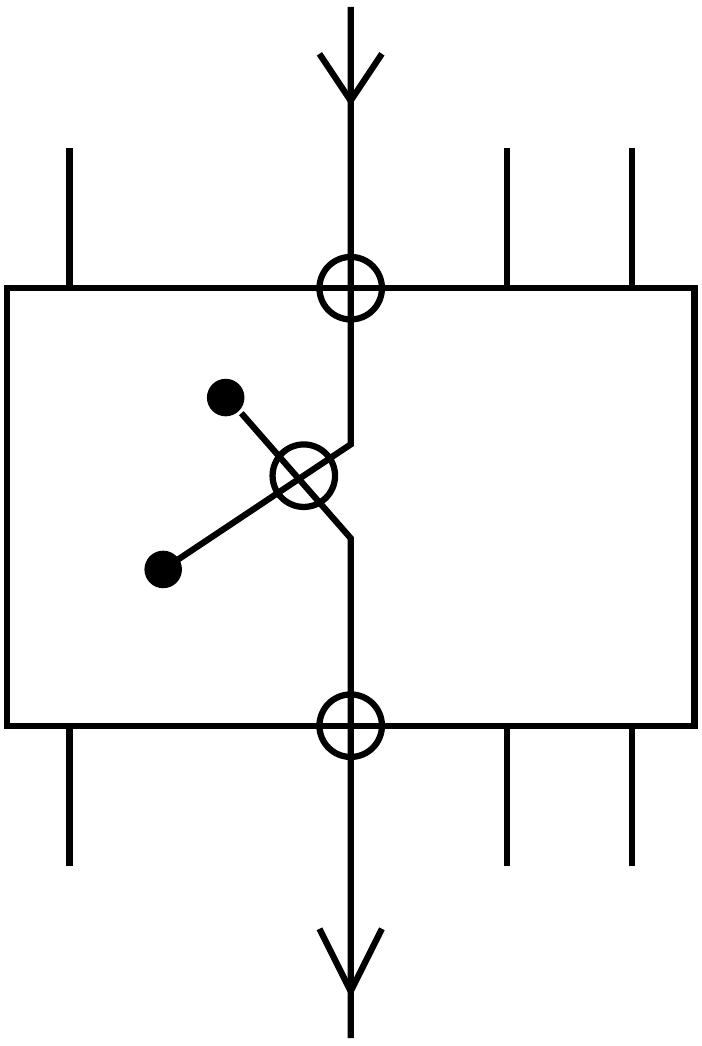}}
\]
\caption{Left and right virtual $L_v$-move}\label{virt Lv}
\end{figure}

Note that the closures of the braids involved in a real or virtual $L_v$-move are isotopic diagrams via a detour move and an $R1$ move or $V1$ move, respectively.

\begin{definition}
A \textit{threaded $L_v$-move} on a virtual trivalent braid consists of a virtual $L_v$-move in which the two new braids are pulled over or under a third strand, called the \textit{thread}. The move can be performed either on the left or on the right of the arc of the braid, giving rise to the \textit{left} and \textit{right} versions of the move. This process is called threading and is equivalent to performing an $R2$ move before cutting the arc of the braid and stretching the endpoints as explained earlier. Depending whether the virtual kink is pulled over or under the thread, the move is called an \textit{over-threaded $L_v$-move} or an \textit{under-threaded $L_v$-move}. We allow these moves to be performed in the opposite direction, as well.
\end{definition}

Figure~\ref{fig:threaded-move} shows both left and right under-threaded $L_v$-moves. Note that a threaded move does not involve isotopic braids, due to the forbidden moves. However, the closures of the two braids involved in the move are isotopic. 

In addition, if more than one thread is used in the move, we create a \textit{multi-threaded $L_v$-move}. (See~\cite[Fig. 14]{KauLamb}.)

\begin{figure}[ht]
\[
\reflectbox{\raisebox{-35pt}{\includegraphics[height=1in]{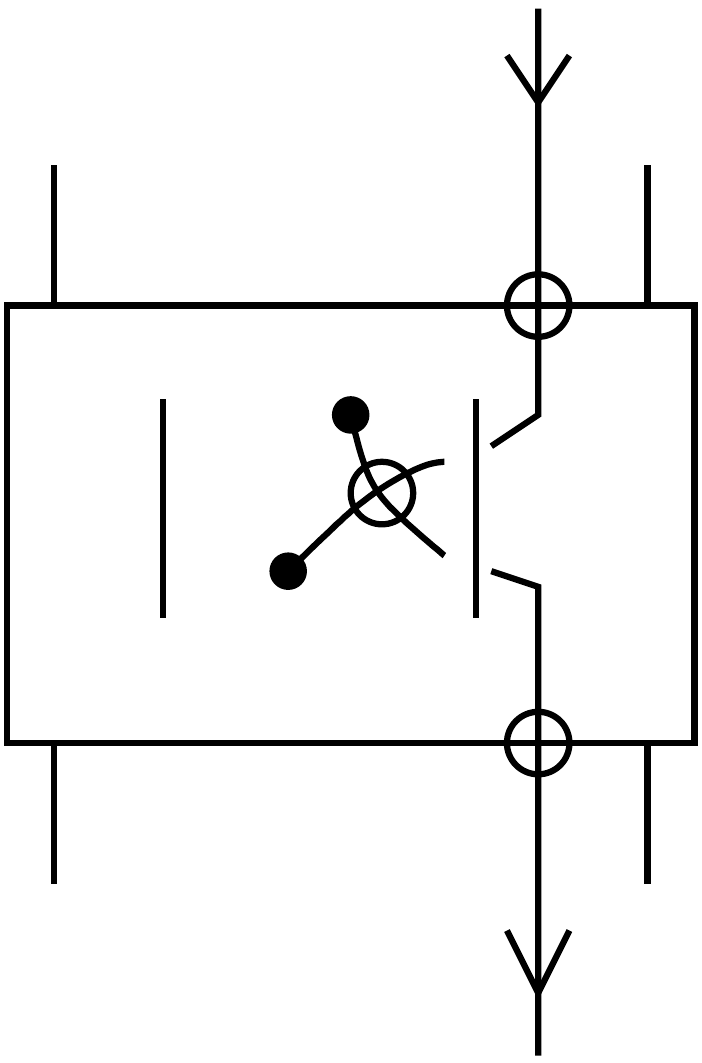}}}\,\,\,\, \underset{\text{$L_v$-move}}{\overset{\text{left under-threaded}}{\longleftrightarrow}} \,\,\,\, \raisebox{-25pt}{\includegraphics[height=.75in]{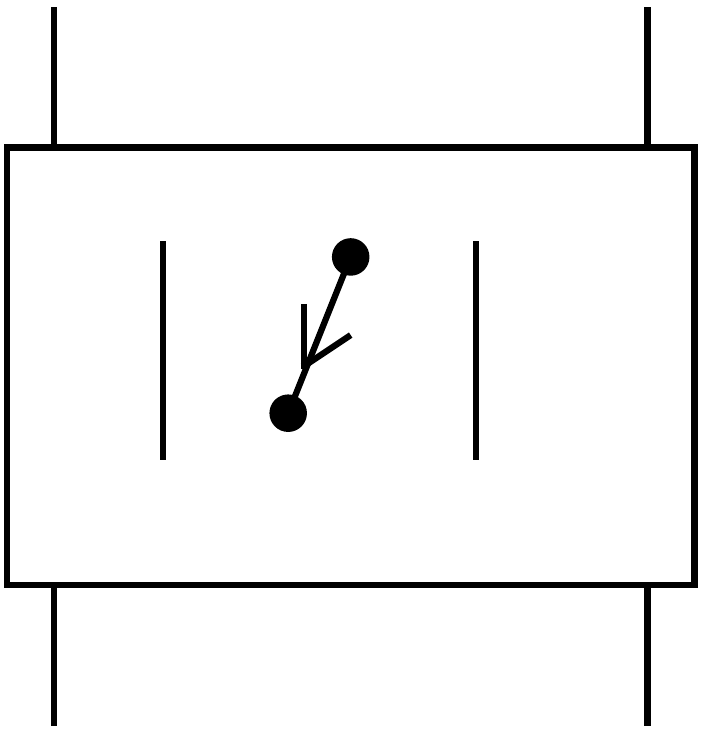}}\,\,\,\, \underset{\text{$L_v$-move}}{\overset{\text{right under-threaded}}{\longleftrightarrow}} \,\, \,\, \raisebox{-35pt}{\includegraphics[height=1in]{right-threaded-move}}
\]
\caption{Left and right under-threaded $L_v$-moves}\label{fig:threaded-move}
\end{figure}

Through working in the setting of virtual trivalent braids, we must extend the set of previously established $L_v$-moves to enclose additional moves near a trivalent vertex and further define an equivalence on the set of virtual trivalent braids.

\begin{definition} 
The \textit{trivalent $L_v$-move} (or $TL_v$-move) is performed near a $\lambda$-vertex and involves a braid strand that first crosses virtually one of the lower legs of the vertex and then classically the other lower leg. The classical crossings on the two sides of the move are of opposite type. This move comes in two versions, namely \textit{right} (Fig.~\ref{right TL}) and \textit{left} (Fig.~\ref{left TL}). 
\end{definition}

\begin{figure}
\[
\raisebox{-0.75in}{\includegraphics[height=1.6in]{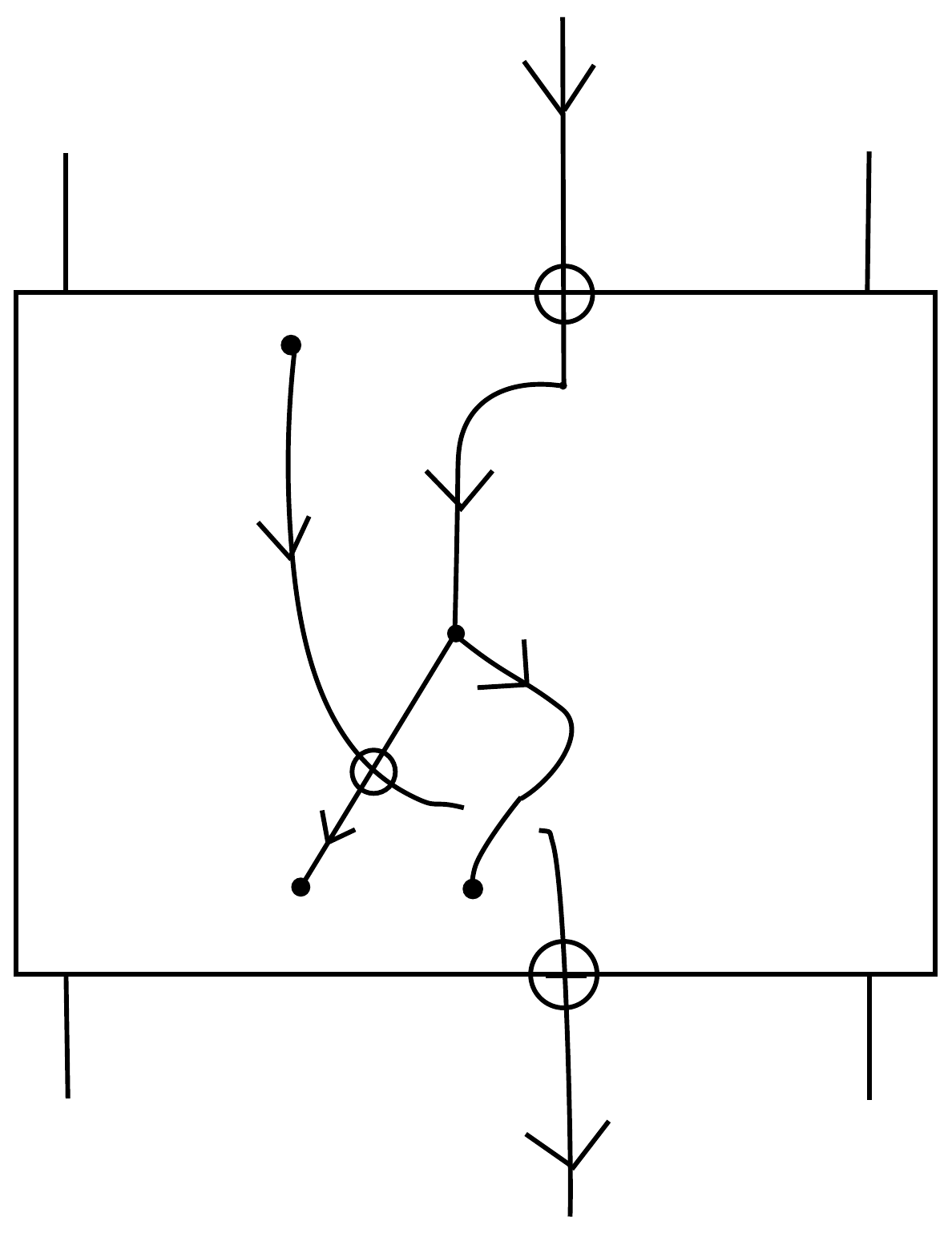}}\hspace{0.5cm} \underset{\text{$TL_v$-move}}{\overset{\text{right}}{\longleftrightarrow}} \hspace{0.5cm} \raisebox{-0.75in}{\includegraphics[height=1.6in]{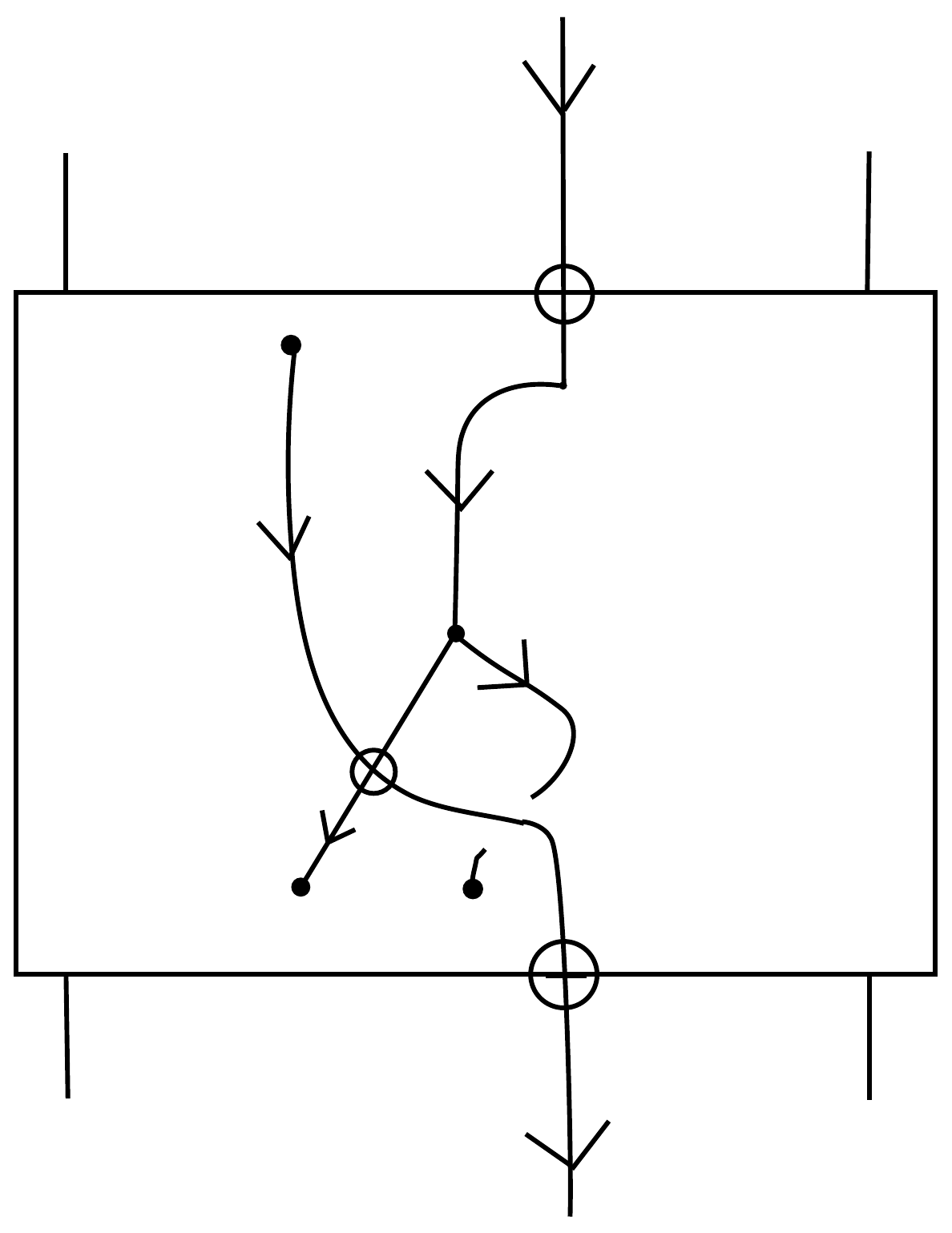}}
\]
\caption{Right trivalent $L_v$-move} \label{right TL}
\end{figure}
\begin{figure}
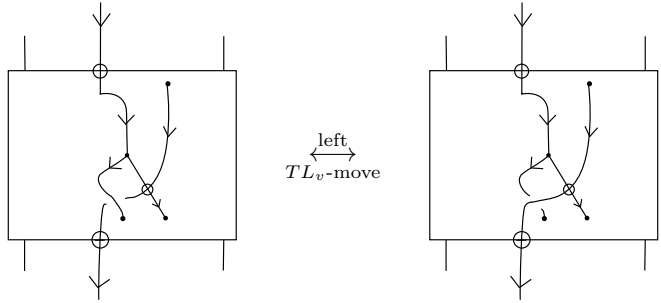

\[
\raisebox{-0.75in}{\reflectbox{\includegraphics[height=1.6in]{boxTL4rn-3b}}}\hspace{0.5cm} \underset{\text{$TL_v$-move}}{\overset{\text{left}}{\longleftrightarrow}} \hspace{0.5cm} \raisebox{-0.75in}{\reflectbox{\includegraphics[height=1.6in]{boxTL4rp-3b}}}
\]
\caption{Left trivalent $L_v$-move} \label{left TL}
\end{figure}

Note that the trivalent $L_v$-moves cannot be applied in a braid, since they do not involve isotopic braids. However, the closures of the two braids involved in such move are isotopic virtual STG diagrams.

\begin{definition}
Given a virtual trivalent braid $\omega \in VTB_n^n$, we say that the braids $\omega \sigma _i^{\pm 1} \sim \sigma_i^{\pm 1} \omega$, for $1\leq i \leq n-1$, are related by \textit{real conjugation}. Similarly, the braids $\omega v_i \sim v_i \omega$ are related by \textit{virtual conjugation}, where $1\leq i \leq n-1$ (see Fig. ~\ref{Conj}). 
\end{definition}

\begin{figure}[ht]
\[ 
\raisebox{-35pt}{\includegraphics[height=1in]{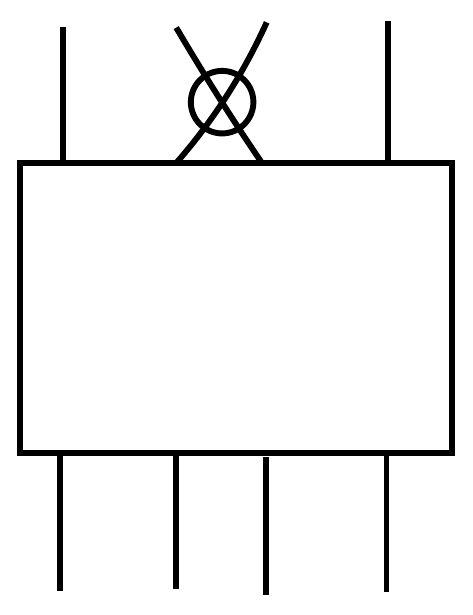}}\,\, \sim \,\, \raisebox{-35pt}{\includegraphics[height=1 in]{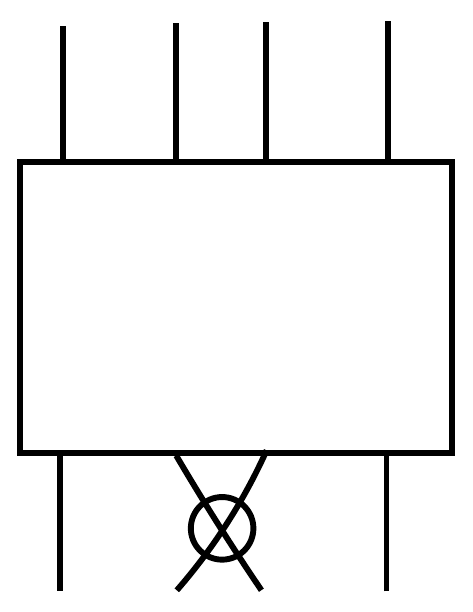}}\hspace{1cm}
\raisebox{-35pt}{\includegraphics[height=1 in]{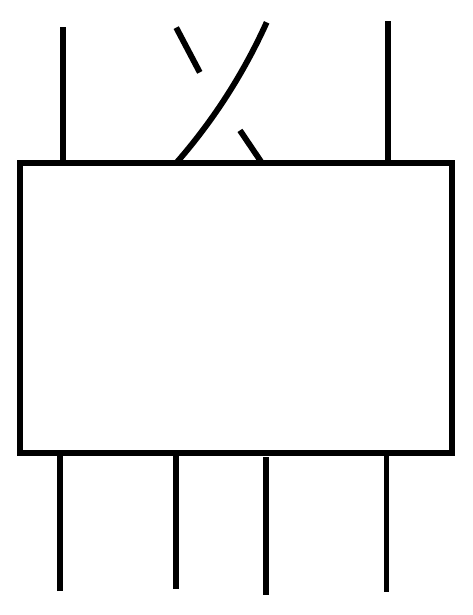}}\,\,\sim \,\, \raisebox{-35pt}{\includegraphics[height=1in]{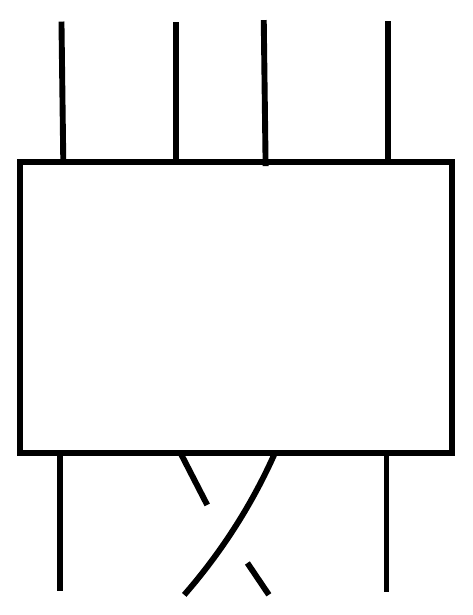}}
 \put(-265, 0){\fontsize{9}{9}$\alpha$}
 \put(-190, 0){\fontsize{9}{9}$\alpha$}
  \put(-105, 0){\fontsize{9}{9}$\alpha$}
  \put(-30, 0){\fontsize{9}{9}$\alpha$}
   \]
  \caption{Virtual and real conjugation}\label{Conj} 
\end{figure}

Note that since $v_i$ is its own inverse by the braid version of $V2$, virtual conjugation is equivalent to $\omega \sim v_i \omega v_i$ while the real conjugation takes the equivalent form $\omega \sim \sigma_i^{\mp 1} \omega \sigma_i^{\pm 1}$. In addition, it is easy to see that the closures of two braids related by conjugation are isotopic virtual STG diagrams through an $R2$ or $V2$ move.

\subsection{$TL_v$-equivalence and Markov-type theorem based on $L_v$-moves} \label{ssec:MarkovThm1}
To establish our first Markov-type theorem, we seek an equivalence relation on virtual trivalent braids which involves the minimum number of moves.


\begin{remark}
The following results have been proven in ~\cite{KauLamb}:
\begin{itemize}
\item Virtual conjugation follows from basic and virtual $L_v$-moves together with braid isotopy (specifically, braid detour)~\cite[Lemma 1]{KauLamb}.
\item Basic and left virtual $L_v$-moves follow from braid isotopy and right virtual $L_v$-moves~\cite[Lemma 2]{KauLamb}.
\item Over-threaded $L_v$-moves can be obtained through a sequence of real conjugation, right real $L_v$-moves, right virtual $L_v$-moves, and right and left under-threaded $L_v$-moves~\cite[Lemma 4]{KauLamb}.

\end{itemize}
\label{remark1}
\end{remark}

\begin{lemma}
The left $TL_v$-move follows from virtual trivalent braid isotopy, the right $TL_v$-move, and $L_v$-moves.
\end{lemma}

\begin{proof}
The proof is given in Fig.~\ref{left TLmove}.
\end{proof}

\begin{figure}[ht]
\[\raisebox{-.7in}{\includegraphics[height=1.5in]{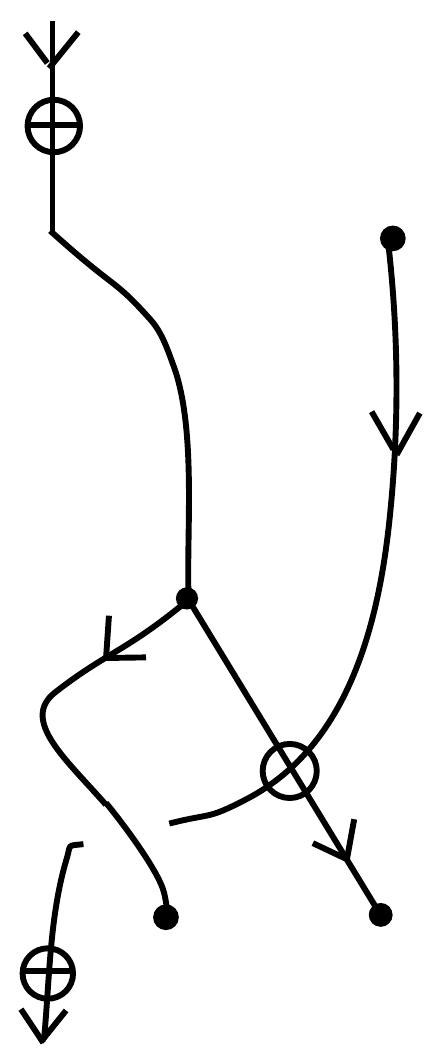}}
\hspace{0.05in} \underset{\text{$L_v$-move}}{\overset{\text{right real}}{\longleftrightarrow}}
\hspace{0.05in}
\raisebox{-.7in}{\includegraphics[height=1.5in]{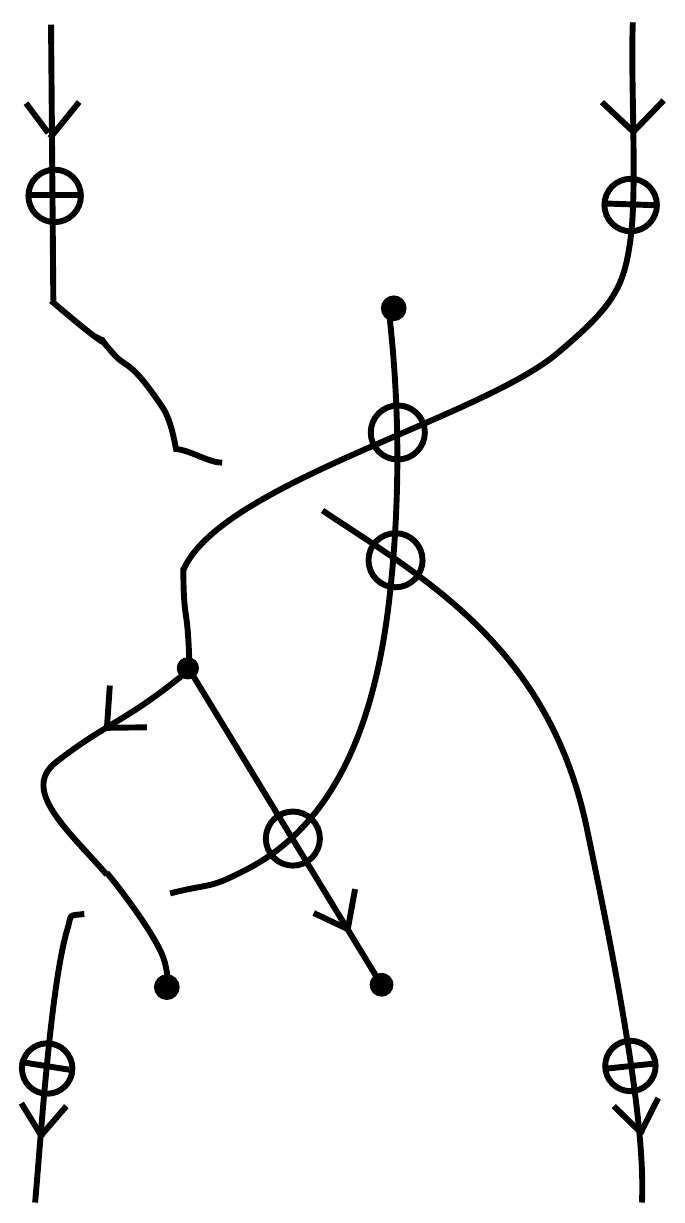}}
\hspace{0.05in}
\underset{\text{$R4$}}{\overset{\text{braid}}{\longleftrightarrow}} 
\hspace{0.05in}
\raisebox{-.7in}{\includegraphics[height=1.5in]{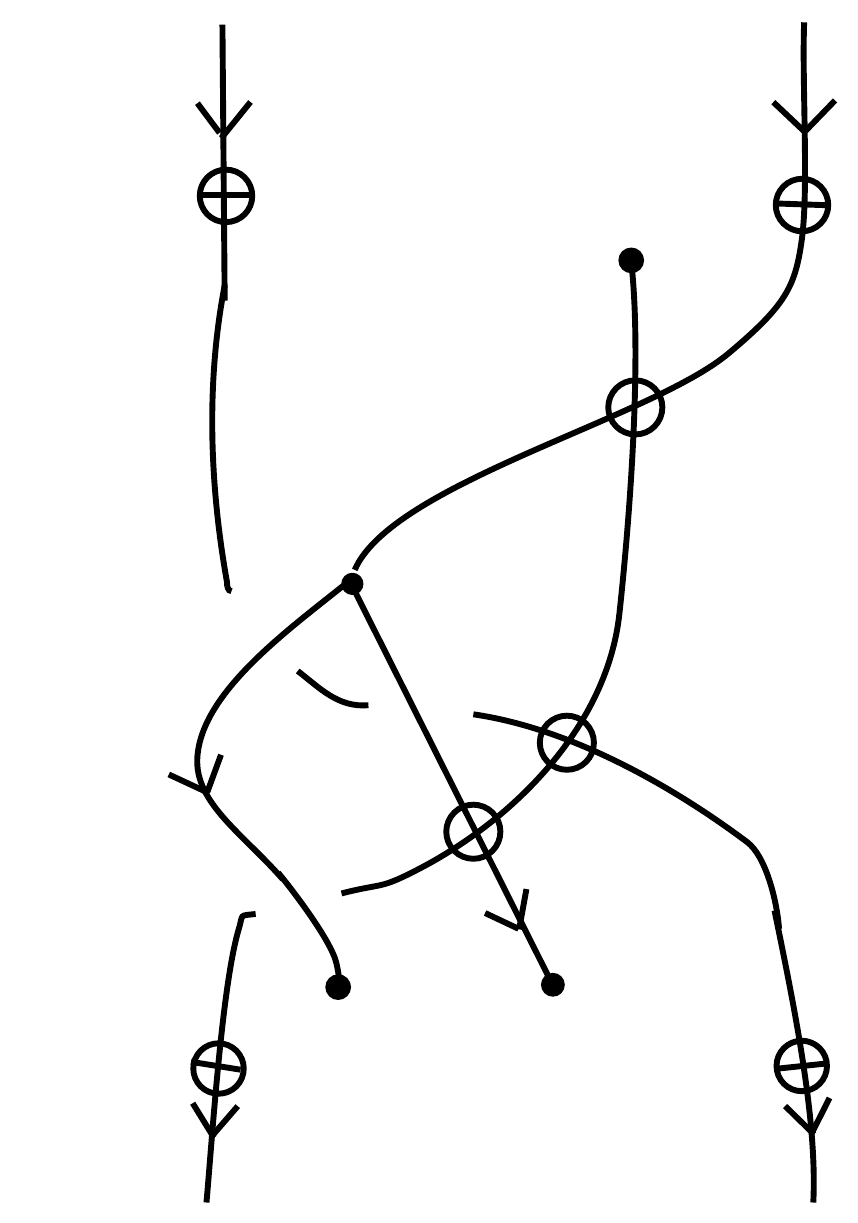}} \hspace{0.05in}\underset{\text{$VR3$}}{\overset{\text{braid}}{\longleftrightarrow}}\]
\\
\[\raisebox{-.7in}{\includegraphics[height=1.5in]{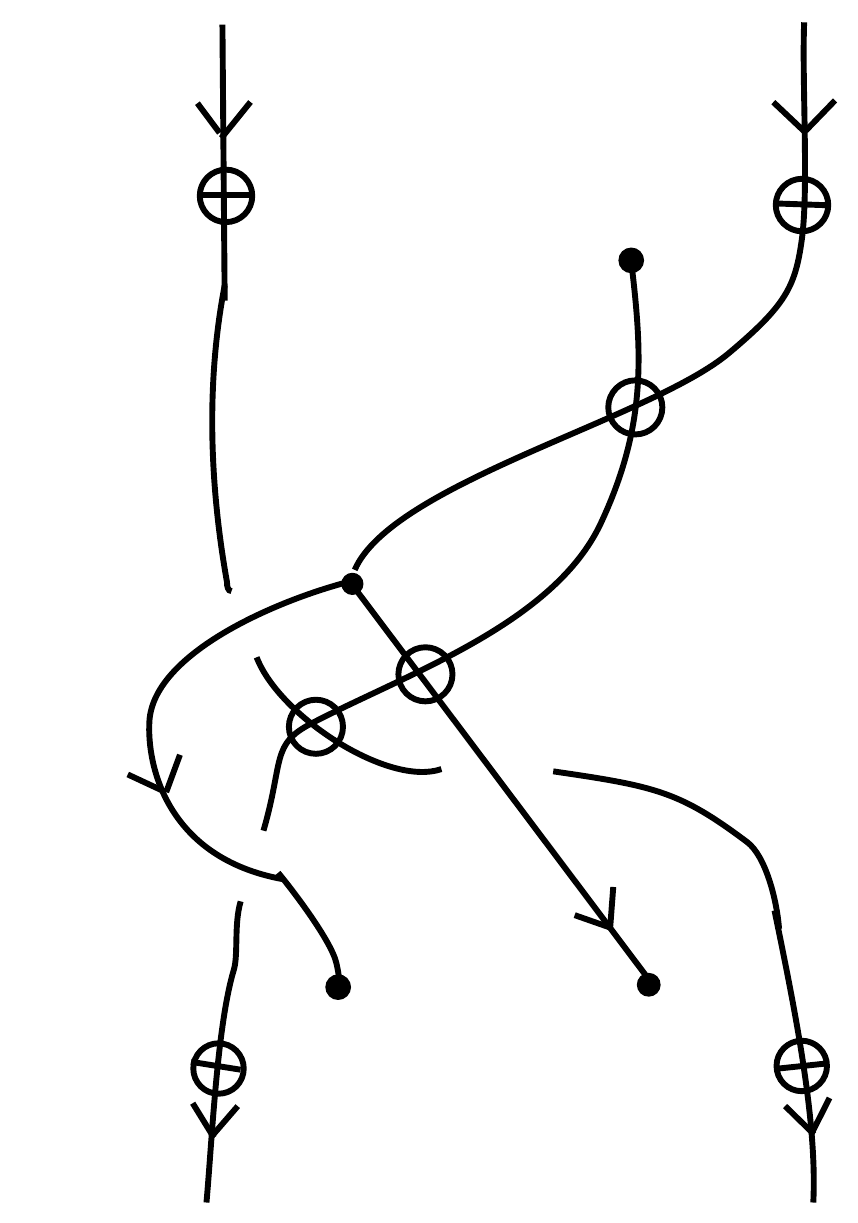}}\hspace{0.05in}\underset{\text{$vL_v$-move}}{\overset{\text{under-threaded}}{\longleftrightarrow}}\hspace{0.05in}\raisebox{-.7in}{\includegraphics[height=1.5in]{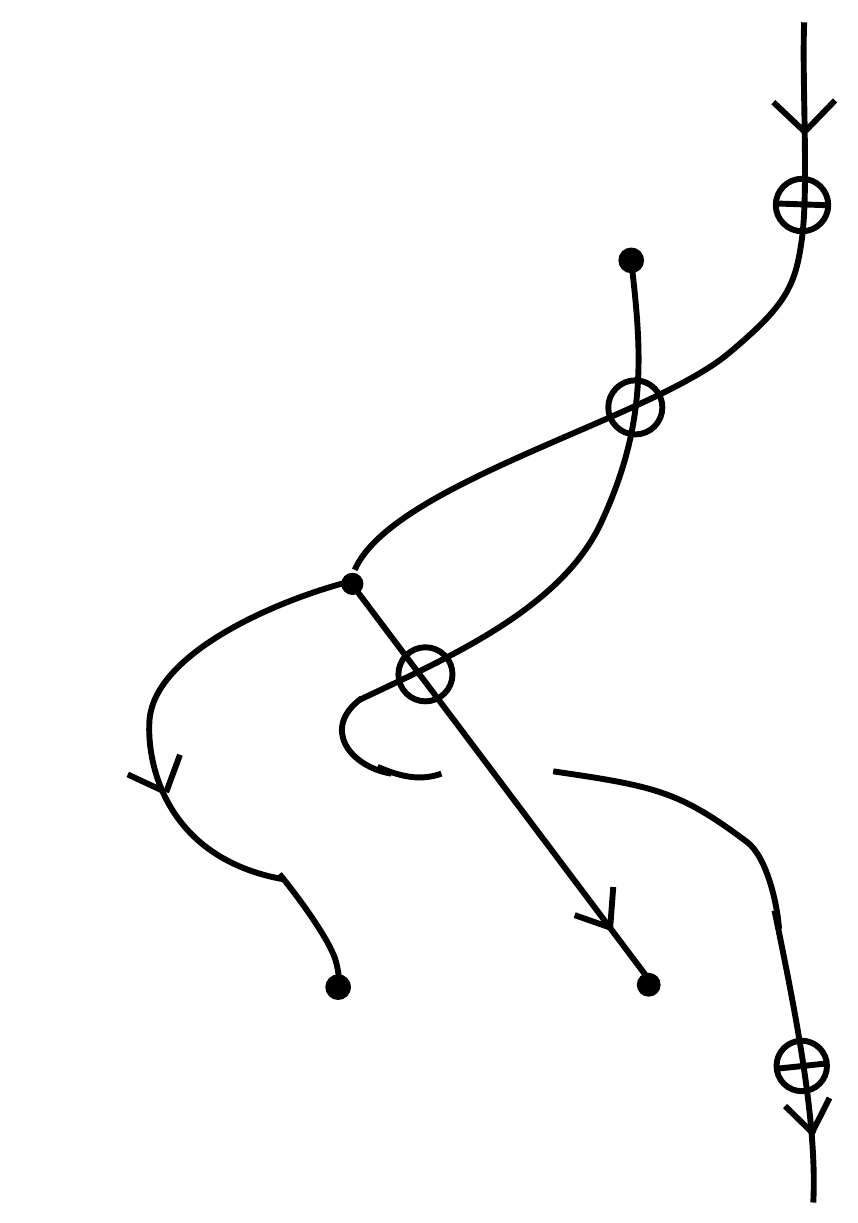}}
\hspace{0.05in} \underset{\text{$R4$}}{\overset{\text{braid}}{\longleftrightarrow}}\hspace{0.05in}
\raisebox{-.7in}{\includegraphics[height=1.5in]{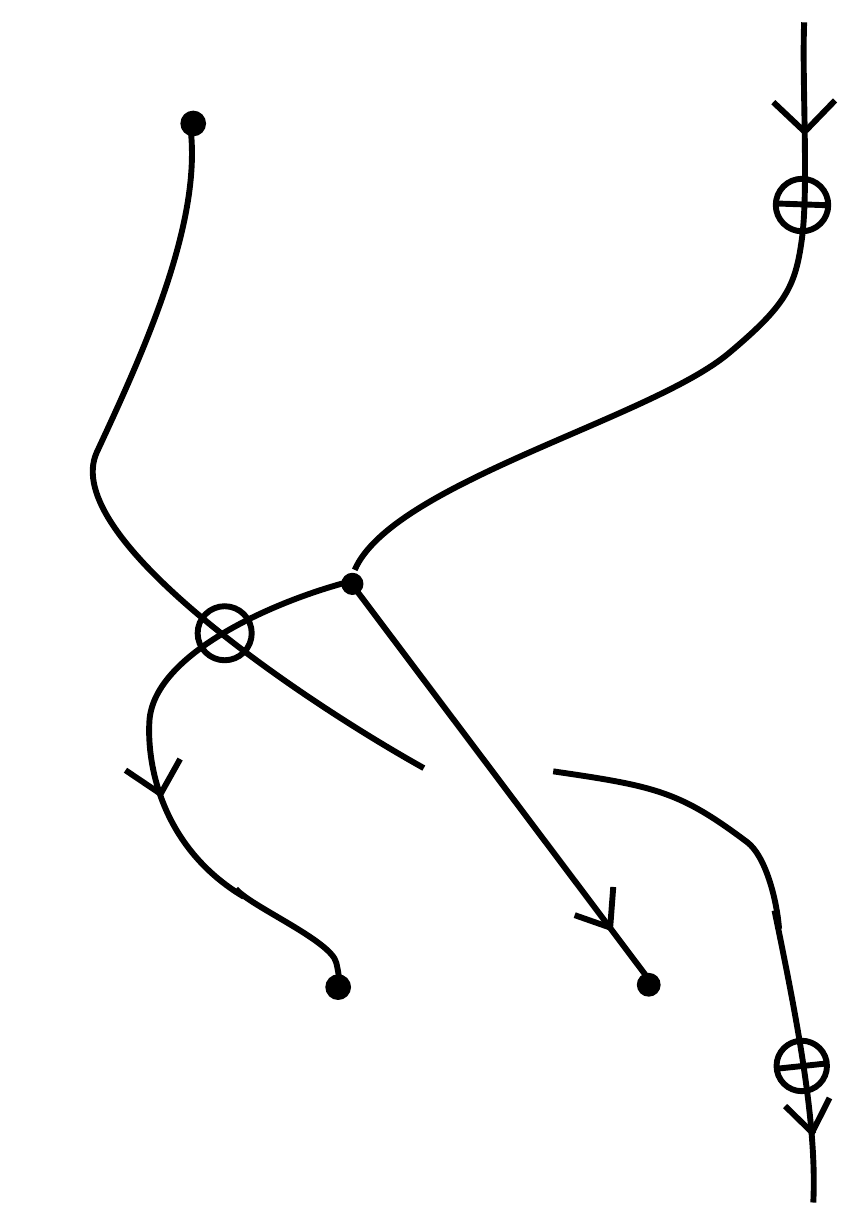}}\hspace{0.05in}
\underset{\text{$TL_v$-move}}{\overset{\text{right}}{\longleftrightarrow}}
\]
\\
\[\raisebox{-.7in}{\includegraphics[height=1.5in]{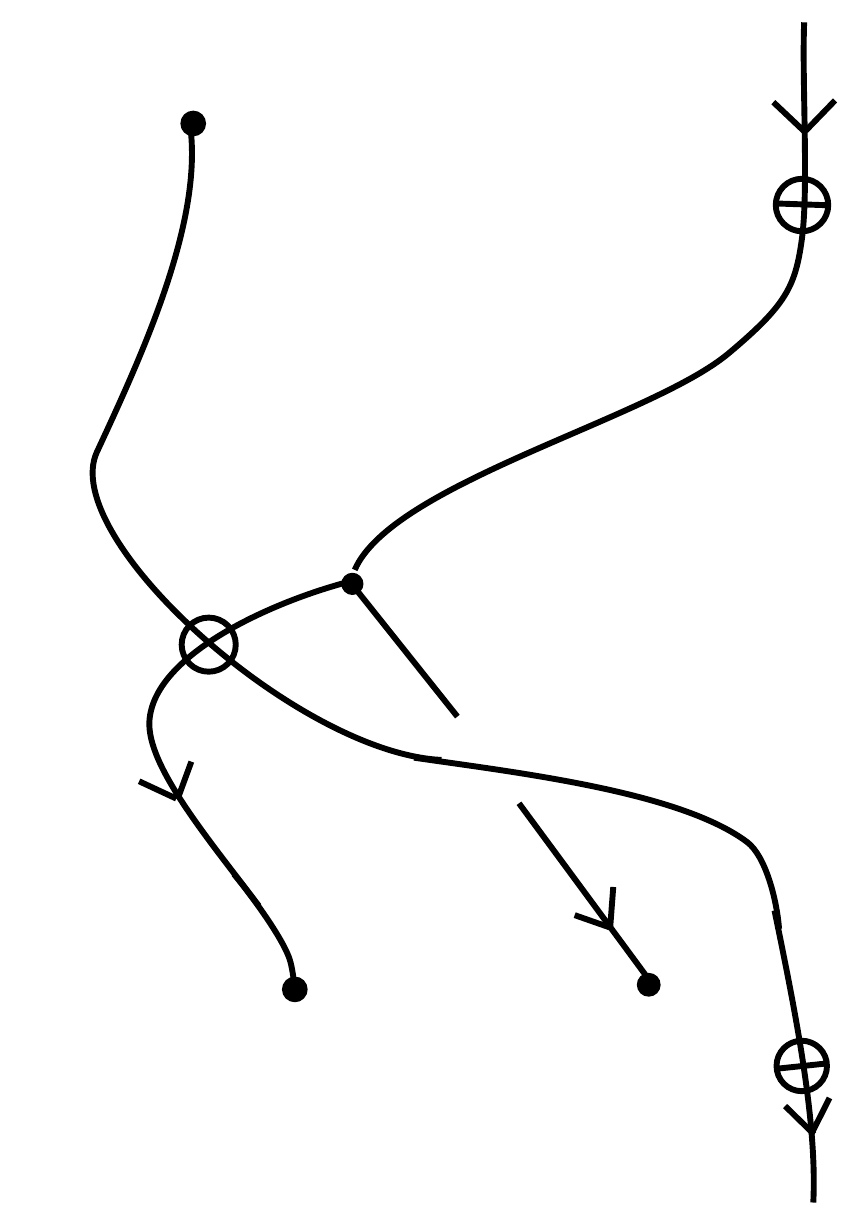}} \hspace{0.05in}\underset{\text{$R4$}}{\overset{\text{braid}}{\longleftrightarrow}} \hspace{0.05in}
\raisebox{-.7in}{\includegraphics[height=1.5in]{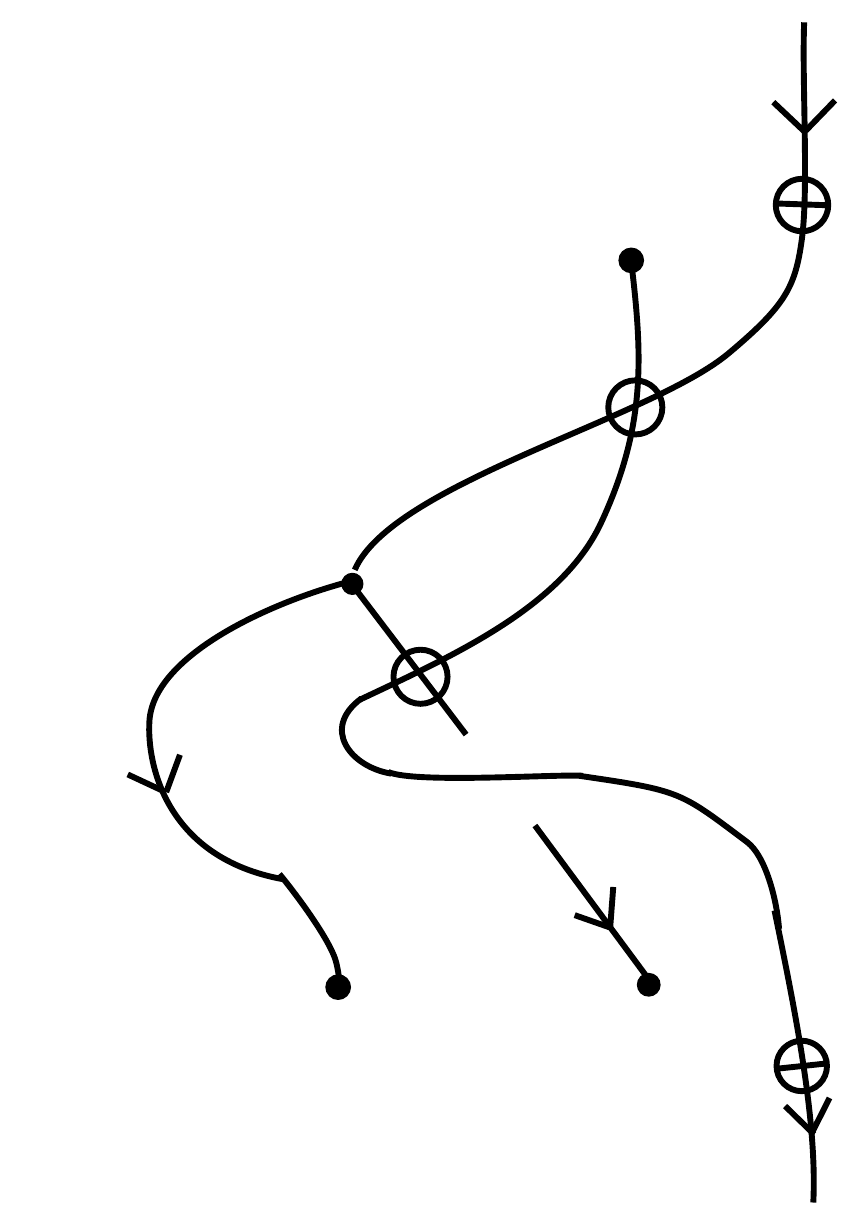}}\hspace{0.05in}
\underset{\text{$vL_v$-move}}{\overset{\text{over-threaded}}{\longleftrightarrow}} \hspace{0.05in}
\raisebox{-.7in}{\includegraphics[height=1.5in]{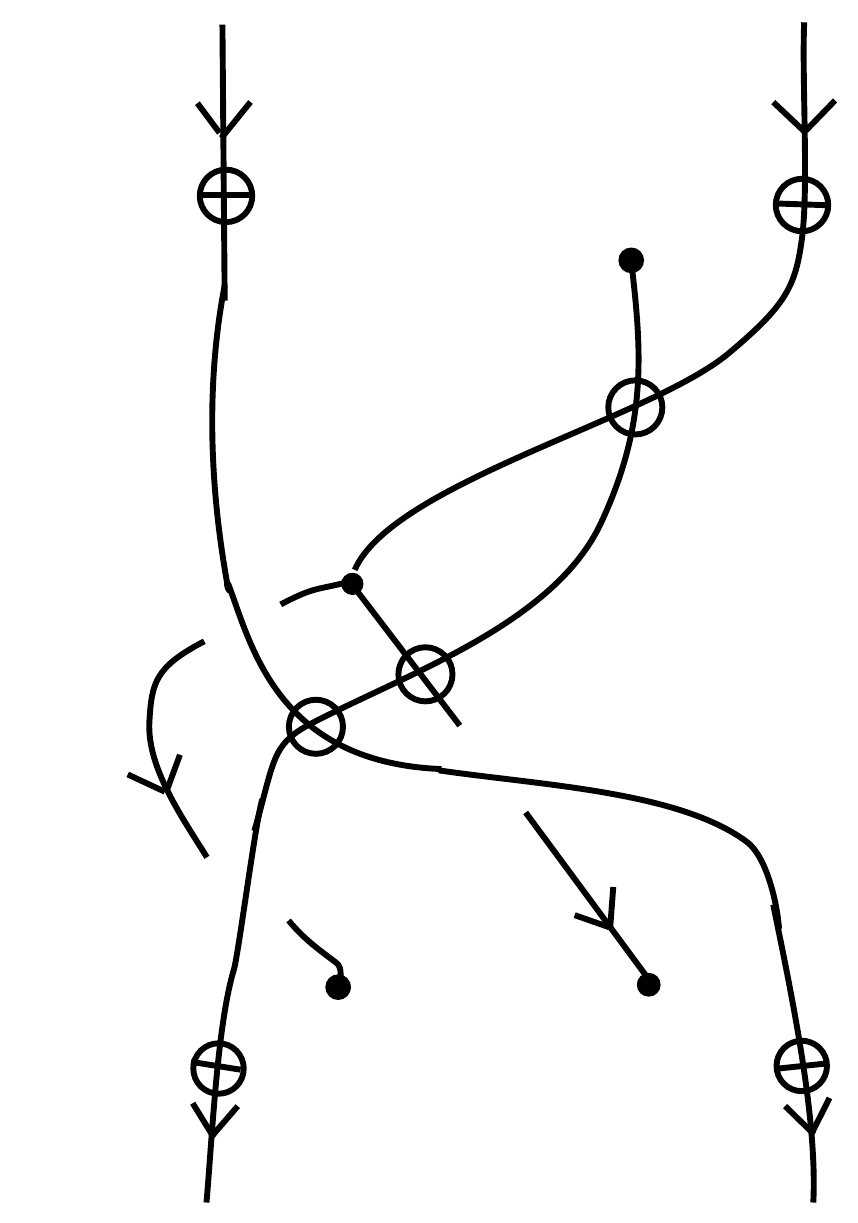}}
\hspace{0.05in} \underset{\text{$VR3$}}{\overset{\text{braid}}{\longleftrightarrow}}\]
\\
\[
\raisebox{-.7in}{\includegraphics[height=1.5in]{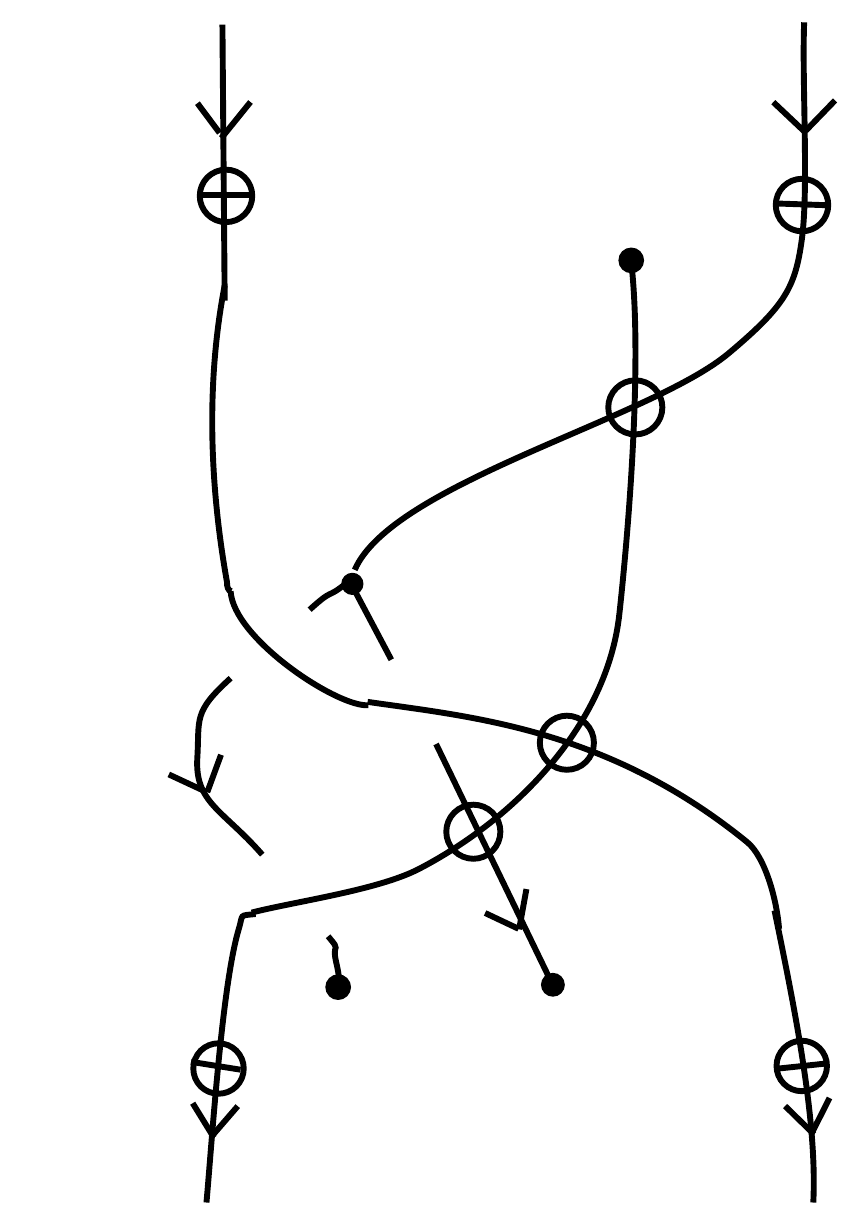}}\hspace{0.05in}
\underset{\text{$R4$}}{\overset{\text{braid}}{\longleftrightarrow}} \hspace{0.05in}
\raisebox{-.7in}{\includegraphics[height=1.5in]{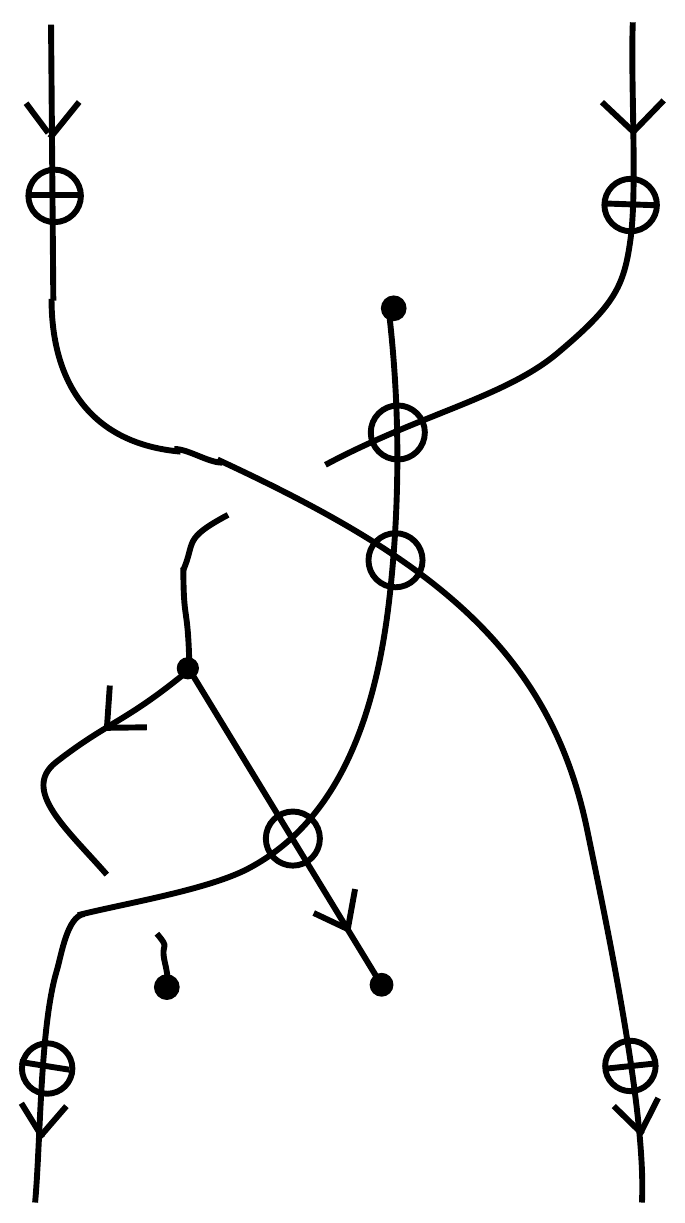}} \hspace{0.05in}\underset{\text{$L_v$-move}}{\overset{\text{right real}}{\longleftrightarrow}} \hspace{0.05in}
\raisebox{-.7in}{\includegraphics[height=1.5in]{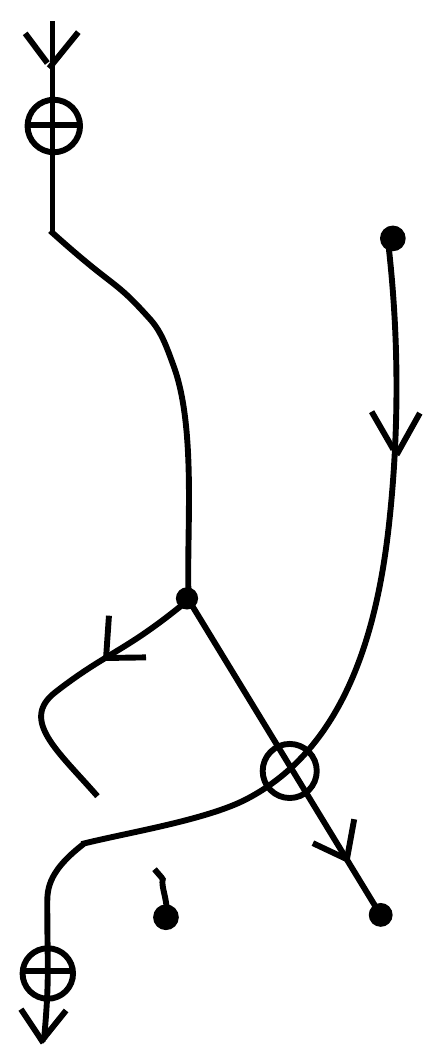}}\]
\caption{Left $TL_v$-move in terms of the right $TL_v$-move} \label{left TLmove}
\end{figure}

We are now ready to define an equivalence on the set of virtual trivalent braids that will be used to prove a Markov-type theorem for virtual trivalent braids and virtual spatial trivalent graphs.

\begin{definition}
We say that two virtual trivalent braids are \textit{$TL_v$-equivalent} if they are related by virtual trivalent braid isotopy and a finite sequence of the following moves:
\begin{enumerate}[(i)]
\item Real conjugation
\item Right virtual  and right real $L_v$-moves
\item Right and left under-threaded $L_v$-moves
\item Right trivalent $L_v$-moves ($TL_v$-moves)
\end{enumerate}
\end{definition}

We are now ready to state and prove our first Markov-type theorem.

\begin{theorem}[$L$-move Markov-type Theorem] \label{Markov-typeThrm}
Two well-oriented virtual spatial trivalent graph diagrams are isotopic if and only if any two of their corresponding virtual trivalent braids are $TL_v$-equivalent.
\end{theorem}

\begin{proof}
It is easy to see that $TL_v$-equivalent virtual trivalent braids have isotopic closures.

We have to show the converse. For this, we need to compare virtual trivalent braids obtained from making different choices when applying the braiding algorithm to a given virtual STG diagram. In addition, we need to compare virtual trivalent braids corresponding to diagrams that are related by different choices made when bringing a diagram to regular position and then to general position. Finally, we need to compare the virtual trivalent braids that correspond to any two isotopic virtual STG diagrams. Therefore, we split the proof into three parts.

\textit{Part I.} First, we show that all of the different choices that are made during the braiding process result in virtual trivalent braids that are $TL_v$-equivalent. These different choices involve both the order in which the local areas are braided, as well as how the subdivision points are chosen. 

Except for the trivalent $L_v$-moves ($TL_v$-moves) all of the other $L_v$-moves are applied away from trivalent vertices. Then, applying a similar argument as in~\cite[Remark 1]{KauLamb}, we have that the order in which the free up-arcs are braided does not affect the isotopy type of the resulting diagram; this is due to the detour move. Also, the order in which the crossings are braided is irrelevant, because of the narrow zone condition for the crossings. 

Using a similar approach to subdivision points as in~\cite[Corollary 2]{KauLamb}, it can be seen that given different subdivisions of a virtual STG diagram, the corresponding virtual trivalent braids are $TL_v$-equivalent.

\textit{Part II.} Second, we show that all of the different choices made when bringing a diagram into general position result in virtual trivalent braids that are $TL_v$-equivalent. We focus only on local moves on a diagram and their corresponding braided diagrams, because we assume that the two braids are identical except in a neighborhood where the move is applied.

The first consideration for this step is to check the possible choices in bringing a diagram to regular position, namely those cases when we make use of an $R5$ move to put a vertex in regular position. There are two versions of $R5$ moves, depending on the sign of the classical crossing involved in the move. Moreover, according to the chart in Fig.~\ref{regpos}, for a $Y$- or $\lambda$-vertex near which the three edges have upwards orientation, there are four choices for putting that vertex in regular position (see rows 1 and 3 in the first column of the chart in Fig.~\ref{regpos}). 

 When two diagrams in regular position differ by the type of classical crossing (that resulted by the application of the $R5$ move), we call this change of crossing a \textit{switch move.} We show that if two virtual STG diagrams in regular position differ by a switch move, then their corresponding braids are $TL_v$-equivalent. There are six cases to consider. 

Figure~\ref{fig:RegPosBraidR5} shows the first cases of a switch move on a $\lambda$-vertex in regular position. We proceed to show that the braids corresponding to the two diagrams on the left side of Fig.~\ref{fig:RegPosBraidR5} are $TL_v$-equivalent. The case of the switch move on the right side of Fig.~\ref{fig:RegPosBraidR5} is verified in a similar manner, and thus it is omitted here.
 
\begin{figure}[ht]
\[
\raisebox{-.2in}{\includegraphics[height=0.5in]{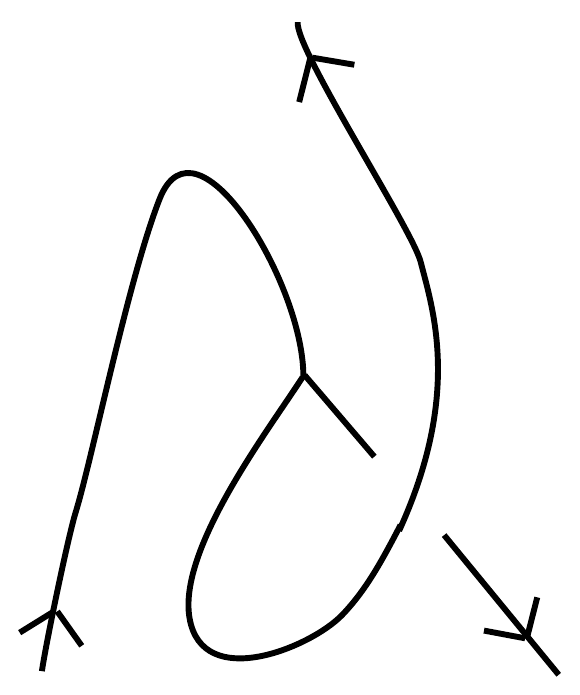}}\,\, \hspace{.3cm} \overset{\text{switch}}{\underset{\text{move}} \longleftrightarrow} \,\, \hspace{.3cm}
\raisebox{-.2in}{\includegraphics[height=.5in]{LUDU2}}
\hspace{1.5cm}
\reflectbox{\raisebox{-.2in}{\includegraphics[height=0.5in]{LUDU2p}}}\,\, \hspace{.3cm} \overset{\text{switch}}{\underset{\text{move}} \longleftrightarrow} \,\, \hspace{.3cm}
\reflectbox{\raisebox{-.2in}{\includegraphics[height=.5in]{LUDU2}}}
\]
\caption{Switch moves on a $\lambda$-vertex in regular position} \label{fig:RegPosBraidR5}
\end{figure}

Since these diagrams differ only by the classical crossing we can ignore all free up-arcs that are identical in both diagrams. We refer to this process as the `segmenting process' (see Fig.~\ref{segmenting}). We will assume from this point on that any other diagrams in this part of the proof will be segmented first to simplify the braiding and avoid identical regions in the resulting diagrams.

\begin{figure}[ht]
\[
\raisebox{-.2in}{\includegraphics[height=0.75in]{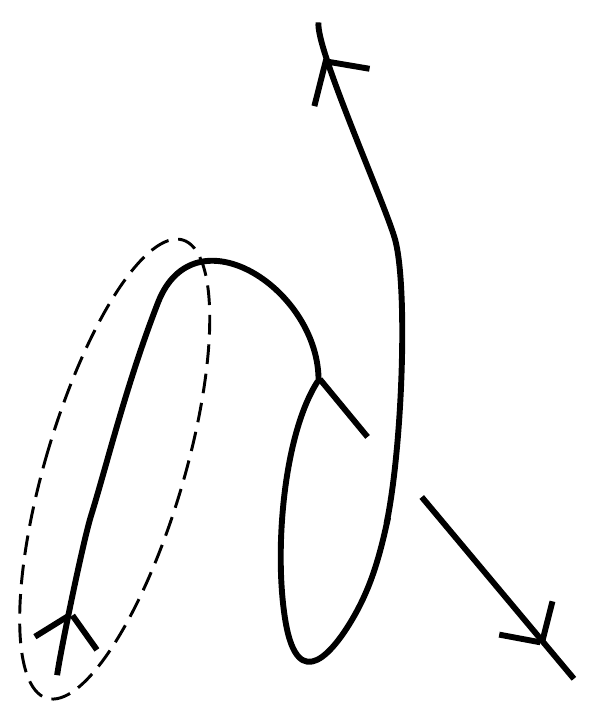}}
\longrightarrow
\raisebox{-.2in}{\includegraphics[height=0.75in]{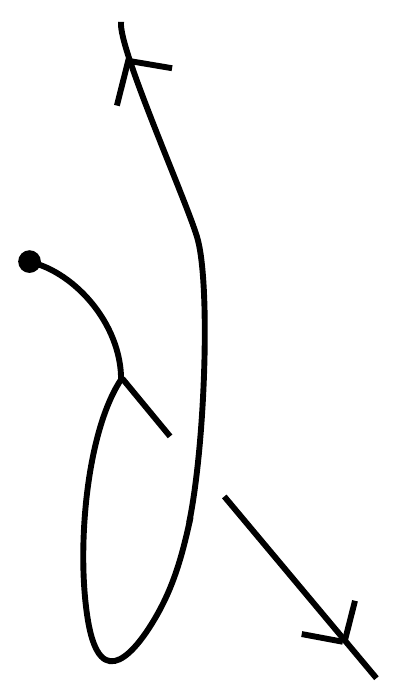}}
\]
\caption{Segmenting a diagram} \label{segmenting}
\end{figure}

We then apply the braiding algorithm to both sides of the segmented diagrams. Since the resulting braids differ by an $R5$ move in braid form--as shown in Fig.~\ref{fig:braidR5}--it follows that these braids are $TL_v$-equivalent. The dotted light gray ovals in the first and last diagrams of Fig.~\ref{fig:braidR5} show the local regions where the braiding algorithm is applied.

\begin{figure}[ht]
\[
\raisebox{-.4in}{\includegraphics[height=2cm]{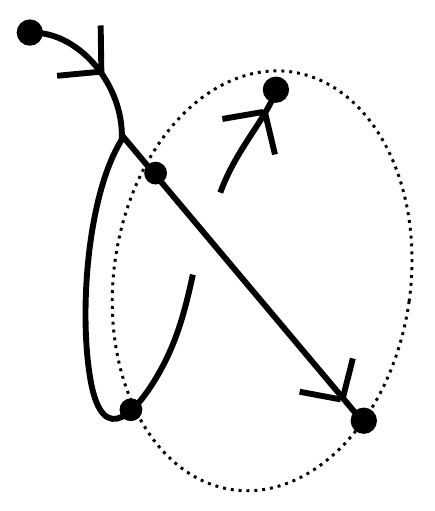}}
\hspace{0.1in} \stackrel{\text{braiding}}{\longrightarrow} \hspace{0.1in}%
\raisebox{-.4in}{\includegraphics[height=3.25cm]{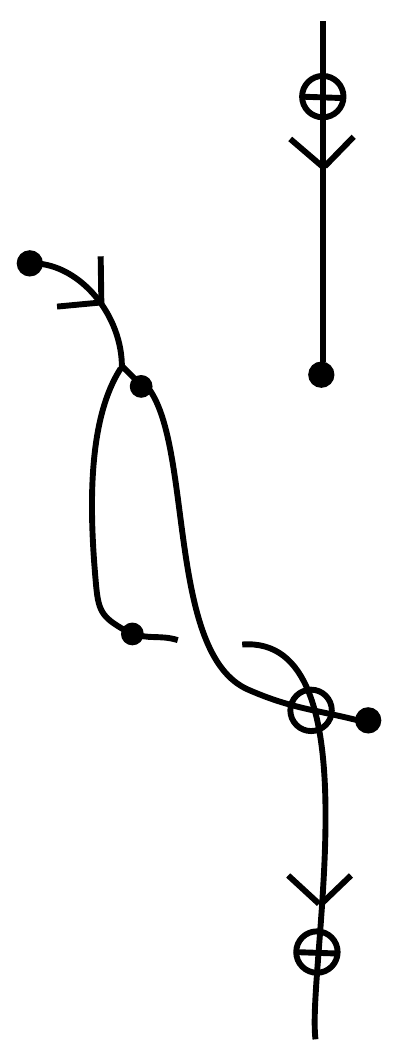}}
\hspace{0.1in} \stackrel{\text{braid R5}}{\longleftrightarrow} \hspace{0.1in}
\raisebox{-.4in}{\includegraphics[height=3.25cm]{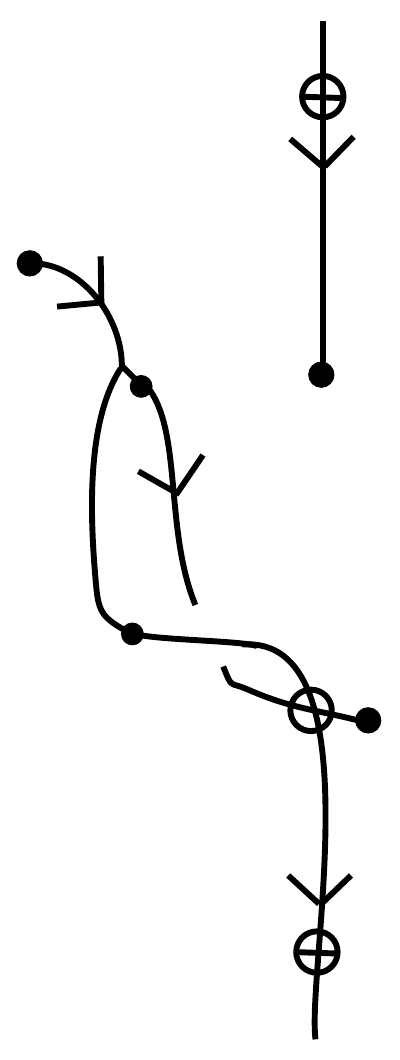}}
\hspace{0.1in} \stackrel{\text{braiding}}{\longleftarrow} \hspace{0.1in}
\raisebox{-.4in}{\includegraphics[height=2cm]{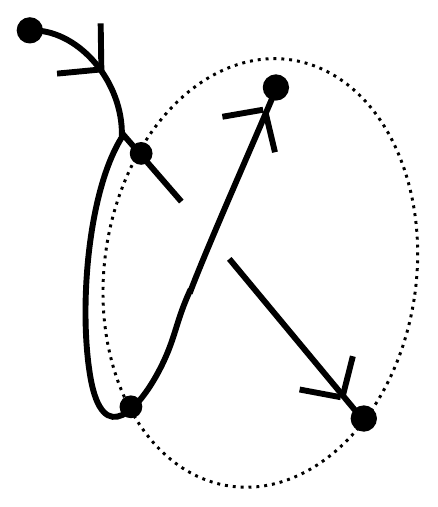}}
\]
\caption{A switch move on a $\lambda$-vertex corresponds to braids that differ by a braid $R5$ move} \label{fig:braidR5}
\end{figure}

The next case involves a $Y$-vertex differing by a switch move as shown in Fig.~\ref{fig:RegPosL}. We verify again the move on the left, since the move on the right follows similarly. After applying the braiding algorithm to the segmented diagrams in both sides of the move on the left, the resulting braids differ by a basic $L_v$-move and braid isotopy, as shown in Fig.~\ref{fig:Lmove}. We remark that the basic $L_v$-move is applied on the small arc bounded by the two subdivision points on the third and, respectively, fifth diagram. Moreover, the dotted light gray ovals in the first and last diagrams show the local region where the braiding algorithm is applied.

\begin{figure}[ht]
\[\raisebox{-.2in}{\includegraphics[height=0.5in]{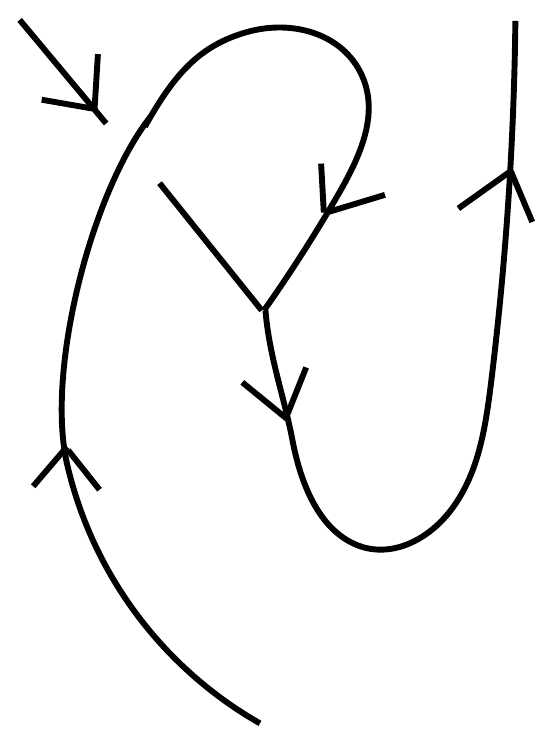}} \hspace{.3cm} \overset{\text{switch}}{\underset{\text{move}}\longleftrightarrow} \hspace{.3cm}
\raisebox{-.2in}{\includegraphics[height=.5in]{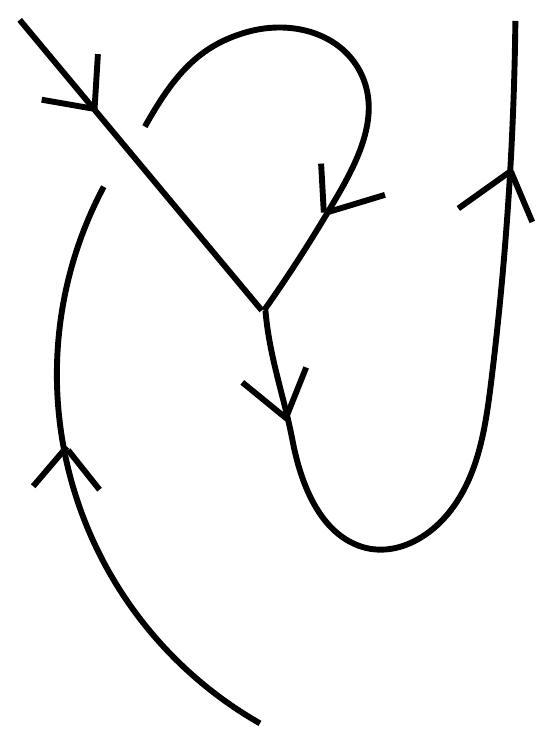}}
\hspace{1.5cm}
\reflectbox{\raisebox{-.2in}{\includegraphics[height=0.5in]{TL1p}}} \hspace{.3cm} \overset{\text{switch}}{\underset{\text{move}}\longleftrightarrow} \hspace{.3cm}
\reflectbox{\raisebox{-.2in}{\includegraphics[height=.5in]{TL1n}}}
\]
\caption{Switch moves on a $Y$-vertex in regular position} \label{fig:RegPosL}
\end{figure}

\begin{figure}[ht]
\[\raisebox{-.4in}{\includegraphics[height=2.5cm]{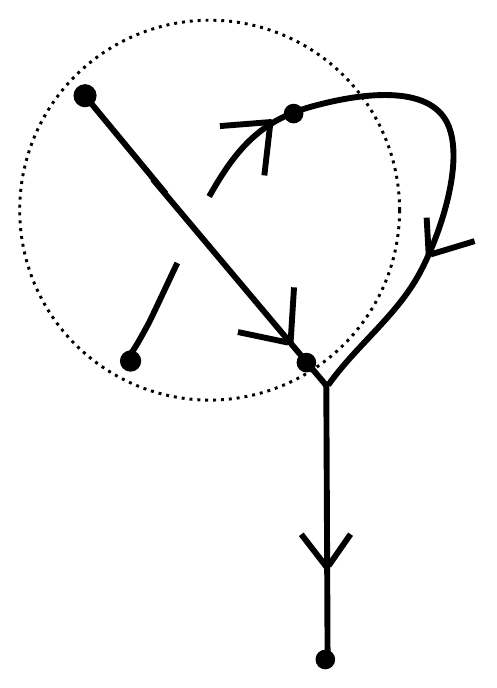}}
\hspace{0.1in} \stackrel{\text{braiding}}{\longrightarrow} \hspace{0.1in}
\raisebox{-.75in}{\includegraphics[height=4.5cm]{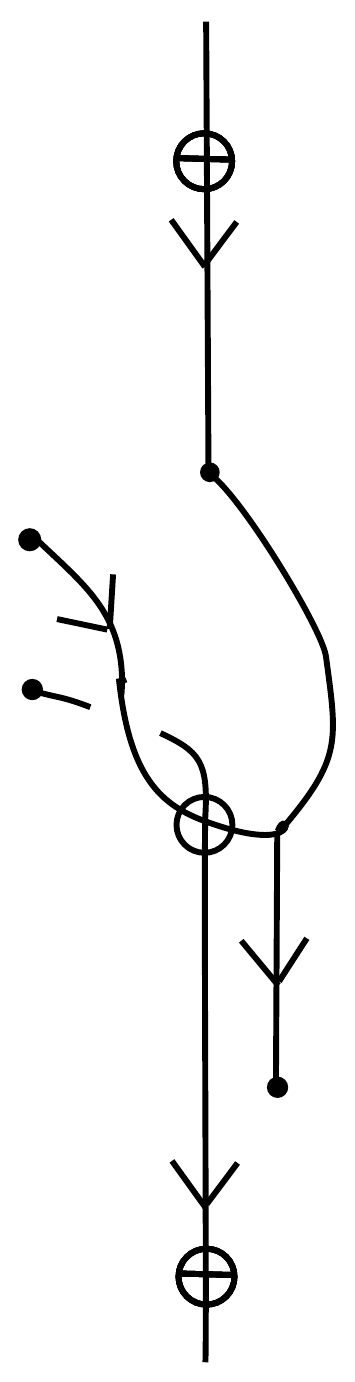}}\hspace{0.1in}
\underset{\text{$L_v$-move}}{\overset{\text{basic}}{\longleftrightarrow}} 
\raisebox{-.4in}{\includegraphics[height=2.5cm]{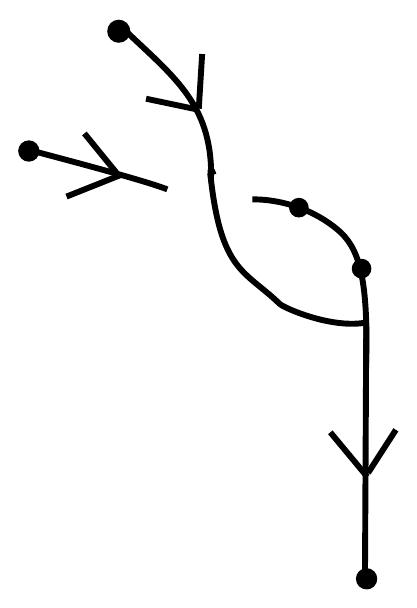}} \stackrel{\text{braid R5}}{\longleftrightarrow} \hspace{0.1in}
\raisebox{-.4in}{\includegraphics[height=2.5cm]{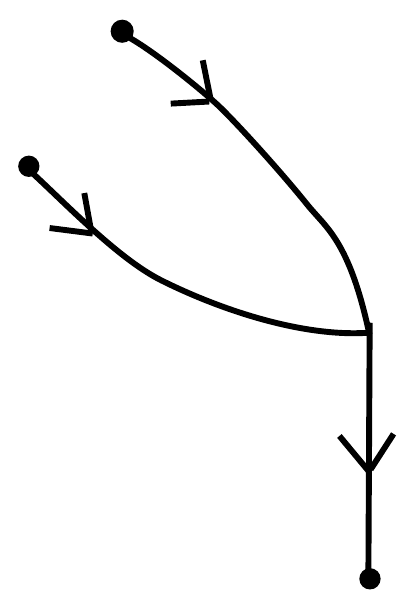}}\]
\\

\[\stackrel{\text{braid R5}}{\longleftrightarrow} \hspace{0.1in}
\raisebox{-.4in}{\includegraphics[height=2.5cm]{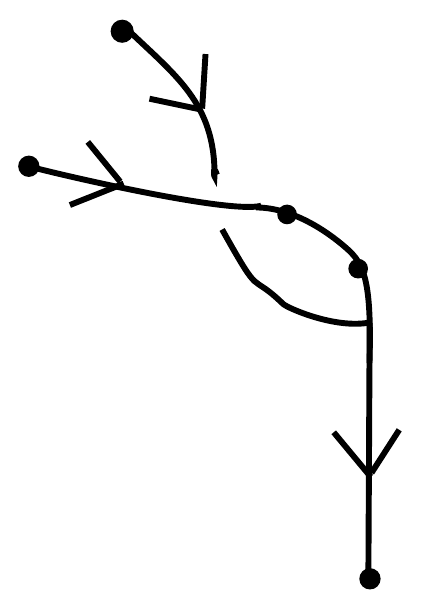}} \hspace{0.1in}
\underset{\text{$L_v$-move}}{\overset{\text{basic}}{\longleftrightarrow}} 
\hspace{0.1in}
\raisebox{-.75in}{\includegraphics[height=4.5cm]{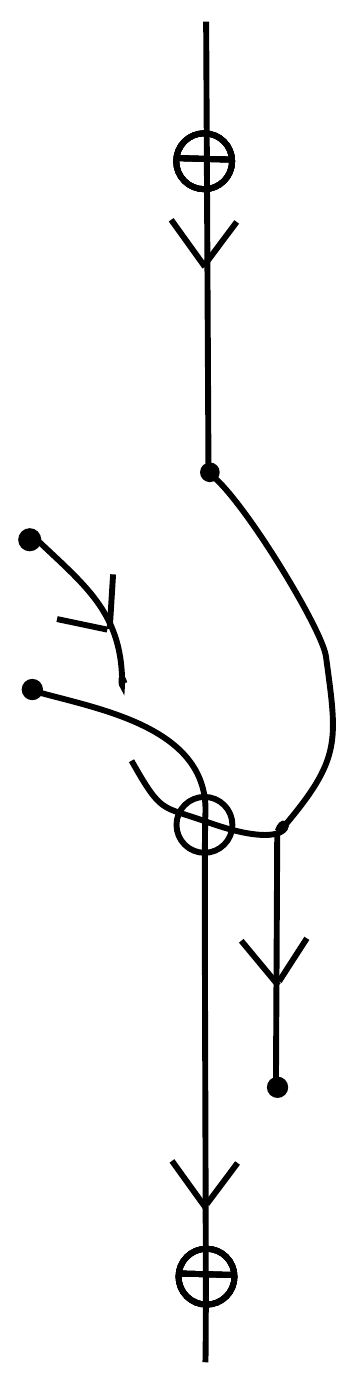}}
\hspace{0.1in} \stackrel{\text{braiding}}{\longleftarrow} \hspace{0.1in}
\raisebox{-.4in}{\includegraphics[height=2.5cm]{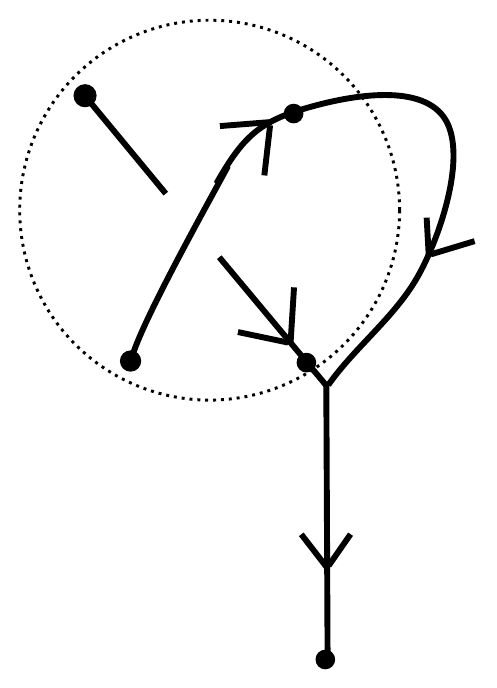}}
\]

\caption{A switch move on a $Y$-vertex corresponds to braids that differ by $L_v$-moves and a braid $R5$ move} \label{fig:Lmove}
\end{figure}

Another case of a switch move involving a $Y$-vertex is shown in Fig.~\ref{TL2switch}. By applying the braiding algorithm to the segmented diagrams corresponding to the two sides of the move, the resulting braids are $TL_v$-equivalent, as shown in Fig.~\ref{TL2equ}. The first basic $L_v$-move applied to the second diagram in Fig.~\ref{TL2equ} takes place near the bottom right part of the diagram where the rightmost strand is extended to the bottom (using braid isotopy); then we extend the top and bottom of the rightmost strands of the resulting trivalent braid and apply virtual conjugation to arrive at the third diagram in the first row. Similarly, the get the third diagram from the second diagram in the last row of Fig.~\ref{TL2equ}, we apply first virtual conjugation followed by a basic $L_v$-move on an arc near the bottom right of the diagram.

\begin{figure}[ht]
\[\raisebox{-.2in}{\includegraphics[height=0.5in]{YUUU2b}}\hspace{.3cm}  \overset{\text{switch}}{\underset{\text{move}}\longleftrightarrow} \hspace{.3cm}
\raisebox{-.2in}{\includegraphics[height=.5in]{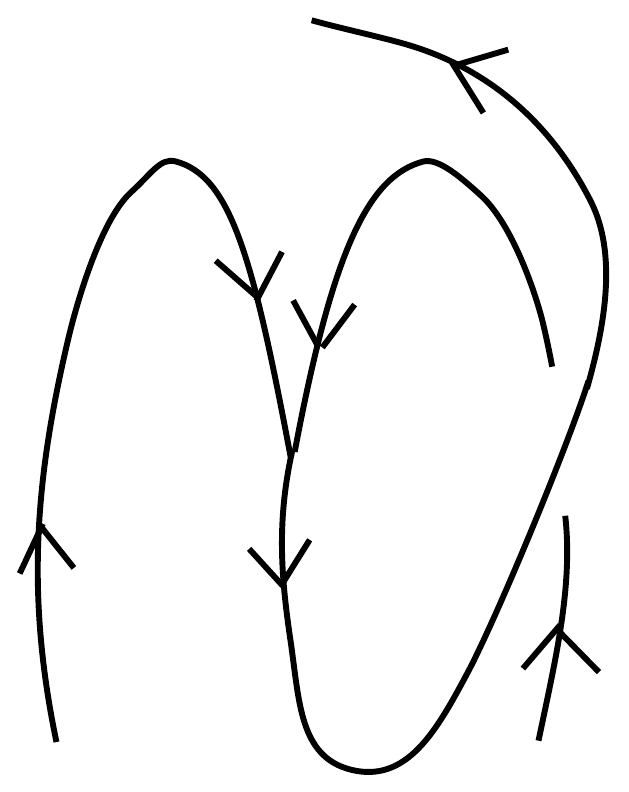}}
\]
\caption{A switch move on a $Y$-vertex in regular position} \label{TL2switch} 
\end{figure}

\begin{figure}[ht]
\[\raisebox{-.3in}{\includegraphics[height=2.5cm]{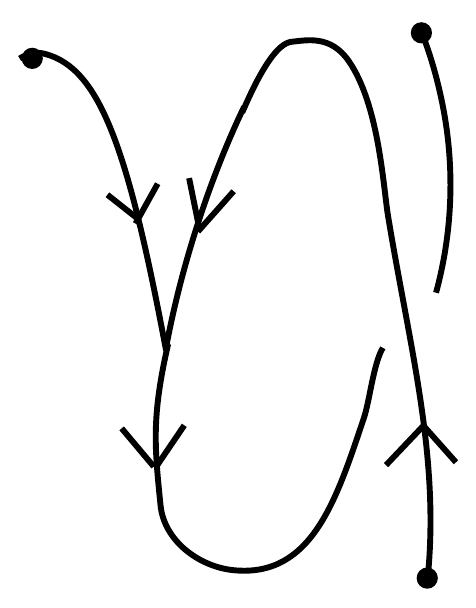}}
\hspace{0.05in} \stackrel{\text{braiding}}{\longrightarrow} \hspace{0.05in}
\raisebox{-.75in}{\includegraphics[height=4.5cm]{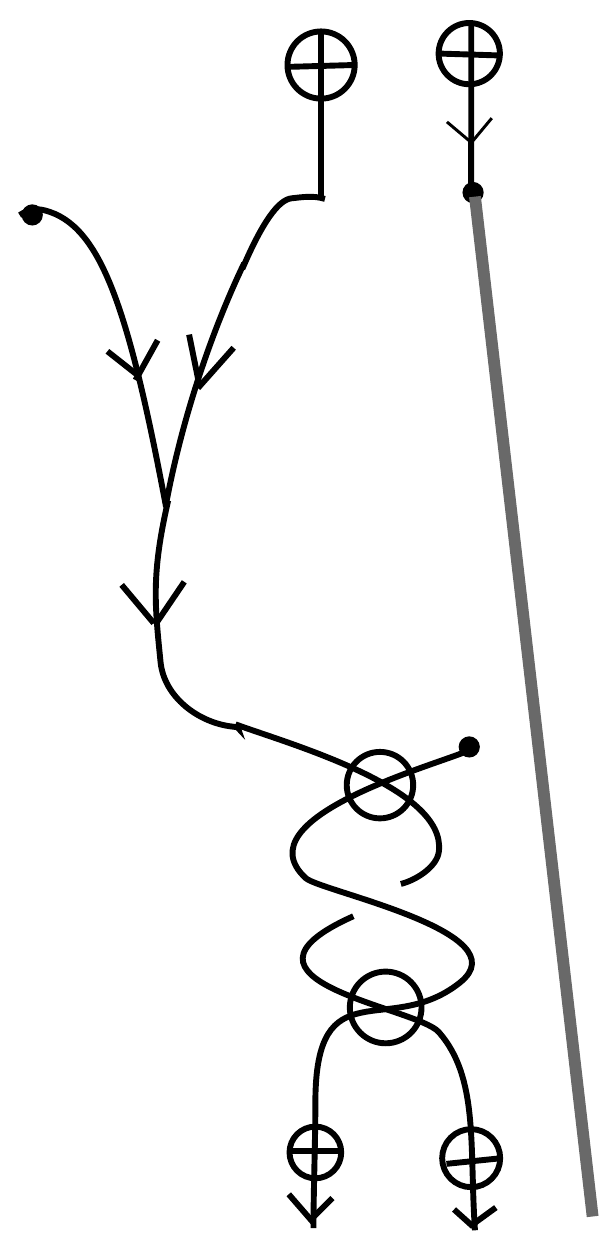}}
\hspace{0.05in} 
\underset{\text{virt. conj.}}{\overset{\text{basic $L_v$-move}}{\longleftrightarrow}} \hspace{0.05in}
\raisebox{-.75in}{\includegraphics[height=4.5cm]{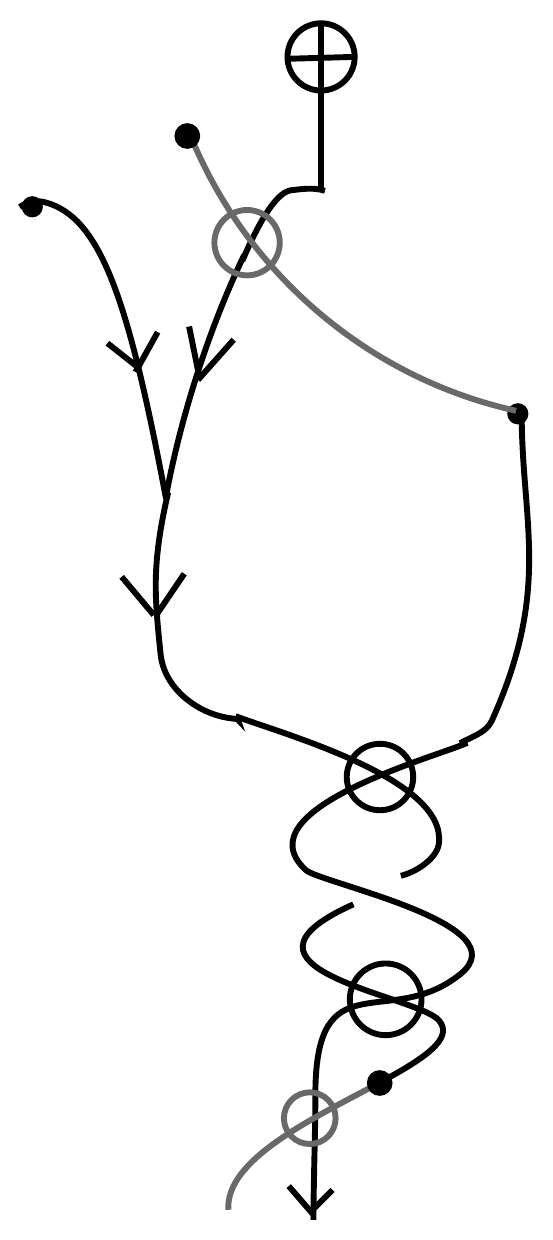}}
\underset{\text{V4, V2}}{\overset{\text{braid}}{\longleftrightarrow}} \hspace{0.05in}
\raisebox{-.75in}{\includegraphics[height=4.5cm]{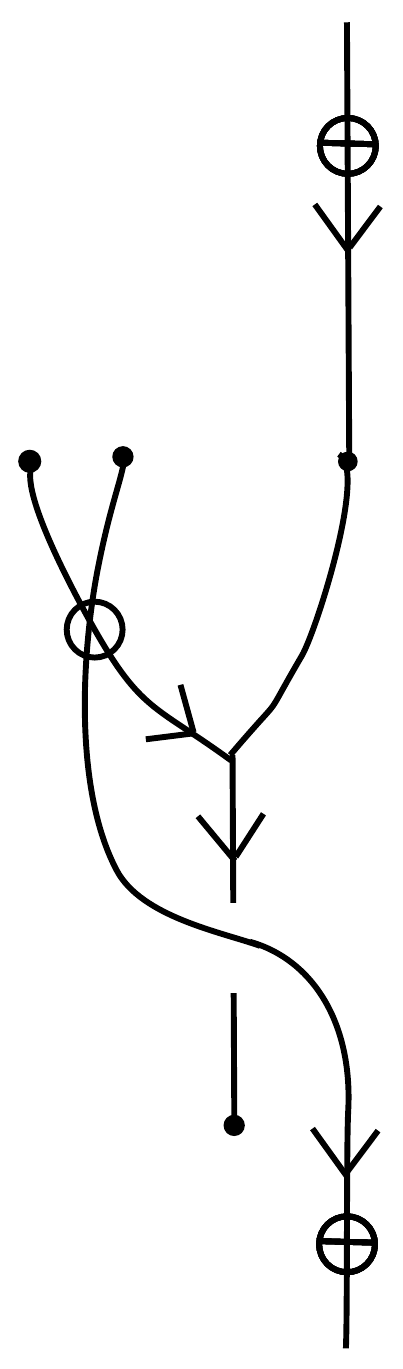}}
\]
\[
\underset{\text{R4}}{\overset{\text{braid}}{\longleftrightarrow}} \hspace{0.05in}
\raisebox{-.75in}{\includegraphics[height=4.5cm]{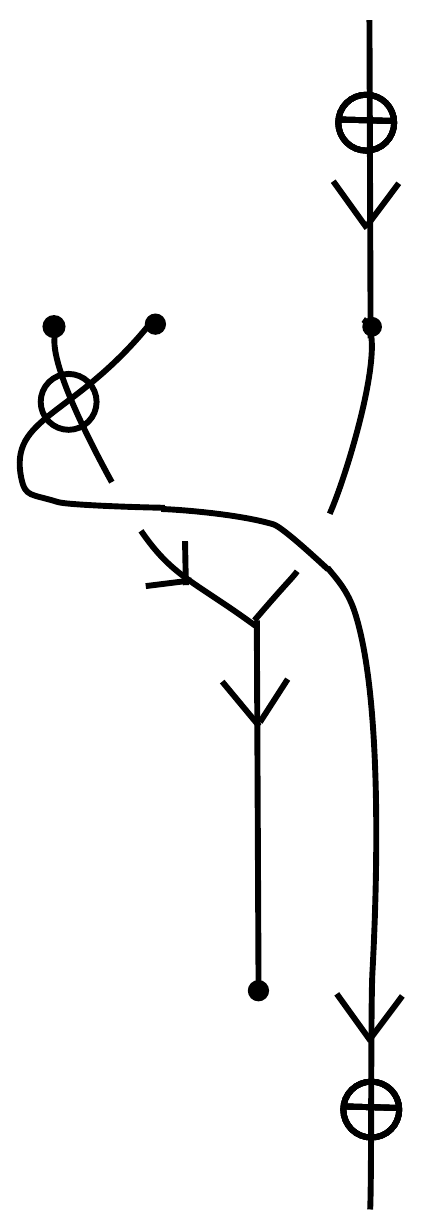}}
\hspace{0.05in}  
\underset{\text{$L_v$-move}}{\overset{\text{right real}}{\longleftrightarrow}} \hspace{0.05in}
\raisebox{-.4in}{\includegraphics[height=3cm]{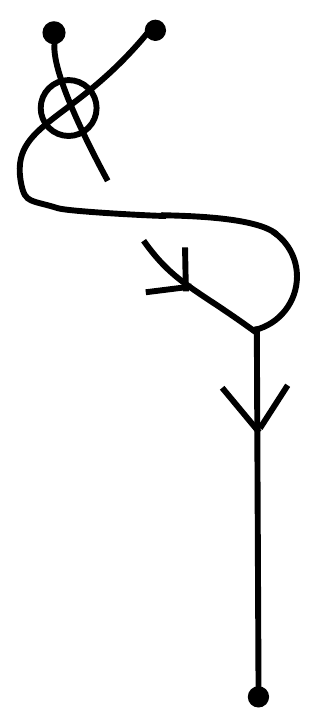}}
\hspace{0.05in} 
\underset{\text{R5}}{\overset{\text{braid}}{\longleftrightarrow}} \hspace{0.05in}
\raisebox{-.4in}{\includegraphics[height=3cm]{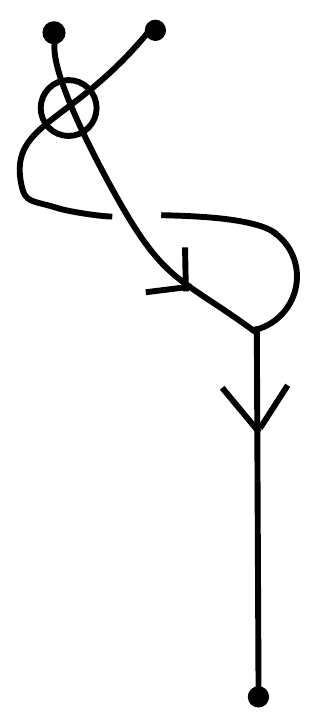}}
\hspace{0.05in} 
\underset{\text{ $L_v$-move}}{\overset{\text{right real}}{\longleftrightarrow}}
\hspace{0.05in}
\raisebox{-.75in}{\includegraphics[height=4.5cm]{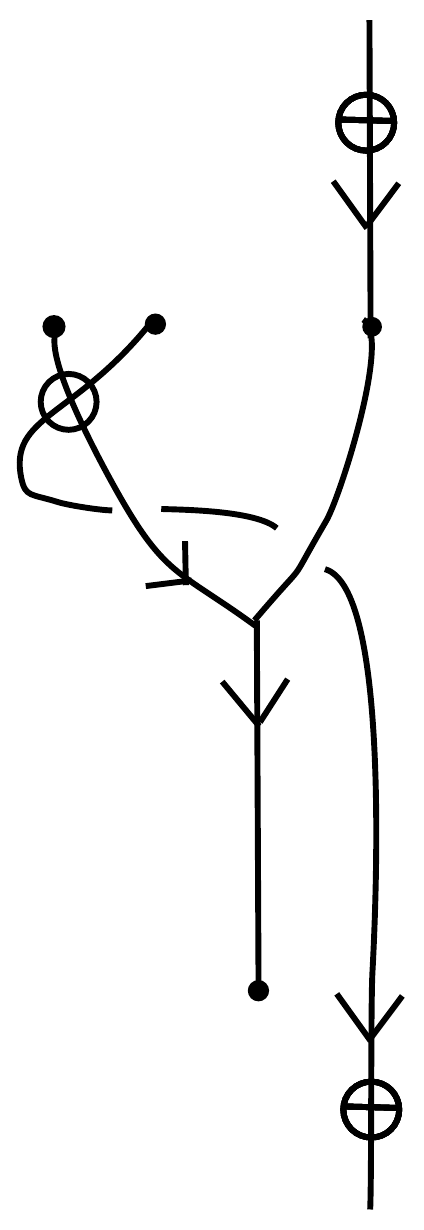}}
\underset{\text{R4}}{\overset{\text{braid}}{\longleftrightarrow}} \hspace{0.05in}
\]
\[
\raisebox{-.75in}{\includegraphics[height=4.5cm]{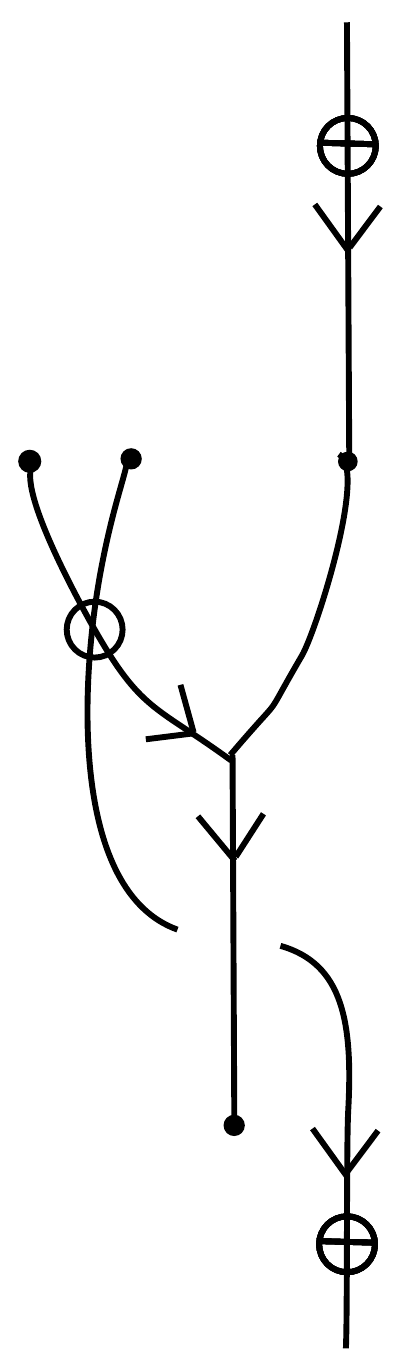}}
\underset{\text{V4, V2}}{\overset{\text{braid}}{\longleftrightarrow}} \hspace{0.05in}
\raisebox{-.75in}{\includegraphics[height=4.5cm]{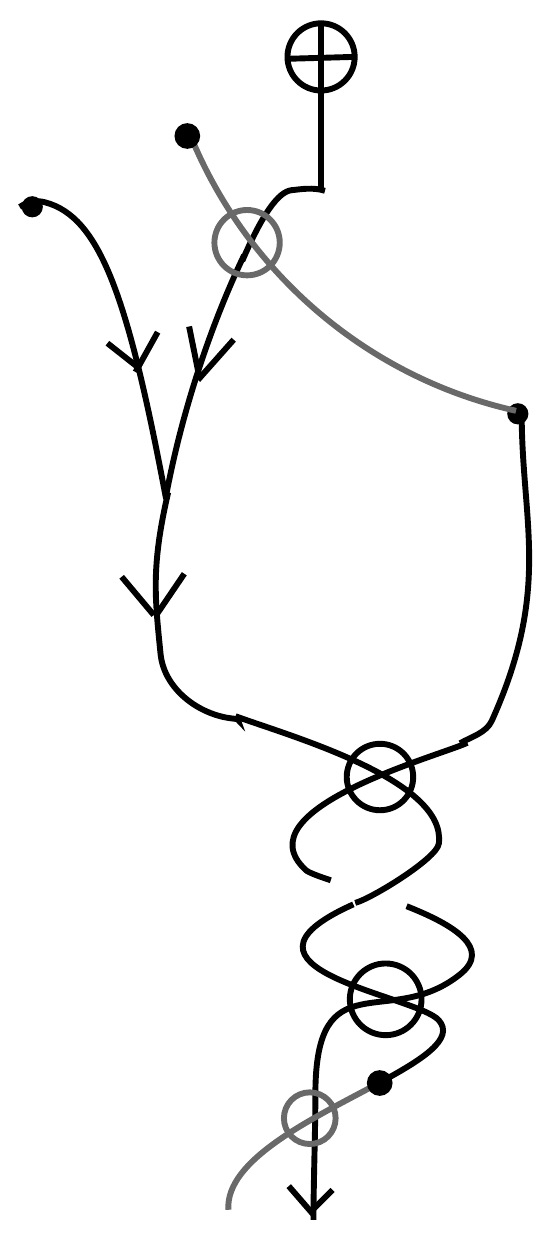}}
\underset{\text{basic L-move}}{\overset{\text{virt. conj.}}{\longleftrightarrow}} \hspace{0.05in}
\raisebox{-.75in}{\includegraphics[height=4.5cm]{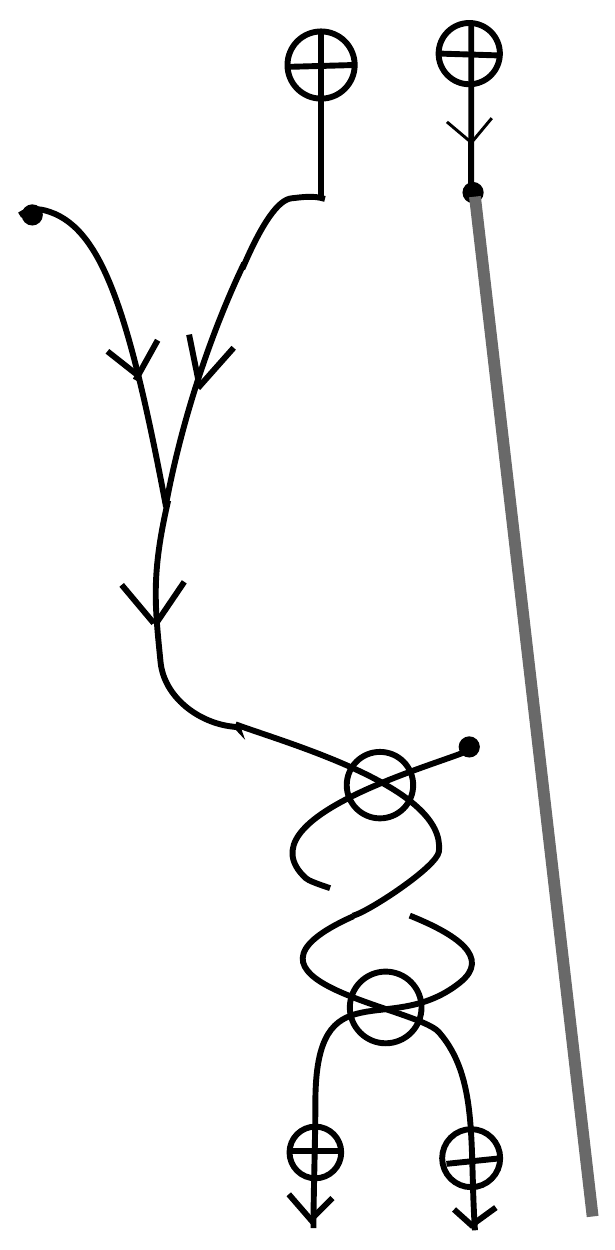}}
\hspace{0.05in} \stackrel{\text{braiding}}{\longleftarrow} \hspace{0.05in}
\raisebox{-.3in}{\includegraphics[height=2.5cm]{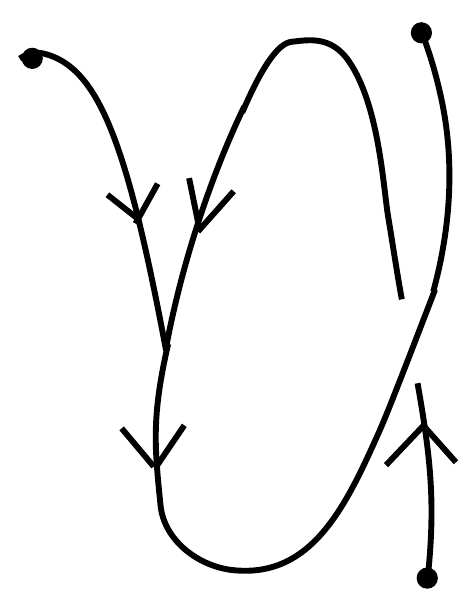}}
\]
\caption{Checking a switch move on a $Y$-vertex} \label{TL2equ}
\end{figure}

Lastly, consider the switch move on a $\lambda$-vertex in regular position, as shown in Fig.~\ref{piece} (where the dotted regions represent the arcs to be segmented). The corresponding braids for the associated segmented diagrams are shown in Fig.~\ref{fig:TL}. Note that these two braids differ by basic $L_v$-moves, braid isotopy and a right $TL_v$-move, and hence are $TL_v$-equivalent. The basic $L_v$-move and virtual conjugation applied to the second diagram to get the third diagram (in both rows of Fig.~\ref{fig:TL}) are performed in a similar manner as in Fig.~\ref{TL2equ}.

\begin{figure}[ht] 
\[
\raisebox{-.4in}{\includegraphics[height=1.75cm]{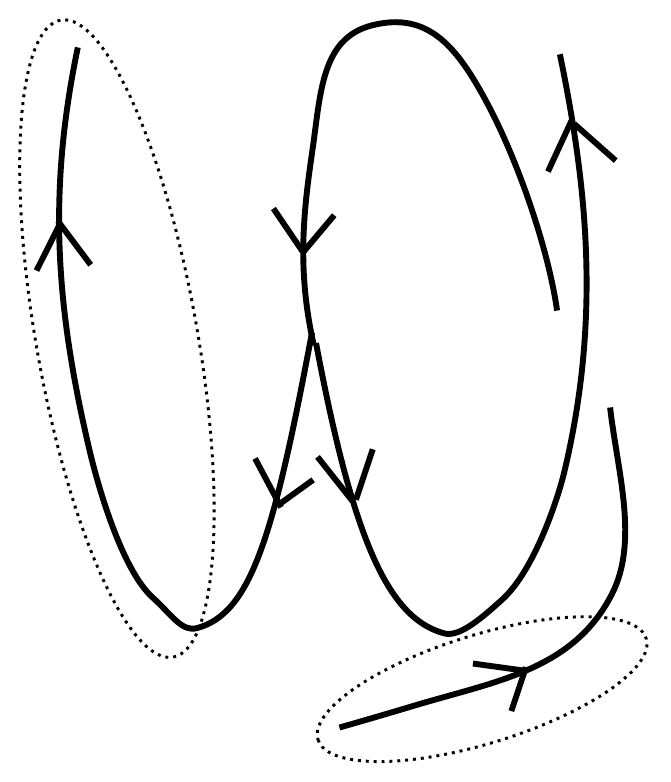}}
\hspace{0.3cm} 
\overset{\text{switch}}{\underset{\text{move}} \longleftrightarrow} \hspace{0.3cm}
\raisebox{-.4in}{\includegraphics[height=1.75cm]{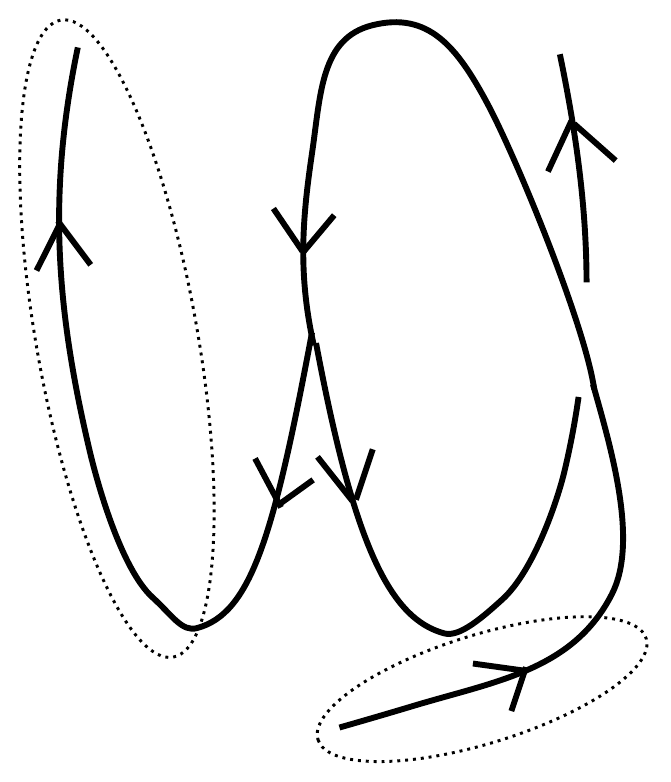}}
\]
\caption{A switch move on  a $\lambda$-vertex in regular position} \label{piece}
\end{figure}

\begin{figure}[ht]
\[
\raisebox{-.4in}{\includegraphics[height=2cm]{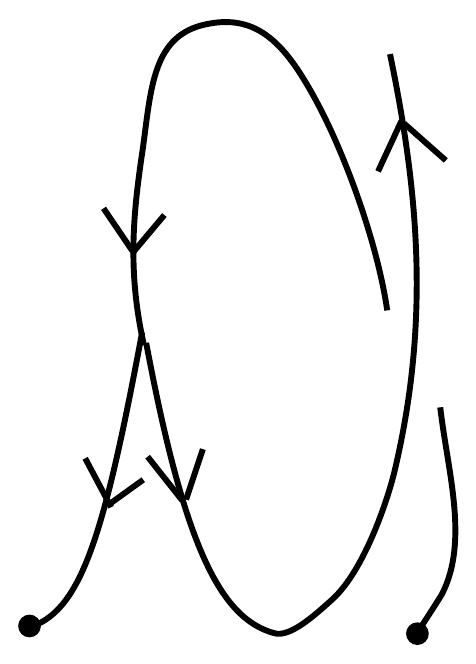}}
\hspace{0.1in} \stackrel{\text{braiding}}{\longrightarrow} \hspace{0.1in}
\raisebox{-.75in}{\includegraphics[height=4.5cm]{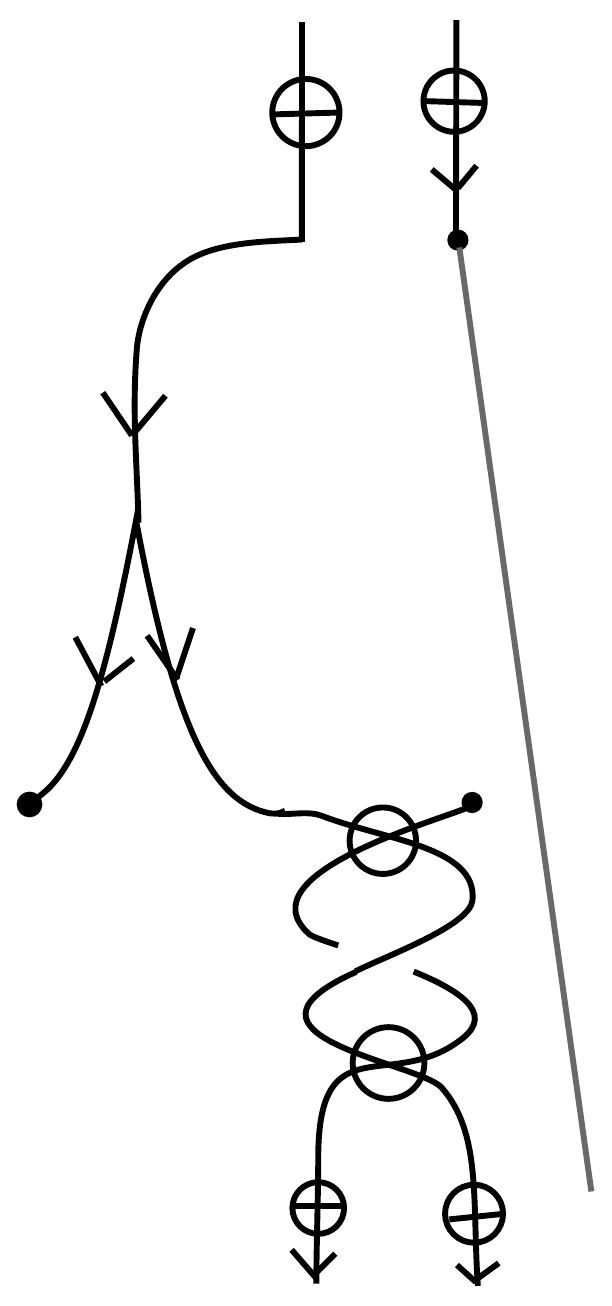}}
\hspace{0.1in} \underset{\text{virt. conj.}}{\overset{\text{basic $L_v$-move}}{\longleftrightarrow}}
\hspace{0.1in}
\raisebox{-.75in}{\includegraphics[height=4.35cm]{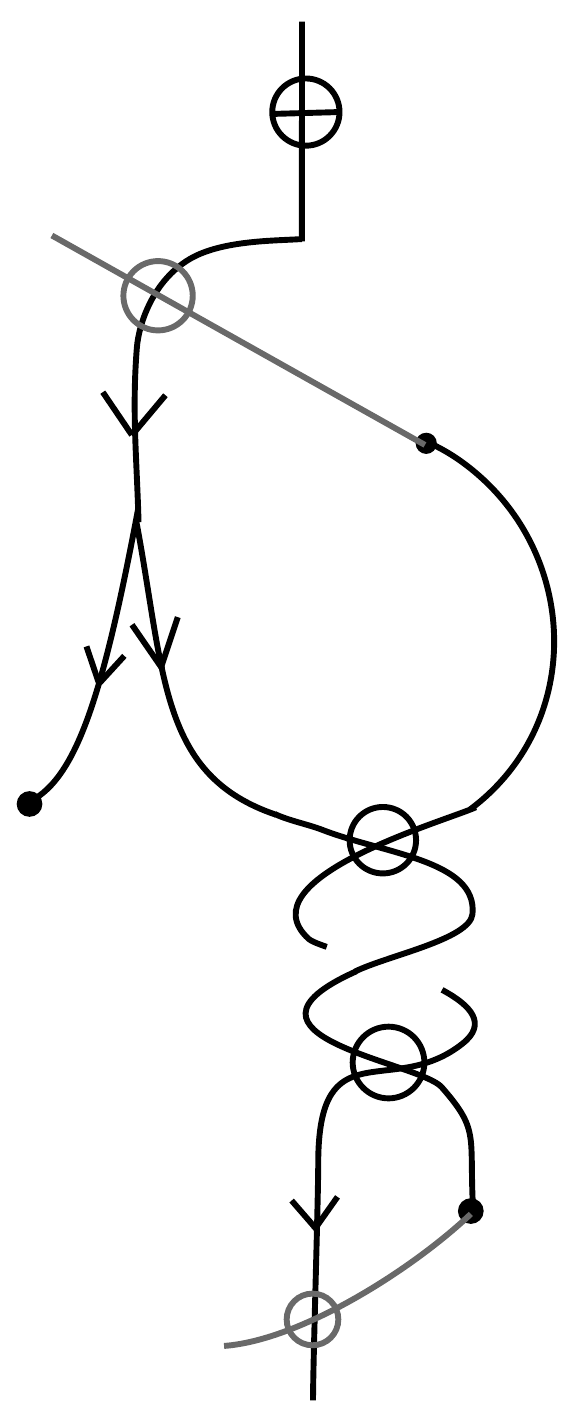}}
\hspace{0.1in} \underset{V4, V2} {\overset{\text{braid}}{\longleftrightarrow}} \hspace{0.1in}
\raisebox{-.7in}{\includegraphics[height=4.1cm]{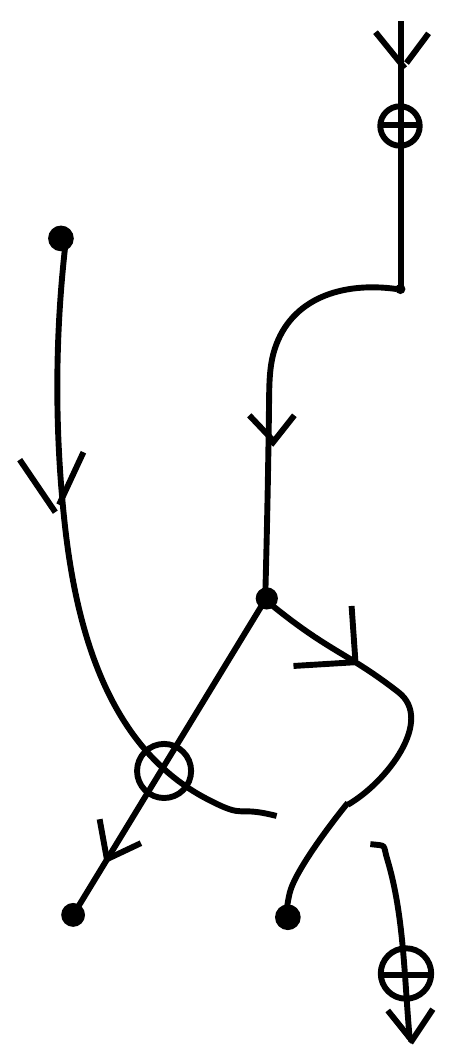}}
\]
\[
\raisebox{-.4in}{\includegraphics[height=2cm]{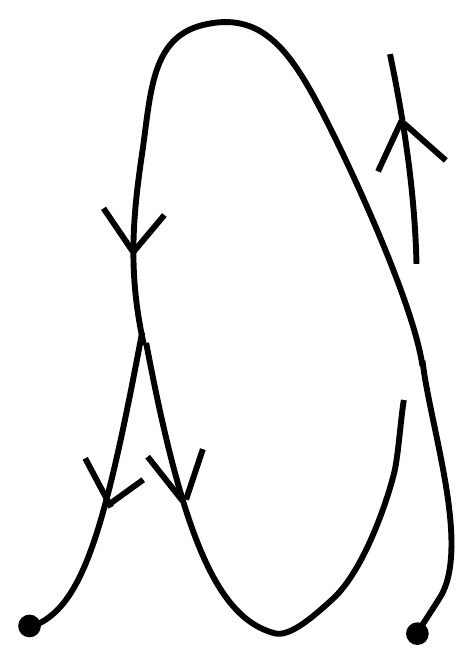}}
\hspace{0.1in} \stackrel{\text{braiding}}{\longrightarrow} \hspace{0.1in}
\raisebox{-.75in}{\includegraphics[height=4.35cm]{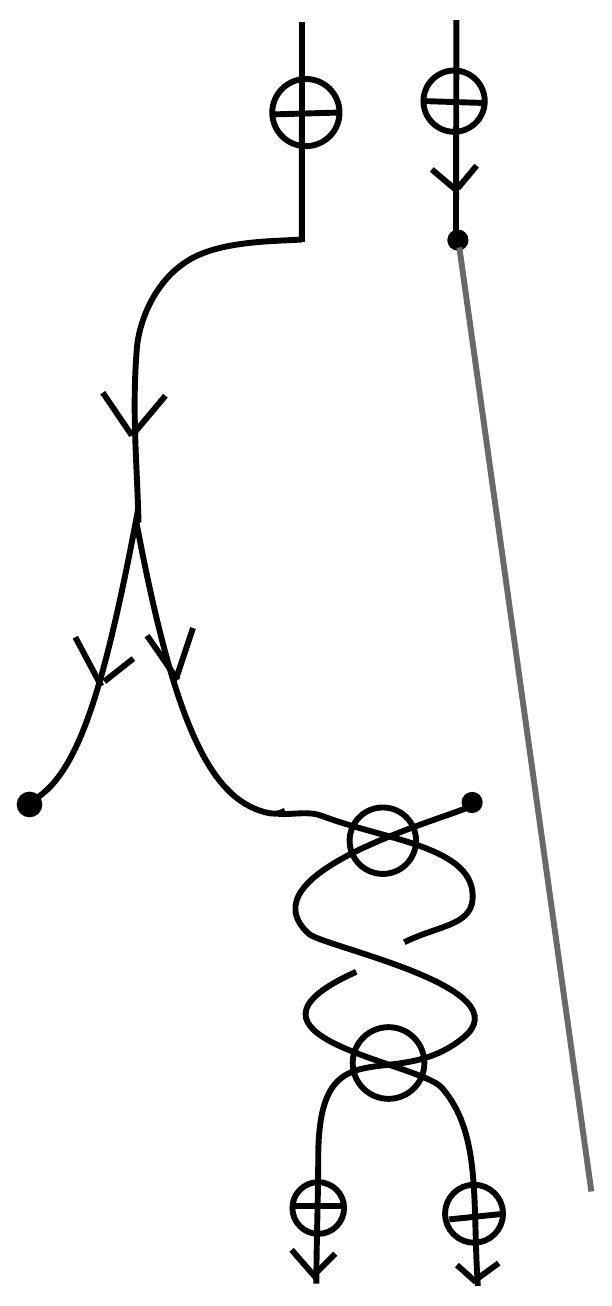}}
\hspace{0.1in} \underset{\text{virt. conj.}}{\overset{\text{basic $L_v$-move}}{\longleftrightarrow}}
\hspace{0.1in}
\raisebox{-.75in}{\includegraphics[height=4.35cm]{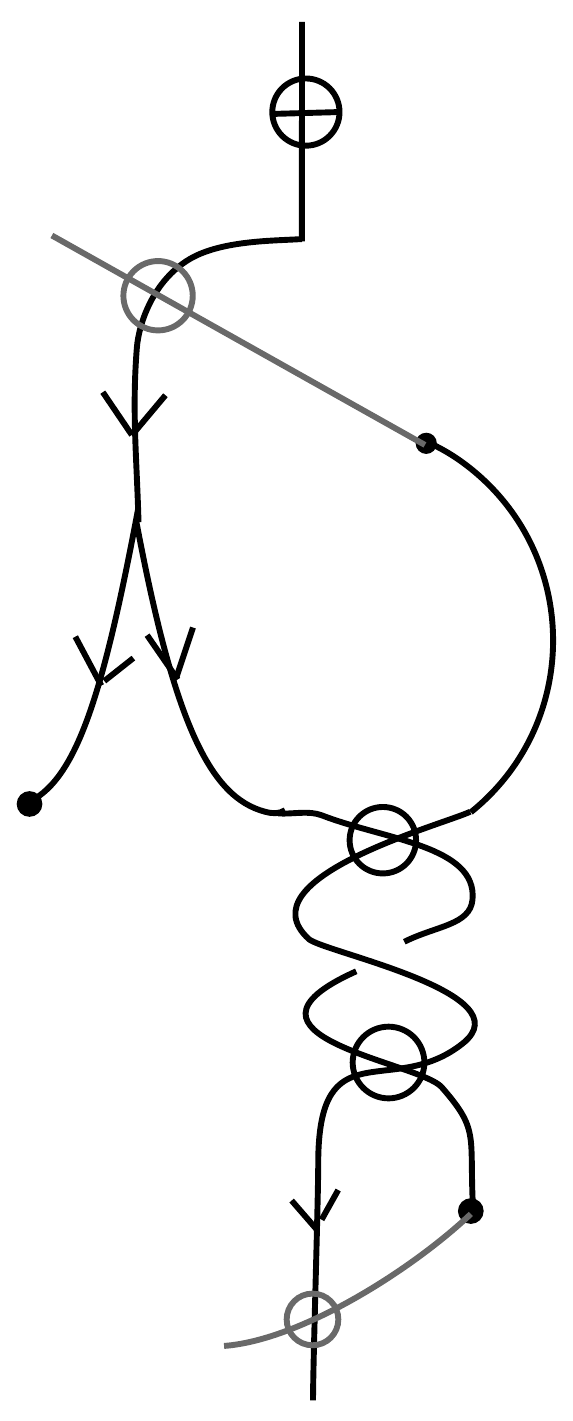}}
\hspace{0.1in} \underset{V4, V2} {\overset{\text{braid}}{\longleftrightarrow}}
\hspace{0.1in}
\raisebox{-.7in}{\includegraphics[height=4.1cm]{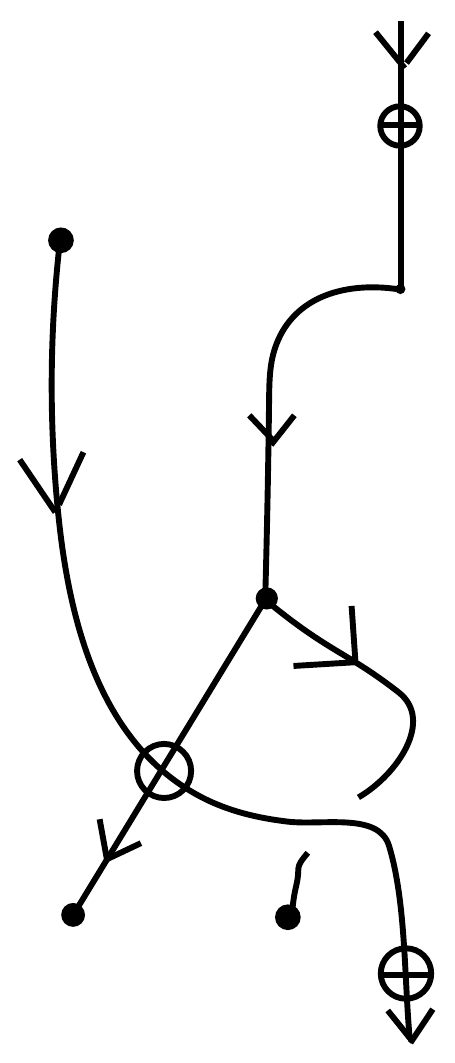}}
\]

\caption{A case of the switch move on a $\lambda$-vertex via the right $TL_v$-move}\label{fig:TL}
\end{figure}

This concludes the part of the proof involving a switch move.

Once the diagram is in regular position, it must be put into general position.  We need to show that all of the different choices that can be made when bringing a regular diagram into general position result in virtual trivalent braids that are $TL_v$-equivalent. It can be easily verified using a similar approach as in~\cite[Lemma 7]{KauLamb} that applying the braiding algorithm (away from the vertices) to diagrams that differ by direction sensitive moves results in braids that are $TL_v$-equivalent. 

\textit{Part III.} Next we need to prove that two well-oriented virtual STG diagrams in general position that differ by extended Reidemeister moves for virtual STG diagrams (given in Fig.~\ref{fig:R-moves}) correspond to closures of virtual trivalent braids that are $TL_v$-equivalent. The extended Reidemeister moves in the virtual spatial trivalent graph setting encapsulate the virtual isotopy moves in the virtual knot setting, and by~\cite[Lemma 8]{KauLamb} we know that the virtual isotopy moves correspond to virtual braids that differ by braid isotopy, real conjugation and $L_v$-moves. Hence we only need to consider the extended Reidemeister moves involving trivalent vertices, specifically the moves $R4, V4$ and $R5$. We will also need to consider all of the possible orientations of the strands involved in the moves. When checking the moves $R4, V4$ and $R5$, we will assume that the other extended Reidemeister moves correspond to virtual trivalent braids that are $TL_v$-equivalent.

We consider first the $R4$ move with the free strand sliding under the vertex. The $R4$ move with the  strand sliding over the vertex is verified similarly, and thus it is omitted to avoid repetition.

The case of the move in which all the strands are oriented downward (see for example Fig.~\ref{DownY1}) is a virtual trivalent braid relation, thus there is nothing to verify in this case.

\begin{figure}[ht]
\[
\raisebox{-.5cm}{\includegraphics[height=1.5cm]{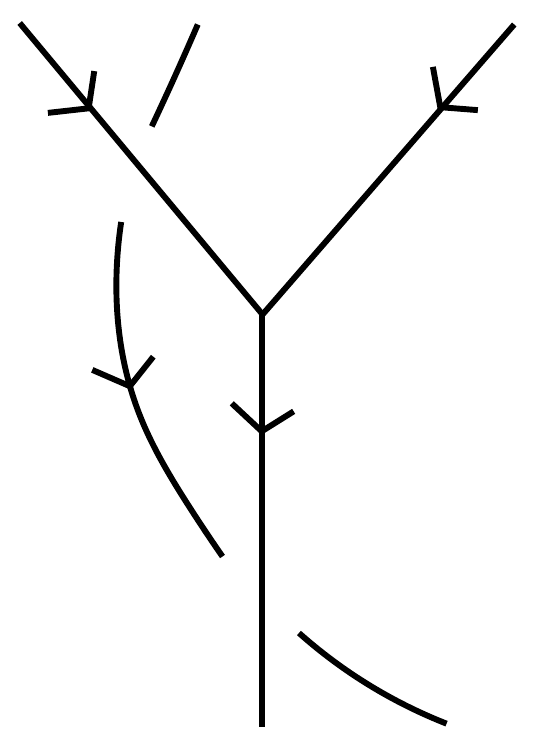}} 
\underset{\text{isotopy}}{\overset{\text{braid}}{\sim}} 
\raisebox{-.5cm}{\includegraphics[height=1.5cm]{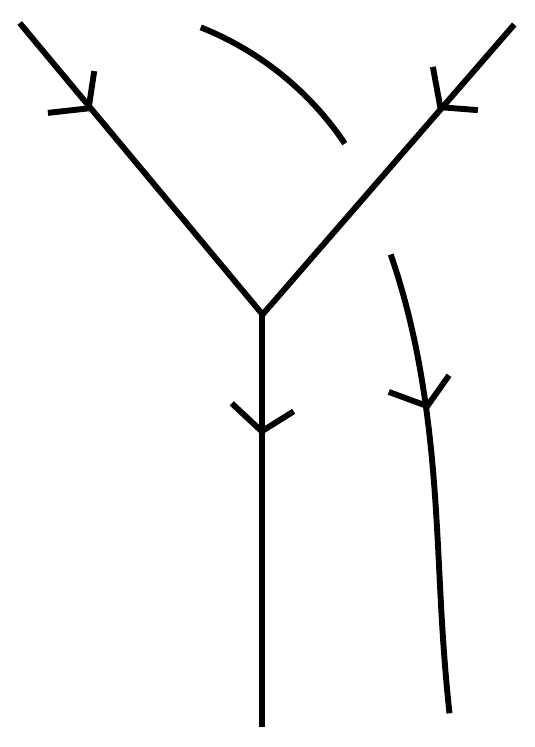}}
\] 
\caption{An $R4$ move in braid form}\label{DownY1}
\end{figure}

We will restrict next to the version of the move in which the free strand is an up-arc; the other case is checked similarly.

The version of the $R4$ move in which the free strand has upward orientation and the edges meeting at the vertex have downward orientation can be obtained using $R2$ moves along with the $R4$ move in braid form, as shown in Figs.~\ref{DownYUp} and~\ref{fig:DDDUL}. 

\begin{figure}[ht]
\[
\raisebox{-.5cm}{\includegraphics[height=1.75cm]{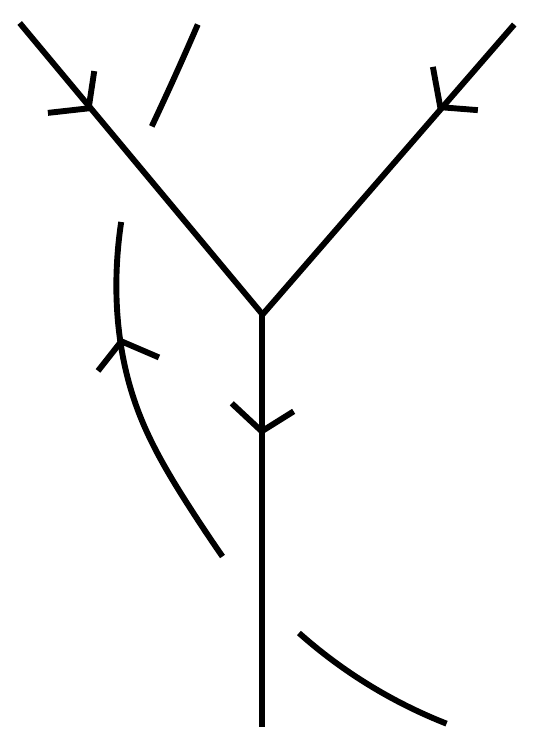}}  \hspace{.05in} \stackrel{\text{R2}}{\longleftrightarrow} \hspace{.05in}
\raisebox{-.5cm}{\includegraphics[height=1.75cm]{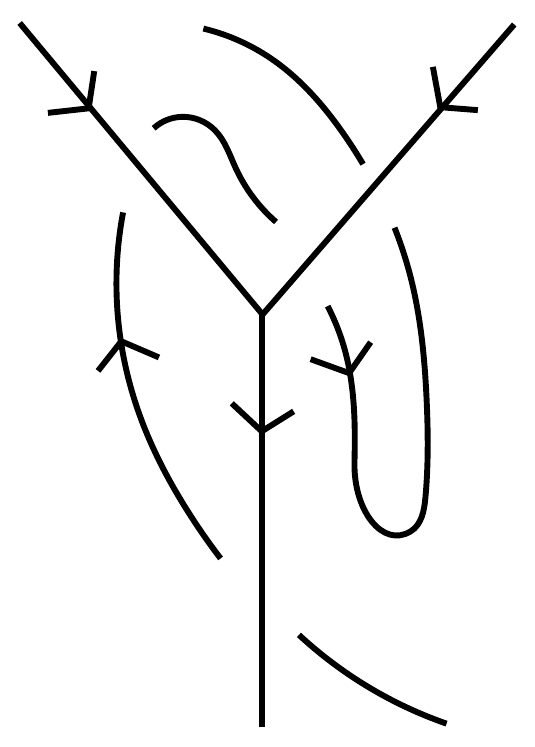}} \hspace{.05in}
\underset{\text{R4}}{\overset{\text{braid}}{\longleftrightarrow}} \hspace{.05in}
\raisebox{-.5cm}{\includegraphics[height=1.75cm]{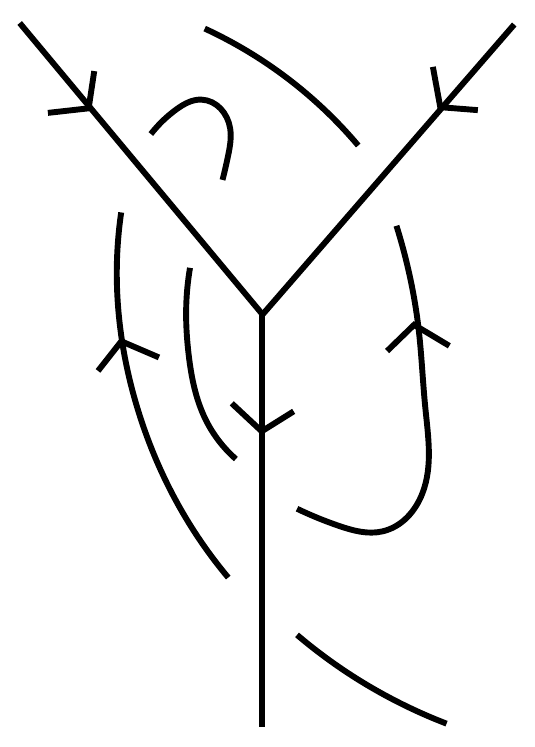}}\hspace{.05in}
\stackrel{\text{R2}}{\longleftrightarrow}\hspace{.05in}
\raisebox{-.5cm}{\includegraphics[height=1.75cm]{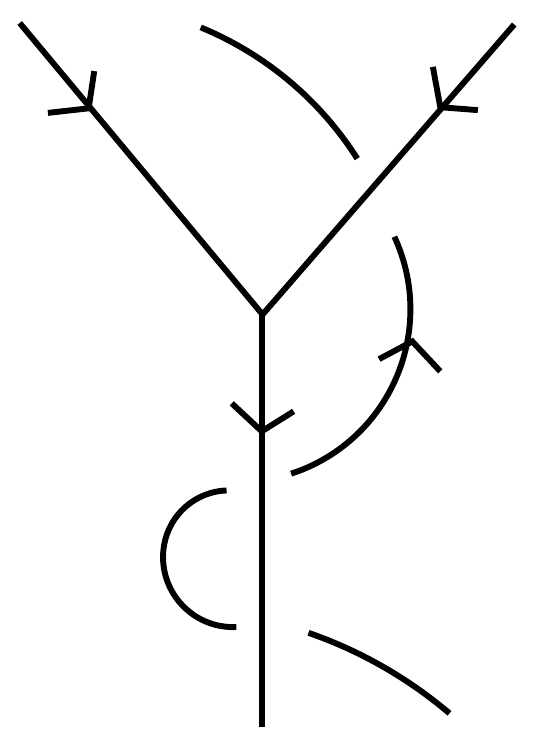}}
 \stackrel{\text{R2}}{\longleftrightarrow}\hspace{.05in}
\raisebox{-.5cm}{\includegraphics[height=1.75cm]{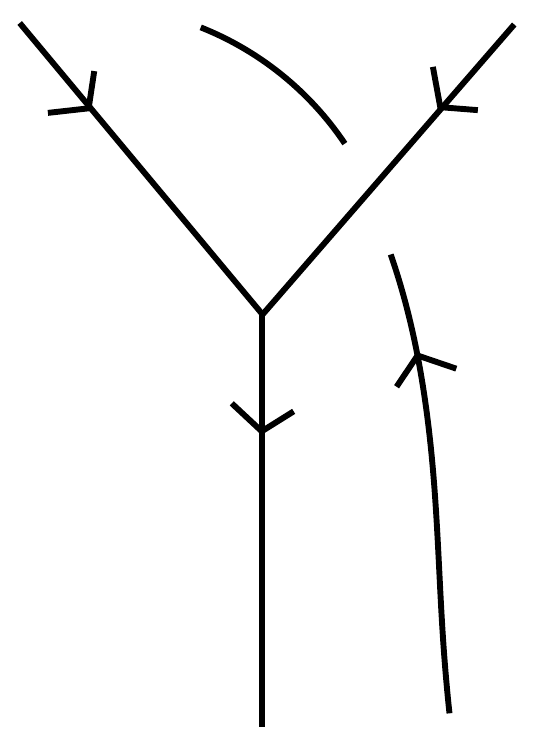}}
\]
\caption{An $R4$ move on a $Y$-vertex with 3 down arcs}\label{DownYUp}
\end{figure}

\begin{figure}[ht] 
\[
\raisebox{-.5cm}{\includegraphics[height=1.775cm]{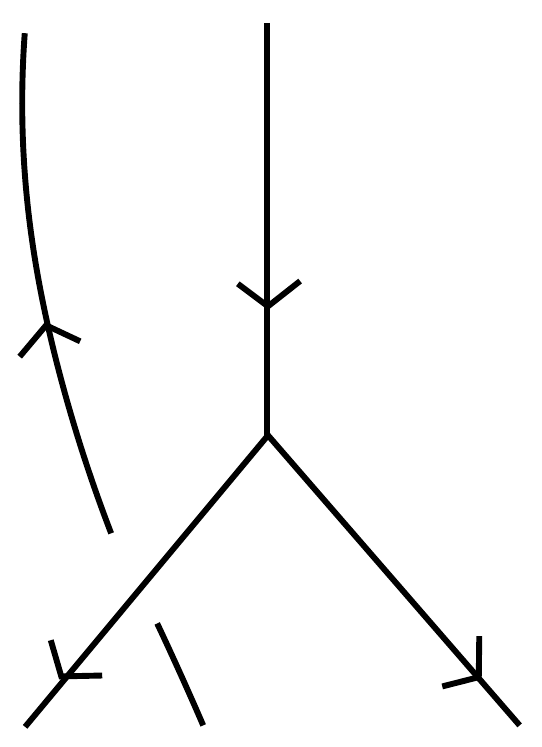}}  \hspace{.05in} \stackrel{\text{R2}}{\longleftrightarrow} \hspace{.05in}
\raisebox{-.5cm}{\includegraphics[height=1.75cm]{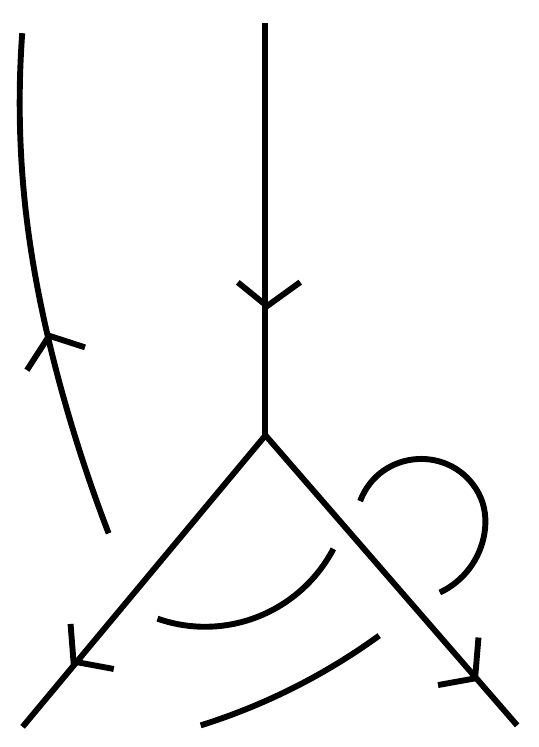}} \hspace{.05in}
\underset{\text{R4}}{\overset{\text{braid}}{\longleftrightarrow}} \hspace{.05in}
\raisebox{-.5cm}{\includegraphics[height=1.75cm]{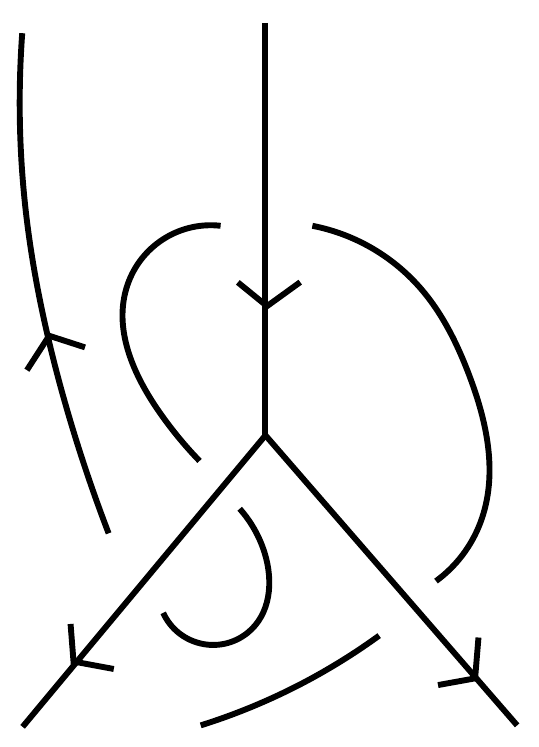}}\hspace{.05in}
 \stackrel{\text{R2}}{\longleftrightarrow}\hspace{.05in}
\raisebox{-.5cm}{\includegraphics[height=1.75cm]{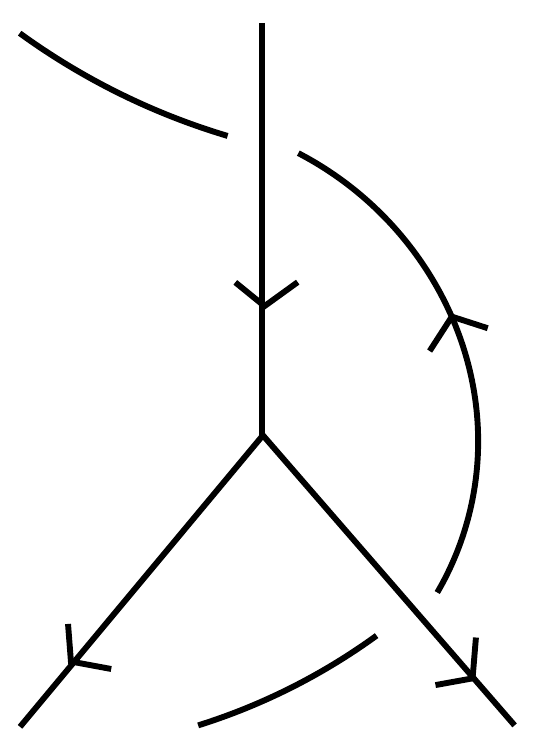}}
\] 
\caption{$R4$ move on a $\lambda$-vertex with down-arcs}\label{fig:DDDUL}
\end{figure}

The next case we consider is an $R4$ move applied locally to a $Y$-vertex with one up-arc. We demonstrate one such case in Fig.~\ref{R42Up}. We must first put the vertex in regular position and then we apply the $R4$ move on a $\lambda$-vertex with three down-arcs (which was checked in the previous case).  

\begin{figure}[ht]
\[
\raisebox{-.5cm}{\includegraphics[height=1.75cm]{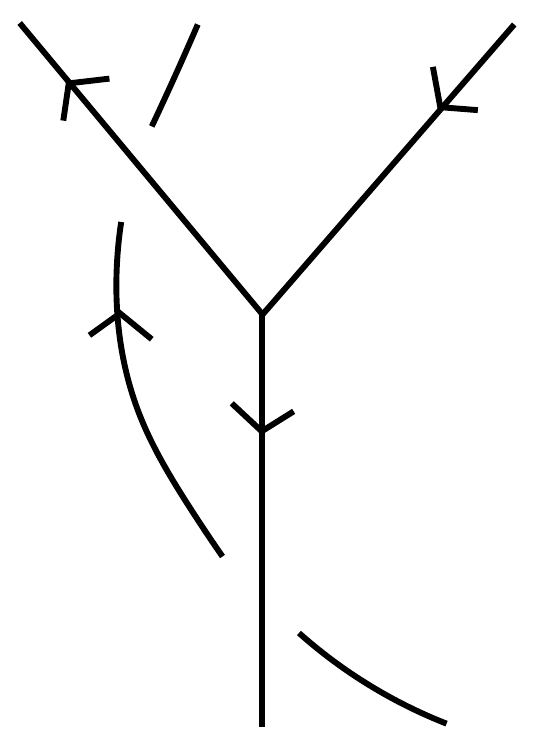}}  \hspace{.05in} \stackrel{\text{RP}}{\longrightarrow} \hspace{.05in}
\raisebox{-.5cm}{\includegraphics[height=1.75cm]{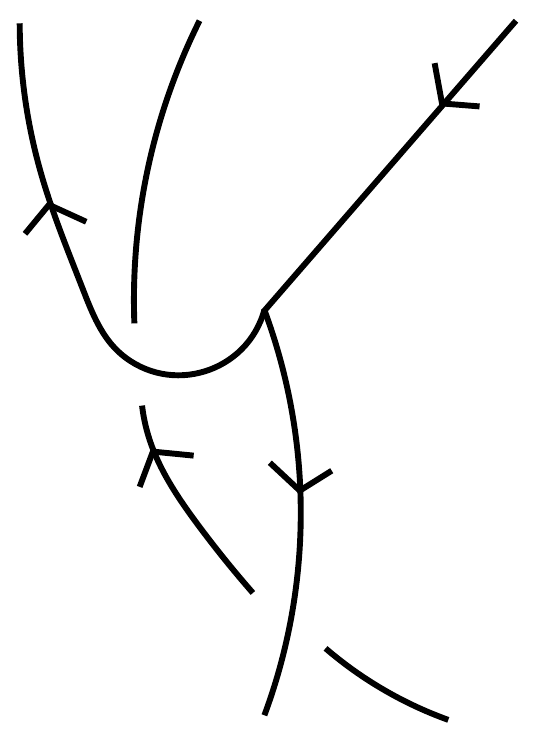}} \hspace{.05in}
\stackrel{\text{R4}}{\longleftrightarrow}\hspace{.05in}
\raisebox{-.5cm}{\includegraphics[height=1.75cm]{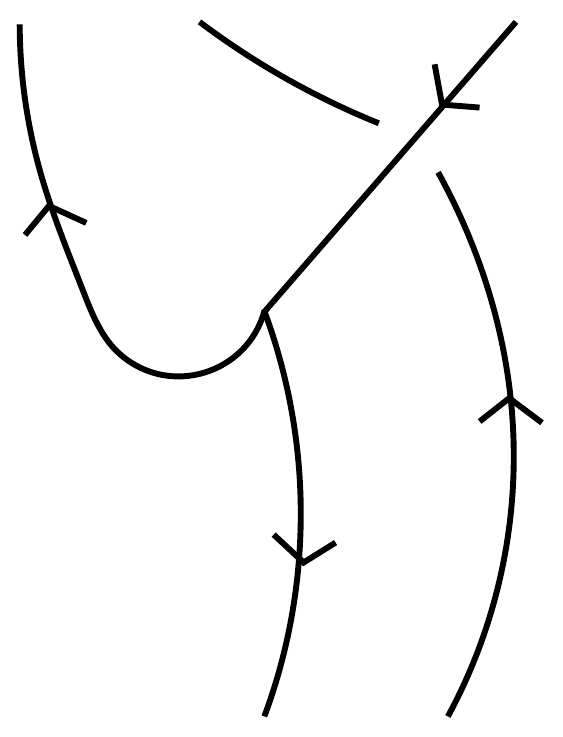}}
\hspace{.05in} \stackrel{\text{RP}}{\longleftarrow} \hspace{.05in}
\raisebox{-.5cm}{\includegraphics[height=1.75cm]{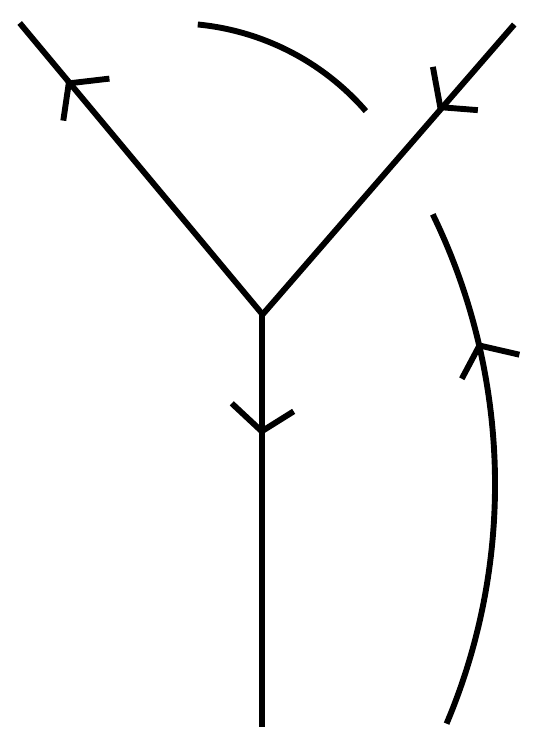}}
\]
\caption{An $R4$ move on a $Y$-vertex with one up-arc}\label{R42Up}
\end{figure}

We now consider the case of an $R4$ move applied locally to a $Y$-vertex with two up-arcs. We demonstrate a particular case in Fig.~\ref{R42Up2}. After the vertex is put in regular position, the move can be performed using $R2$ and $R3$ moves together with an $R4$ move involving a $Y$-vertex with downward oriented arcs (which has been checked). 

\begin{figure}[ht]
\[
\raisebox{-.5cm}{\includegraphics[height=1.75cm]{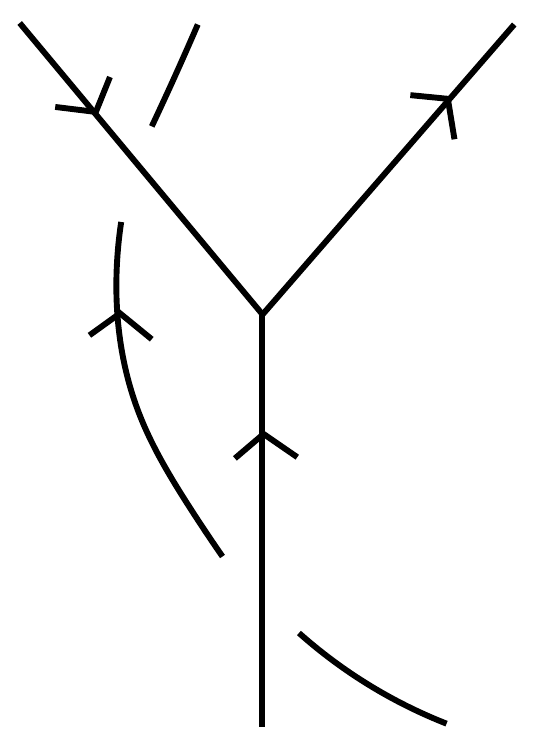}}  \hspace{.05in} \stackrel{\text{RP}}{\longrightarrow} \hspace{.05in}
\raisebox{-.5cm}{\includegraphics[height=1.75cm]{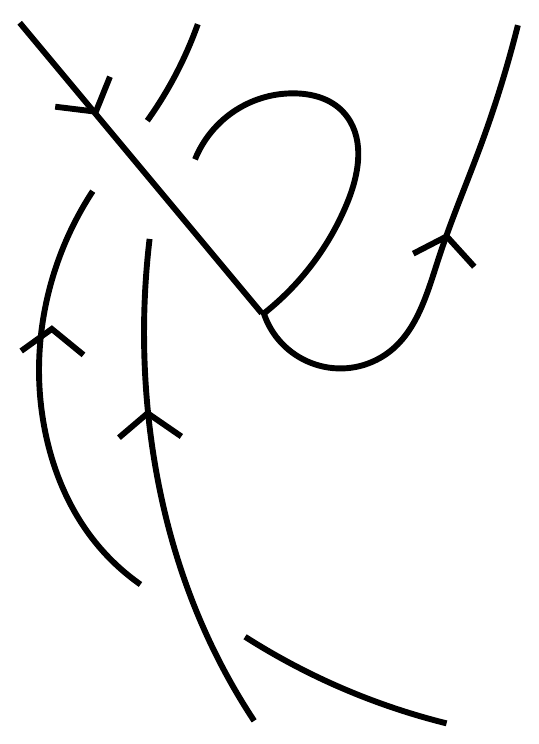}} \hspace{.05in}
 \stackrel{\text{R2, R3}}{\longleftrightarrow}\hspace{.05in}
\raisebox{-.5cm}{\includegraphics[height=1.75cm]{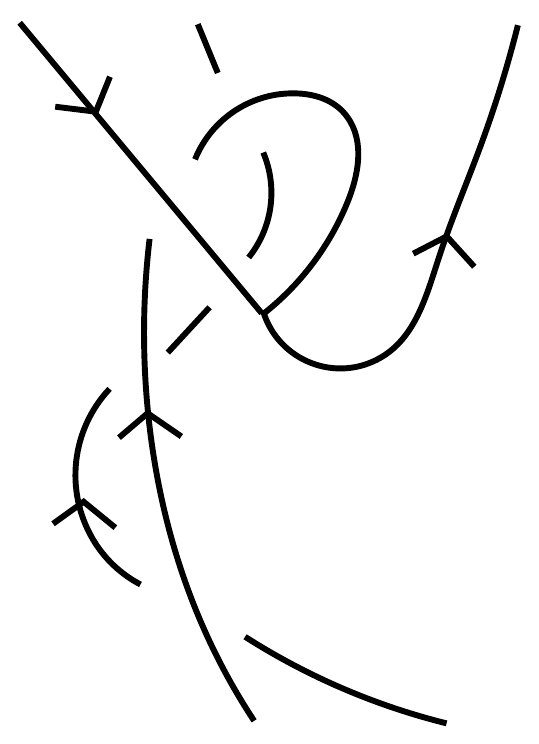}}\hspace{.05in}
\stackrel{\text{R4}}{\longleftrightarrow}\hspace{.05in}
\raisebox{-.5cm}{\includegraphics[height=1.75cm]{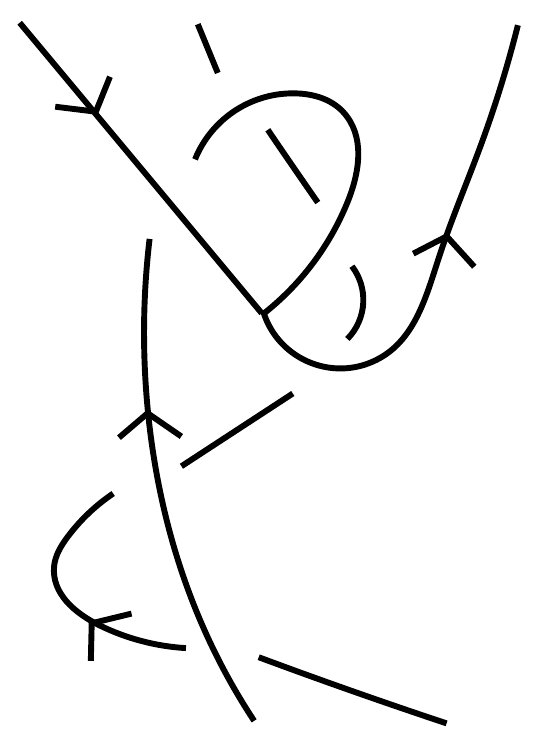}}
 \stackrel{\text{R2}}{\longleftrightarrow}\hspace{.05in}
\raisebox{-.5cm}{\includegraphics[height=1.75cm]{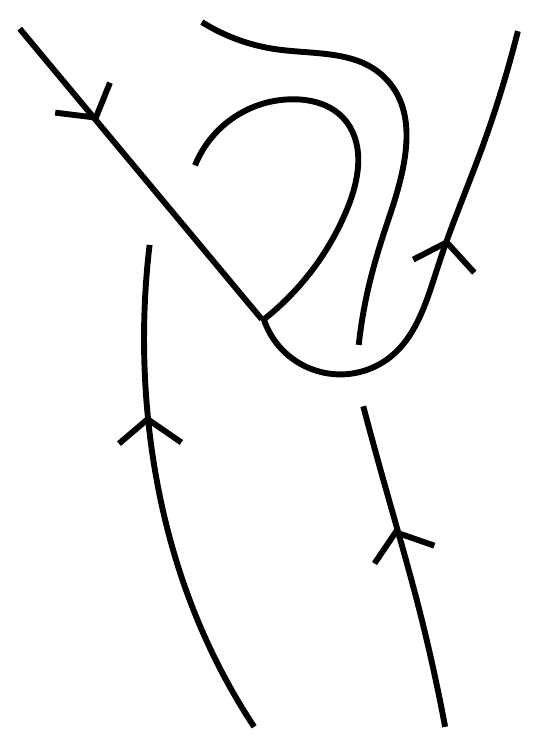}}
 \stackrel{\text{RP}}{\longleftarrow}\hspace{.05in}
\raisebox{-.5cm}{\includegraphics[height=1.75cm]{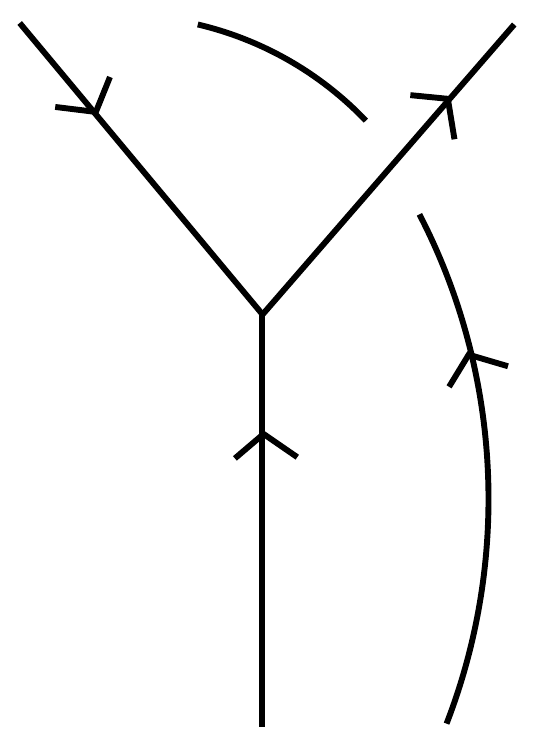}}
\]
\caption{An $R4$ move on a $Y$-vertex with two up-arcs}\label{R42Up2}
\end{figure}

Next we check an $R4$ move involving a $Y$-vertex with three up-arcs. As in the previous case, after putting the vertex in regular position, we see that this version of the move reduces to an $R4$ move on a $\lambda$-vertex with down-arcs, together with the moves $R2$ and $R3$. This is demonstrated in Fig.~\ref{AllUp}.

\begin{figure}[ht]
\[
\raisebox{-.5cm}{\includegraphics[height=1.75cm]{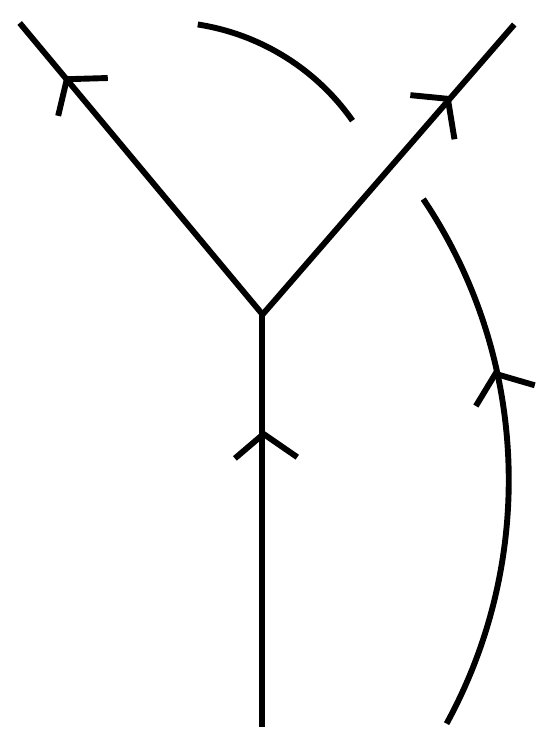}}  \hspace{.03in} \stackrel{\text{RP}}{\longrightarrow} \hspace{.03in}
\raisebox{-.5cm}{\includegraphics[height=1.75cm]{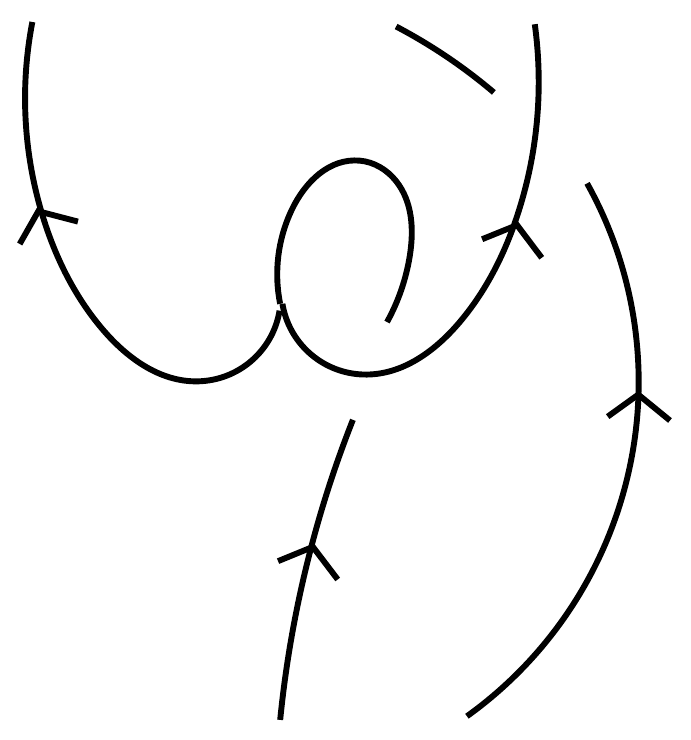}} \hspace{.03in}
\stackrel{\text{R2}}{\longleftrightarrow}\hspace{.05in}
\raisebox{-.5cm}{\includegraphics[height=1.75cm]{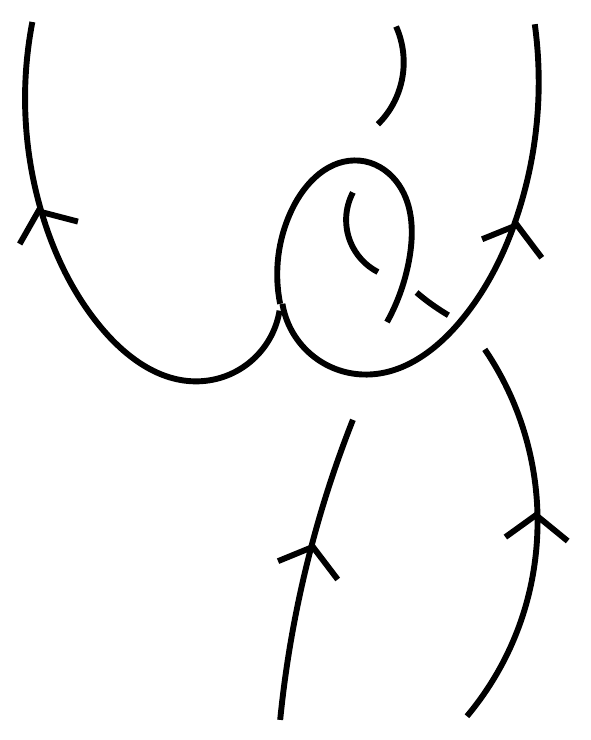}}\hspace{.03in}
 \stackrel{\text{R3}}{\longleftrightarrow}
 \hspace{.03in}
\raisebox{-.5cm}{\includegraphics[height=1.75cm]{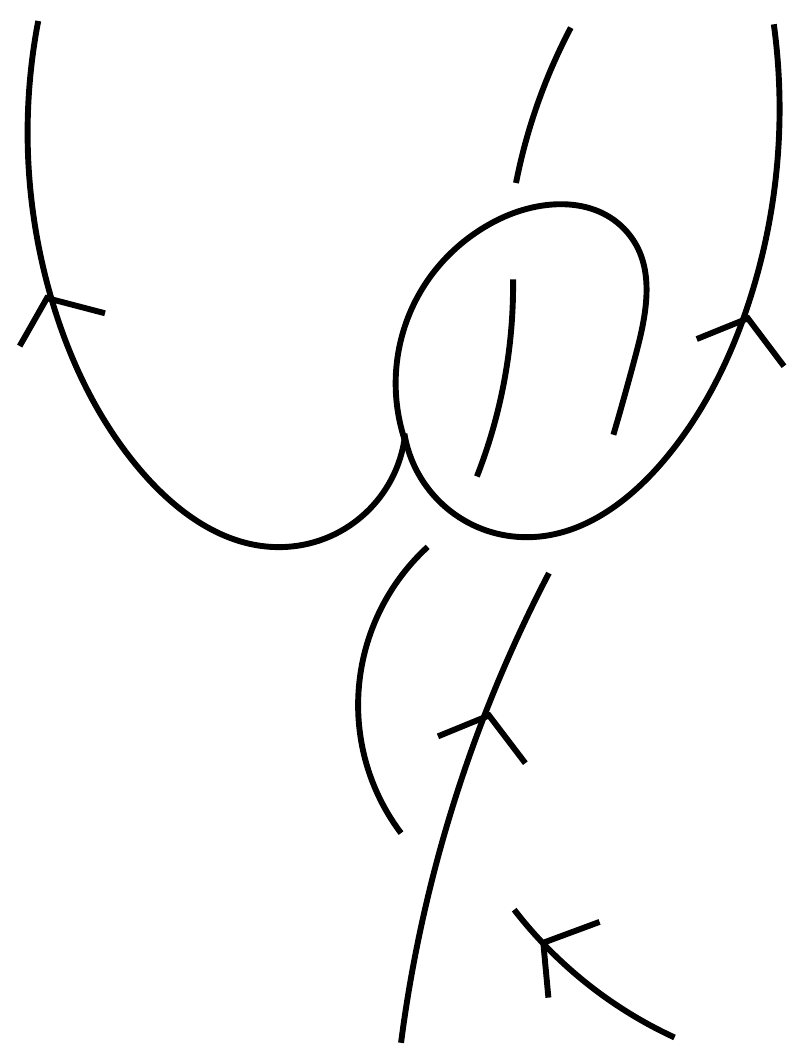}}
 \stackrel{\text{R4}}{\longleftrightarrow}
 \hspace{.03in}
\raisebox{-.5cm}{\includegraphics[height=1.75cm]{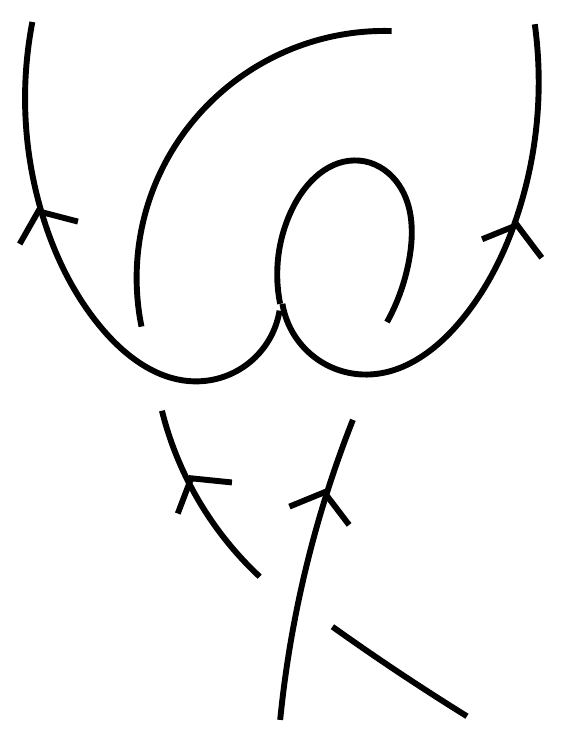}}
 \stackrel{\text{RP}}{\longleftarrow}
 \hspace{.03in}
\raisebox{-.5cm}{\includegraphics[height=1.75cm]{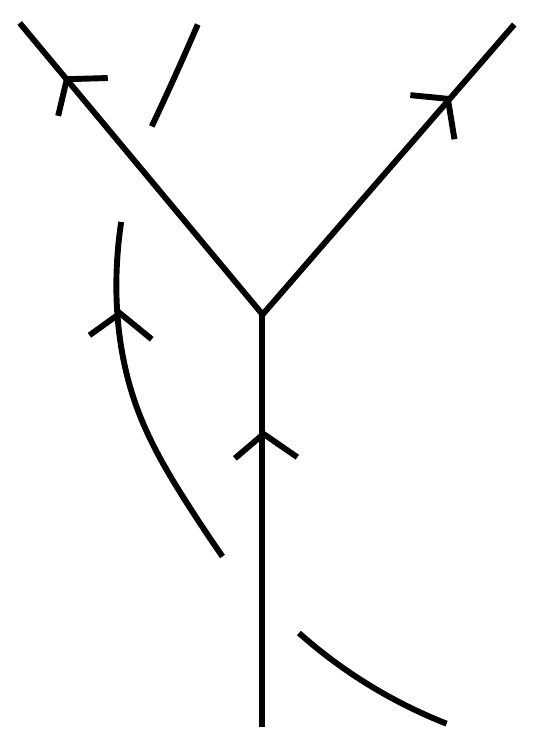}}
\]
\caption{An $R4$ move on a $Y$-vertex with three up-arcs}\label{AllUp}
\end{figure}

Finally, in Fig.~\ref{fig:UUUUL}, we verify an $R4$ move on a $\lambda$-vertex with three up-arcs. After putting the vertex in regular position, the move reduces to an $R4$ move on a $Y$-vertex with three down-arcs. 
\begin{figure}[ht]
\[
\raisebox{-.5cm}{\includegraphics[height=1.775cm]{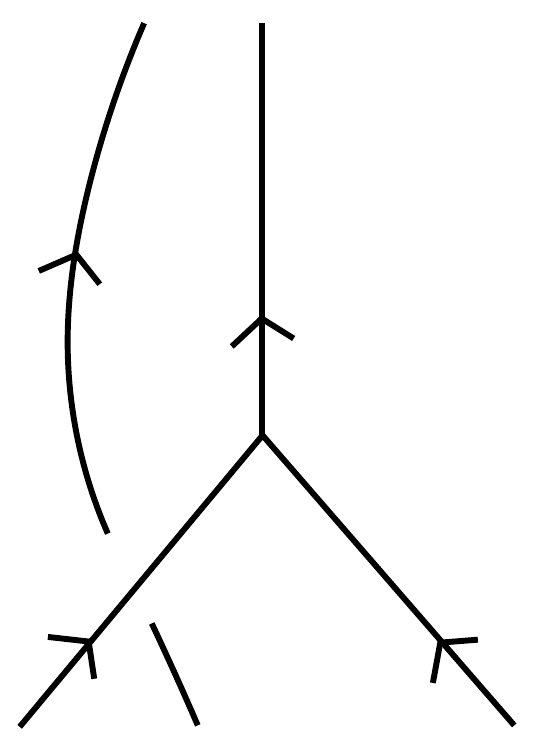}}  \hspace{.05in} \stackrel{\text{RP}}{\longrightarrow} \hspace{.05in}
\raisebox{-.5cm}{\includegraphics[height=1.75cm]{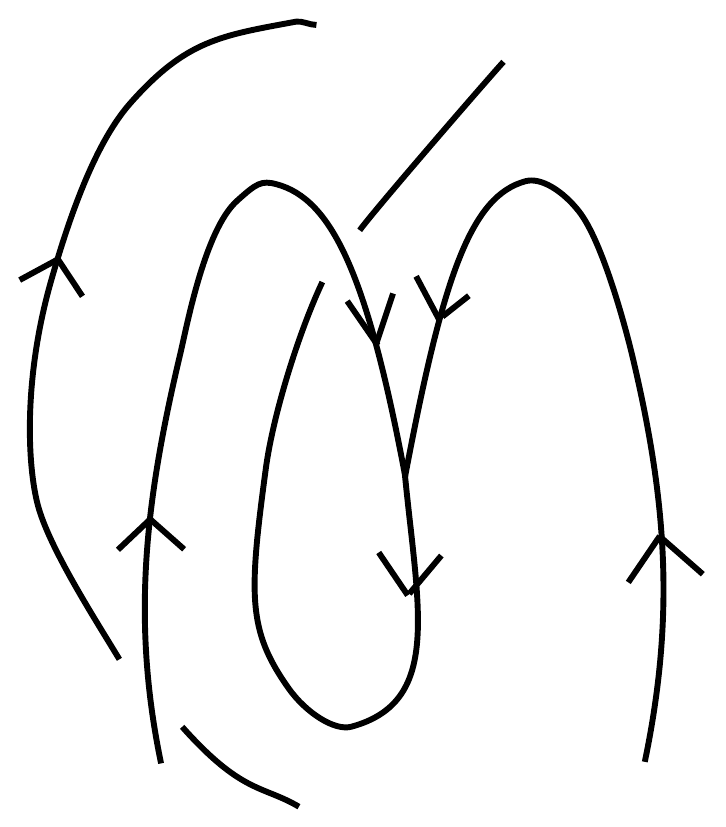}} \hspace{.05in}
 \stackrel{\text{R2, R3}}{\longleftrightarrow}\hspace{.05in}
\raisebox{-.5cm}{\includegraphics[height=1.75cm]{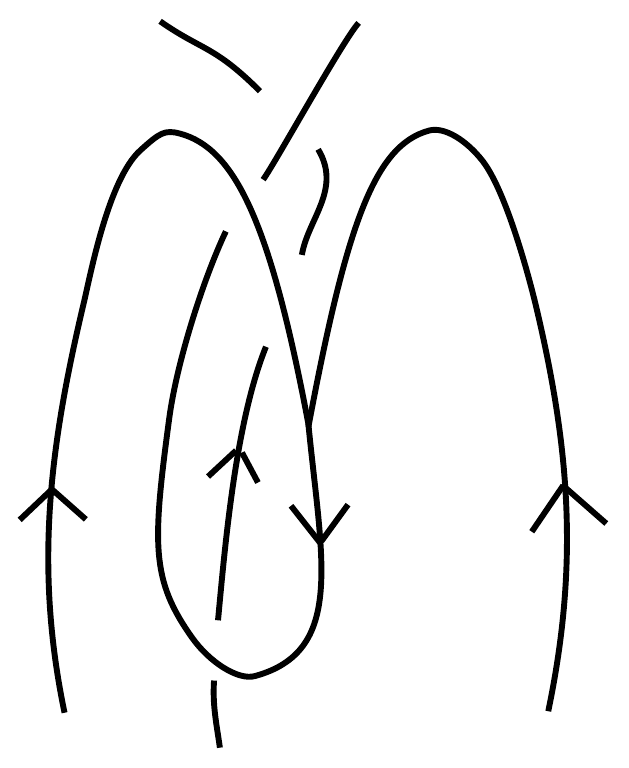}}\hspace{.05in}
 \stackrel{\text{R4, R2}}{\longleftrightarrow}\hspace{.05in}
\raisebox{-.5cm}{\includegraphics[height=1.75cm]{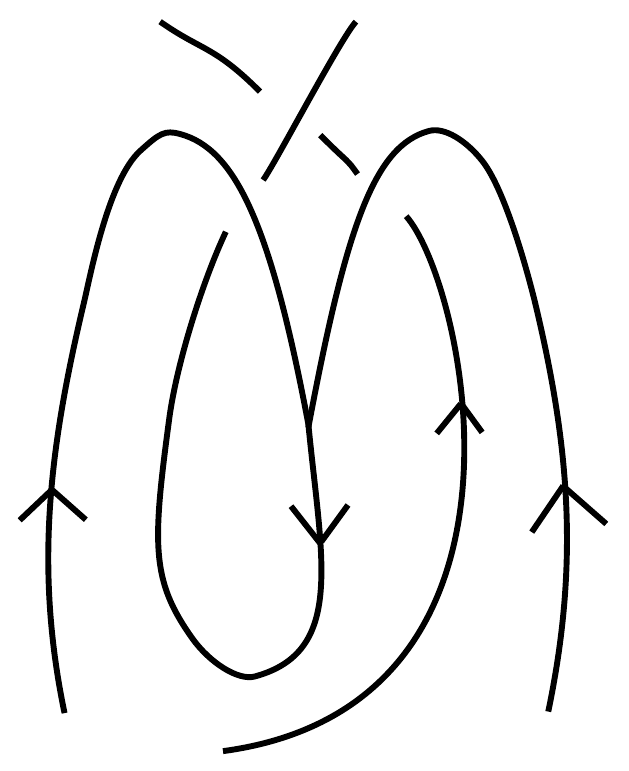}}
 \hspace{.05in} \stackrel{\text{RP}}{\longleftarrow} \hspace{.05in}
\raisebox{-.5cm}{\includegraphics[height=1.75cm]{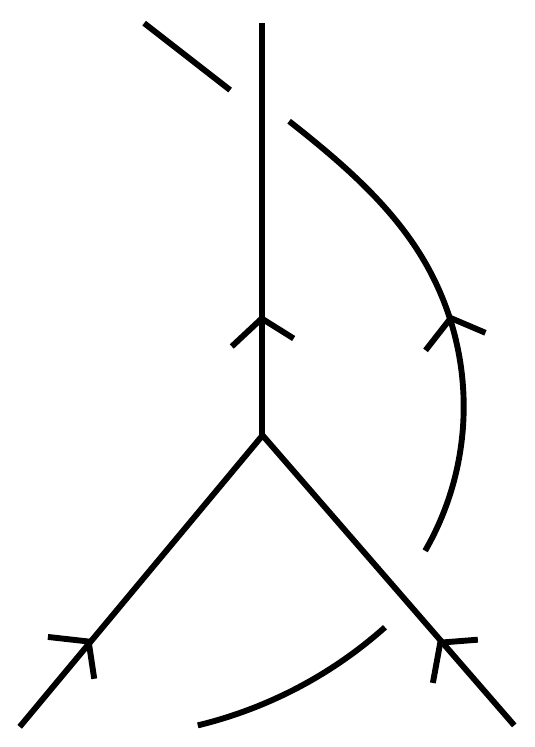}}
\]
\caption{An $R4$ move on a $\lambda$-vertex with three up-arcs} \label{fig:UUUUL}
\end{figure}

The other versions of the $R4$ move are checked in a similar fashion, and follow from cases verified above.

The proof that the versions of the $V4$ move correspond to virtual trivalent braids that are $TL_v$-equivalent is done in a similar manner to the proof for the versions of the $R4$ move, with the main difference that $R2, R3$ and $R4$ moves are replaced by $V2, VR3$ and $V4$ moves, respectively. In Fig.~\ref{AllUp-virt}, we demonstrate a $V4$ move on a $Y$-vertex with three up-arcs.

\begin{figure}[ht]
\[
\raisebox{-.5cm}{\includegraphics[height=1.75cm]{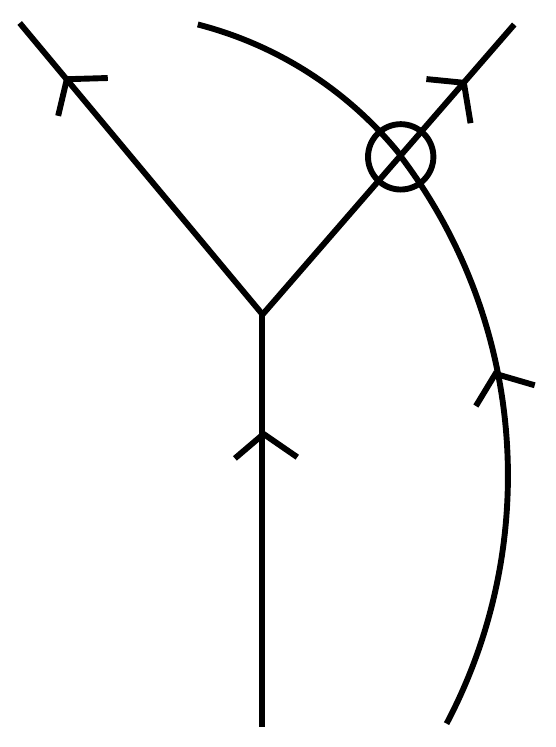}}  \hspace{.03in} \stackrel{\text{RP}}{\longrightarrow} \hspace{.03in}
\raisebox{-.5cm}{\includegraphics[height=1.75cm]{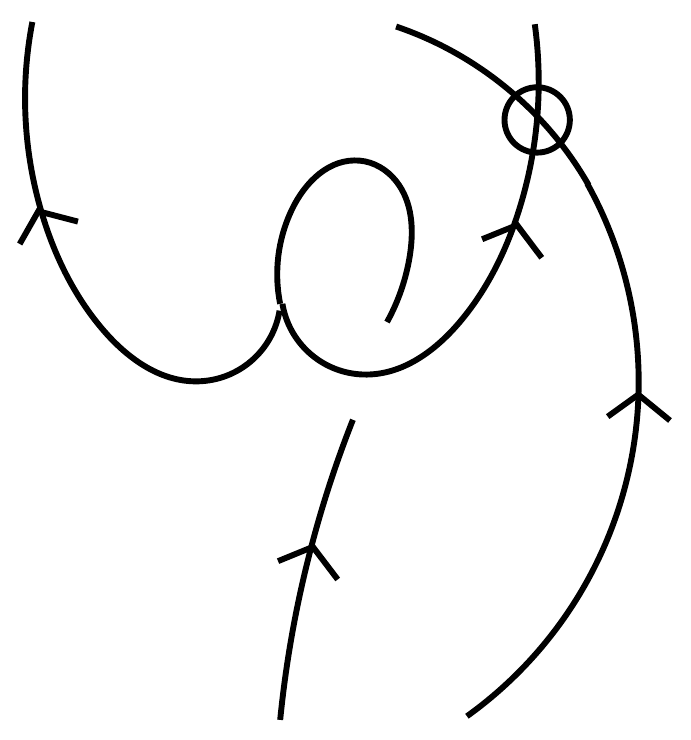}} \hspace{.03in}
\stackrel{\text{V2}}{\longleftrightarrow}\hspace{.05in}
\raisebox{-.5cm}{\includegraphics[height=1.75cm]{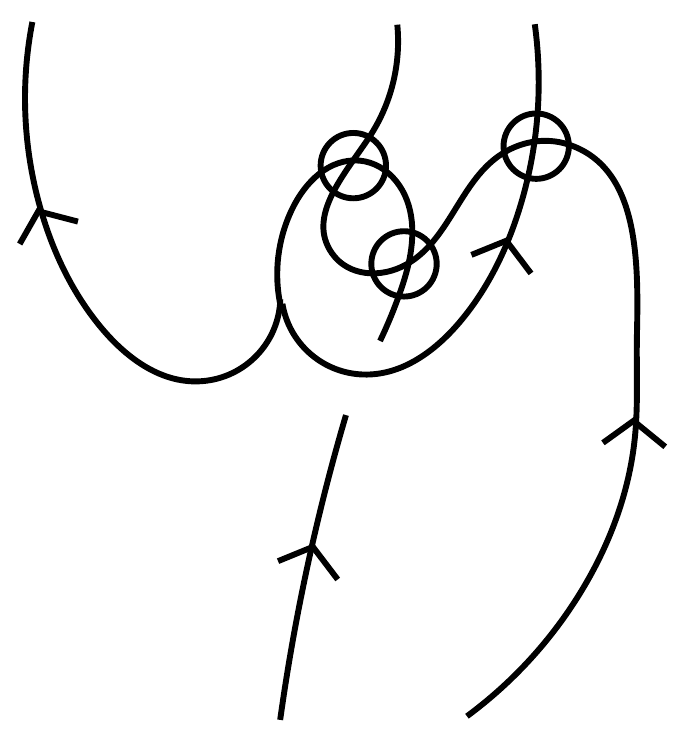}}\hspace{.03in}
 \stackrel{\text{VR3}}{\longleftrightarrow}
 \hspace{.03in}
\raisebox{-.5cm}{\includegraphics[height=1.75cm]{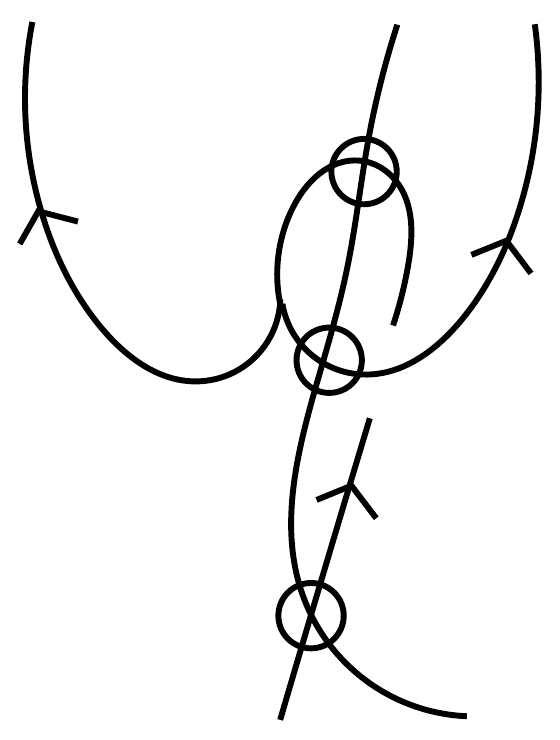}}
 \stackrel{\text{R4}}{\longleftrightarrow}
 \hspace{.03in}
\raisebox{-.5cm}{\includegraphics[height=1.75cm]{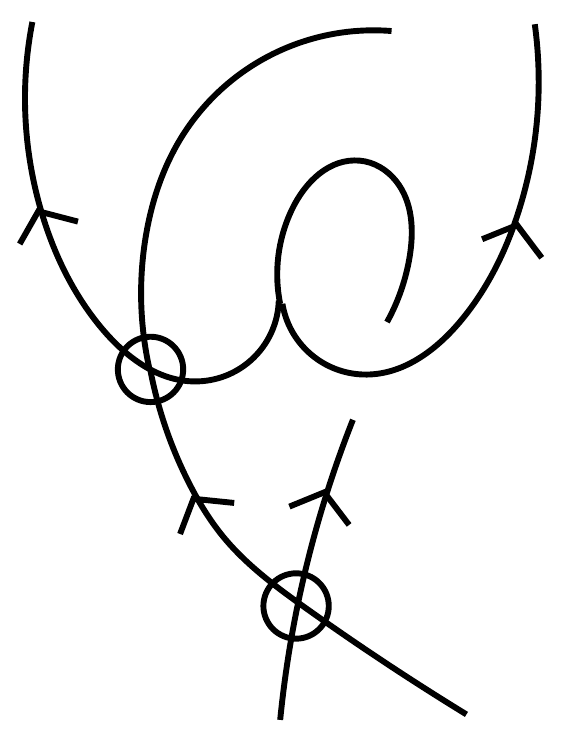}}
 \stackrel{\text{RP}}{\longleftarrow}
 \hspace{.03in}
\raisebox{-.5cm}{\includegraphics[height=1.75cm]{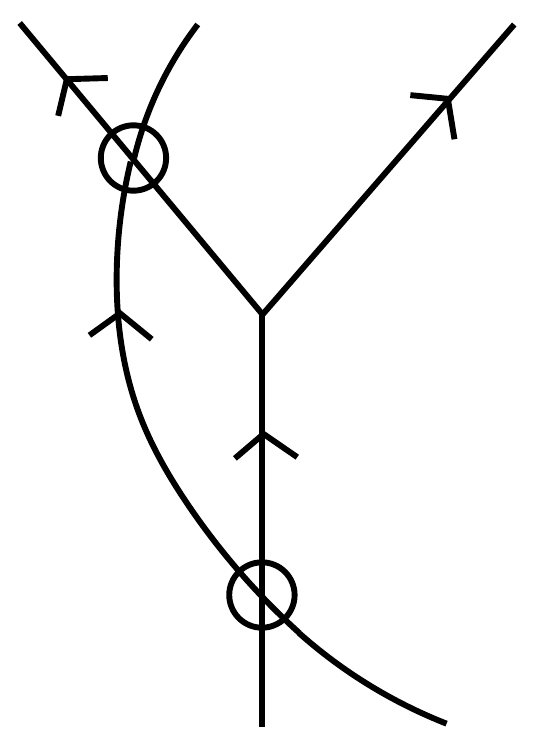}}
\]
\caption{A $V4$ move on a $Y$-vertex with three up-arcs}\label{AllUp-virt}
\end{figure}


Lastly, we consider the $R5$ move. The simplest cases where the three edges meeting at the $Y$-vertex or $\lambda$-vertex are oriented downward are, in fact, braid relations.

We first consider versions of the $R5$ move on  $Y$-vertex and involving the two upper edges meeting at the vertex. A particular case when the $Y$-vertex has one up-arc is shown in Fig.~\ref{fig:YUDD_up}. After putting the vertex in regular position, the move can be performed using $R4$ and $R1$ moves (which have been verified), together with an $R5$ move in braid form. All other cases of the $R5$ move between the two upper edges of a $Y$-vertex are checked in a similar manner.

\begin{figure}[ht]
\[
\raisebox{-.5cm}{\includegraphics[height=1.775cm]{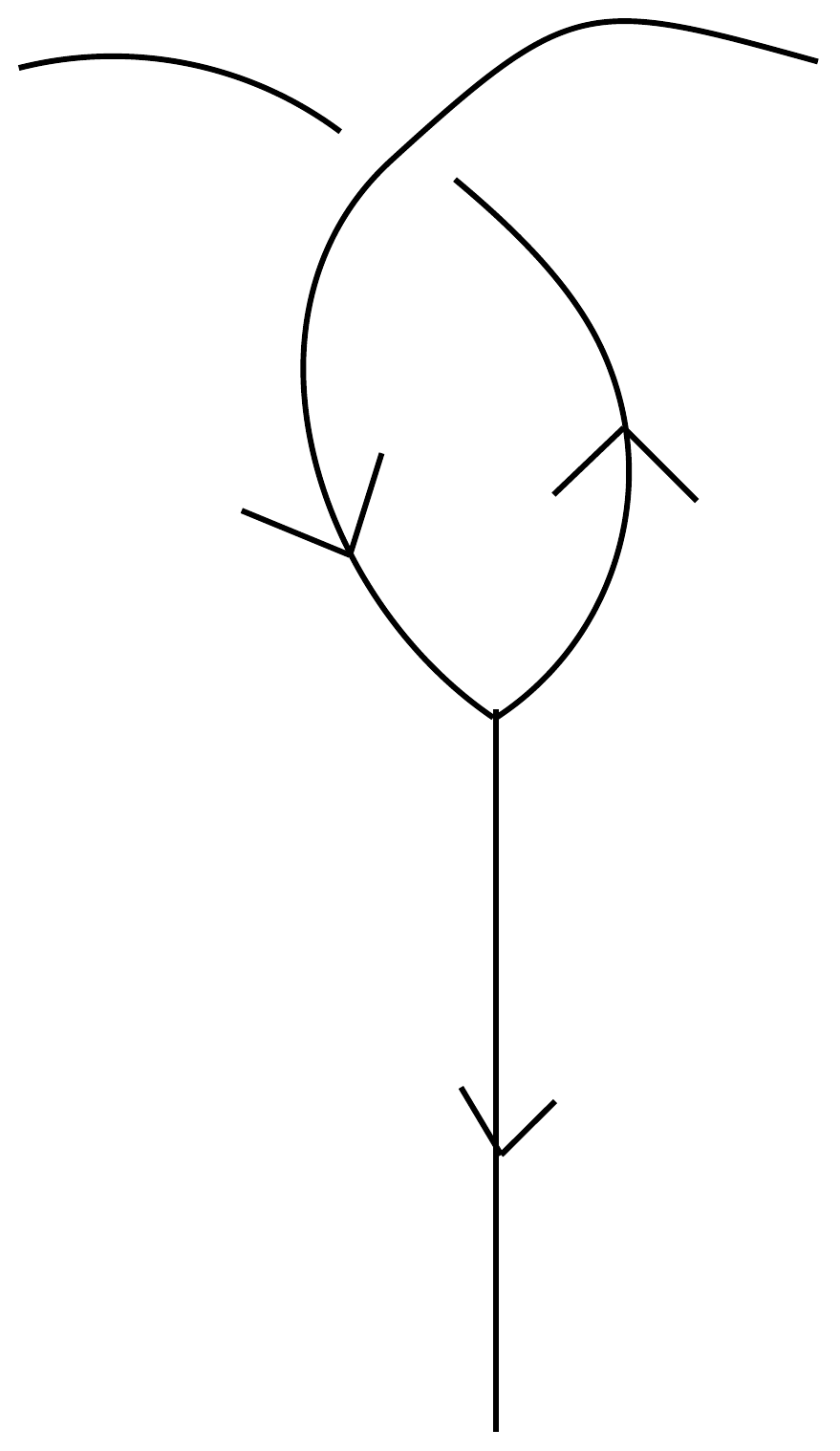}}  \hspace{.05in} \stackrel{\text{RP}}{\longrightarrow} \hspace{.05in}
\raisebox{-.5cm}{\includegraphics[height=1.775cm]{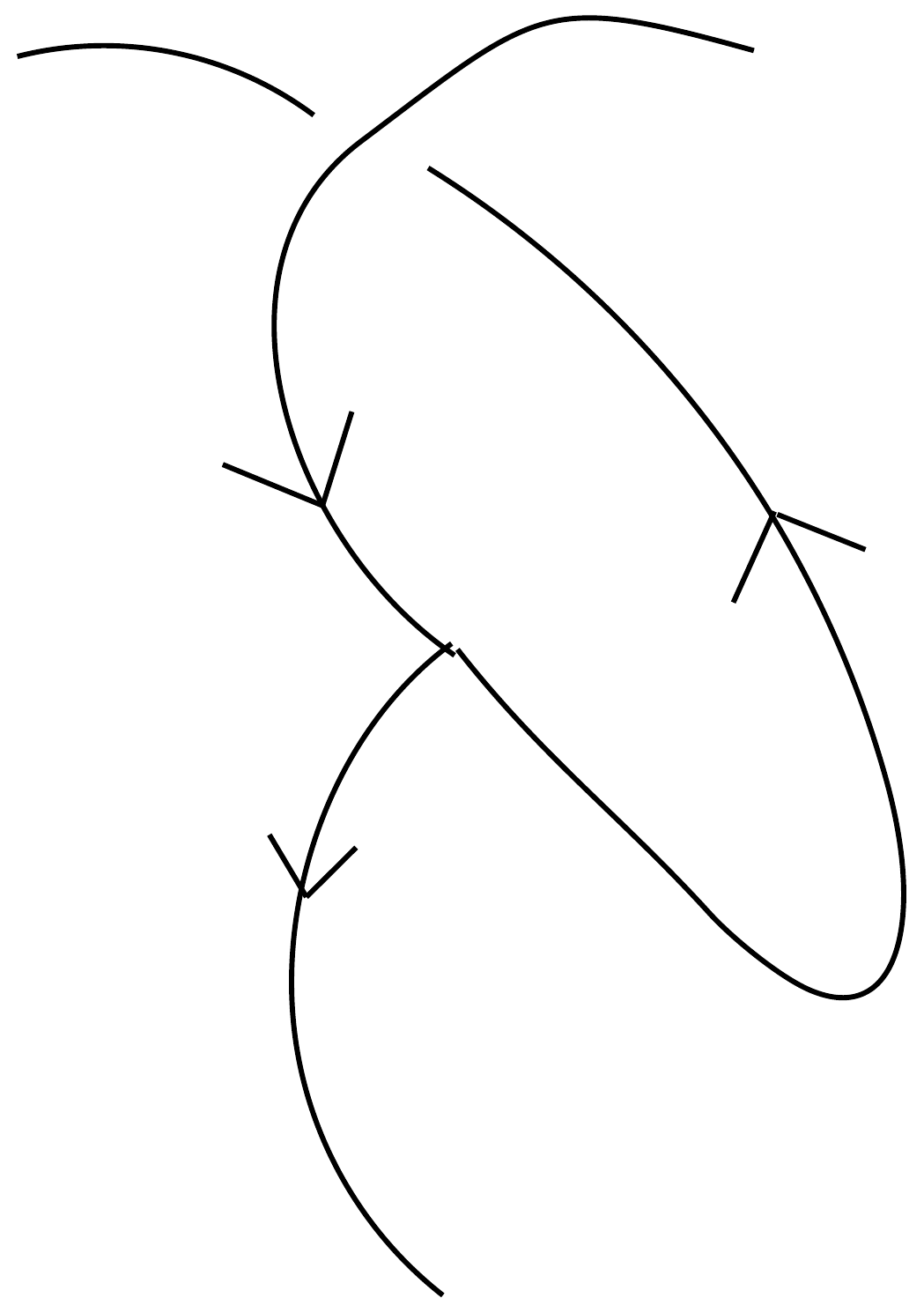}}  \hspace{.05in} \stackrel{\text{R4}}{\longleftrightarrow} \hspace{.05in}
\raisebox{-.5cm}{\includegraphics[height=1.775cm]{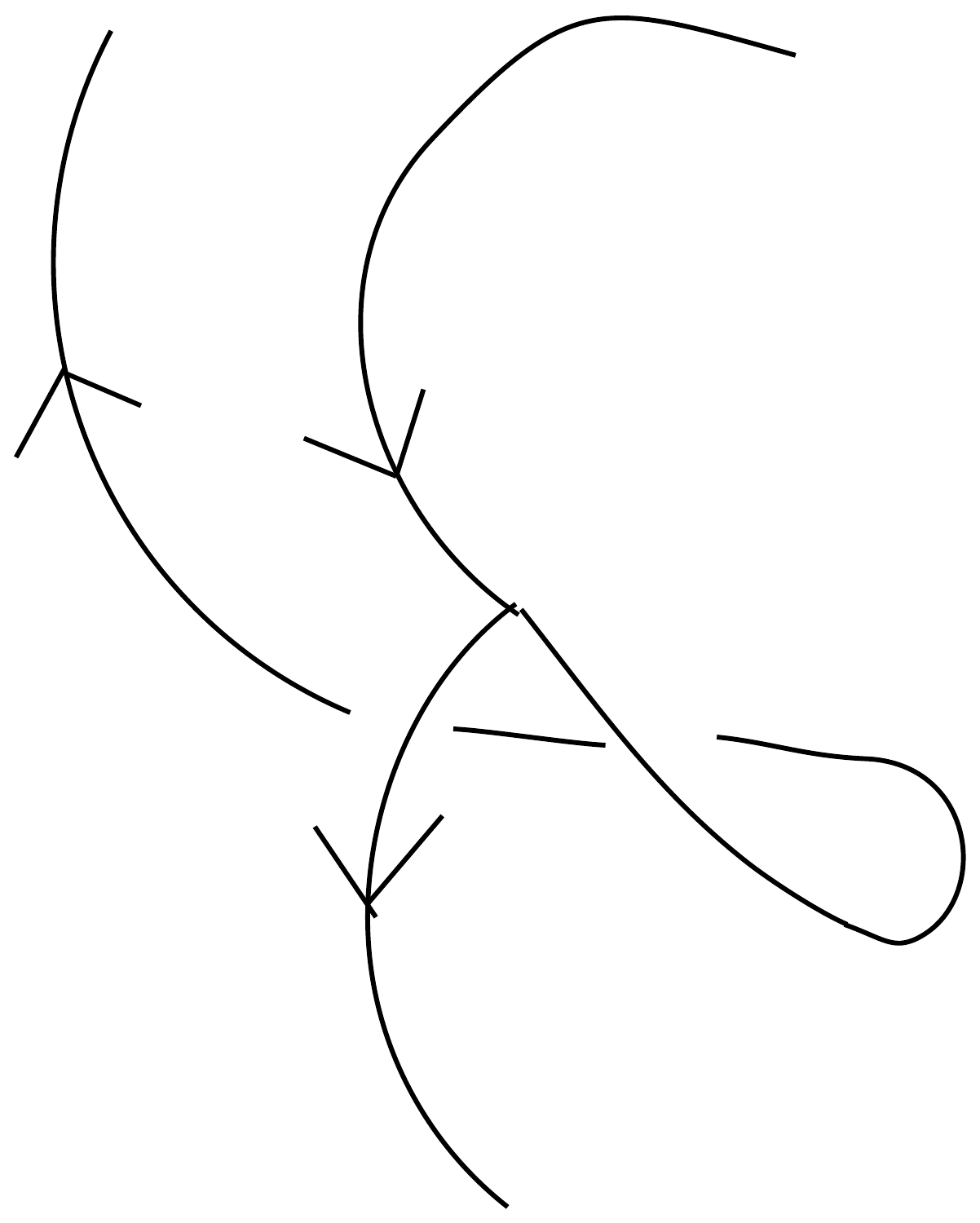}}  \hspace{.05in} \stackrel{\text{R1}}{\longleftrightarrow} \hspace{.05in}
\raisebox{-.5cm}{\includegraphics[height=1.775cm]{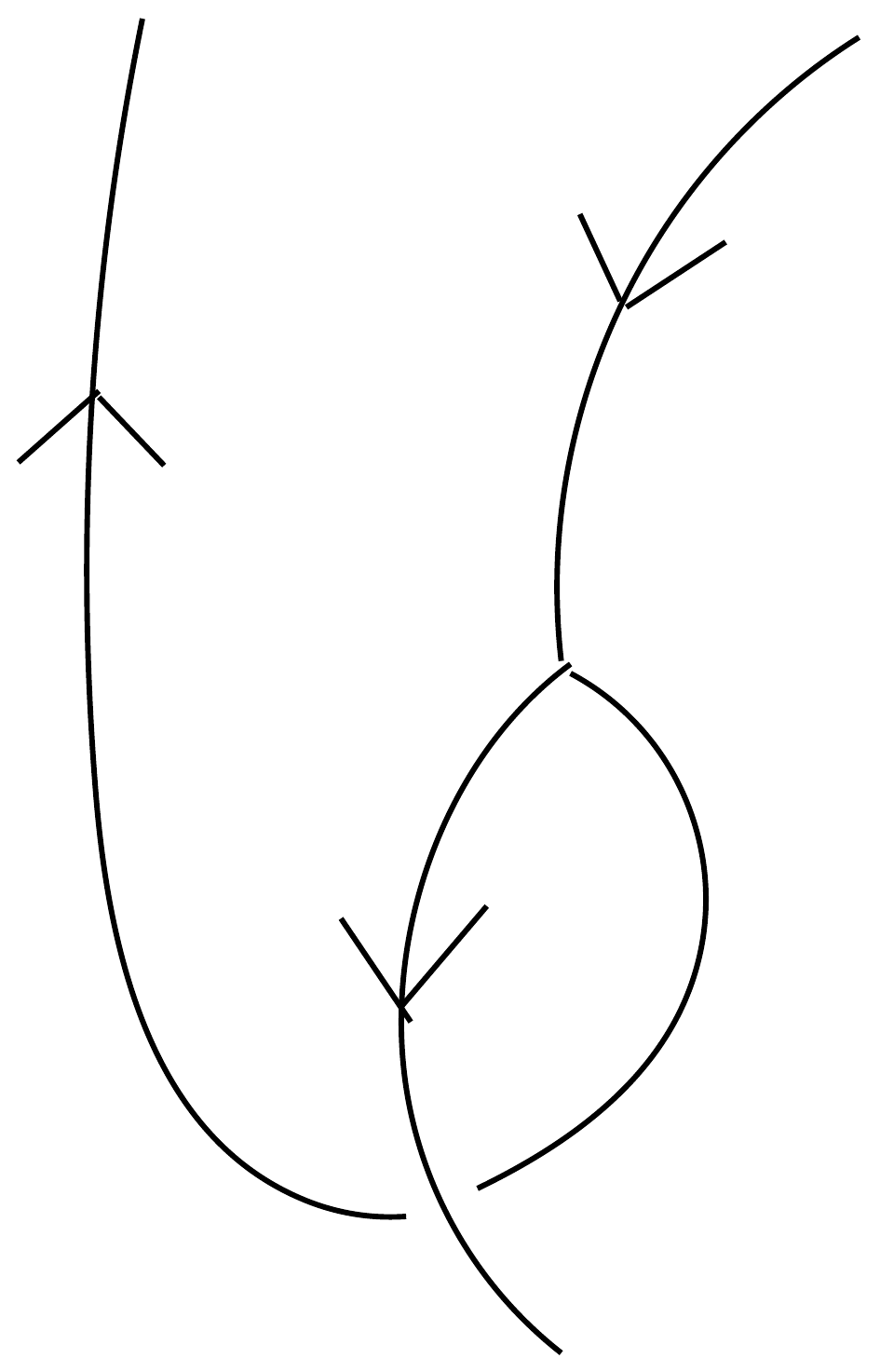}}  \hspace{.05in} \underset{\text{R5}}{\overset{\text{braid}}{\longleftrightarrow}} \hspace{.05in}
\raisebox{-.5cm}{\includegraphics[height=1.775cm]{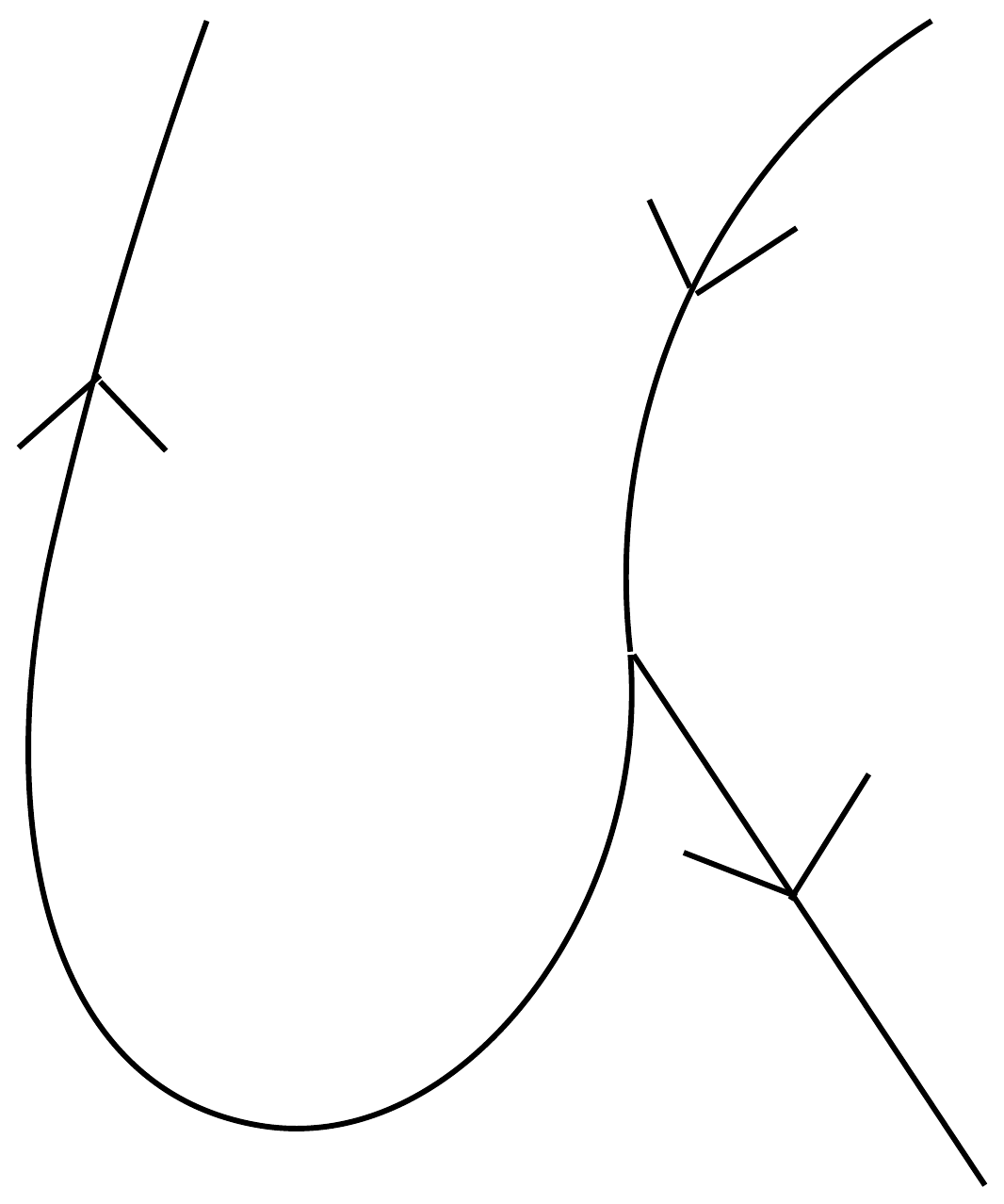}}  \hspace{.05in} \stackrel{\text{RP}}{\longleftarrow} \hspace{.05in}
\raisebox{-.5cm}{\includegraphics[height=1.775cm]{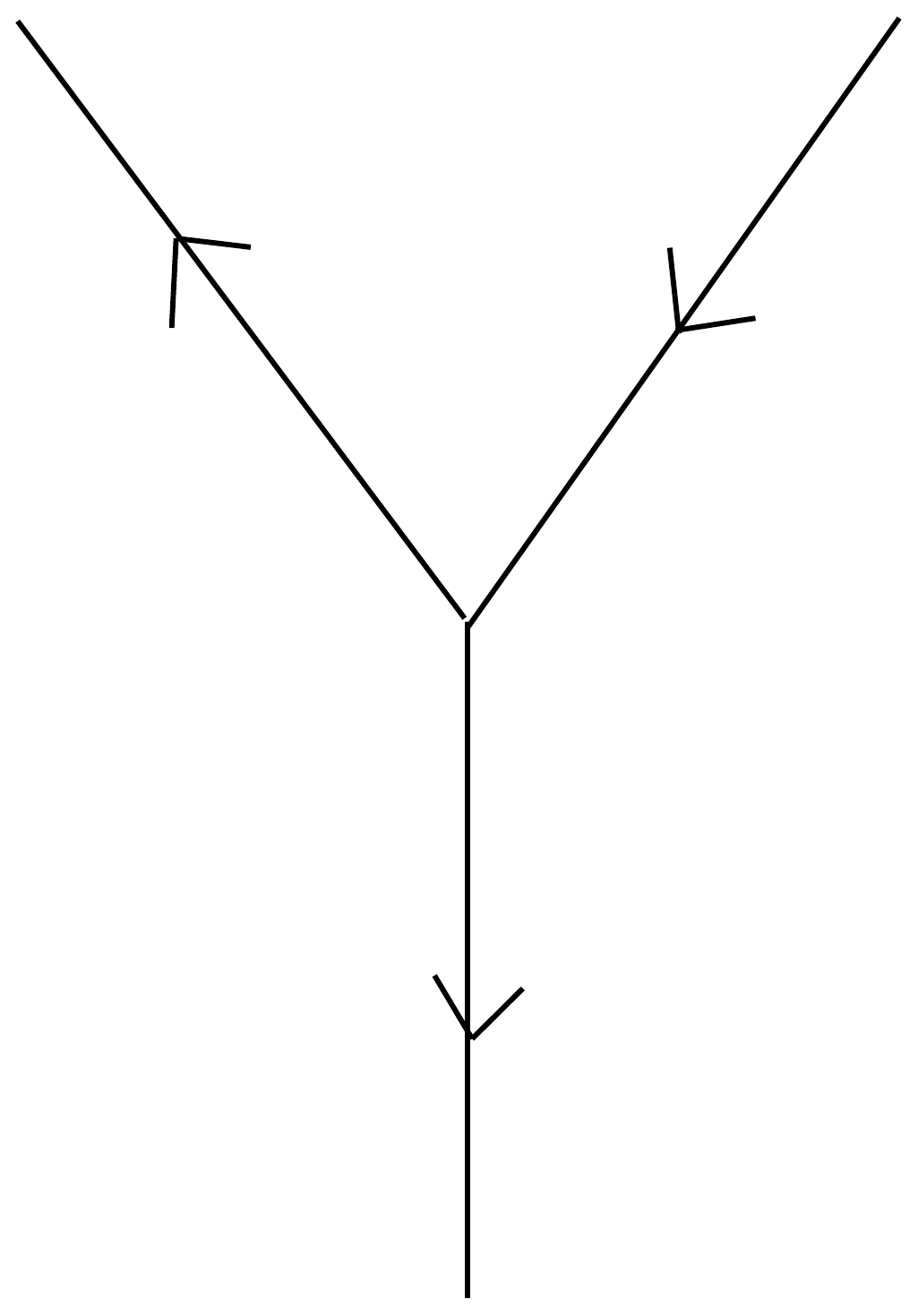}}
\]
\caption{An $R5$ move with one up-arc}\label{fig:YUDD_up}
\end{figure}
We consider next an $R5$ move with a twist between the right upper edge and the lower edge incident with a $Y$-vertex. Figure \ref{fig:YUDU_right} shows a particular case of the move on a $Y$-vertex with two up-arcs. All of the other cases involving different orientations follow similarly, from extended Reidemeister move that have been checked before. 

\begin{figure}[ht]
\[
\raisebox{-.5cm}{\includegraphics[height=1.775cm]{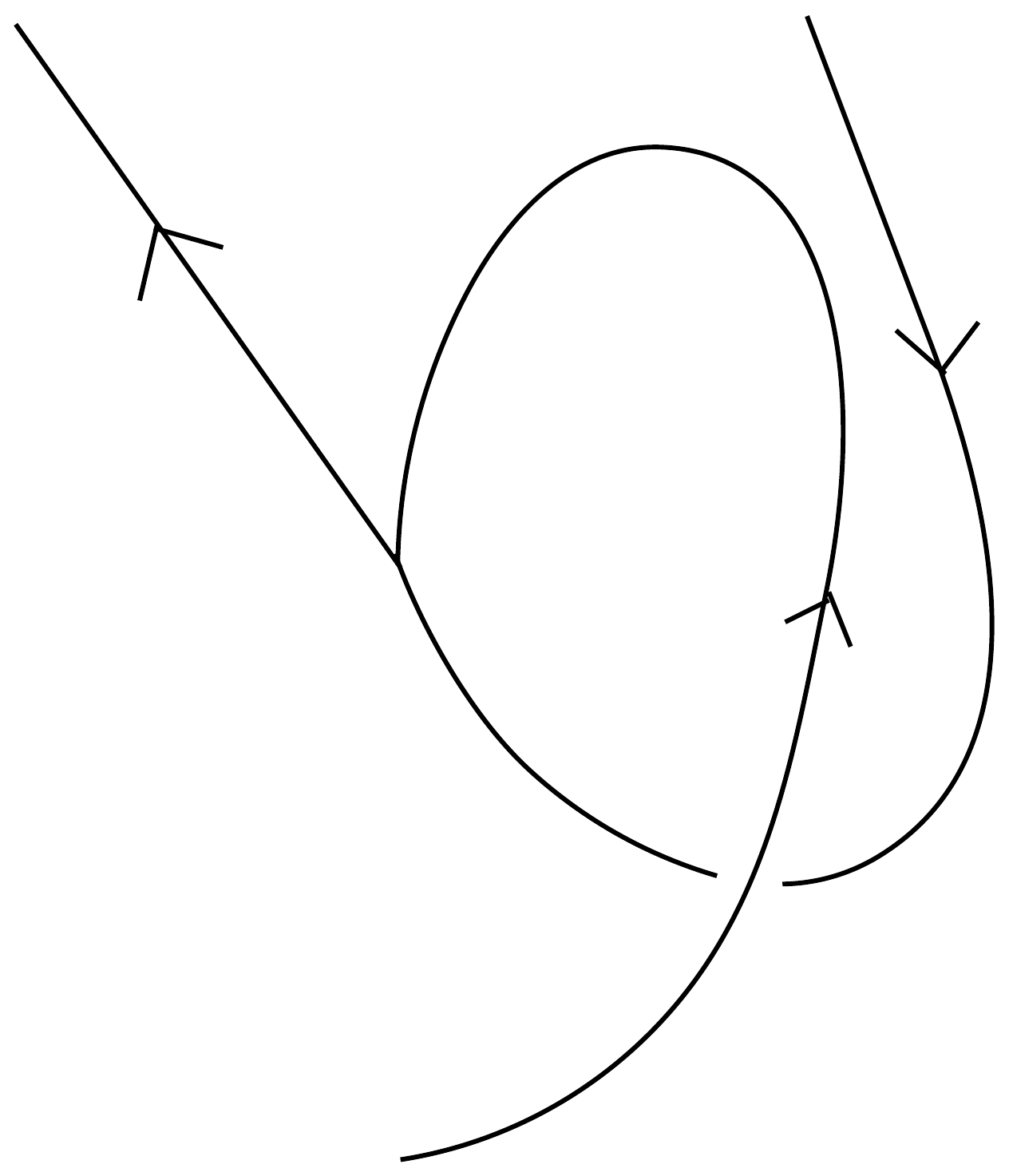}}  \hspace{.05in} \stackrel{\text{RP}}{\longrightarrow} \hspace{.05in}
\raisebox{-.5cm}{\includegraphics[height=1.775cm]{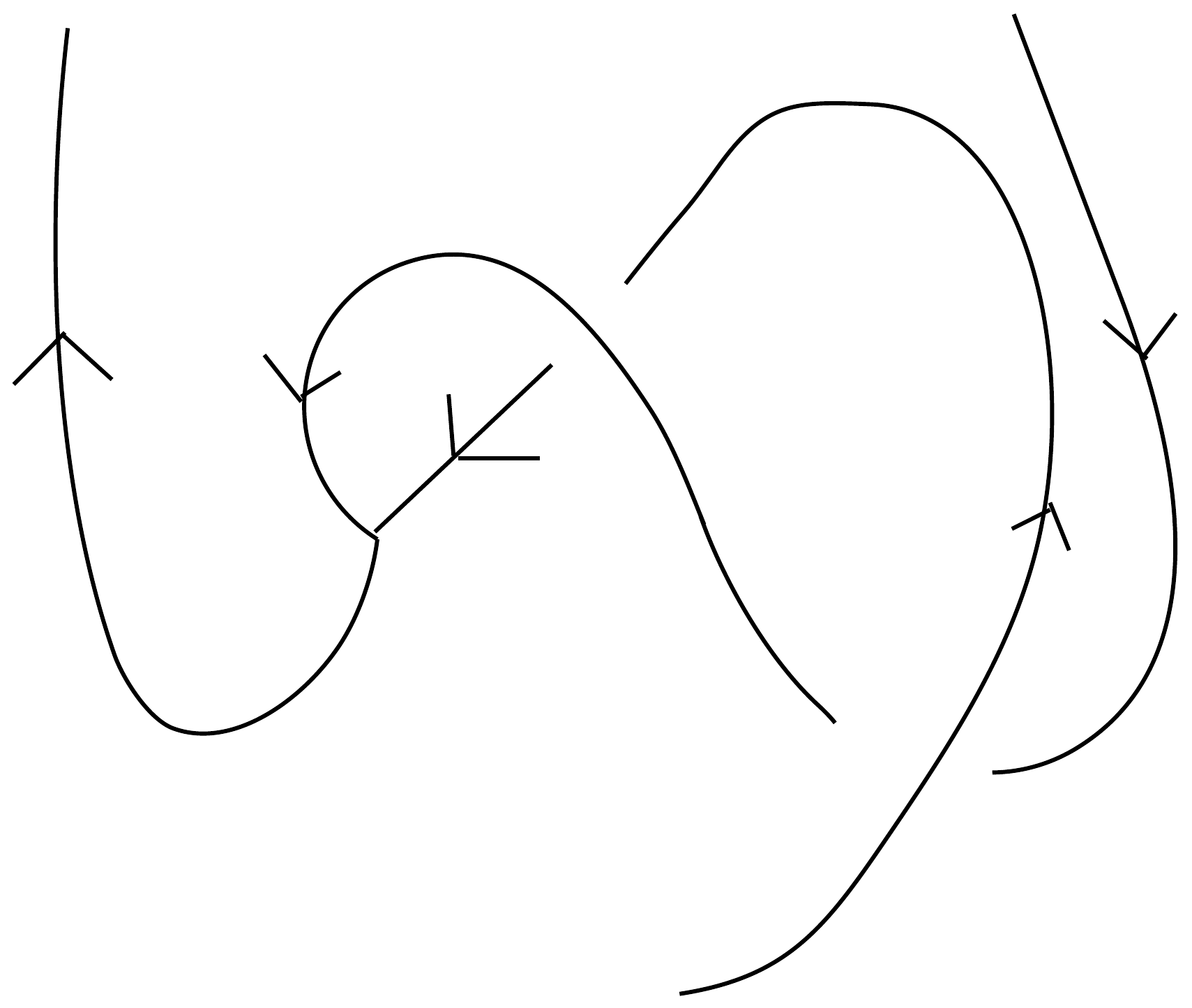}}  \hspace{.05in} \underset{\text{R5}}{\overset{\text{braid}}{\longleftrightarrow}} \hspace{.05in}
\raisebox{-.5cm}{\includegraphics[height=1.775cm]{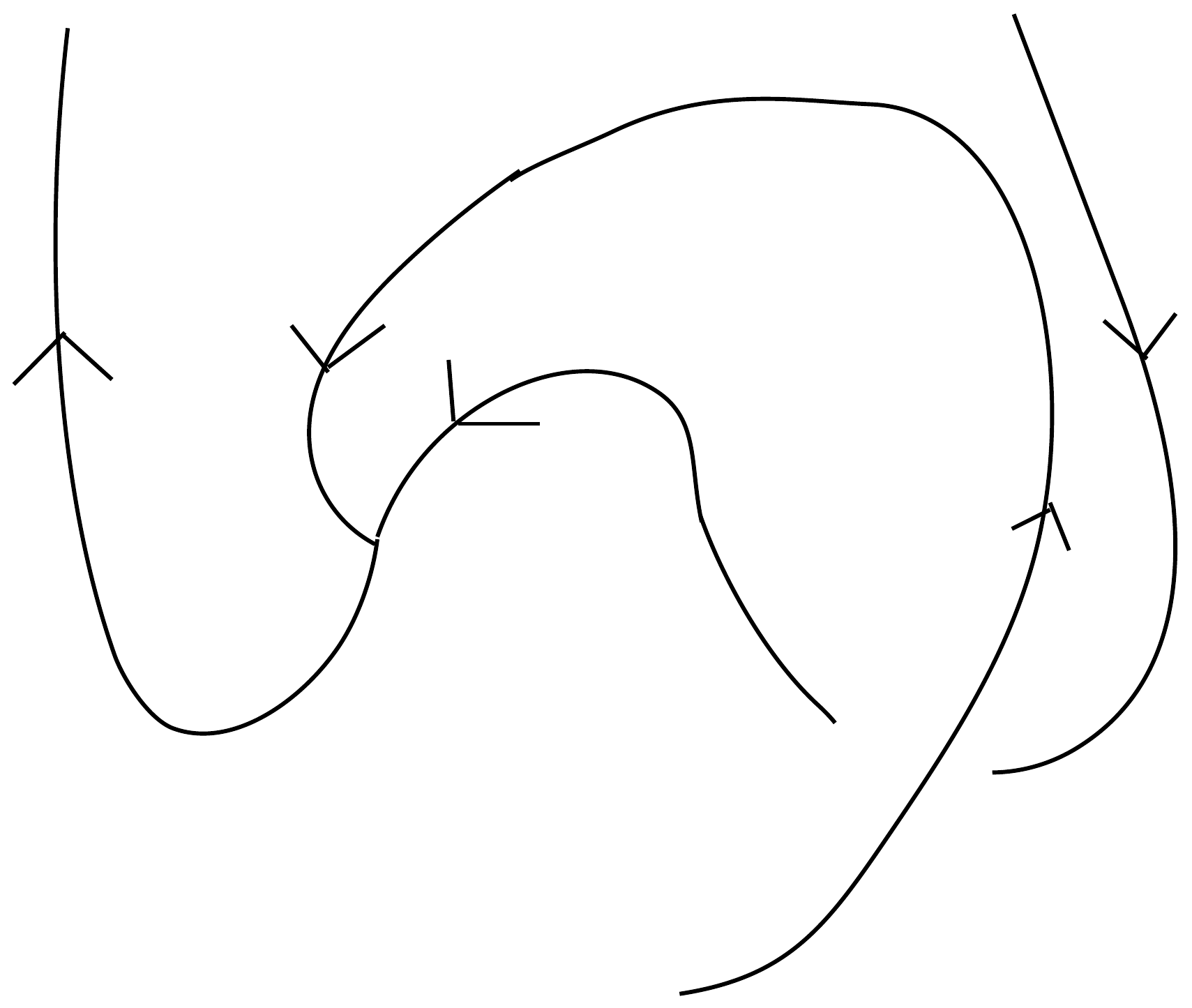}}  \hspace{.05in} \underset{\text{move}}{\overset{\text{swing}}{\longleftrightarrow}} \hspace{.05in}
\raisebox{-.5cm}{\includegraphics[height=1.775cm]{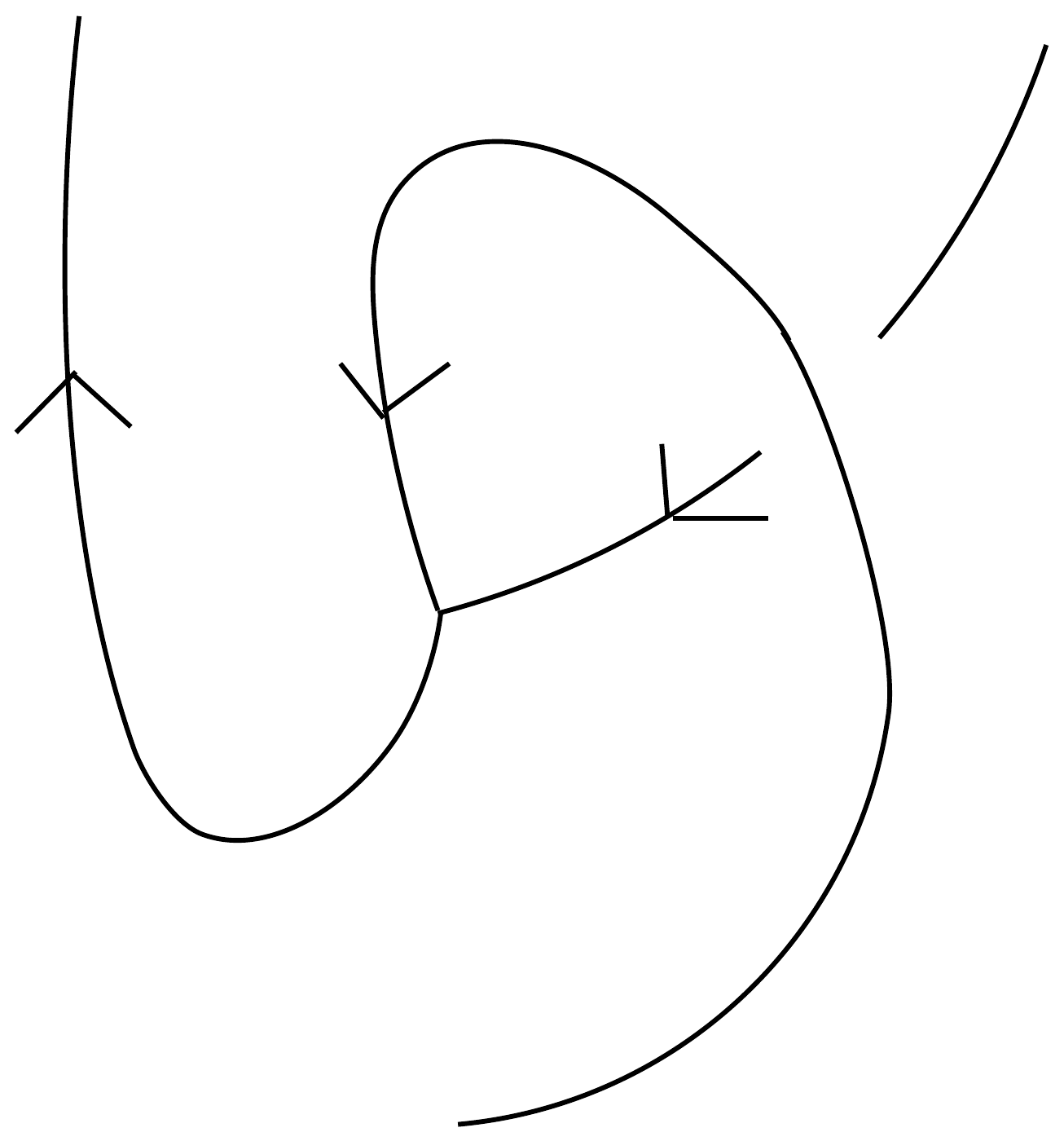}}  \hspace{.05in} \stackrel{\text{RP}}{\longleftarrow} \hspace{.05in}
\raisebox{-.5cm}{\includegraphics[height=1.775cm]{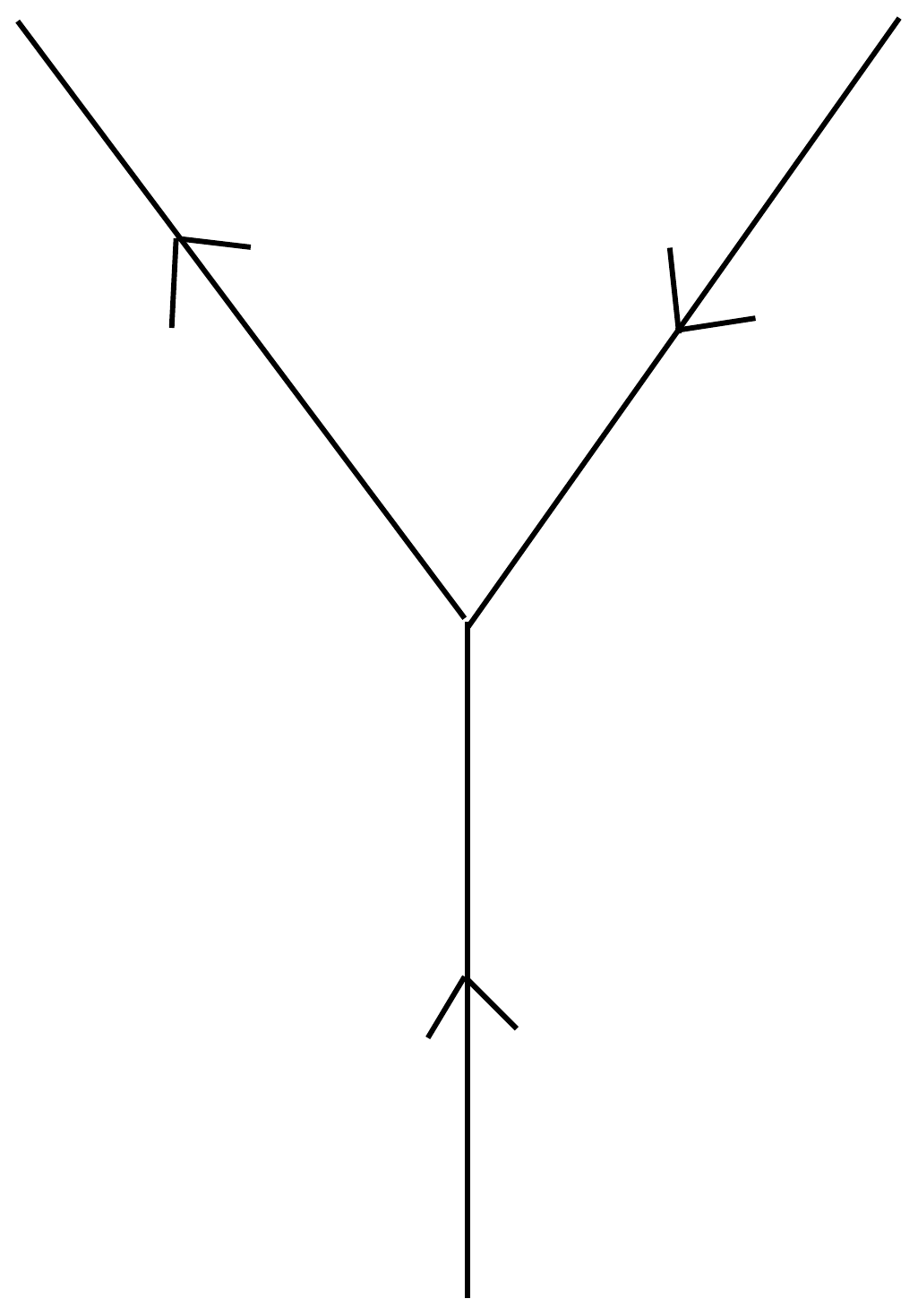}}
\]
\caption{A right $R5$ move on a $Y$-vertex with two up-arcs}\label{fig:YUDU_right}
\end{figure}

Next we consider an $R5$ move with a twist between the left upper edge and the lower edge incident with a $Y$-vertex with two up-arcs (see Fig.~\ref{fig:YUDU_left}).

\begin{figure}[ht]
\[
\raisebox{-.5cm}{\includegraphics[height=1.775cm]{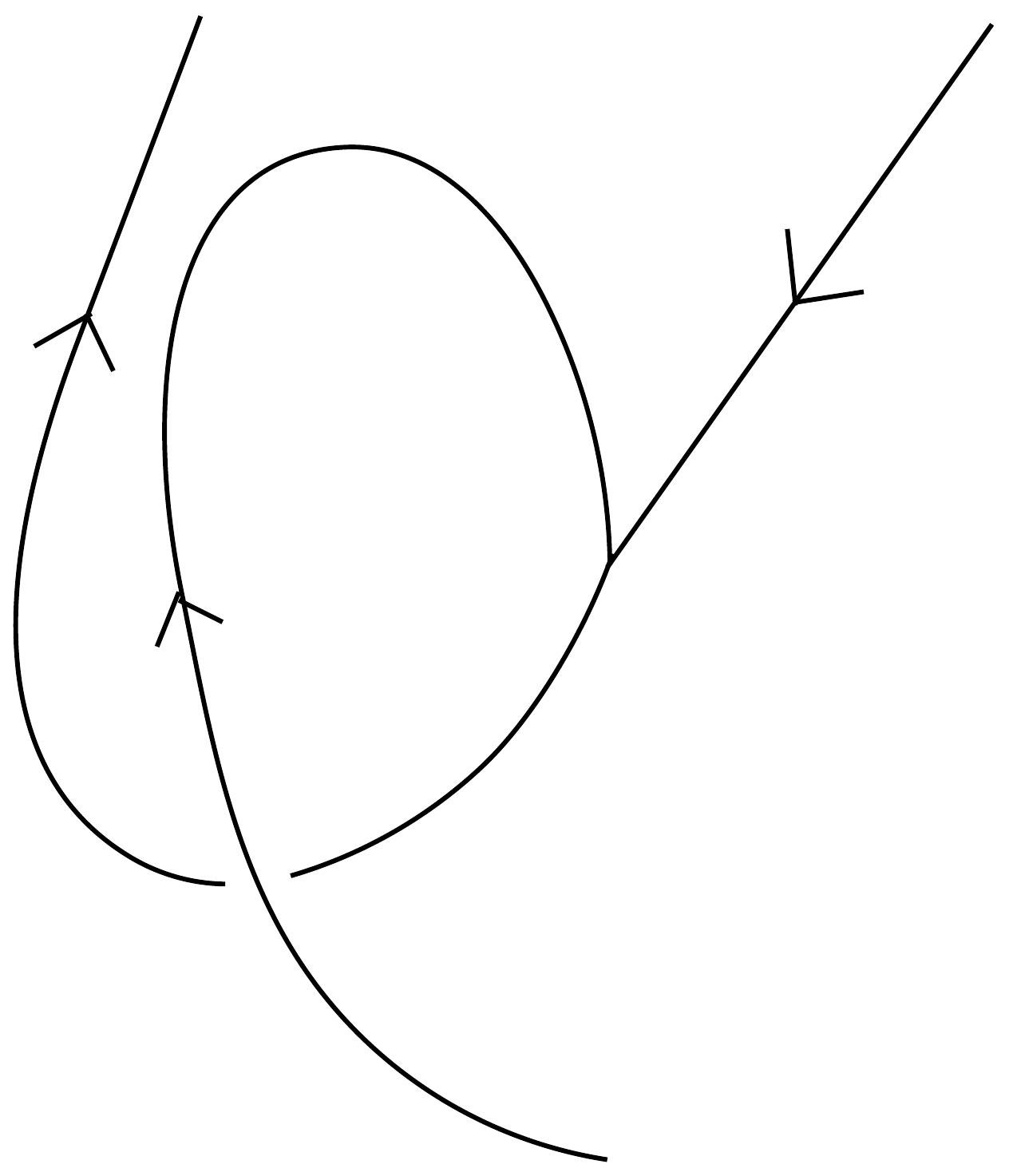}}  \hspace{.05in}
\stackrel{\text{R4}}{\longleftrightarrow} \hspace{.05in}
\raisebox{-.5cm}{\includegraphics[height=1.775cm]{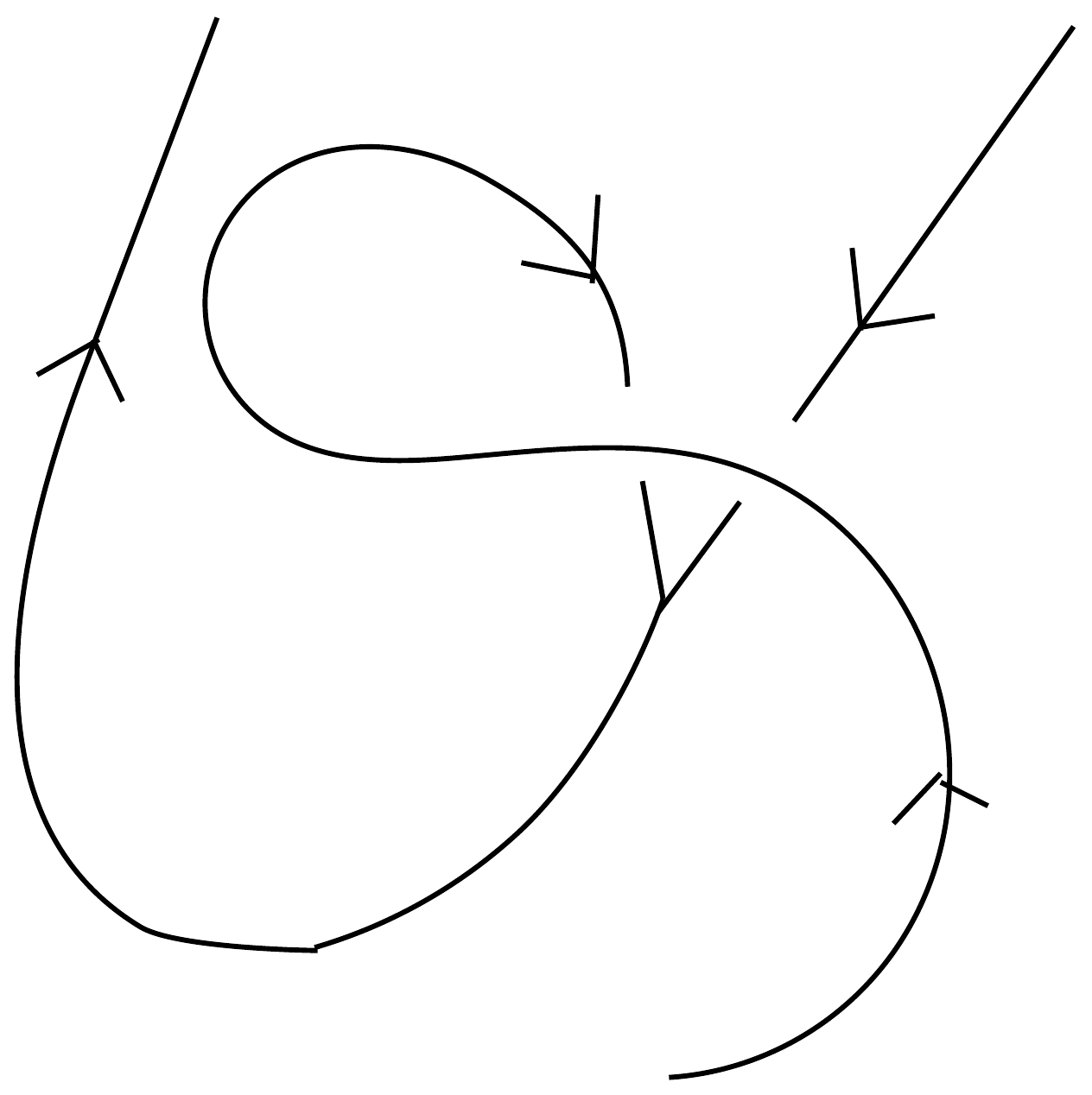}}  \hspace{.05in} \stackrel{\text{R1}}{\longleftrightarrow} \hspace{.05in}
\raisebox{-.5cm}{\includegraphics[height=1.775cm]{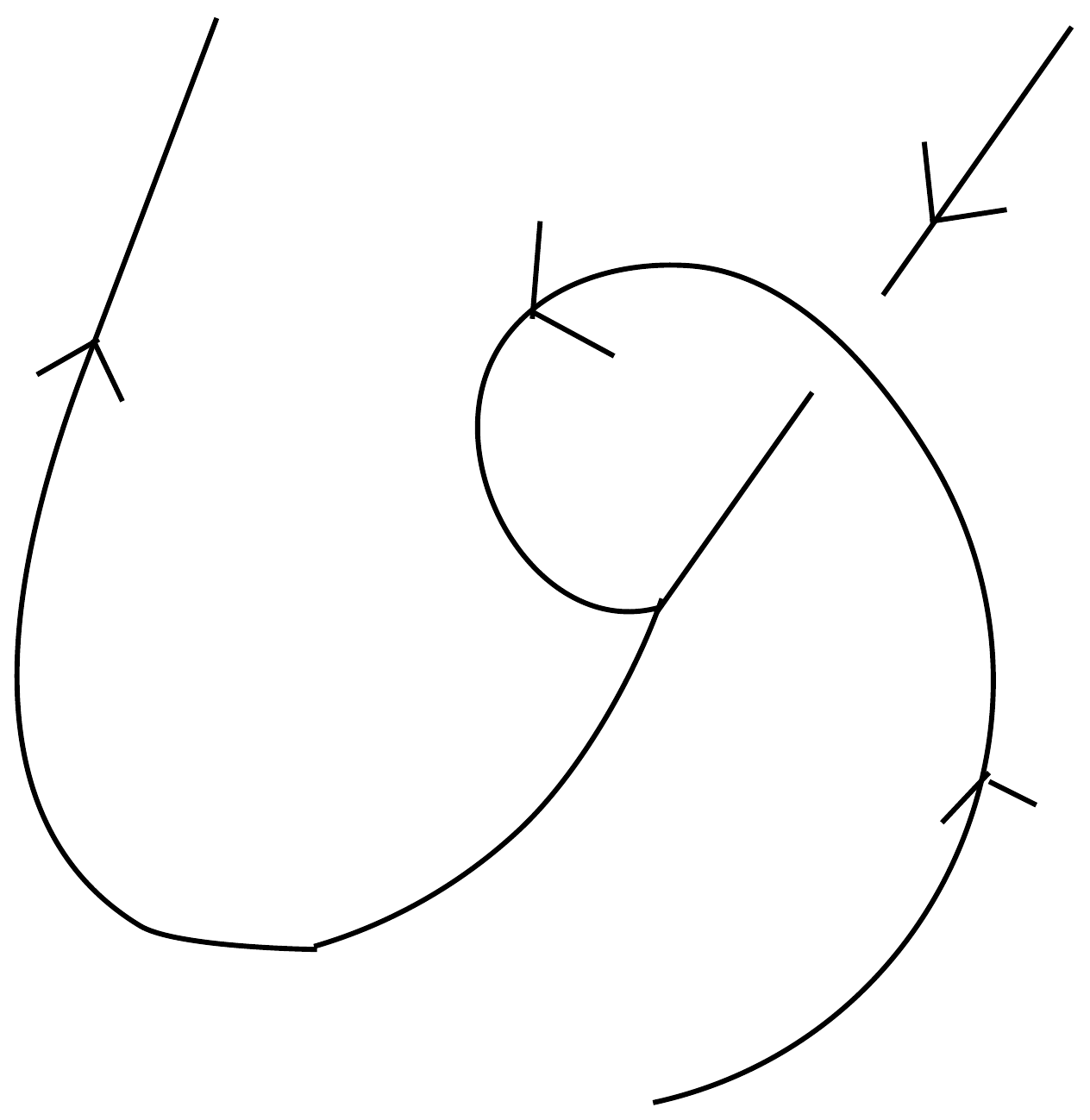}}  \hspace{.05in} \stackrel{\text{RP}}{\longleftarrow} \hspace{.05in}
\raisebox{-.5cm}{\includegraphics[height=1.775cm]{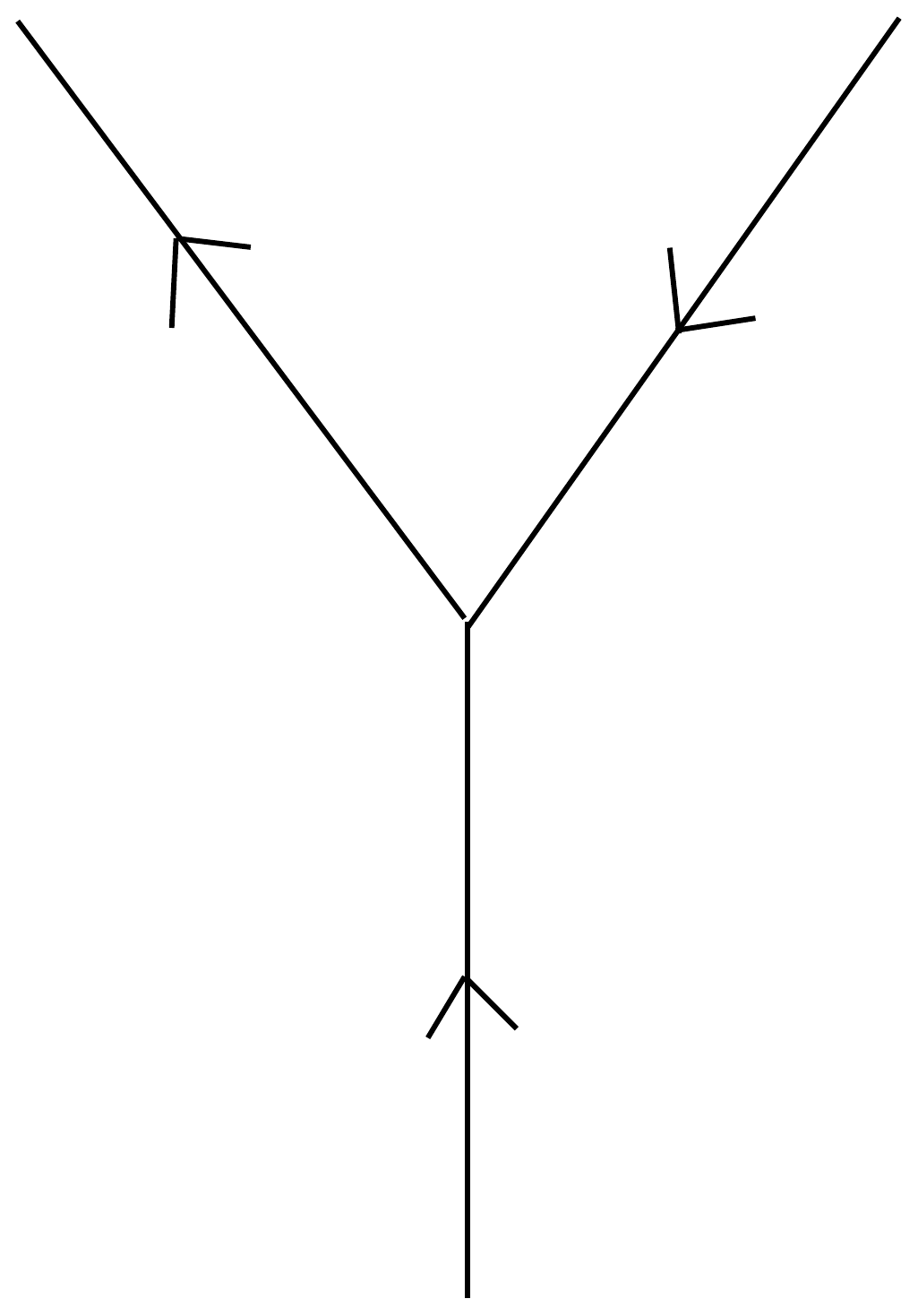}}
\]
\caption{A left $R5$ move on a $Y$-vertex with two up-arcs}\label{fig:YUDU_left}
\end{figure}

All other cases involving a $Y$-vertex are checked in a similar way to the ones shown here, except for the case where all arcs are oriented upward, but this case is similar to the case of a $\lambda$-vertex with three up-arcs, which is shown in Fig.~\ref{fig:LUUU_down}. 

\begin{figure}[ht]
\[
\raisebox{-.5cm}{\includegraphics[height=1.775cm]{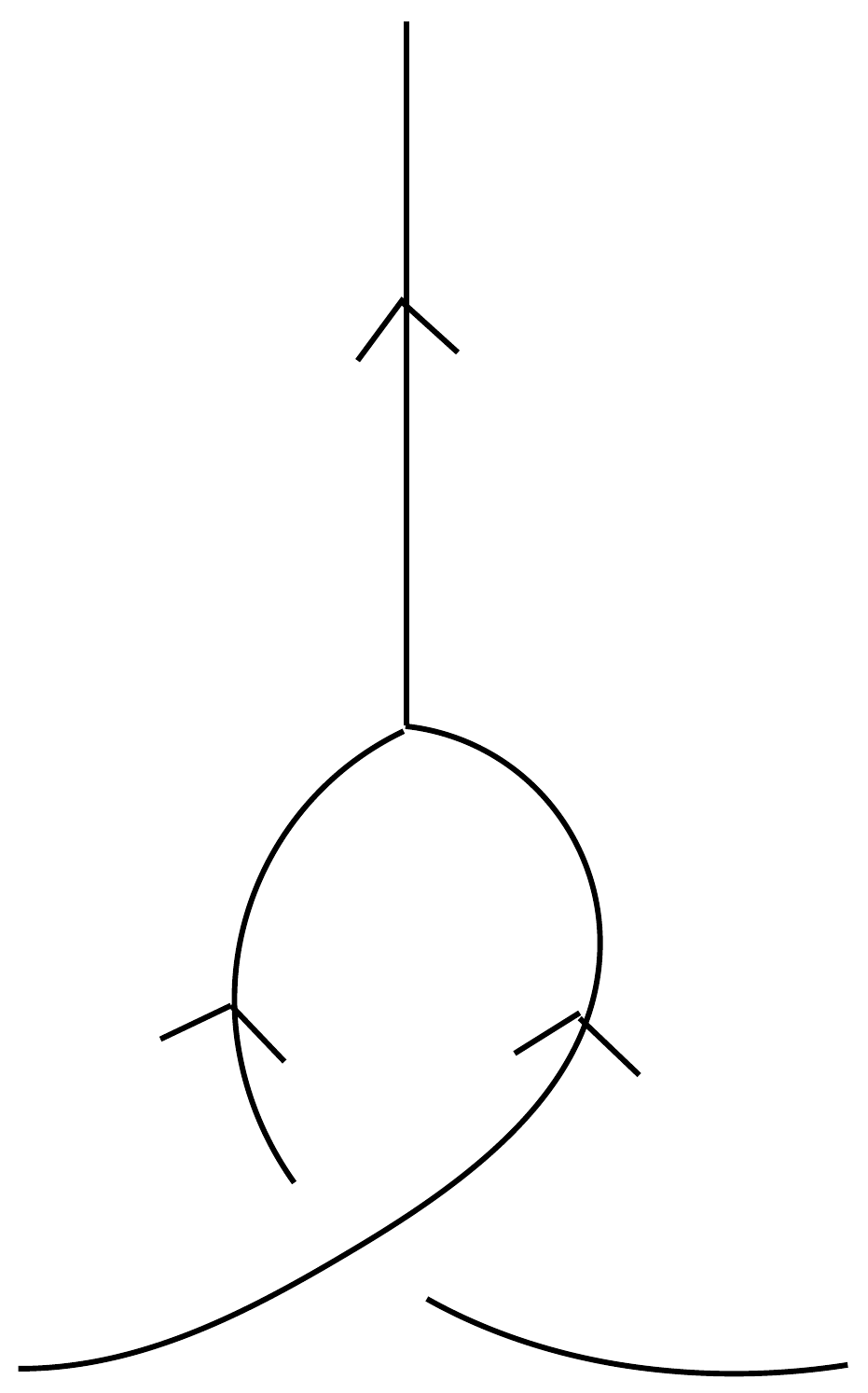}}  \hspace{.03in} \stackrel{\text{RP}}{\longrightarrow} \hspace{.03in}
\raisebox{-.5cm}{\includegraphics[height=1.775cm]{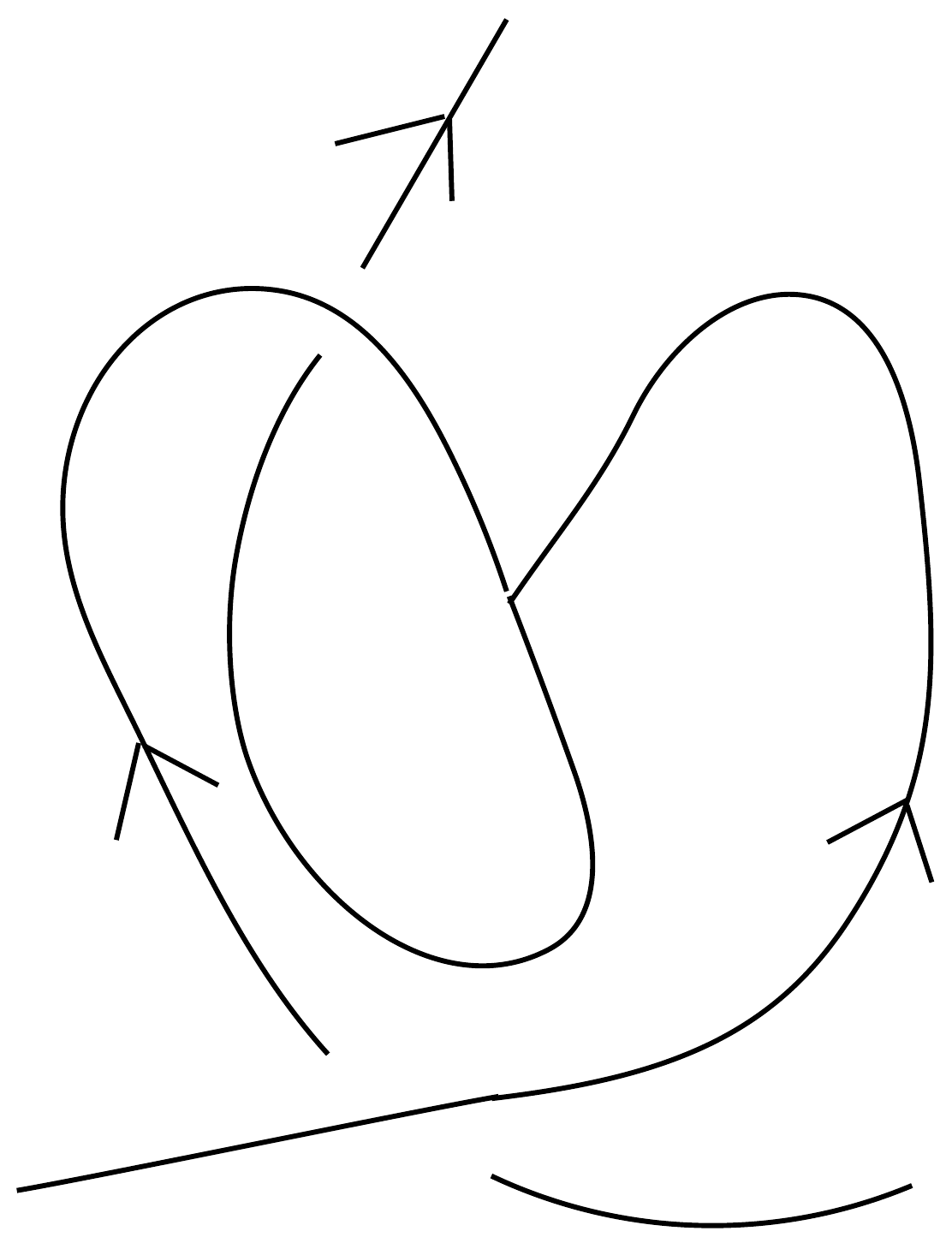}}  \hspace{.03in} \stackrel{\text{R2, R3}}{\longleftrightarrow} \hspace{.03in}
\raisebox{-.5cm}{\includegraphics[height=1.775cm]{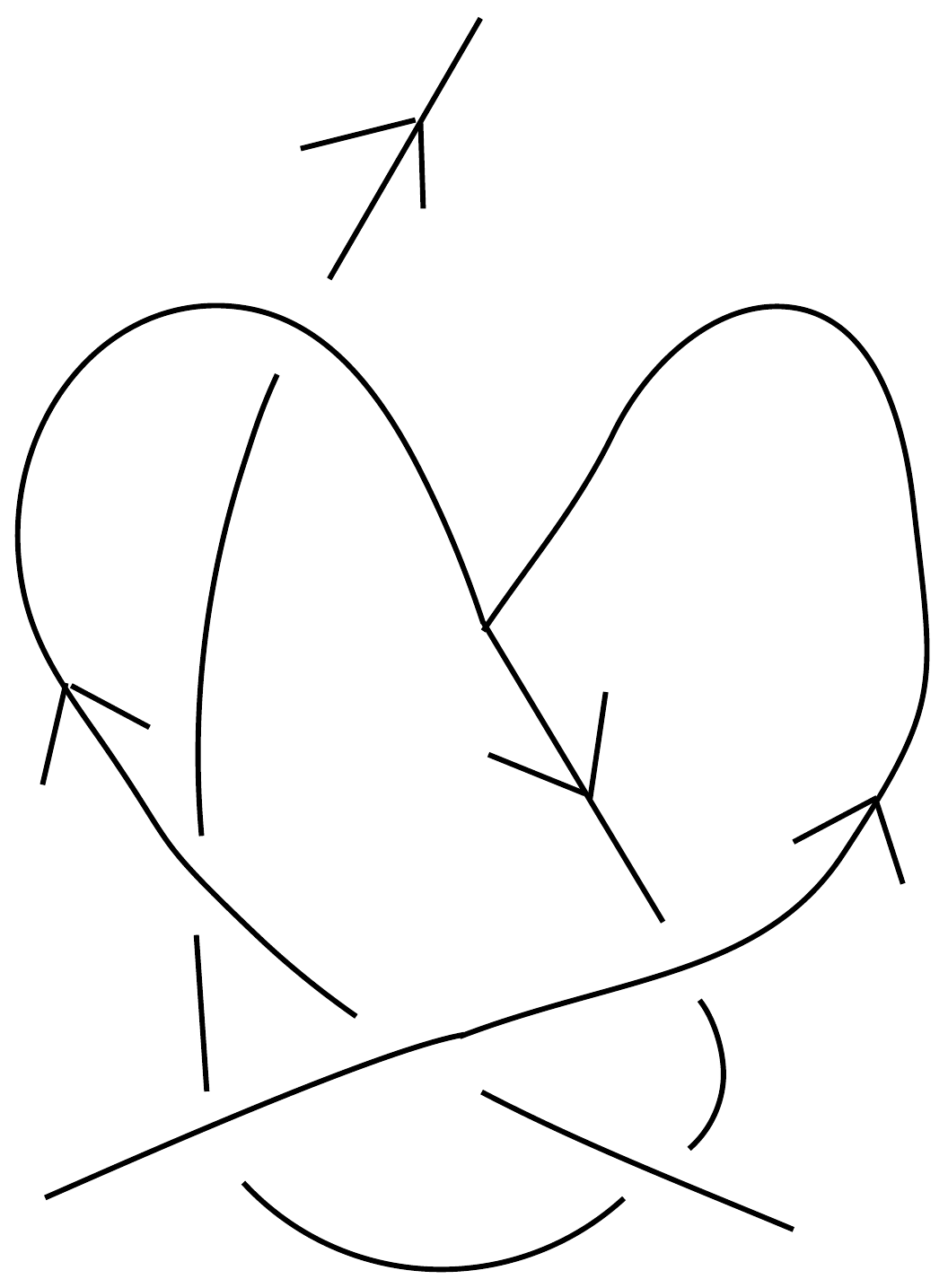}}  \hspace{.03in} \stackrel{\text{R2}}{\longleftrightarrow} \hspace{.03in}
\raisebox{-.5cm}{\includegraphics[height=1.775cm]{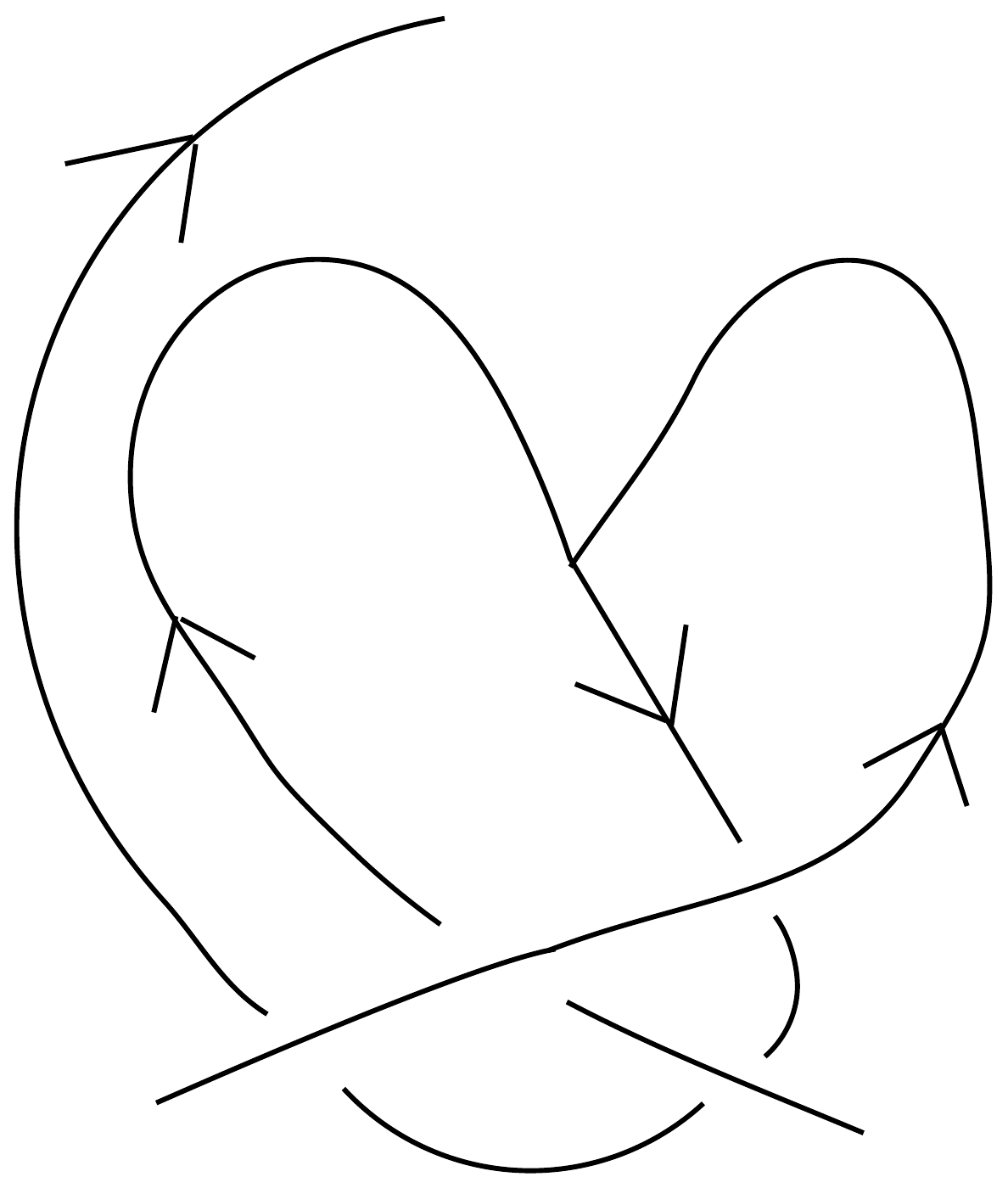}}  \hspace{.03in} \stackrel{\text{R4}}{\longleftrightarrow} \hspace{.03in}
\raisebox{-.5cm}{\includegraphics[height=1.775cm]{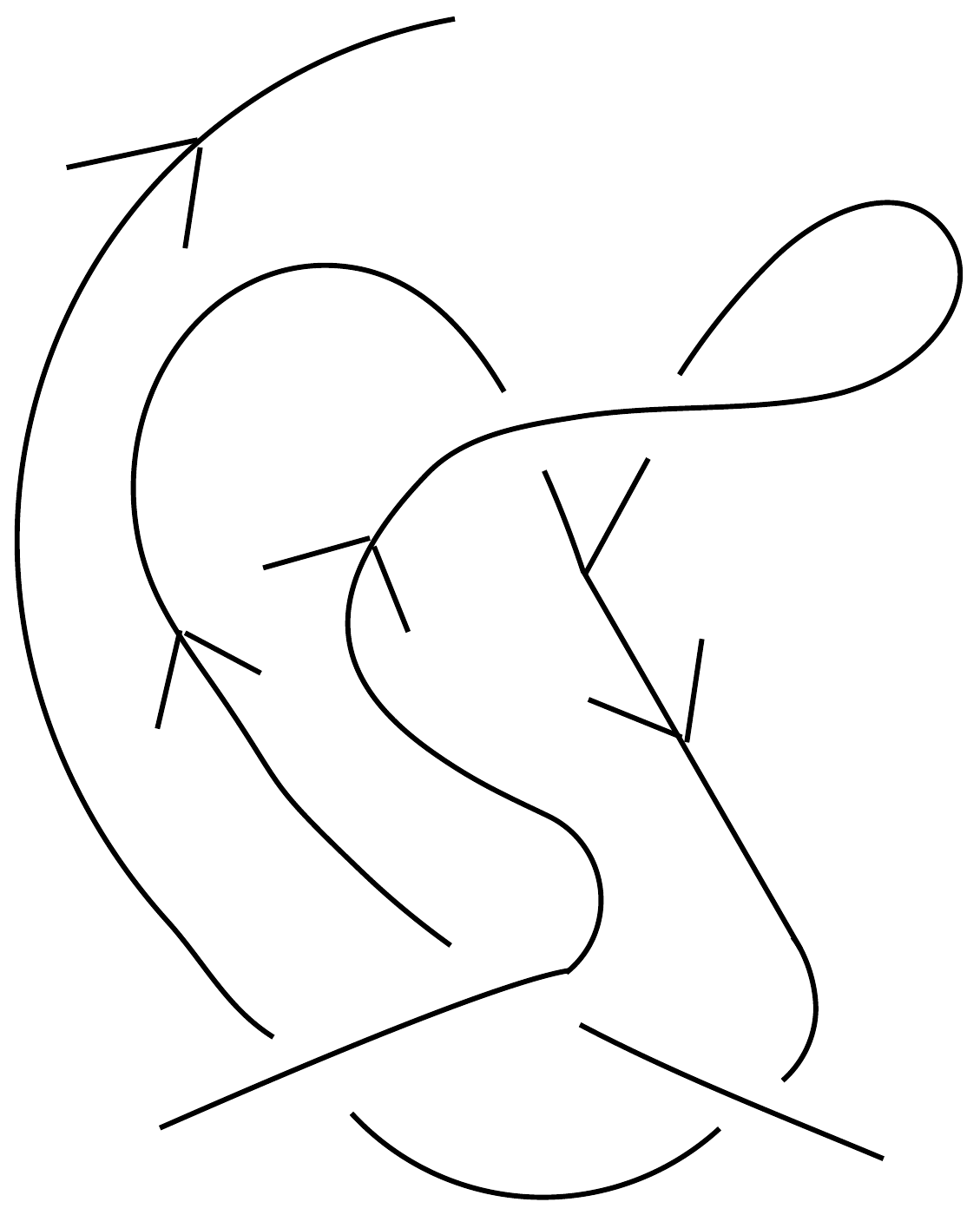}}  \hspace{.03in}  \underset{\text{swing move}} {\overset{\text{$R2$, $R1$}}{\longleftrightarrow}} \hspace{.05in}
\]
\vspace{.3cm}
\[
\raisebox{-.5cm}{\includegraphics[height=1.775cm]{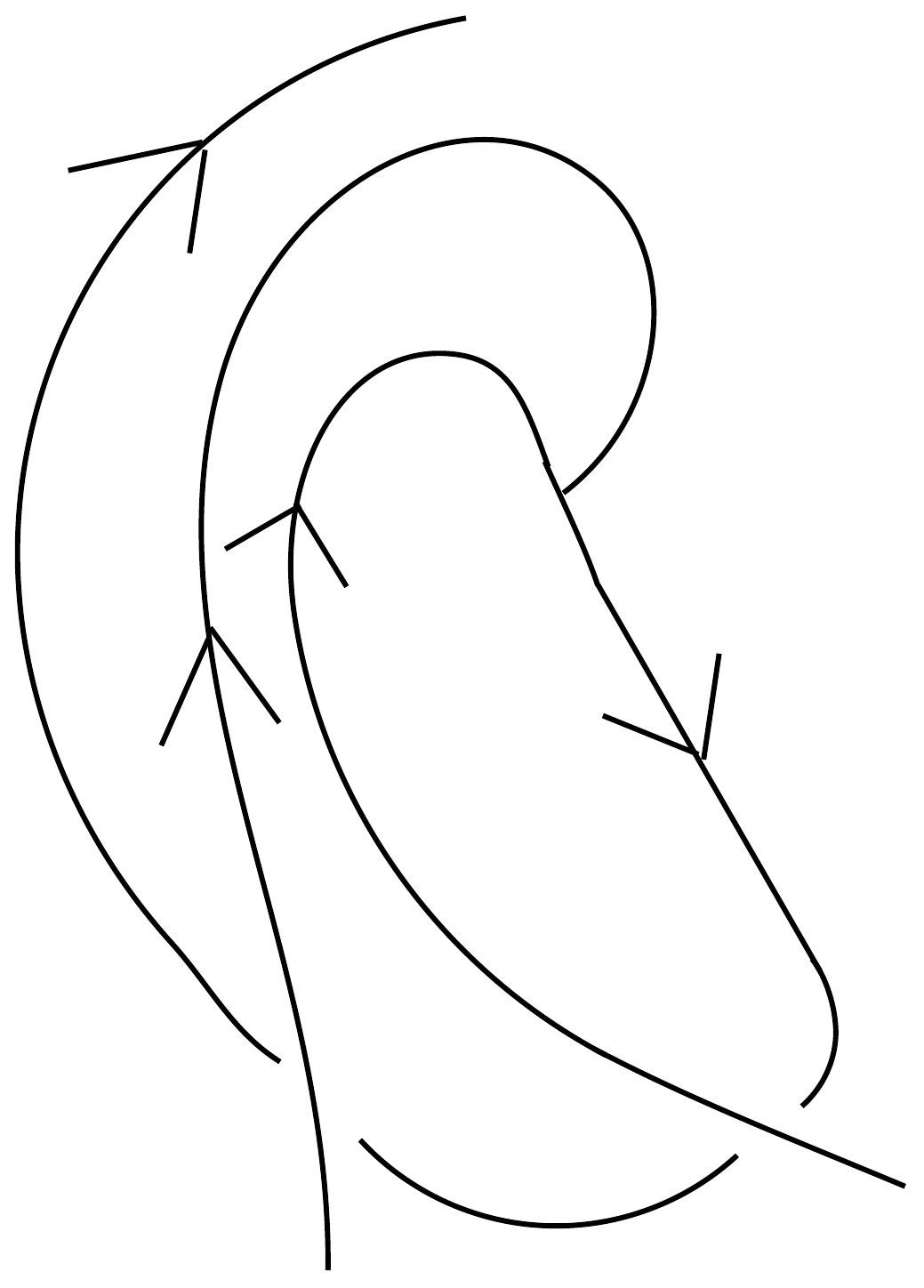}}  \hspace{.03in}
\stackrel{\text{$R4$}}{{\longleftrightarrow}}
\hspace{.03in}
\raisebox{-.5cm}{\includegraphics[height=1.775cm]{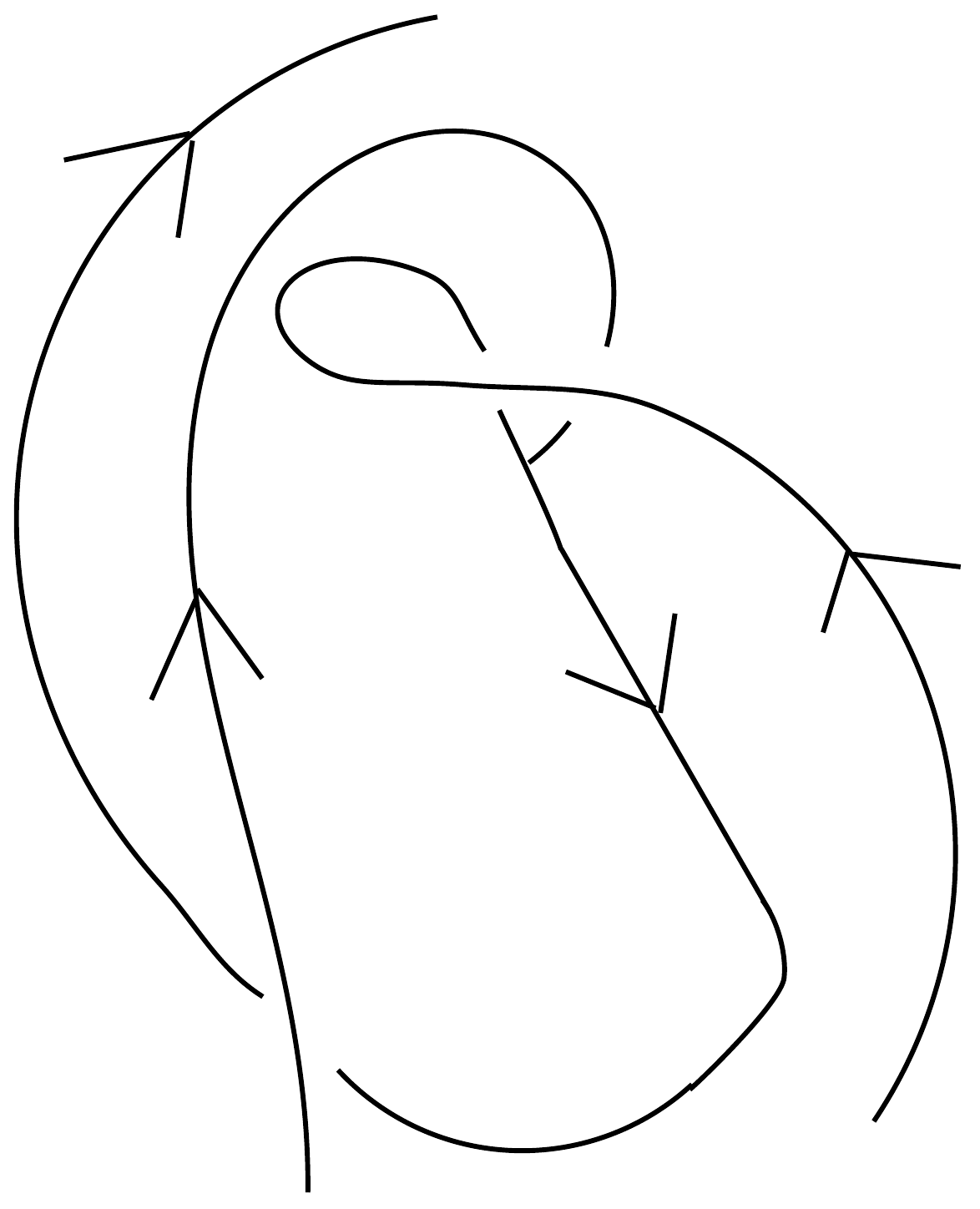}}  \hspace{.03in} \stackrel{\text{$R1$}}{\longleftrightarrow} \hspace{.03in}
\raisebox{-.5cm}{\includegraphics[height=1.775cm]{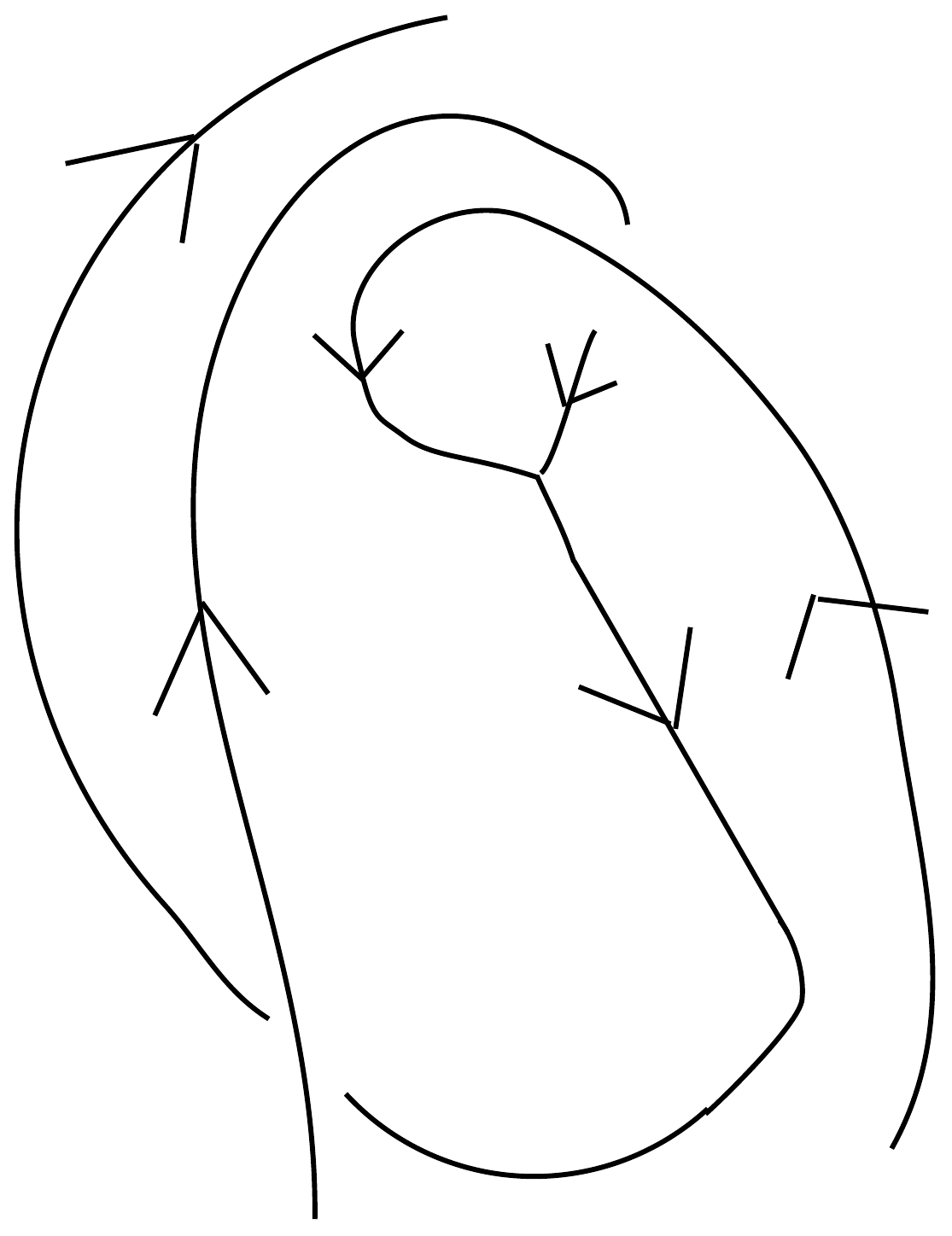}}  \hspace{.03in} \underset{\text{$R5$}}{\overset{\text{braid}}{\longleftrightarrow}} \hspace{.05in}
\raisebox{-.5cm}{\includegraphics[height=1.775cm]{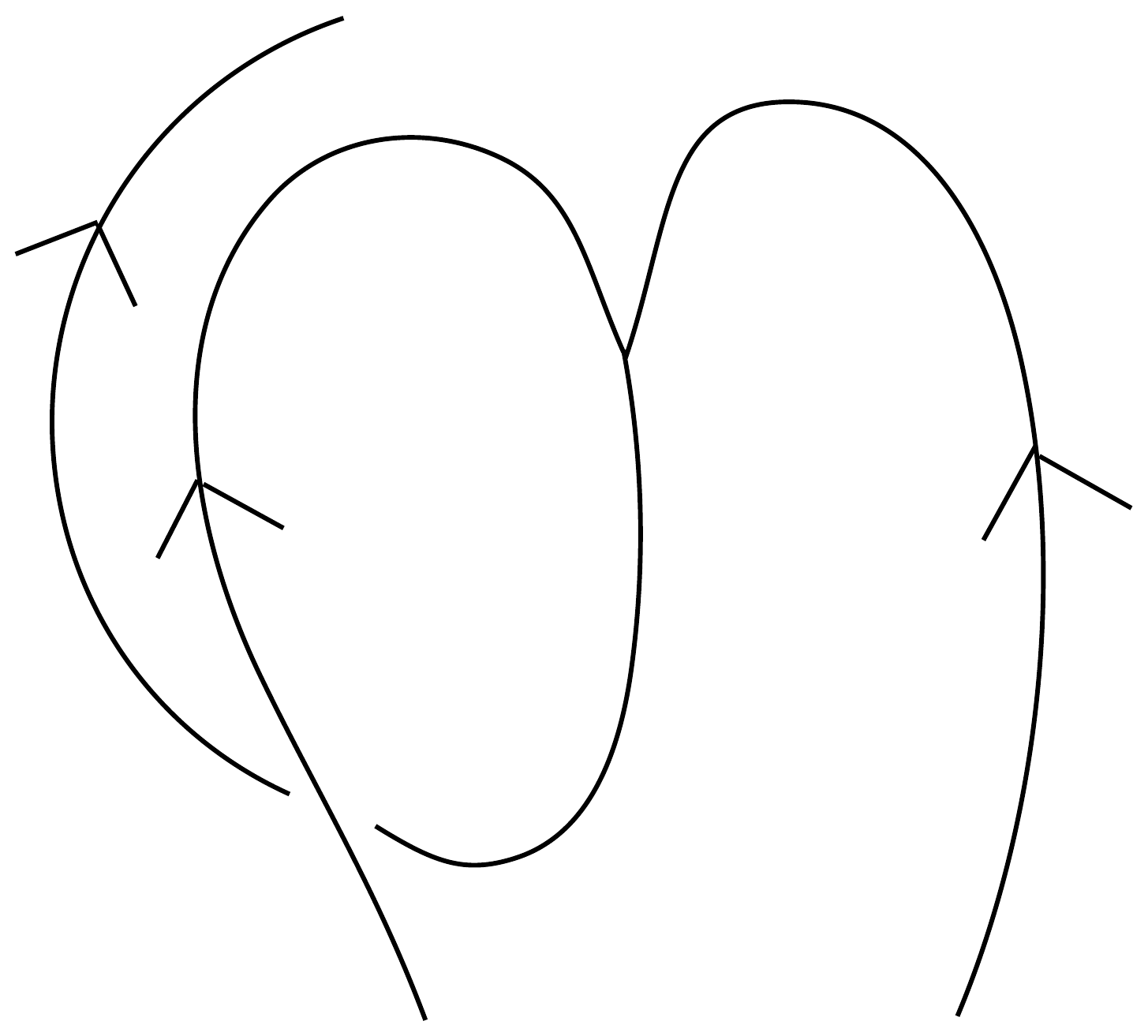}}  \hspace{.03in} \stackrel{\text{RP}}{\longleftarrow} \hspace{.03in}
\raisebox{-.5cm}{\includegraphics[height=1.775cm]{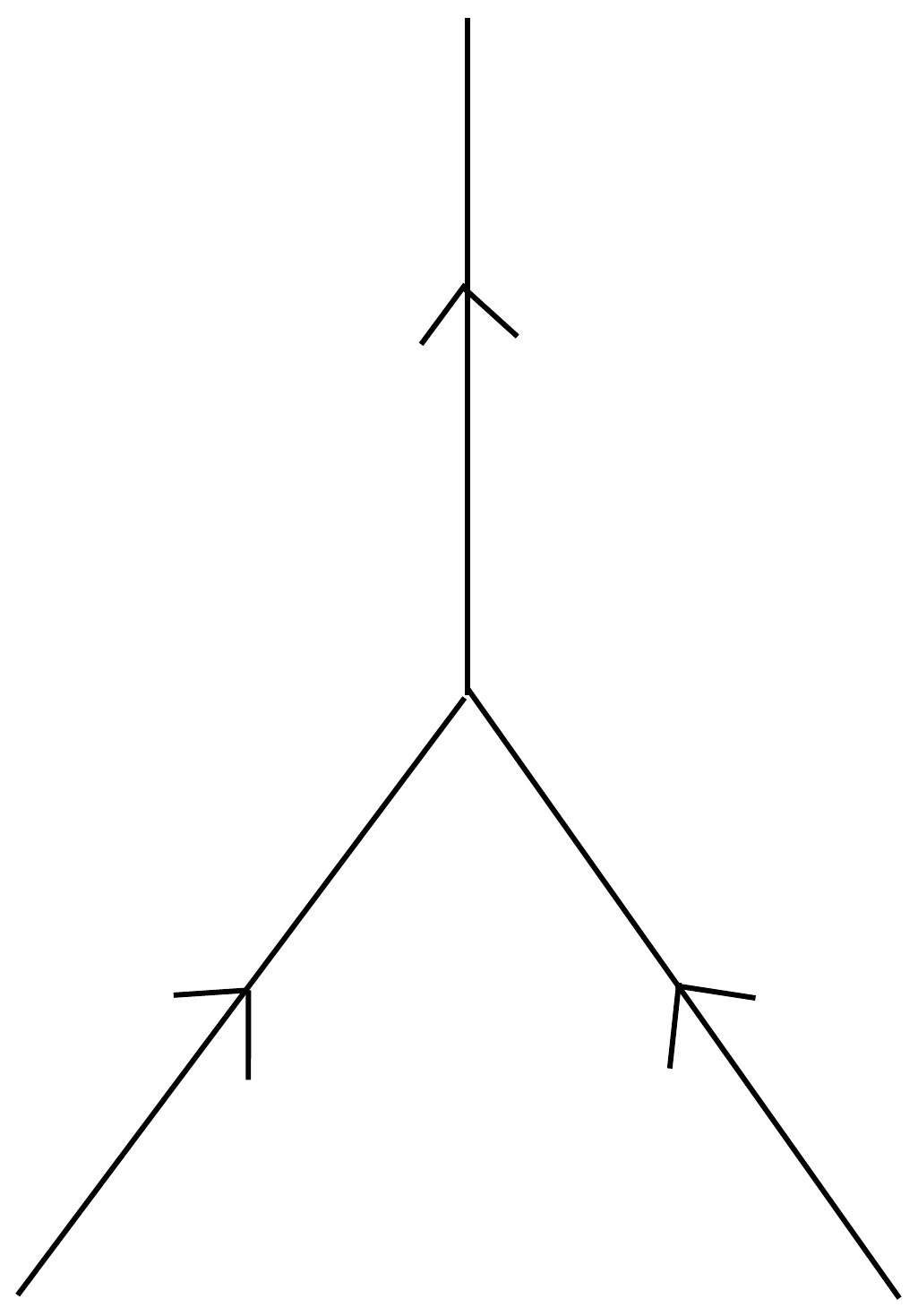}}  
\]
\caption{A down $R5$ move on a $\lambda$-vertex with three up-arcs}\label{fig:LUUU_down}
\end{figure}

Figure~\ref{fig:UDDL} covers the case of the $R5$ move on a $\lambda$-vertex with two up-arcs; the twist is applied between the two lower edges meeting at the vertex.

\begin{figure}[ht]
\[
\raisebox{-.5cm}{\includegraphics[height=1.775cm]{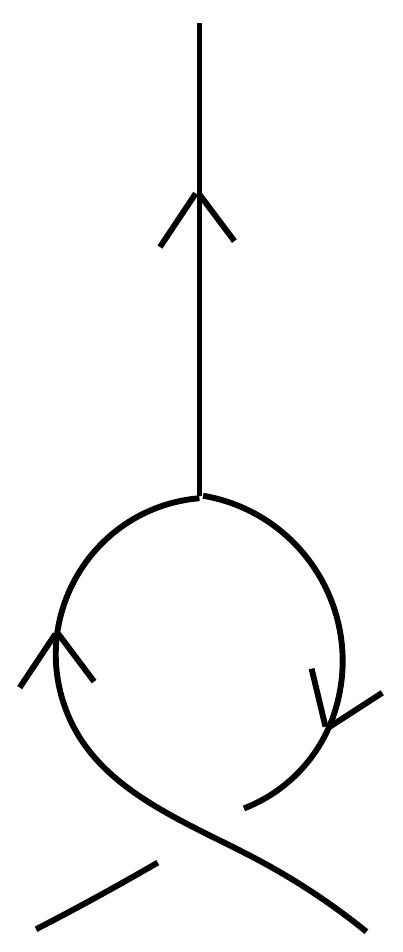}}  \hspace{.05in} \stackrel{\text{RP}}{\longrightarrow} \hspace{.05in}
\raisebox{-.5cm}{\includegraphics[height=1.775cm]{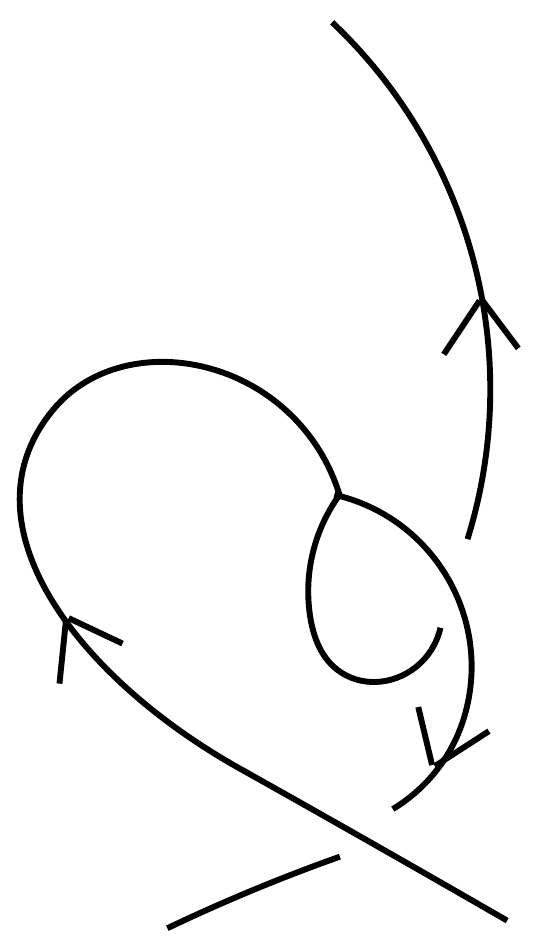}}  \hspace{.05in} 
\underset{\text{$R5$}}{\overset{\text{braid}}{\longleftrightarrow}}
\raisebox{-.5cm}{\includegraphics[height=1.775cm]{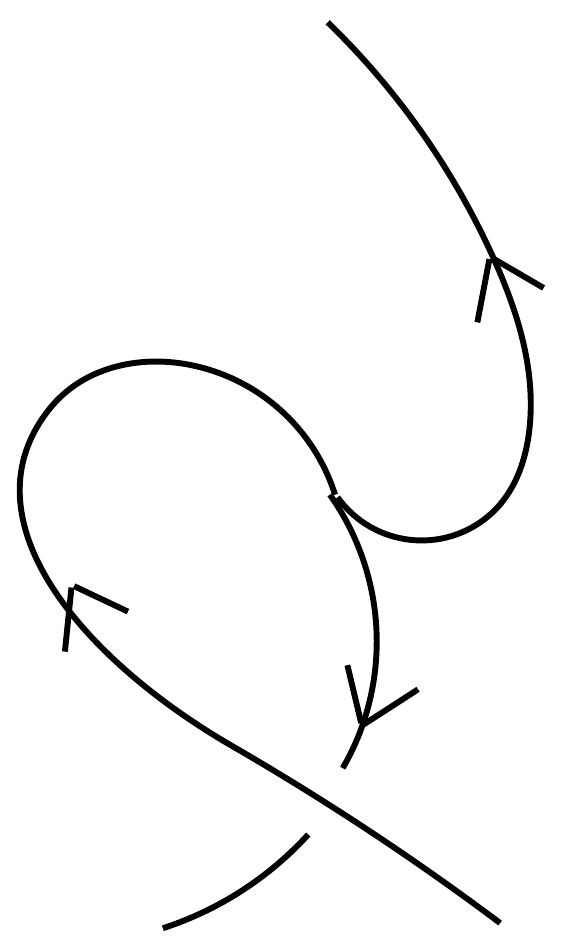}}  \hspace{.05in} \stackrel{\text{$R4$}}{\longleftrightarrow} \hspace{.05in}
\raisebox{-.5cm}{\includegraphics[height=1.775cm]{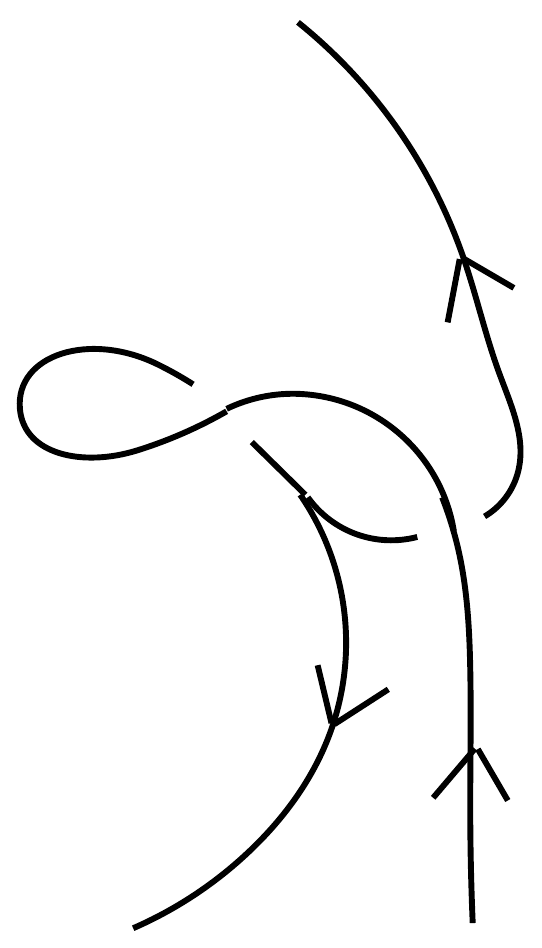}}  \hspace{.05in} 
\]
\[
\stackrel{\text{$R1$}}{\longleftrightarrow} \hspace{.05in}
\raisebox{-.5cm}{\includegraphics[height=1.775cm]{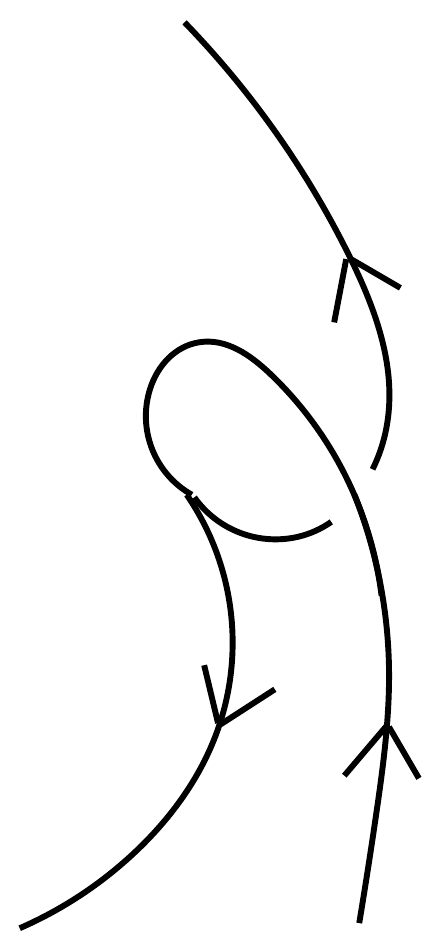}}  \hspace{.05in} \stackrel{\text{$R2$, $R4$}}{\longleftrightarrow} \hspace{.05in}
\raisebox{-.5cm}{\includegraphics[height=1.775cm]{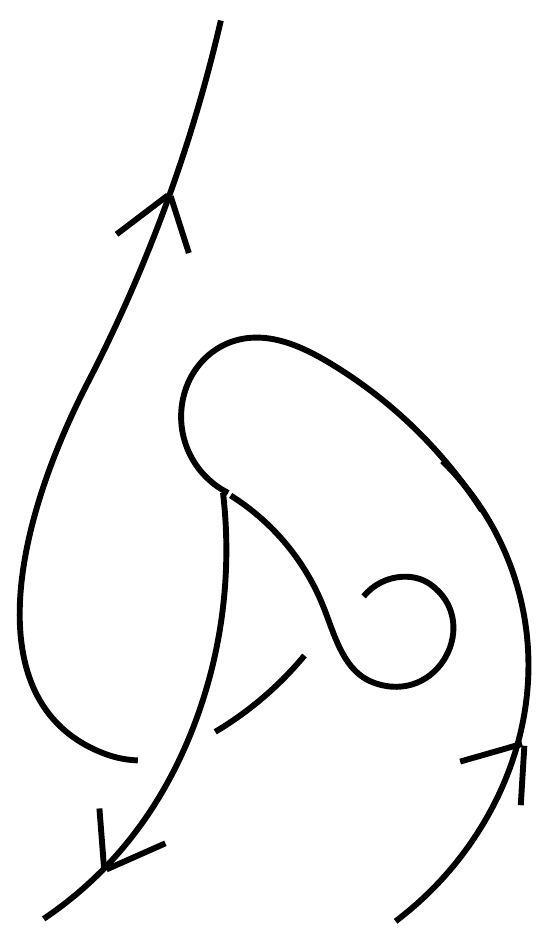}}
\stackrel{\text{$R1$}}{\longleftrightarrow} \hspace{.05in}
\raisebox{-.5cm}{\includegraphics[height=1.775cm]{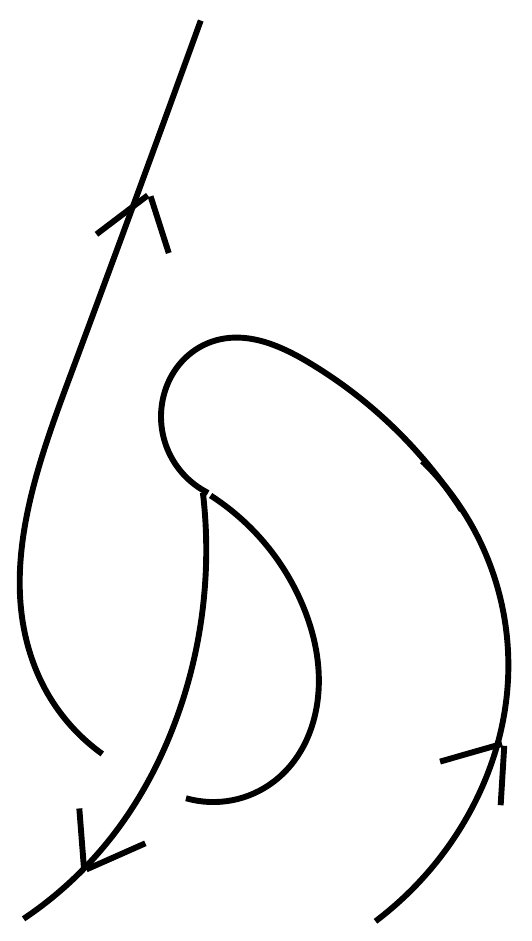}}
\stackrel{\text{RP}}{\longleftarrow} \hspace{.05in}
\raisebox{-.5cm}{\includegraphics[height=1.775cm]{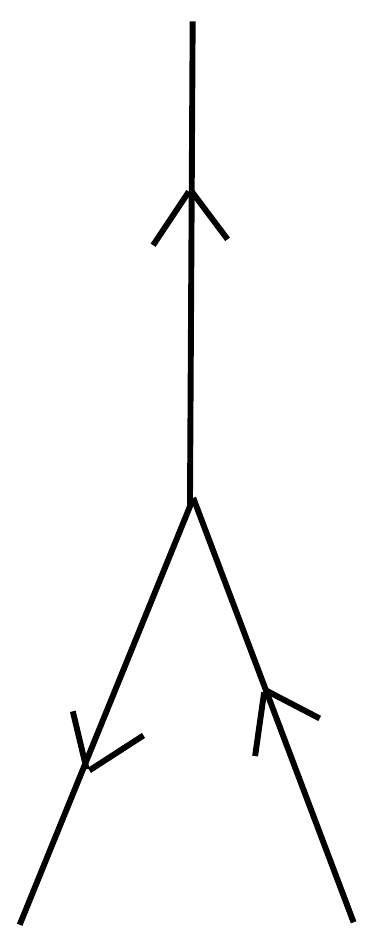}}
\]
\caption{A down $R5$ move on a $\lambda$-vertex with two up-arcs} \label{fig:UDDL}
\end{figure}

In Fig.~\ref{fig:DUDL}, we verify an $R5$ move applied to a $\lambda$-vertex with one up-arc, where the twist is between the right lower edge and the upper edge of the $\lambda$-vertex. 

\begin{figure}[ht]
\[
\raisebox{-.5cm}{\includegraphics[height=1.775cm]{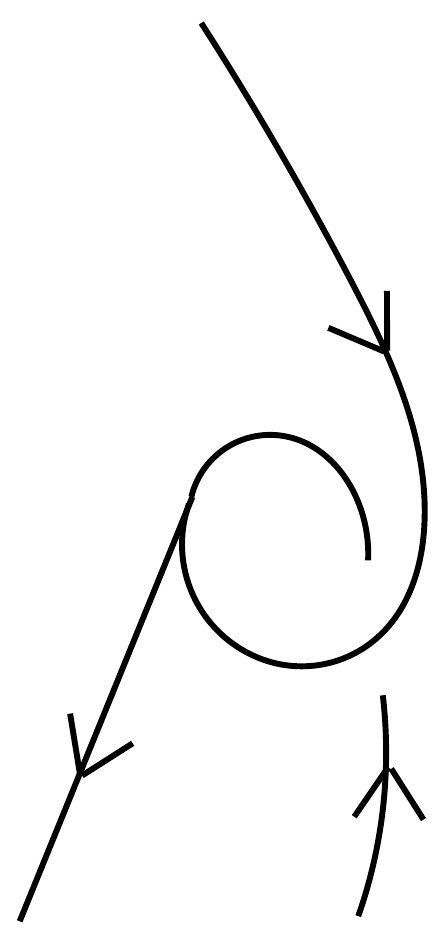}}  \hspace{.05in} \stackrel{\text{RP}}{\longrightarrow} \hspace{.05in}
\raisebox{-.5cm}{\includegraphics[height=1.775cm]{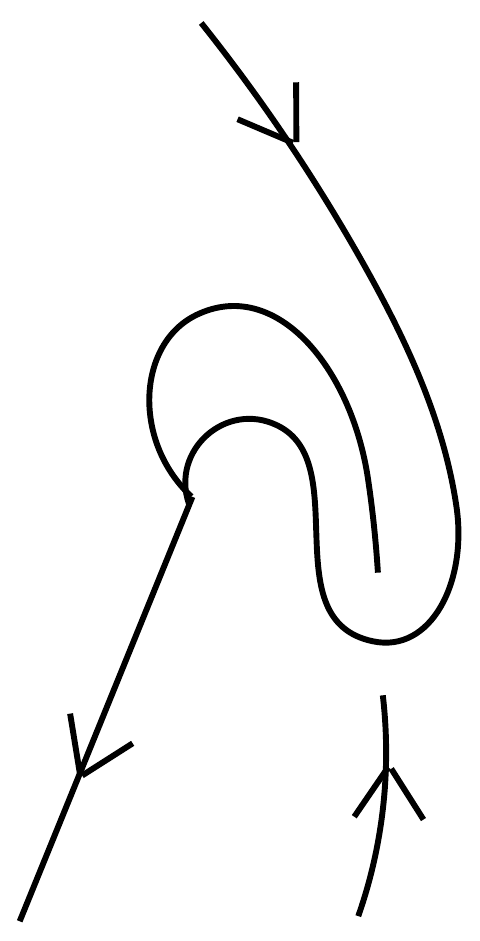}}  \hspace{.05in} 
\underset{\text{$R5$}}{\overset{\text{braid}}{\longleftrightarrow}}
\raisebox{-.5cm}{\includegraphics[height=1.775cm]{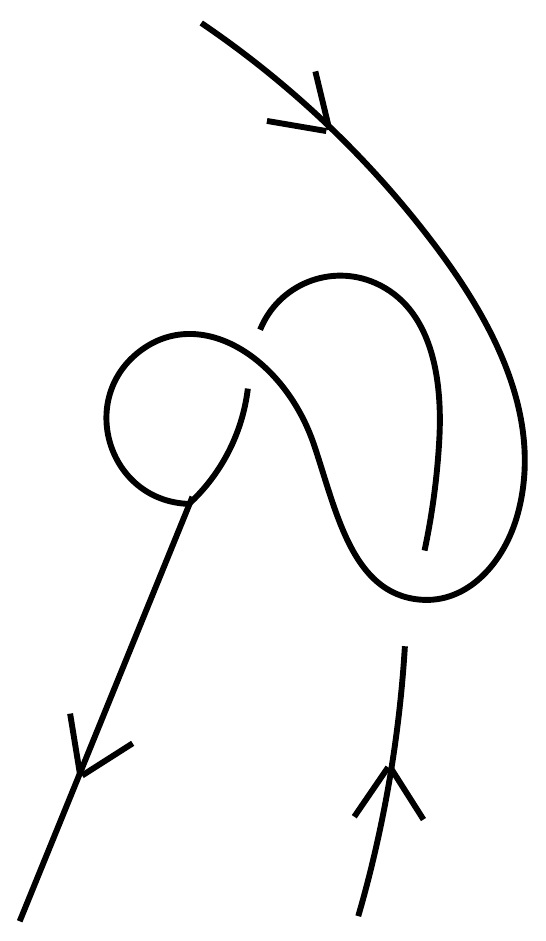}}  \hspace{.05in} \stackrel{\text{$R2$}}{\longleftrightarrow} \hspace{.05in}
\raisebox{-.5cm}{\includegraphics[height=1.775cm]{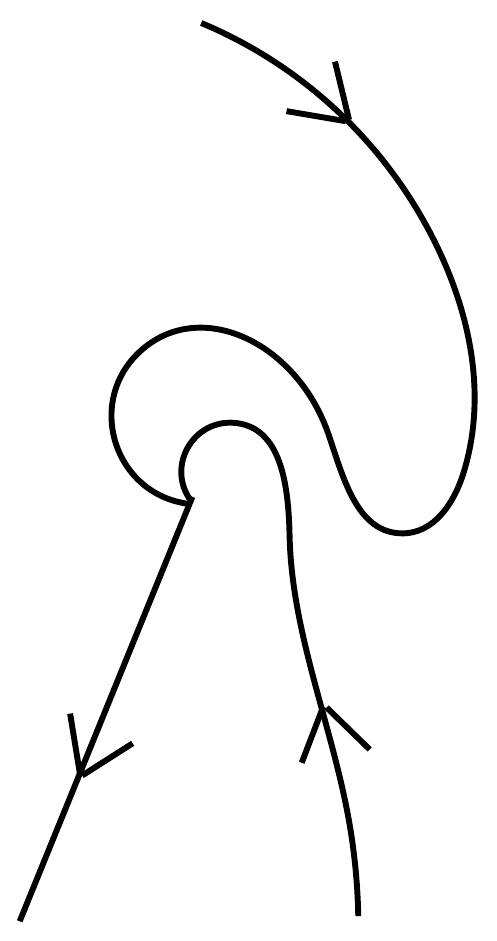}}  \hspace{.05in} 
\underset{\text{isotopy}}{\overset{\text{planar}}{\longleftrightarrow}}
\raisebox{-.5cm}{\includegraphics[height=1.775cm]{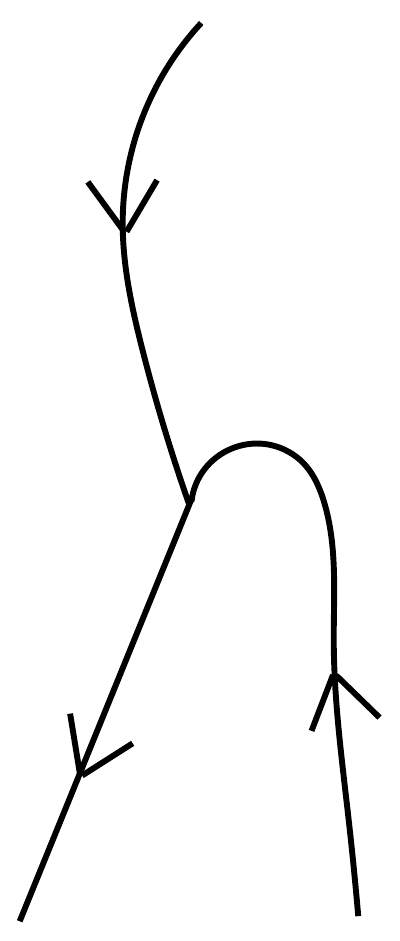}}  \hspace{.05in} \stackrel{\text{RP}}{\longleftarrow} \hspace{.05in}
\raisebox{-.5cm}{\includegraphics[height=1.775cm]{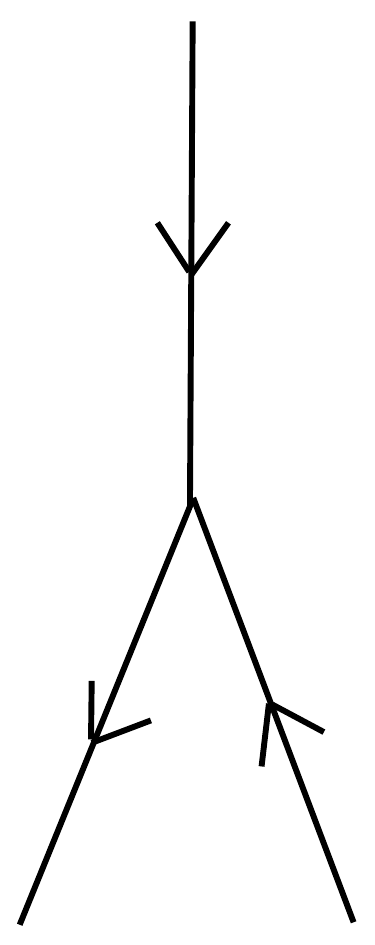}}  \hspace{.05in}
\]
\caption{A right $R5$ move on a $\lambda$-vertex with one up-arc} \label{fig:DUDL}
\end{figure}

Finally, we consider an $R5$ move on a $\lambda$-vertex with all edges oriented downward and with the twist between the left lower edge and upper edge meeting at the vertex. This case is demonstrated in Fig.~\ref{fig:DDDL}. 

\begin{figure}[ht]
\[
\raisebox{-.5cm}{\includegraphics[height=1.775cm]{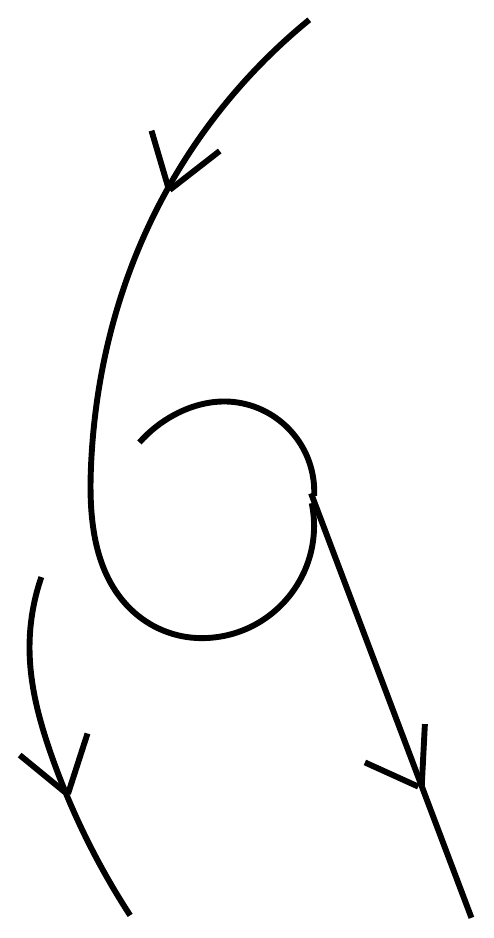}}  \hspace{.05in} \stackrel{\text{RP}}{\longrightarrow} \hspace{.05in}
\raisebox{-.5cm}{\includegraphics[height=1.775cm]{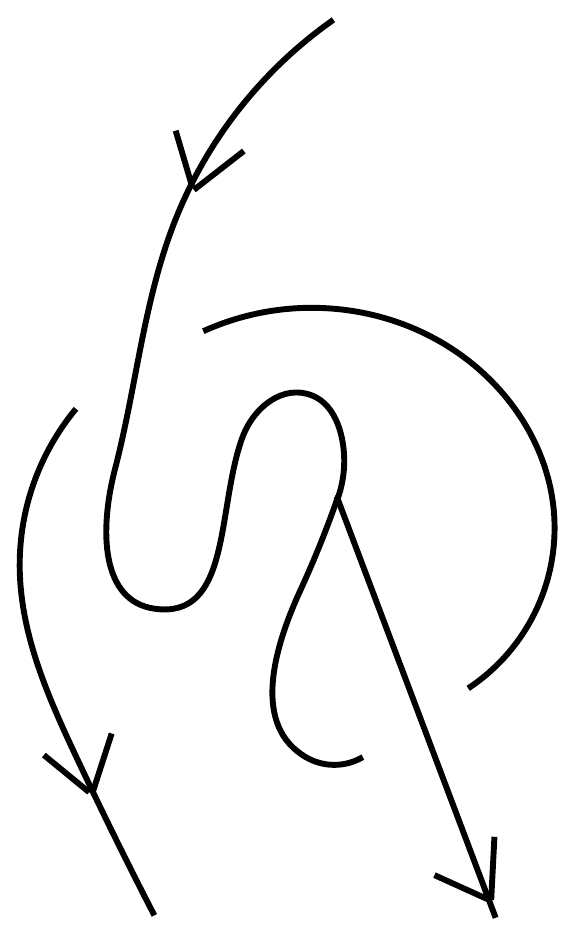}}  \hspace{.05in} 
\underset{\text{isotopy}}{\overset{\text{planar}}{\longleftrightarrow}}
\raisebox{-.5cm}{\includegraphics[height=1.775cm]{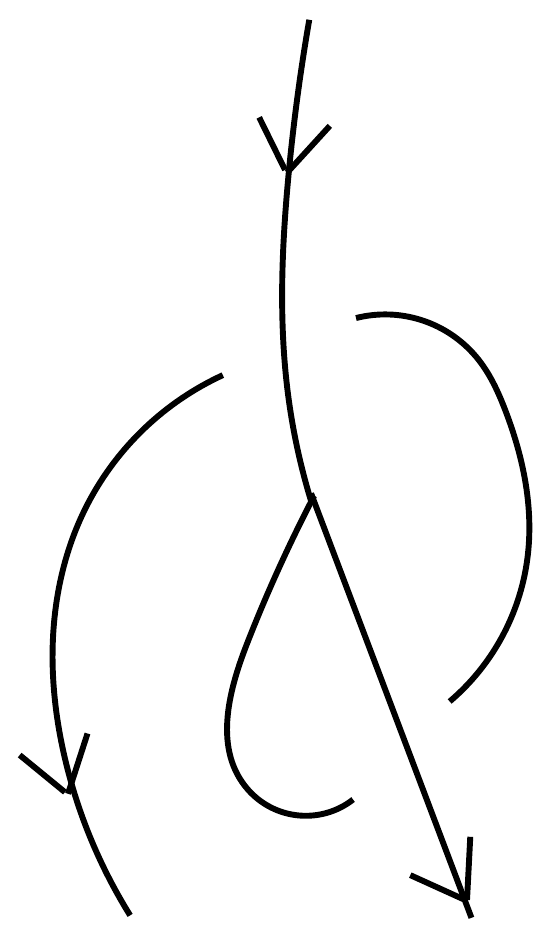}}  \hspace{.05in} \stackrel{\text{$R4$}}{\longleftrightarrow} \hspace{.05in}
\raisebox{-.5cm}{\includegraphics[height=1.775cm]{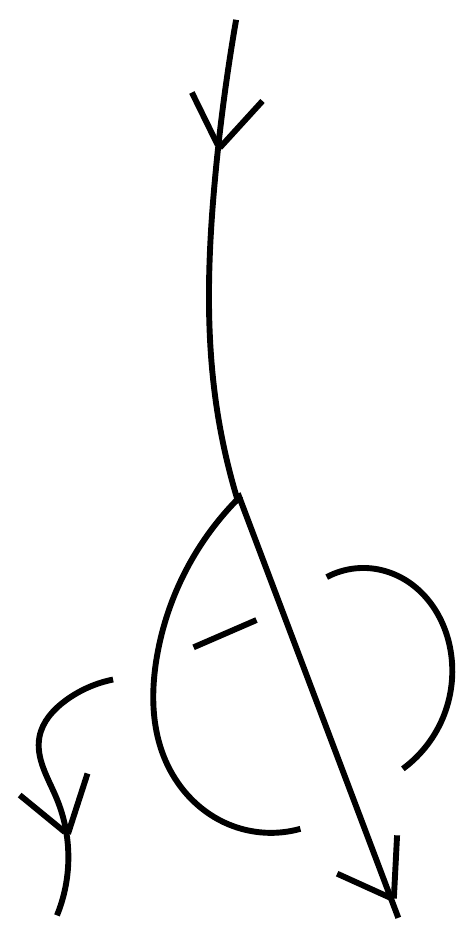}}  \hspace{.05in} 
\stackrel{\text{$R2$}}{\longleftrightarrow} \hspace{.05in}
\raisebox{-.5cm}{\includegraphics[height=1.775cm]{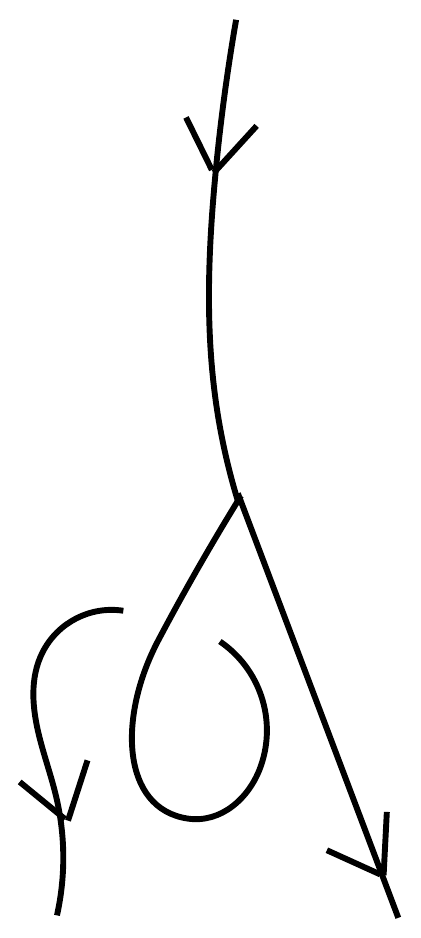}}  \hspace{.05in} \stackrel{\text{$R1$}}{\longleftrightarrow} \hspace{.05in}
\raisebox{-.5cm}{\includegraphics[height=1.775cm]{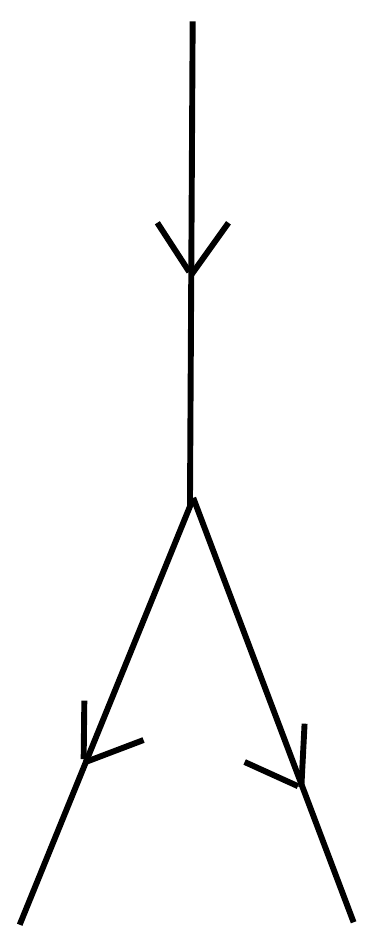}}  \hspace{.05in}
\]
\caption{A left $R5$ move on a $\lambda$-vertex with three down-arcs} \label{fig:DDDL}
\end{figure}
All of the other versions of the $R5$ move on a $\lambda$-vertex are verified similarly. This completes the proof of the theorem.
\end{proof}

\begin{remark}
Part III of Theorem~\ref{Markov-typeThrm} could be shortened if one relies on a minimal set of moves that generates all of the oriented versions of the extended Reidemeister moves for well-oriented STG diagrams. With such a minimal set of moves at hand, one would only need to verify that those particular moves yield $TL_v$-equivalent braids upon our braiding algorithm. We are not aware of a paper in the literature providing such a generating set of moves for STG diagrams. Nonetheless, that could be an interesting and useful small project.
\end{remark}
\subsection{Algebraic Markov-type theorem} \label{ssec:MarkovThm2}

In this section, we provide an algebraic Markov-type theorem for virtual trivalent braids and virtual spatial trivalent graphs.

Generally, we regard $\alpha\in VTB_{n}^m$ as an element of $VTB_{n+1}^{m+1}$  by adding an identity strand on the right of $\alpha$.  (We will not introduce any new notation when we regard $\alpha\in VTB_{n+1}^{m+1}$.) Using this operation of adding a single identity strand on the right of a braid, we can think of $VTB_{n}^m$ as a subset of $VTB_{n+1}^{m+1}$, and we define $VTB:=\cup_{m,n=1}^{\infty}VTB_{n}^m$. In what follows, we also allow the addition of an identity strand at the left of $\alpha\in VTB_{n}^m$ and we denote this braid in $VTB_{n+1}^{m+1}$ by $i(\alpha)$.

\begin{theorem}[Algebraic Markov-type Theorem] \label{thm:AMarkov}
Two well-oriented virtual spatial trivalent graph diagrams are isotopic if and only if any two of their corresponding virtual trivalent braids differ by a finite sequence of braid relations in $VTB$ and the following algebraic moves:
\begin{enumerate}[i.]
\item Virtual and real conjugation (see Fig.~\ref{Conj}): 
\[v_i\alpha  \sim \alpha v_i\, , \,  \sigma_i \alpha  \sim \alpha \sigma_i \,, \,\, \text{where} \, \, \alpha \in VTB_{n}^{n} \,\, \text{and} \,\, 1\leq i \leq n-1\]

\item Right virtual and real stabilization (see Fig.~\ref{Stab}):
\[ \alpha v_n \beta \sim \alpha \beta \sim  \alpha \sigma_n^{\pm 1}\beta, \, \,  \text{where} \,\, \alpha, \beta \in VTB^n_n \]

\item Algebraic right and left under-threading (see Fig.~\ref{UndThd}):  
\[ \alpha \beta \sim  \alpha \sigma_n^{-1} v_{n-1} \sigma_n \beta \,\, , \,\,  \alpha \beta \sim  i(\alpha)\sigma_{1} v_2 \sigma_{1}^{-1} i(\beta), \,  \text{where} \,\,  \alpha, \beta \in VTB^n_n\]

\item Algebraic right trivalent relation (see Fig.~\ref{AlgTr}):
\[ \alpha \lambda_n v_{n-1} \sigma_n \beta \sim  \alpha \lambda_n v_{n-1} \sigma_n^{-1} \beta, \, \, \text{where} \, \, \alpha \in VTB^{n}_{n-1} \,  \text{and} \,\, \beta \in VTB^n_n.\]

\end{enumerate}
\end{theorem}

\begin{figure}[ht]
\[
\raisebox{-35pt}{\includegraphics[height=1.1in]{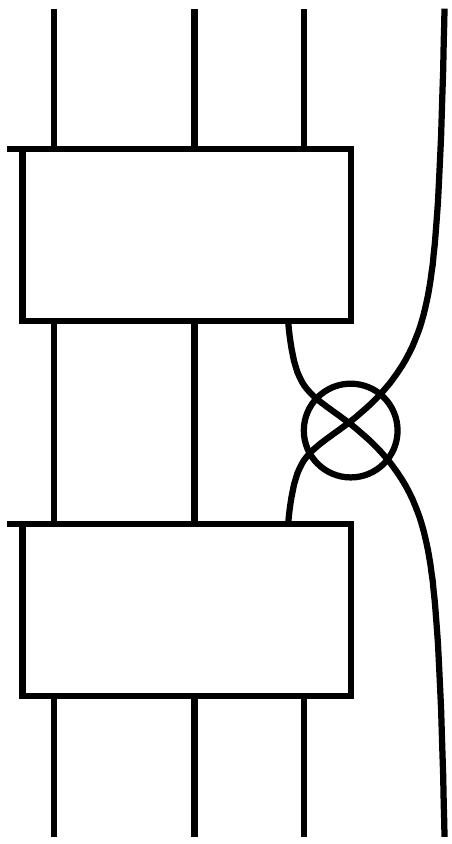}}
\hspace{.2in}
{\sim}
\hspace{.2in}
\raisebox{-35pt}{\includegraphics[height=1.1in]{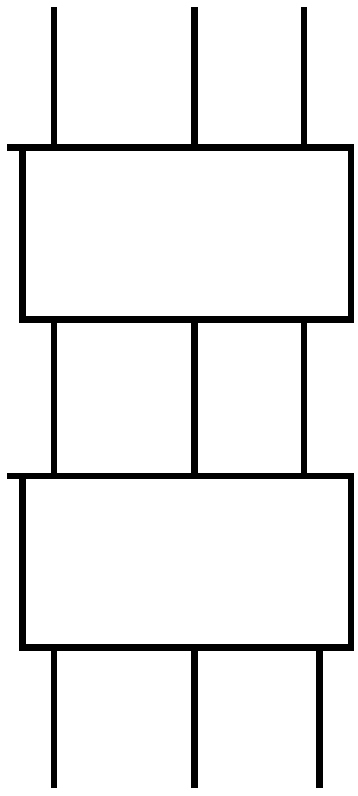}}
\hspace{.2in}
{\sim}
\hspace{.2in}
\raisebox{-35pt}{\includegraphics[height=1.1in]{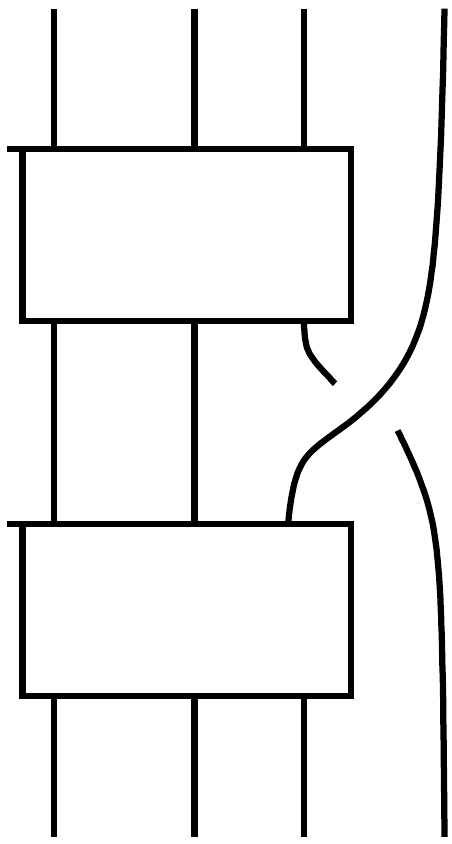}}
 \put(-178, 20){\fontsize{9}{9}$\alpha$}
 \put(-97,20){\fontsize{9}{9}$\alpha$}
 \put(-30,20){\fontsize{9}{9}$\alpha$}
  \put(-178, -15){\fontsize{9}{9}$\beta$}
 \put(-97, -15){\fontsize{9}{9}$\beta$}
 \put(-30, -15){\fontsize{9}{9}$\beta$}
\]
\caption{Right virtual and real stabilization}\label{Stab}
\end{figure}

\begin{figure}[ht]
\[
\reflectbox{\raisebox{-35pt}{\includegraphics[height=1.1in]{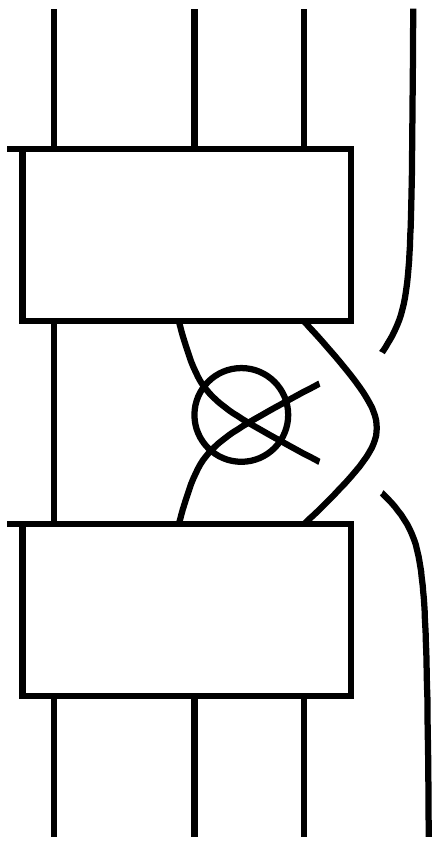}}}
\hspace{.2in}
\sim
\hspace{.2in}
\raisebox{-35pt}{\includegraphics[height=1.1in]{Rstab5-new}}
\hspace{.2in}
\sim
\hspace{.2in}
\raisebox{-35pt}{\includegraphics[height=1.1in]{AlgUnThrdRight-new}}
 \put(-180, 20){\fontsize{9}{9}$\alpha$}
 \put(-105,20){\fontsize{9}{9}$\alpha$}
 \put(-28,20){\fontsize{9}{9}$\alpha$}
  \put(-180, -15){\fontsize{9}{9}$\beta$}
 \put(-103, -15){\fontsize{9}{9}$\beta$}
 \put(-27, -15){\fontsize{9}{9}$\beta$}
\]
\caption{Algebraic left and right under-threading} \label{UndThd}
\end{figure}

\begin{figure}[ht]
\[
\raisebox{-40pt}{\includegraphics[height=1.2in]{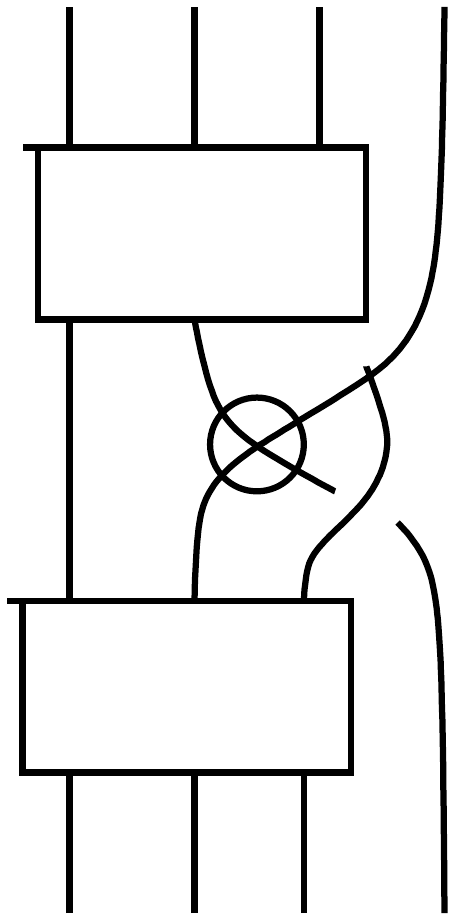}}
\hspace{.2in}
\sim
\hspace{.2in}
\raisebox{-40pt}{\includegraphics[height=1.2in]{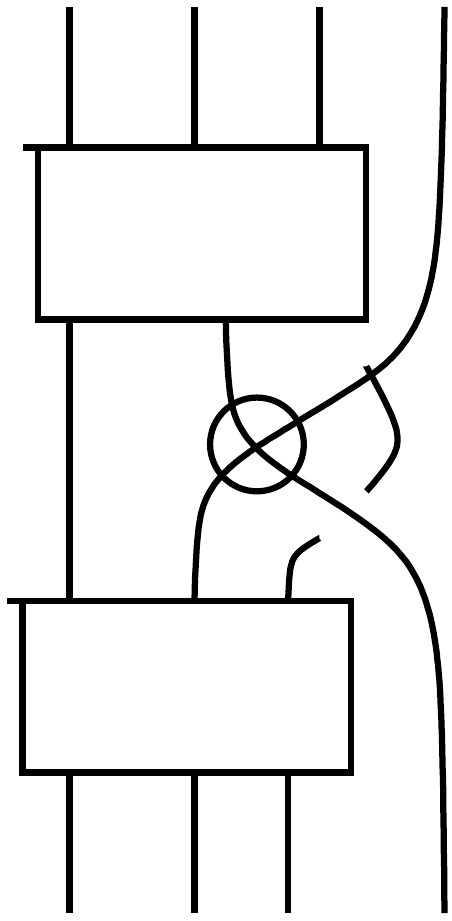}}
\put(-113, 23){\fontsize{9}{9}$\alpha$}
 \put(-28,23){\fontsize{9}{9}$\alpha$}
 \put(-113, -20){\fontsize{9}{9}$\beta$}
 \put(-28,-20){\fontsize{9}{9}$\beta$}
\]
\caption{Algebraic right trivalent relation} \label{AlgTr}
\end{figure}

The braid relations together with the algebraic moves (i)--(iv) given in Theorem~\ref{thm:AMarkov} introduce an equivalence relation on the set of virtual trivalent braids, called \textit{algebraic equivalence}.  We say that two braids belonging to the same algebraic equivalence class are \textit{algebraically-equivalent}.

\begin{proof}
[Proof of Theorem~\ref{thm:AMarkov}]

It is easy to see that algebraically-equivalent braids have isotopic closures. Using Theorem~\ref{Markov-typeThrm}, it is then sufficient to show that the algebraic moves (i)--(iv) follow from the $TL_v$-equivalence. We will prove a stronger statement, namely that the algebraic equivalence implies the $TL_v$-equivalence, and vice versa. Note first that braid isotopy and real conjugation are part of both the algebraic equivalence and $TL_v$-equivalence. Recall also that virtual conjugation follows from $L_v$-moves and braid isotopy (see~\cite[Fig. 17]{KauLamb}).

We show first that the moves comprising the $TL_v$-equivalence follow from the algebraic moves. In Fig.~\ref{rLv-RS} we show that the right real $L_v$-move follows from right real stabilization, virtual conjugation, and braid isotopy. (A similar proof was given in~\cite[Fig. 37]{KauLamb}.) It can be similarly demonstrated that the right virtual $L_v$-move follows from right virtual stabilization, virtual conjugation, and braid isotopy.

\begin{figure}
\[
\hspace{-.03in}
\raisebox{-35pt}{\includegraphics[height=1.1in]{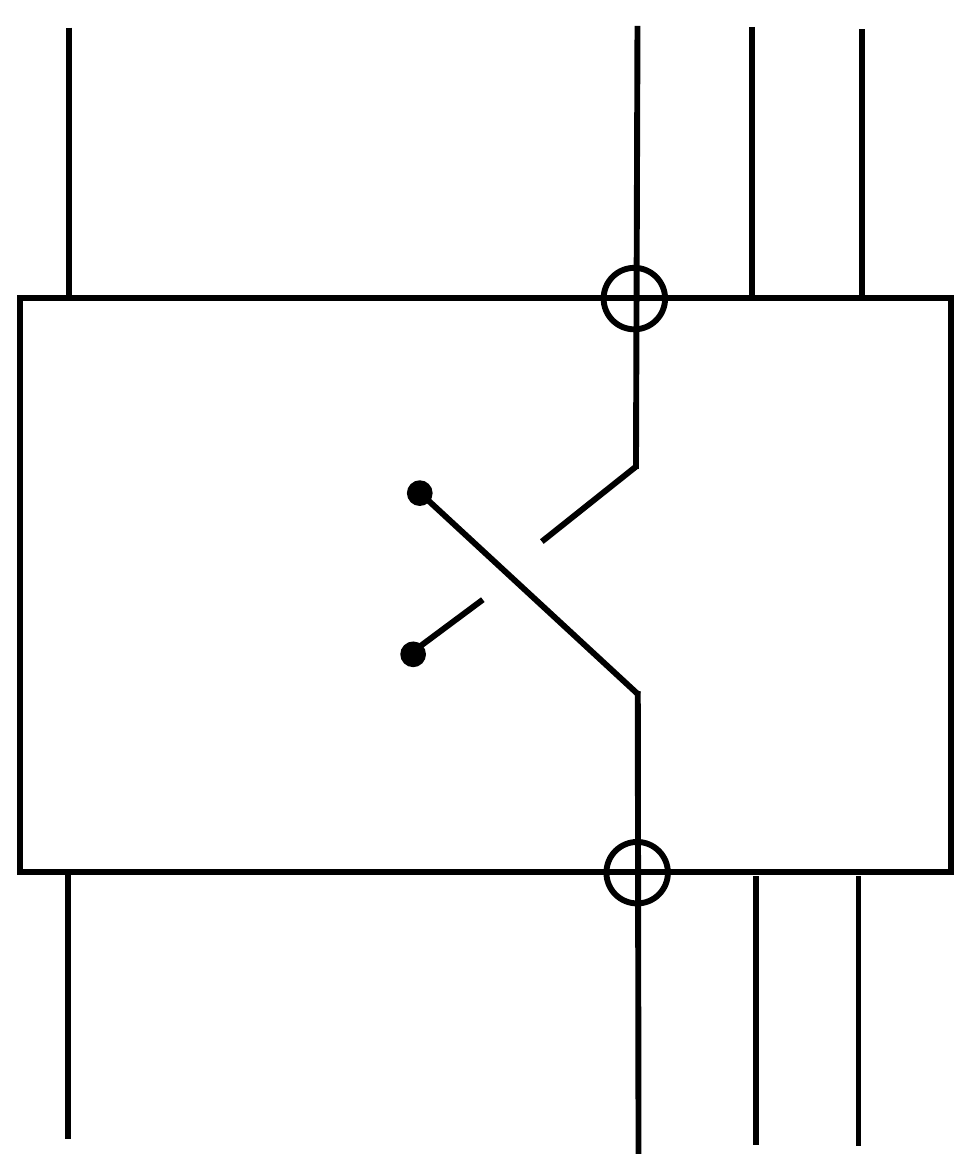}}
\hspace{.07in}
\underset{\text{detour}}{\overset{\text{braid}}{\longleftrightarrow}}
\hspace{.1in}
\raisebox{-35pt}{\includegraphics[height=1.1in]{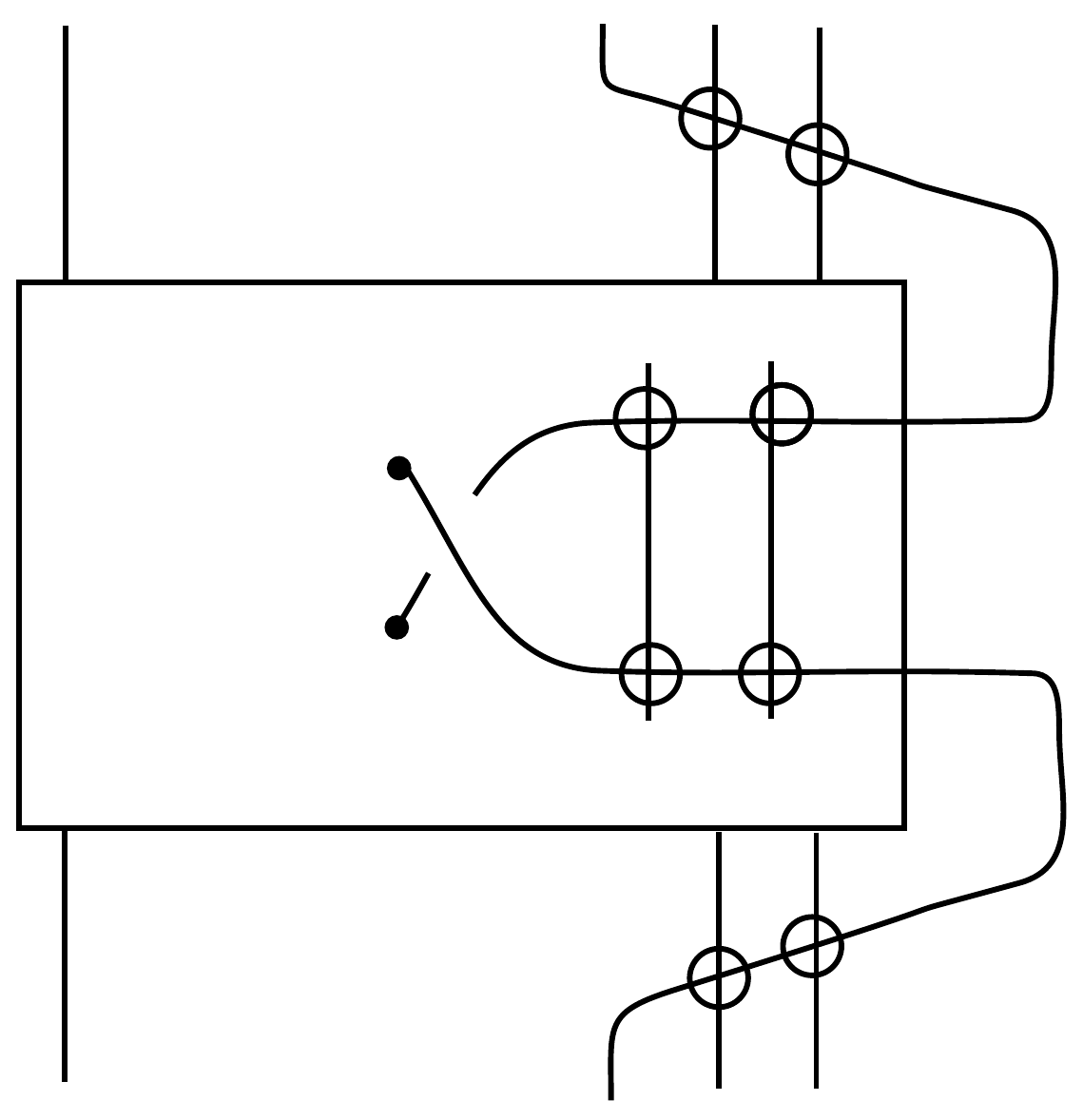}}
\hspace{.1in}
\underset{\text{detour move}}{\overset{\text{virtual conj.}}{\longleftrightarrow}}
\hspace{.1in}
\]
\[
\raisebox{-35pt}{\includegraphics[height=1.1in]{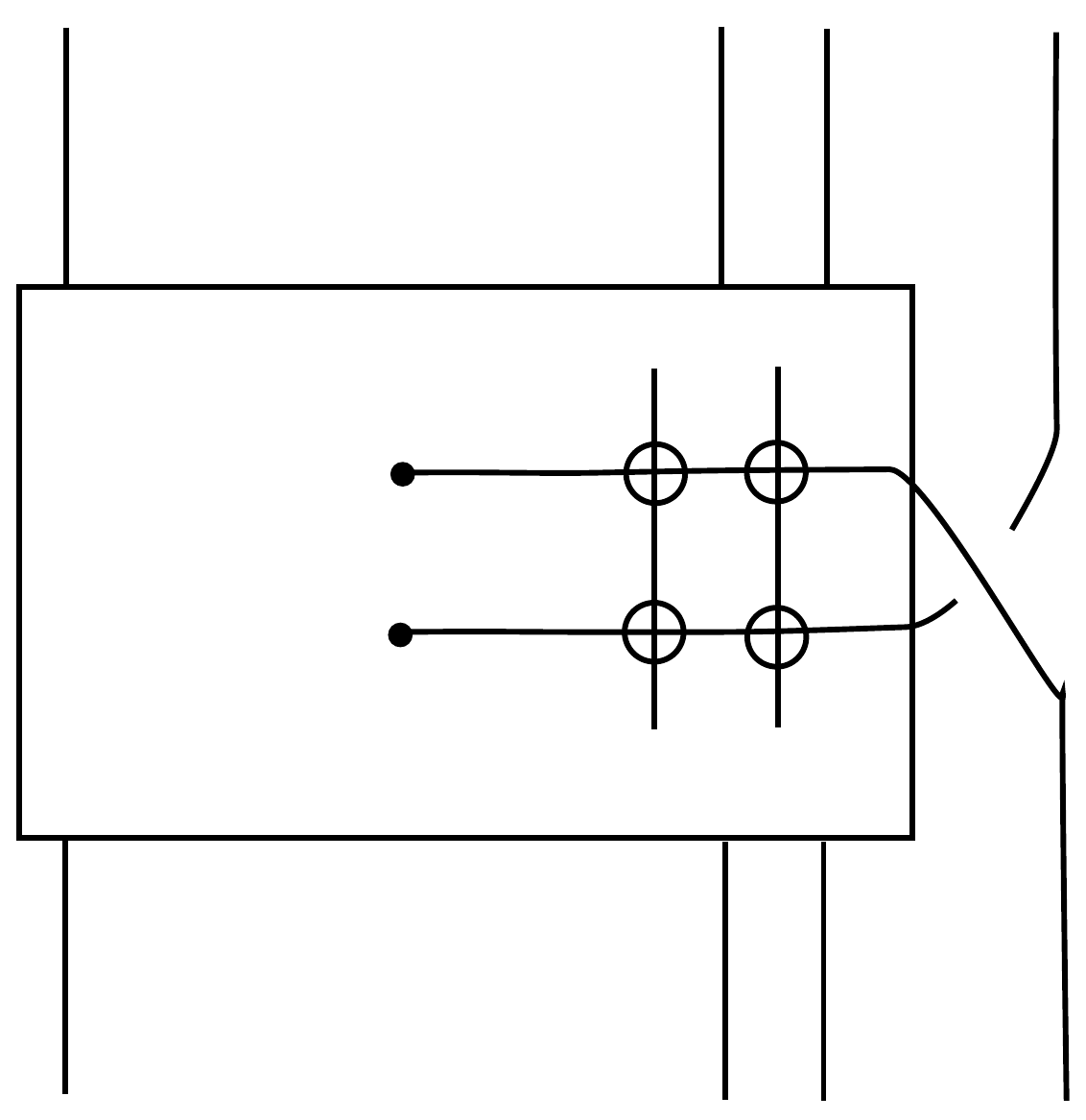}}
\underset{\text{stabilization}}{\overset{\text{right real}}{\longleftrightarrow}}
\hspace{.1in}
\raisebox{-35pt}{\includegraphics[height=1.1in]{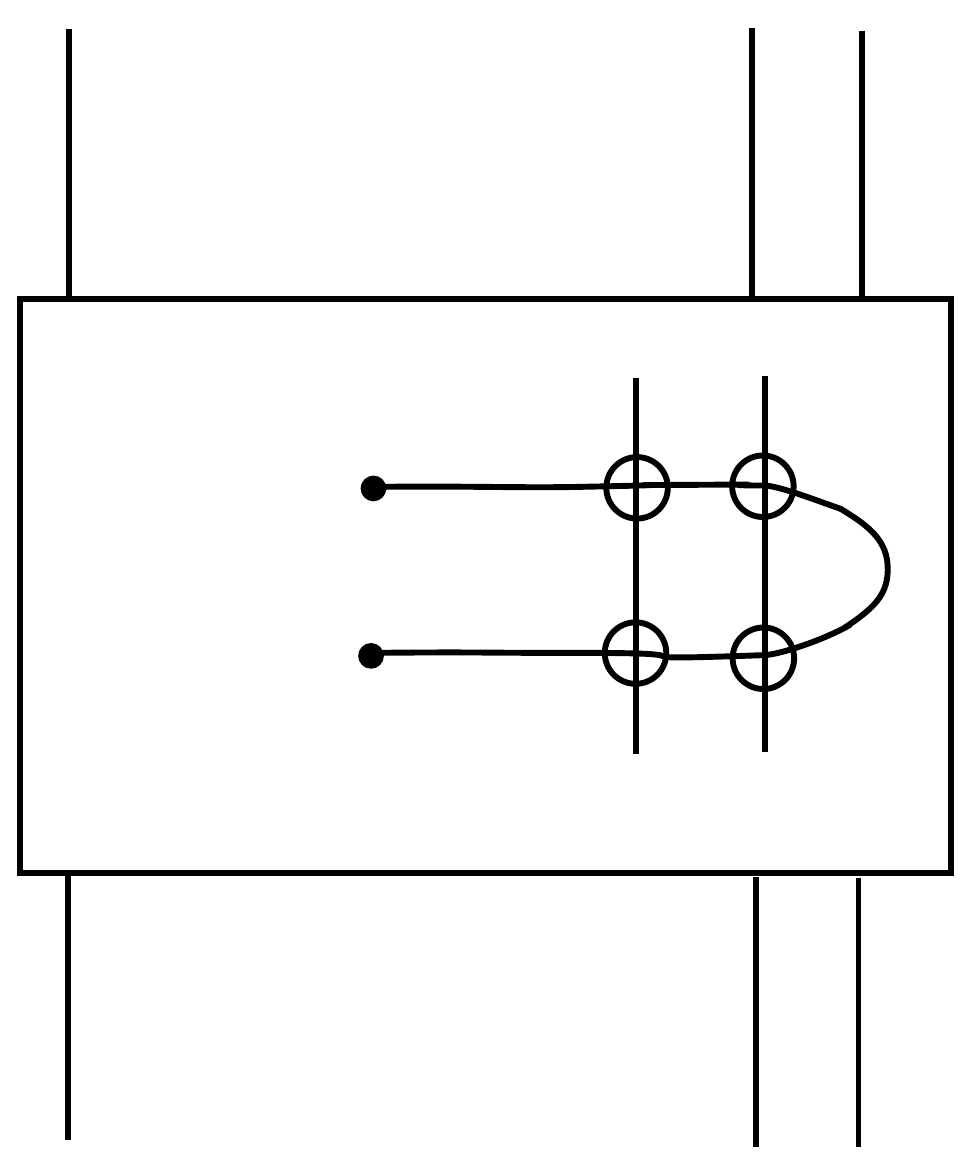}} \hspace{.1in}
{\overset{\text{V2}}{\longleftrightarrow}}
\hspace{.1in}
\raisebox{-35pt}{\includegraphics[height=1.1in]{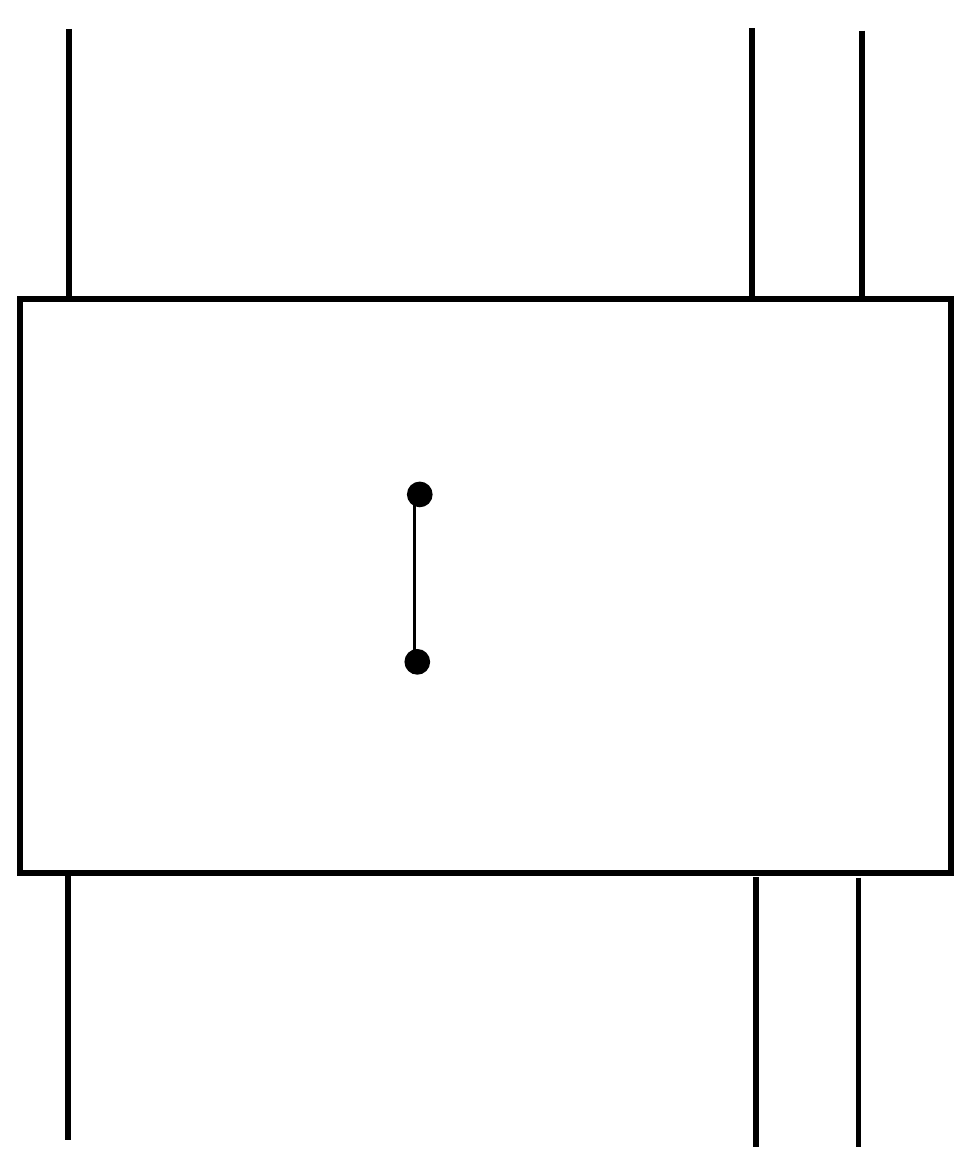}}
\]
\caption{Right real $L_v$-move derived from the algebraic equivalence}\label{rLv-RS}
\end{figure}

 The right trivalent $L_v$-move follows from (virtual trivalent) braid isotopy, virtual conjugation, and algebraic right trivalent relation, as shown in  Fig.~\ref{AlgTL}.

\begin{figure}[ht]
\[
\hspace{-.03in}
\raisebox{-35pt}{\includegraphics[height=1.1in]{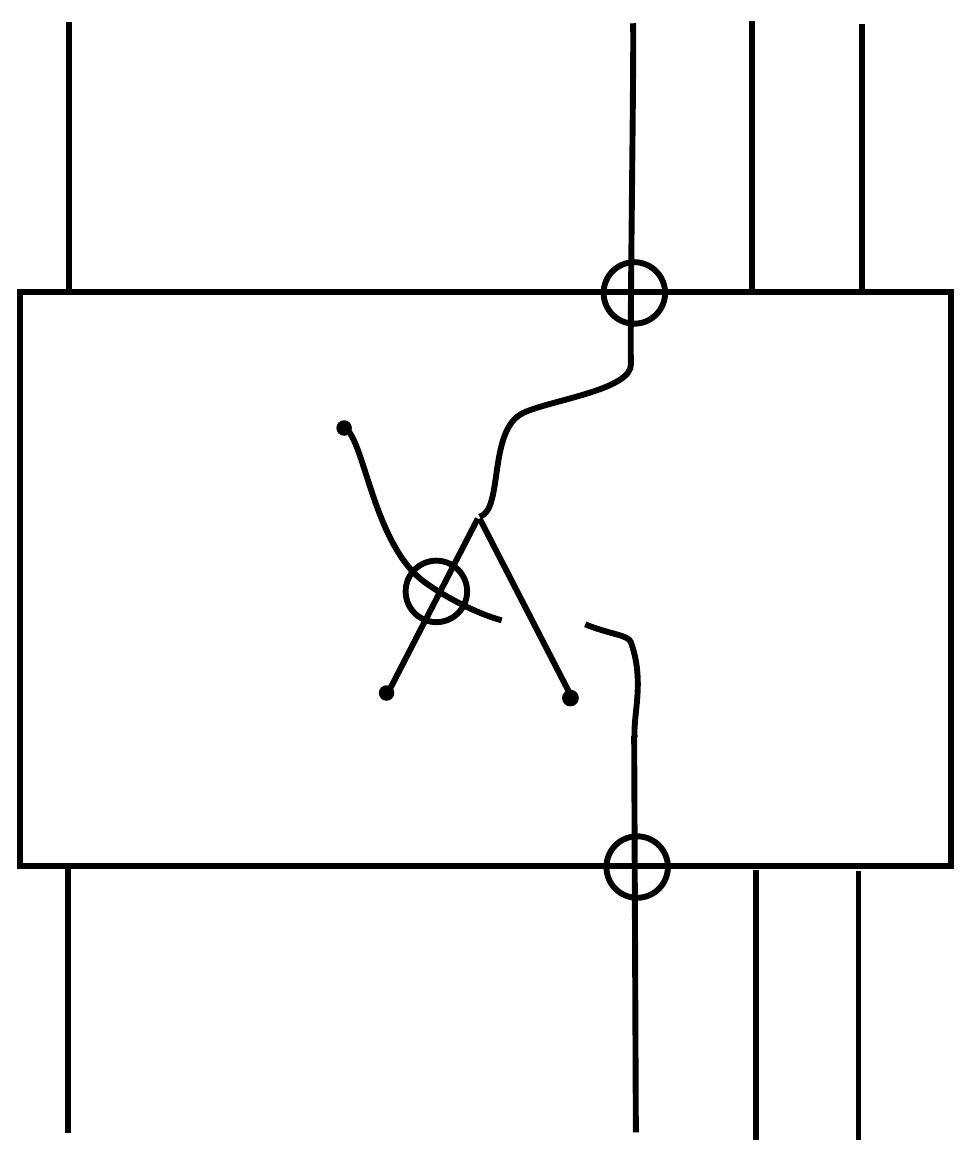}}
\hspace{.07in}
\underset{\text{detour}}{\overset{\text{braid}}{\longleftrightarrow}}
\hspace{.1in}
\raisebox{-35pt}{\includegraphics[height=1.1in]{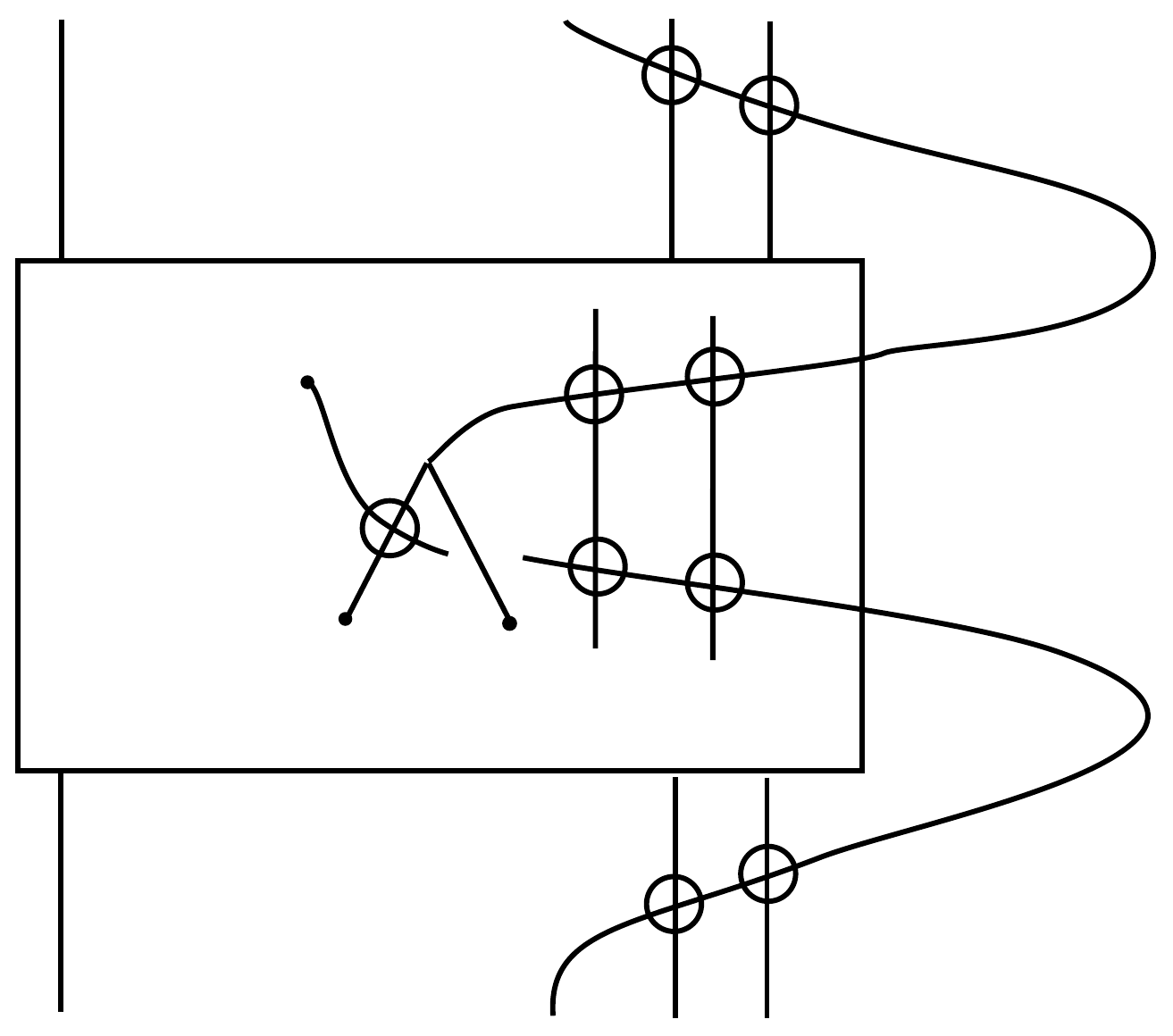}}
\hspace{.1in}
\underset{\text{conj.}}{\overset{\text{virtual}}{\longleftrightarrow}}
\hspace{.1in}
\raisebox{-35pt}{\includegraphics[height=1.1in]{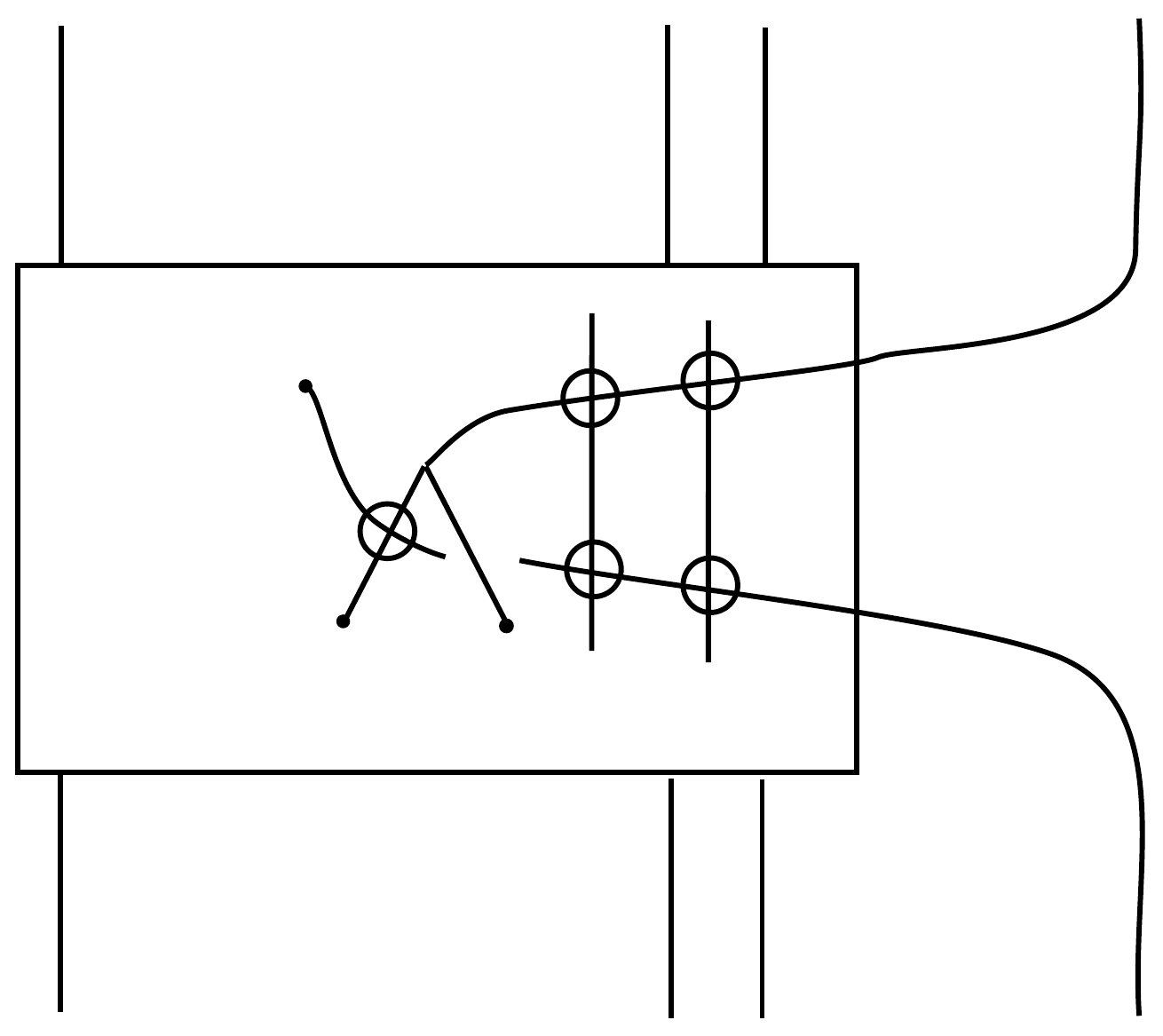}}
\]
\[
\underset{\text{move}}{\overset{\text{detour}}{\longleftrightarrow}}
\hspace{.1in}
\raisebox{-45pt}{\includegraphics[height=1.7in]{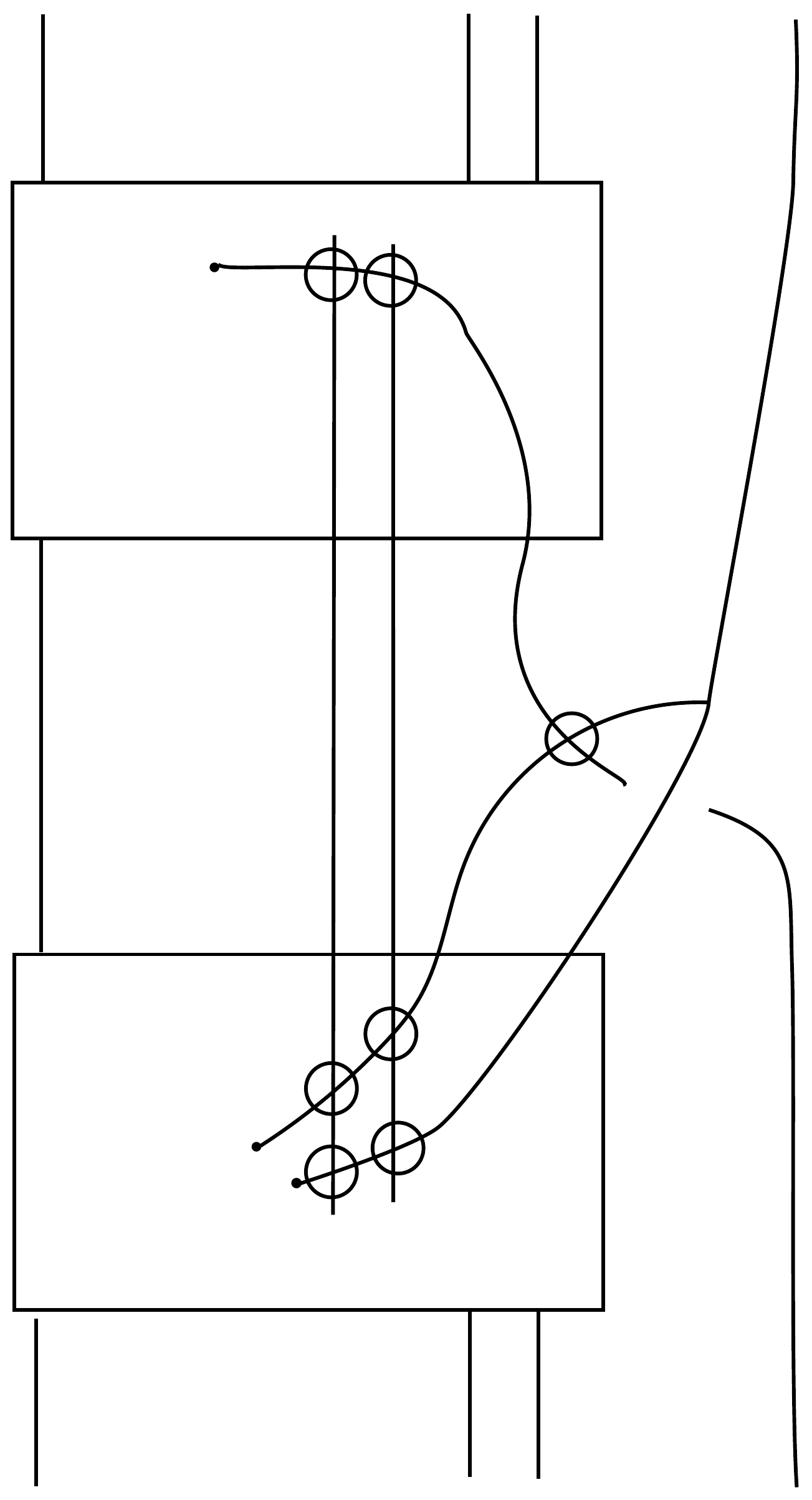}}
\hspace{.1in}
\underset{\text{trival. rel.}} {\overset{\text{alg. right}}
{\longleftrightarrow }}
\hspace{.1in}
\raisebox{-45pt}{\includegraphics[height=1.7in]{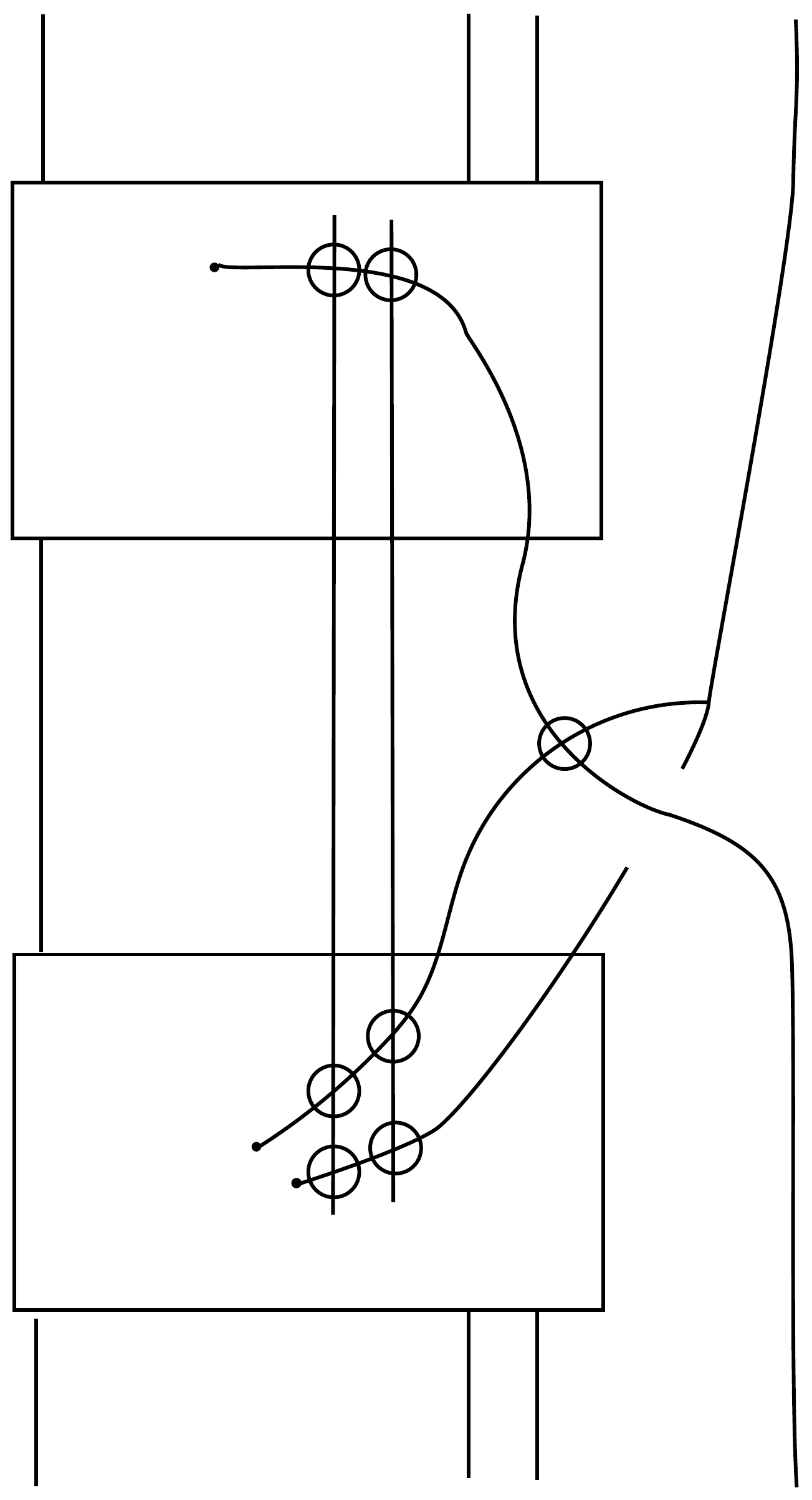}}
\hspace{.1in}
\underset{\text{move}}{\overset{\text{detour}}{\longleftrightarrow}}
\hspace{.1in}
\]
\[
\hspace{-.03in}
\raisebox{-35pt}{\includegraphics[height=1.1in]{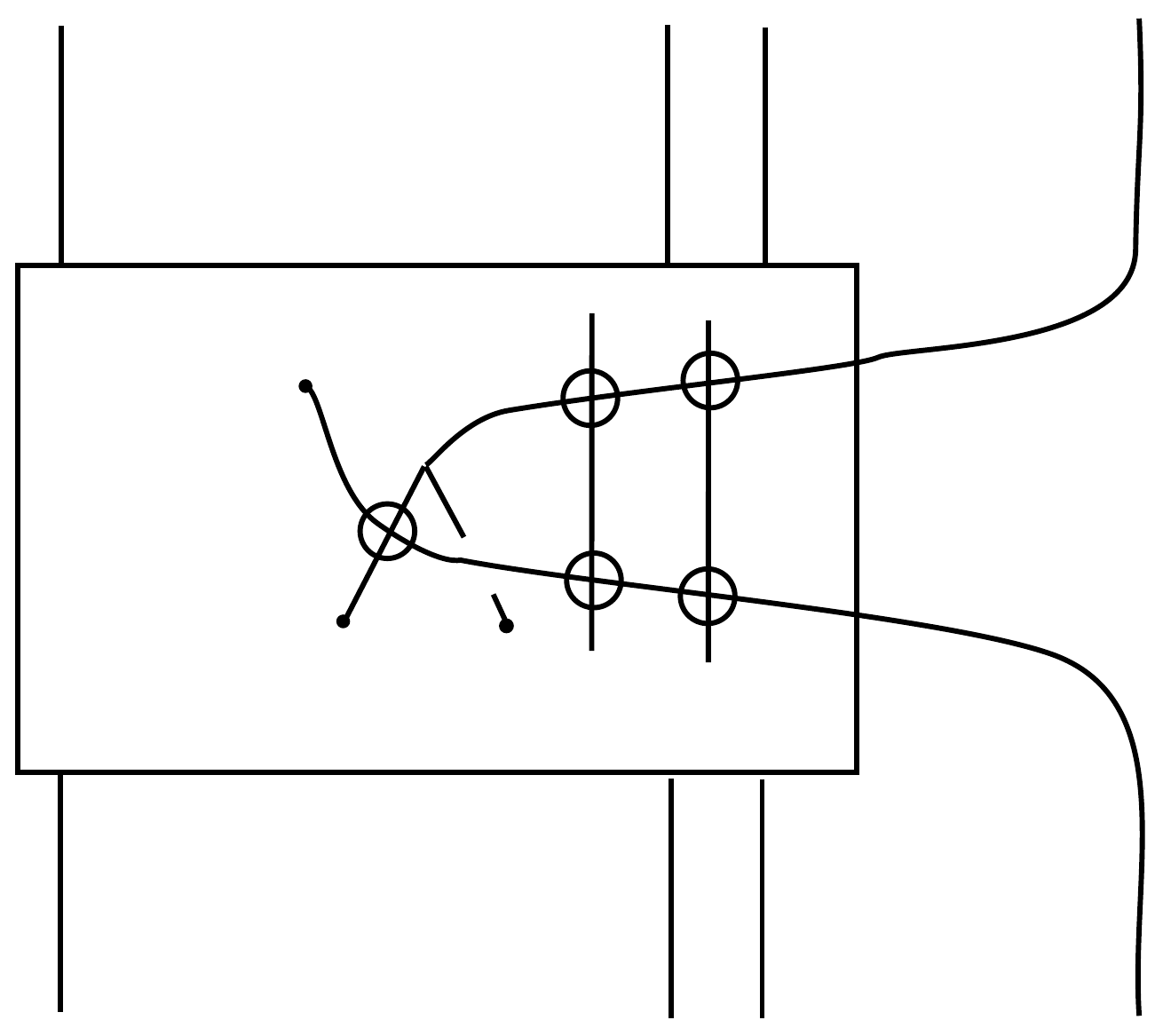}}
\hspace{.07in}
\underset{\text{conj.}}{\overset{\text{virtual}}{\longleftrightarrow}}
\hspace{.1in}
\raisebox{-35pt}{\includegraphics[height=1.1in]{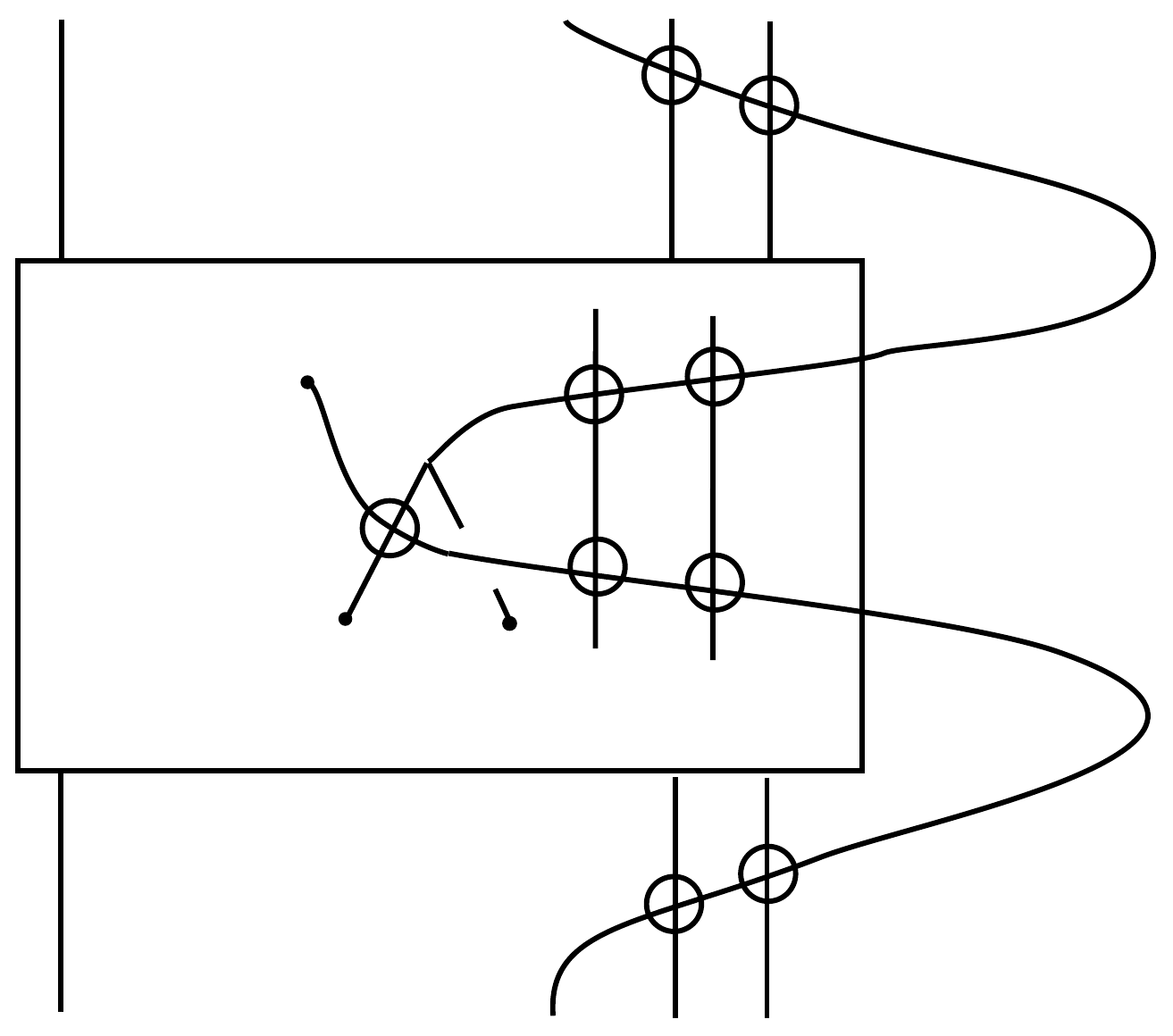}}
\hspace{.1in}
\underset{\text{detour}}{\overset{\text{braid}}{\longleftrightarrow}}
\hspace{.1in}
\raisebox{-35pt}{\includegraphics[height=1.1in]{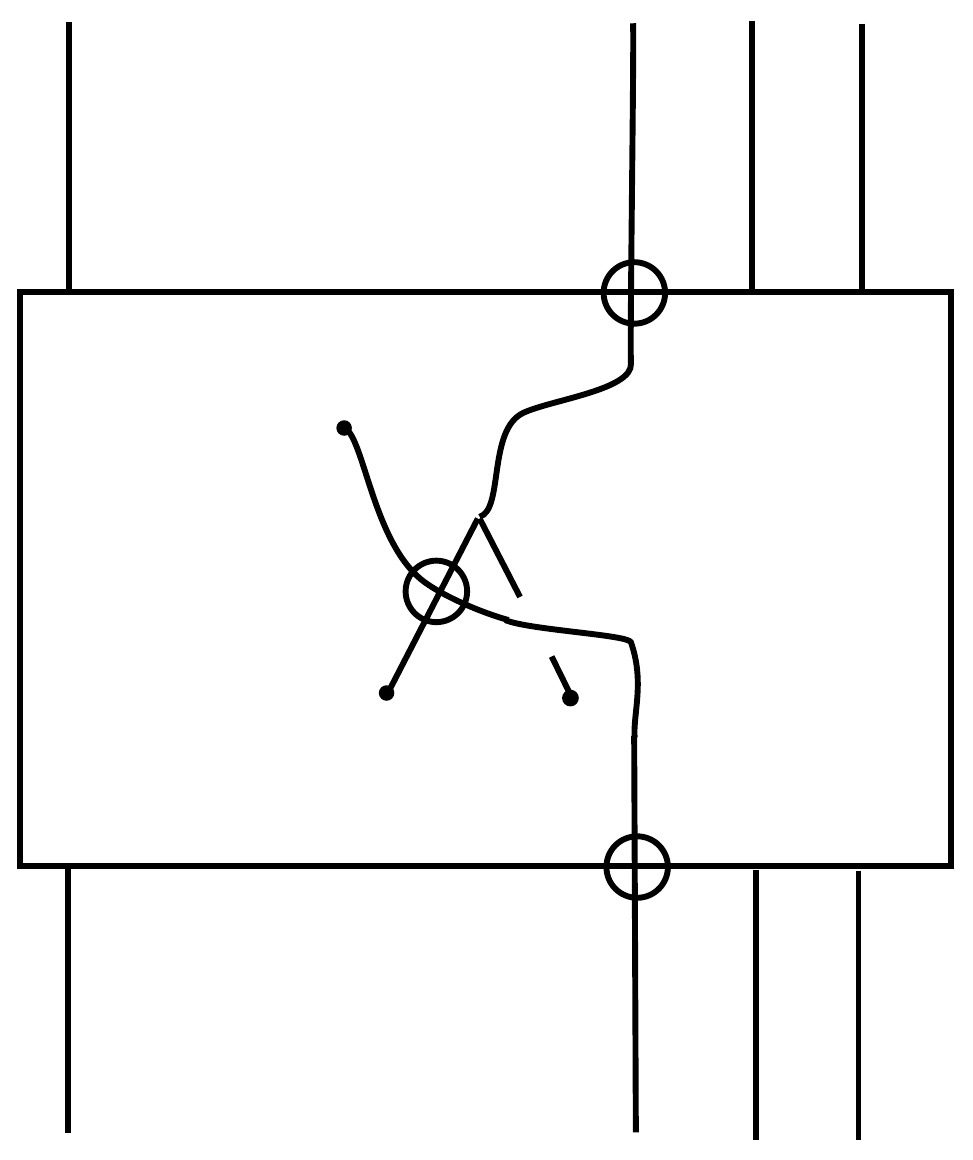}}
\]
\caption{Right trivalent $L_v$-move derived from the algebraic equivalence} \label{AlgTL}
\end{figure}

The left under-threaded $L_v$-move follows from braid isotopy, virtual conjugation, and algebraic left under-threading. This is shown in Fig.~\ref{fig:threadingproof}. Similar steps can also be used to show that the right under-threaded $L_v$-move follows from the algebraic equivalence.

\begin{figure}[ht]
\[ 
\reflectbox{\raisebox{-27pt}{\includegraphics[height=.8in]{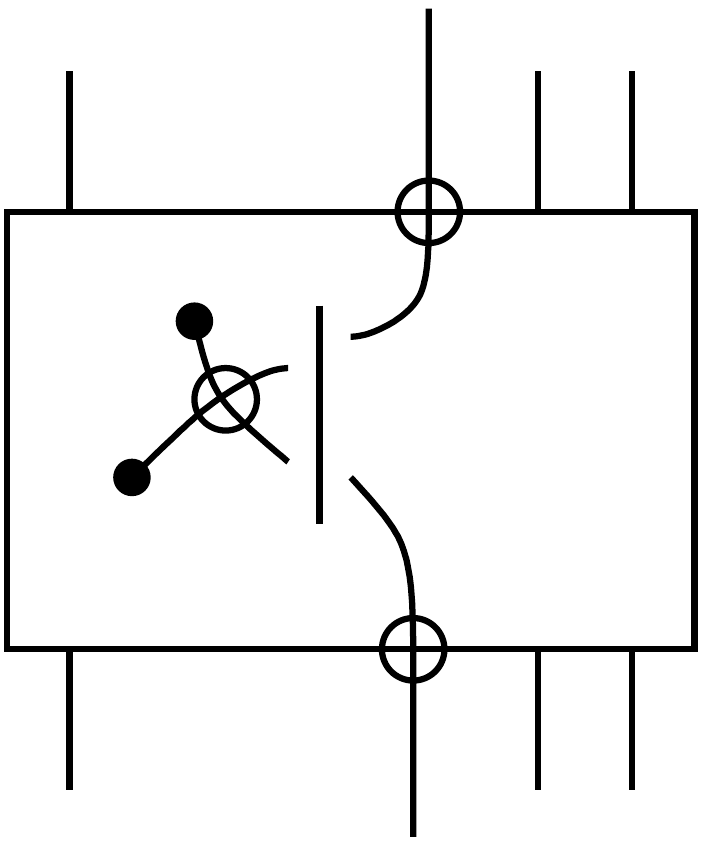}}} \,\, \overset{\text{braid}}{\underset{\text{detour}}{\longleftrightarrow}} \,\, \reflectbox{\raisebox{-30pt}{\includegraphics[height=.87in]{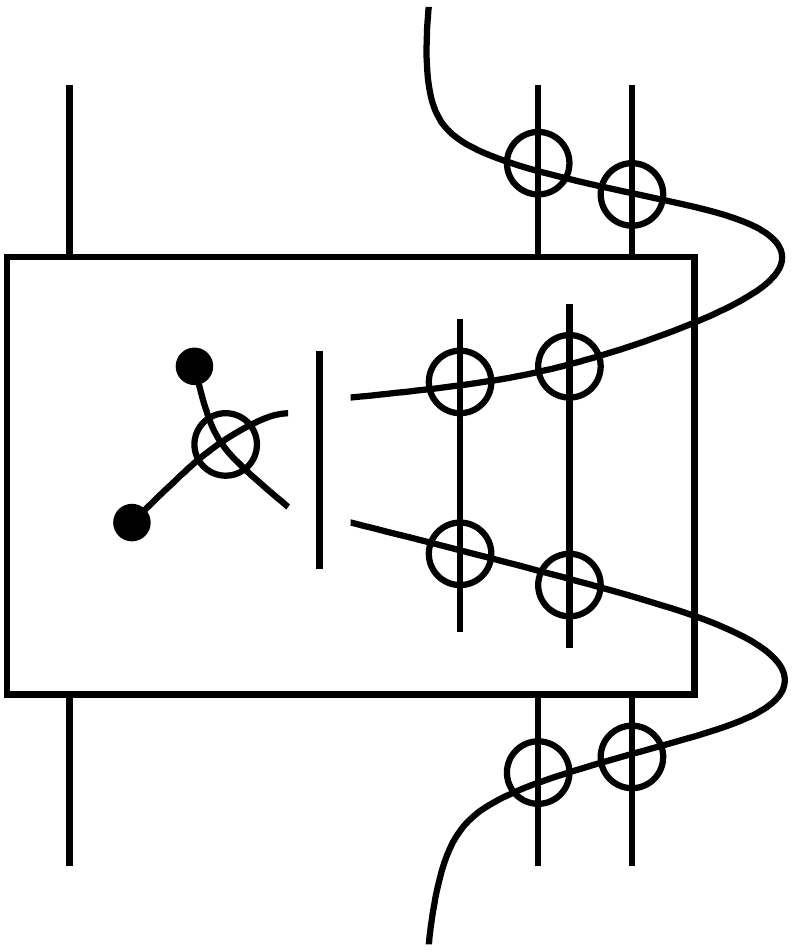}}} \,\, \overset{\text{virtual}}{\underset{\text{conj.}}{\longleftrightarrow}}
\reflectbox{\raisebox{-25pt}{\includegraphics[height=0.72in]{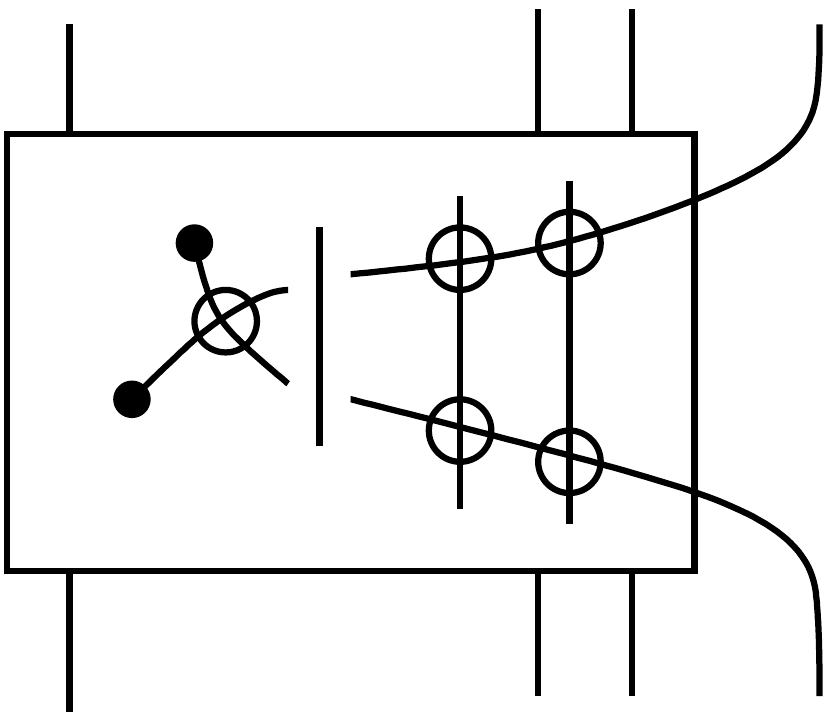}}} \,\, \overset{\text{detour}}{\underset{\text{virt. threads}} {\longleftrightarrow}}
\]
\[ 
\reflectbox{\raisebox{-25pt}{\includegraphics[height=0.72in]{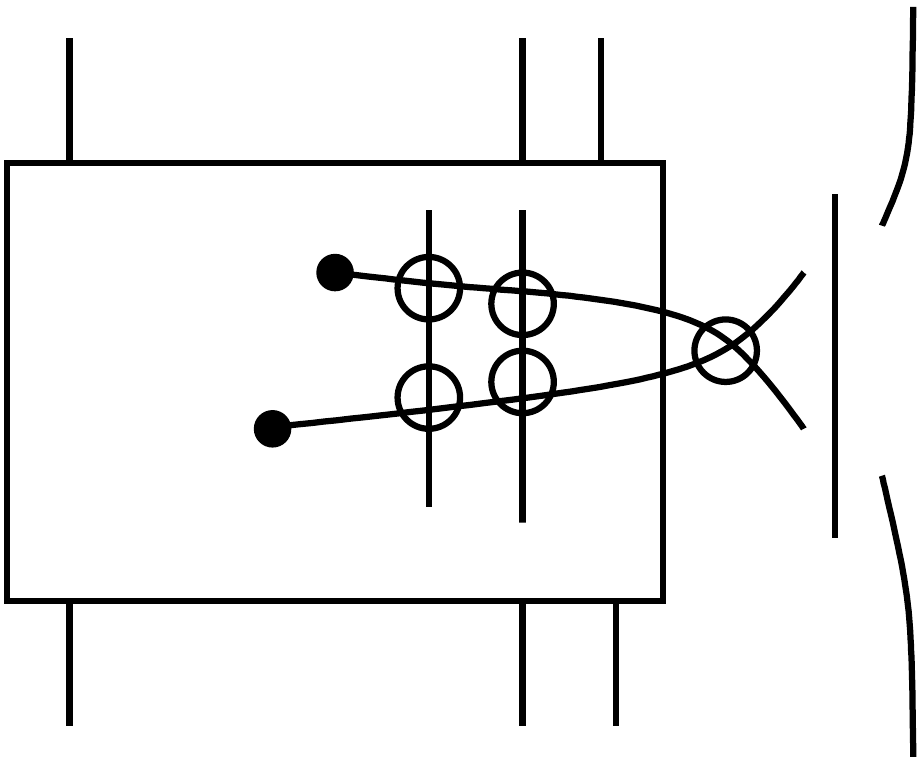}}} \,\, \overset{\text{alg. under-}}{\underset{\text{threading}}{\longleftrightarrow}} \,\,\reflectbox{\raisebox{-25pt}{\includegraphics[height=.72in]{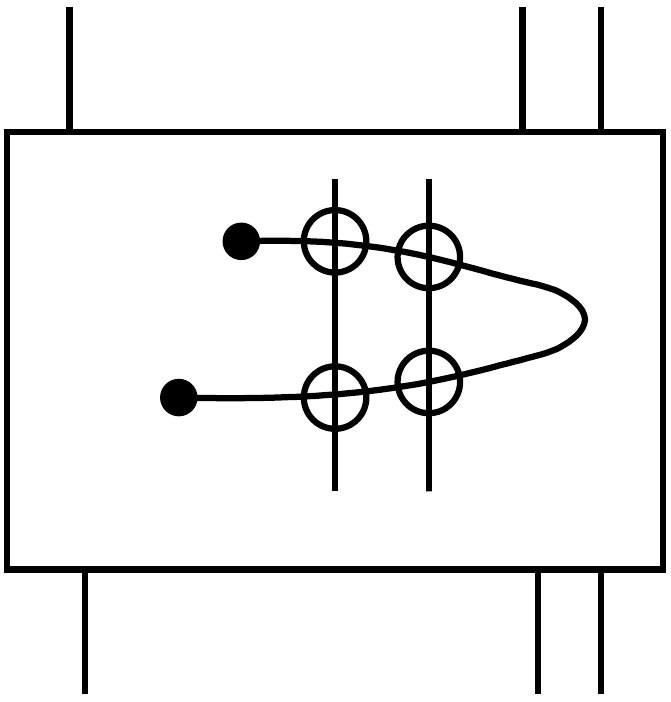}}} \,\,
\stackrel{V2}{\longleftrightarrow} \,\,
\reflectbox{\raisebox{-25pt}{\includegraphics[height=.67in]{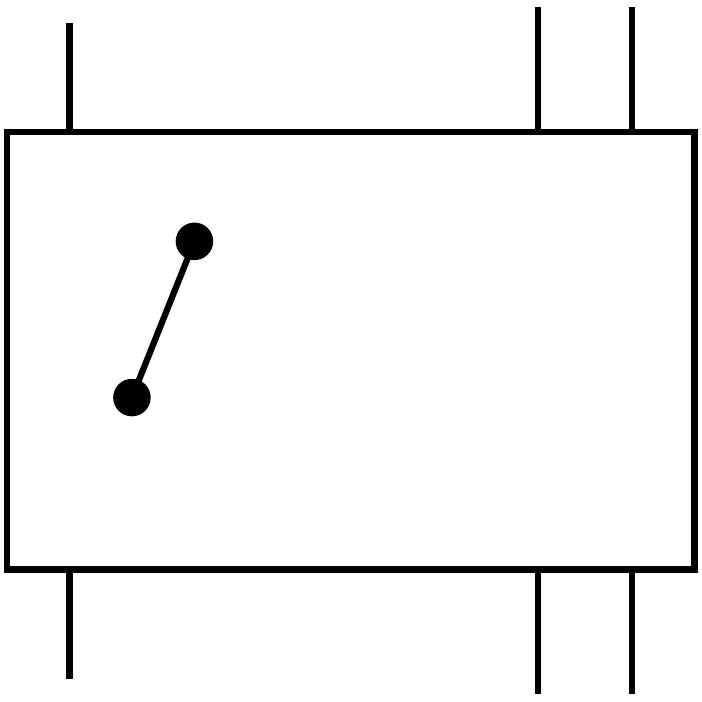}}}
\]
\caption{Left under-threaded $L_v$-move derived from algebraic equivalence}\label{fig:threadingproof}
\end{figure}

Therefore, $TL_v$-equivalence follows from algebraic equivalence. We show now the converse. The right virtual stabilization follows from virtual conjugation and right virtual $L_v$-move, as shown in Fig.~\ref{RightStab}. Right real stabilization follows similarly, with the main different that the right virtual $L_v$-move is replaced by a right real $L_v$-move.

\begin{figure}[ht]
\[
\raisebox{-35pt}{\includegraphics[height=1in]{Rstab1-new}}
\hspace{.1in}
\underset{\text{conj.}}{\overset{\text{virtual}}{\longleftrightarrow}}
\hspace{.1in}
\raisebox{-45pt}{\includegraphics[height=1.3in]{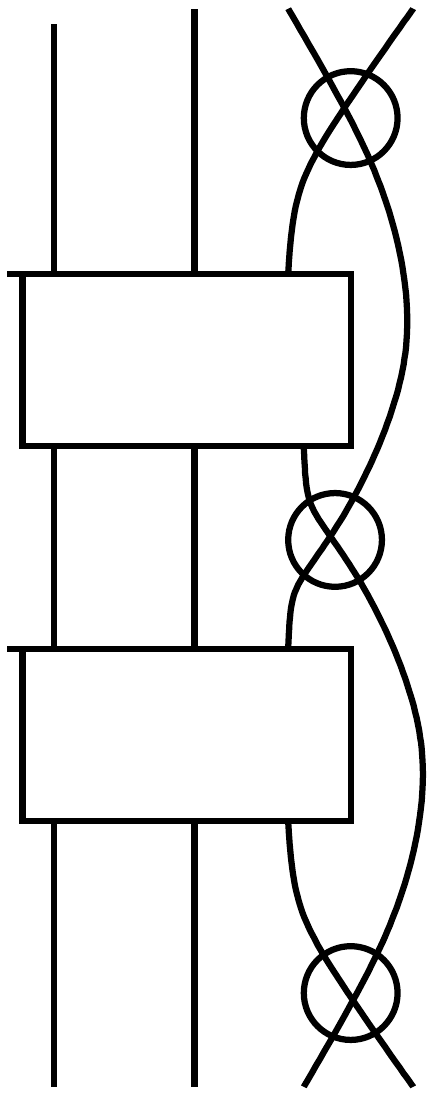}}
\hspace{.1in}
\underset{\text{$vL_v$-move}}{\overset{\text{right}}{\longleftrightarrow}}
\hspace{.1in}
\raisebox{-35pt}{\includegraphics[height=1in]{Rstab5-new}}
\]
\caption{Right virtual stabilization derived from $TL_v$-equivalence} \label{RightStab}
\end{figure}

Left and right algebraic under-threading follow from virtual conjugation and under-threaded virtual $L_v$-move. The case of the left algebraic under-threading is shown in Fig.~\ref{AlgUT}. The case of the right algebraic under-threading follows similarly.

\begin{figure}[ht]
\[
\reflectbox{\raisebox{-35pt}{\includegraphics[height=1.1in]{AlgUnThrdRight-new}}}
\hspace{.05in}
\underset{\text{conj.}}{\overset{\text{virtual}}{\longleftrightarrow}}
\hspace{.05in}
\reflectbox{\raisebox{-40pt}{\includegraphics[height=1.3in]{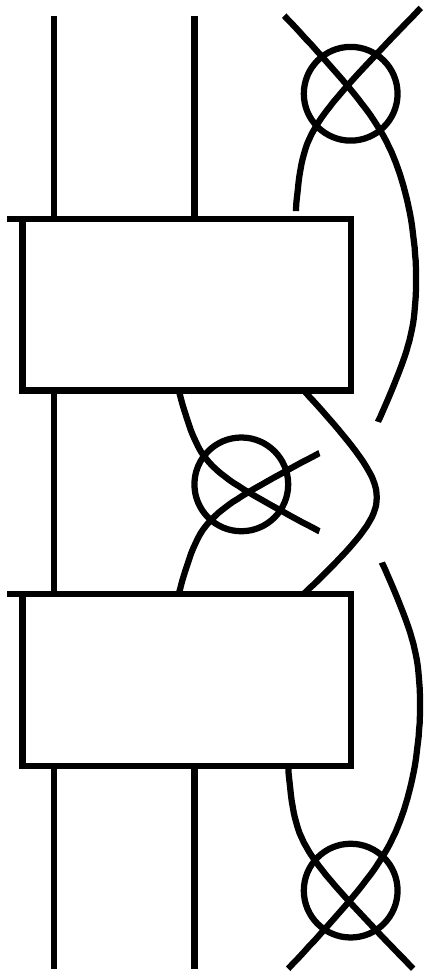}}}
\hspace{.05in}
\underset{\text{$vL_v$-move}}{\overset{\text{under-threaded}}{\longleftrightarrow}}
\hspace{.05in}
\raisebox{-35pt}{\includegraphics[height=1.1in]{Rstab5-new}}
\]
\caption{Left algebraic under-threading derived from $TL_v$-equivalence} \label{AlgUT}
\end{figure}

Because all steps in Fig.~\ref{AlgTL} may be applied in either direction, we can easily see (by following the proof starting from the center diagrams and reading to the right and to the left) that the algebraic right trivalent relation follows from the (geometric) right trivalent $L_v$-move, together with virtual conjugation and braid isotopy. Similarly, the steps in Fig.~\ref{fig:threadingproof} taken in a different order reveal that the algebraic left/right under-threading follows from the left/right under-threaded $TL_v$-move and other moves comprising the $TL_v$-equivalence. 

Therefore, algebraic equivalence follows from $TL_v$-equivalence. This completes the proof of Theorem~\ref{thm:AMarkov}.
\end{proof}

\textit{Final comments.} The theorems proved in this paper shed light on the relationship between virtual trivalent braids and virtual spatial trivalent graphs. The algebraic Markov-type theorem can be used to construct invariants for virtual spatial trivalent graphs.
\medskip

\textbf{Acknowledgments.}  We gratefully acknowledge support from the NSF Grant DMS--1460151 through the \textit{Research Experience for Undergraduates} (REU) Program. The first author was also partially supported by Simons Foundation collaboration grant $\#$ 355640. The second author received partial support from the NSF Grant HRD--1302873 through the \textit{California State University-Louis Stokes Alliance for Minority Participation} (CSU-LSAMP) Program. 


\end{document}